\documentclass[11pt]{article}

\usepackage{mathrsfs}

\usepackage[T1]{fontenc}
\usepackage[utf8]{inputenc}
\usepackage{lmodern}
\usepackage{amssymb,amsmath,amsthm}
\usepackage{bbm}
\usepackage{graphicx}
\usepackage{float}
\usepackage{xcolor}
\usepackage[colorlinks,linkcolor=black!70!black,citecolor=black]{hyperref} 
\usepackage[a4paper,margin=1.5cm]{geometry}
\usepackage{tikz}
\usepackage{tikz-cd}

\usetikzlibrary{decorations.pathreplacing, patterns,shapes,snakes}

\usepackage{pgfplots}

\long\def\metanote#1#2{{\color{#1}\
\ifmmode\hbox\fi{\sffamily\mdseries\upshape [#2]}\ }}

\newcommand{\ra}{\rightarrow}

\newcommand{\be}{\begin{equation}}
\newcommand{\ee}{\end{equation}}
\newcommand{\bi}{\begin{itemize}}
\newcommand{\ei}{\end{itemize}}

\newcommand{\commentout}[1]{}

\newcommand{\bft}{{\bf{t}}}
\newcommand{\bfi}{{\bf{i}}}

\newcommand{\conv}{{\text{conv}}}

\newcommand{\TV}{\text{TV}}

\newcommand{\inte}{\text{int}}

\newcommand{\Law}{{\mathcal{L}}}

\newcommand{\calM}{{\mathcal{M}}}

\newcommand{\osc}{{\text{osc}}}
\newcommand{\Leb}{{\text{Leb}}}

\newcommand{\esssup}{{ess\,sup}}
\newcommand{\essinf}{{ess\,inf}}
\newcommand{\essliminf}{{ess\,lim\,inf}}

\newcommand{\calA}{{\mathcal{A}}}
\newcommand{\calB}{{\mathcal{B}}}
\newcommand{\calC}{{\mathcal{C}}}

\newcommand{\calD}{{\mathcal{D}}}

\newcommand{\calO}{{\mathcal{O}}}

\newcommand{\calV}{{\mathcal{V}}}
\newcommand{\calP}{{\mathcal{P}}}

\newcommand{\calS}{{\mathcal{S}}}

\newcommand{\Nm}{{\mathbb{N}}}

\newcommand{\lv}{{\lvert}}
\newcommand{\rv}{{\rvert}}

\newcommand{\Rm}{{\mathbb R}}

\newcommand{\Pm}{{\mathbb P}}
\newcommand{\Zm}{{\mathbb Z}}
\newcommand{\Qm}{{\mathbb Q}}

\newcommand{\expE}{{\mathbb E}}
\newcommand{\Ind}{{\mathbbm{1}}}

\newtheorem{theo}{Theorem}[section]
\newtheorem{lem}[theo]{Lemma}
\newtheorem{defin}[theo]{Definition}

\newtheorem{prop}[theo]{Proposition}
\newtheorem{cor}[theo]{Corollary}
\newtheorem{rmk}[theo]{Remark}

\newtheorem{assum}[theo]{Assumption}
\newtheorem{prob}[theo]{Problem}

\newtheorem{observ}[theo]{Observation}
\newtheorem{fact}[theo]{Fact}

\newtheorem{Assumletter}{Assumption}[section]

\commentout{

\newcommand{\calA}{{\mathcal{A}}}
\newcommand{\calB}{{\mathcal{B}}}
\newcommand{\calC}{{\mathcal{C}}}
\newcommand{\calD}{{\mathcal{D}}}

\newcommand{\calO}{{\mathcal{O}}}

\newcommand{\calV}{{\mathcal{V}}}
\newcommand{\calP}{{\mathcal{P}}}
\newcommand{\calS}{{\mathcal{S}}}

\newcommand{\Nm}{{\mathbb{N}}}

\newtheorem*{rep@theorem}{\rep@title}
\newcommand{\newreptheorem}[2]{%
\newenvironment{rep#1}[1]{%
 \def\rep@title{#2 \ref{##1}}%
 \begin{rep@theorem}}%
 {\end{rep@theorem}}}
\makeatother
\newreptheorem{theorem}{Theorem}
\newreptheorem{lemma}{Lemma}
\newreptheorem{defin}{Definition}
\newreptheorem{cond}{Condition}
\newreptheorem{prop}{Proposition}
\newreptheorem{lem}{Lemma}
\newreptheorem{cor}{Corollary}
\newreptheorem{rmk}{Remark}
\newreptheorem{exer}{Exercise}
\newreptheorem{conj}{Conjecture}
\newreptheorem{assum}{Assumption}
\newreptheorem{equation}{Equation}
\newreptheorem{problem}{Problem}

\long\def\metanote#1#2{{\color{#1}\
\ifmmode\hbox\fi{\sffamily\mdseries\upshape [#2]}\ }}

}

\newtheorem*{rep@theorem}{\rep@title}
\newcommand{\newreptheorem}[2]{%
\newenvironment{rep#1}[1]{%
 \def\rep@title{#2 \ref{##1}}%
 \begin{rep@theorem}}%
 {\end{rep@theorem}}}
\makeatother
\newreptheorem{theorem}{Theorem}
\newreptheorem{lemma}{Lemma}
\newreptheorem{defin}{Definition}
\newreptheorem{cond}{Condition}
\newreptheorem{prop}{Proposition}
\newreptheorem{lem}{Lemma}
\newreptheorem{cor}{Corollary}
\newreptheorem{rmk}{Remark}
\newreptheorem{exer}{Exercise}
\newreptheorem{conj}{Conjecture}
\newreptheorem{assum}{Assumption}
\newreptheorem{equation}{Equation}

\long\def\metanote#1#2{{\color{#1}\
\ifmmode\hbox\fi{\sffamily\mdseries\upshape [#2]}\ }}

\begin{document}
\setcounter{page}{1}

\title{$L^{\infty}$-convergence to a quasi-stationary distribution}
\author{Oliver Tough\footnote{Institut de Mathématiques, Université de Neuchâtel, Switzerland. (oliver.kelsey@unine.ch)}}
\date{October 24, 2022}

\maketitle

\begin{abstract}
For general absorbed Markov processes $(X_t)_{0\leq t<\tau_{\partial}}$ having a quasi-stationary distribution (QSD) $\pi$ and absorption time $\tau_{\partial}$, we introduce a Dobrushin-type criterion providing for exponential convergence in $L^{\infty}(\pi)$ as $t\ra\infty$ of the density $\frac{d\Law_{\mu}(X_t\lvert \tau_{\partial}>t)}{d\pi}$. We establish this for all initial conditions $\mu$, possibly mutually singular with respect to $\pi$, under an additional ``anti-Dobrushin'' condition. This relies on inequalities we obtain comparing $\Law_{\mu}(X_t\lvert \tau_{\partial}>t)$ with the QSD $\pi$, uniformly over all initial conditions and over the whole space, under the aforementioned conditions. On a PDE level, these probabilistic criteria provide a parabolic boundary Harnack inequality (with an additional caveat) for the corresponding Kolmogorov forward equation. In addition to hypoelliptic settings, these comparison inequalities are thereby obtained in a setting where the corresponding Fokker-Planck equation is first order, with the possibility of discontinuous solutions. As a corollary, we obtain a sufficient condition for a submarkovian transition kernel to have a bounded, positive right eigenfunction, without requiring that any operator is compact. We apply the above to the following examples (with absorption): Markov processes on finite state spaces, degenerate diffusions satisfying parabolic H\"{o}rmander conditions, $1+1$-dimensional Langevin dynamics, random diffeomorphisms, $2$-dimensional neutron transport dynamics, and certain piecewise-deterministic Markov processes. In the last case, convergence to a QSD was previously unknown for any notion of convergence. Our proof is entirely different to earlier work, relying on consideration of the time-reversal of an absorbed Markov process.
\end{abstract}

\section{Introduction}\label{section:introduction}

The classical criteria providing for total variation convergence of a Markov process to its stationary distribution are complemented by a substantial literature providing for other notions of convergence. In particular, given a Markov process $X_t$ with stationary distribution $\pi$, there is a substantial literature providing for long-time convergence of the Radon-Nikodym derivative $\frac{d\Law_{\mu}(X_t)}{d\pi}$ under various metrics. We shall examine the analogous question for absorbed Markov processes $(X_t)_{0\leq t<\tau_{\partial}}$ having some quasi-stationary distribution (QSD), $\pi$. 

The investigation of the long-term behaviour of killed Markov processes conditioned on survival commenced with the work of Yaglom on subcritical Galton-Watson processes \cite{Yaglom1947}. In general, one seeks to establish convergence of the distribution conditioned on survival, $\Law_{\mu}(X_t\lvert \tau_{\partial}>t)$, to the corresponding quasi-stationary distribution (QSD), $\pi$. When the state space is finite, Darroch and Seneta established  convergence in total variation whenever the killed Markov process is irreducible prior to extinction \cite{Darroch1965,Darroch1967}. Much more recently, Champagnat and Villemonais have provided a general criterion for uniform exponential convergence in total variation for killed Markov processes conditioned on survival, \cite[Assumption (A)]{Champagnat2014}. Conditions analogous to \cite[Assumption (A)]{Champagnat2014}, but involving Lyapunov functions and tailored to the situation whereby the convergence of $\Law_{\mu}(X_t\lvert\tau_{\partial}>t)$ to $\pi$ is non-uniform, are provided in \cite{Champagnat2018a}. Whereas the approach in \cite{Champagnat2014,Champagnat2018a} is probabilistic, the most common approach has been to employ spectral arguments, for instance \cite{Pinsky1985,Gong1988,Cattiaux2009,Kolb2012,Benaim2021,Lelievre2021}. Other approaches have included renewal arguments \cite{Ferrari1995} and obtaining uniform (in the number of particles) controls on an associated particle system \cite{Cloez2016}. Our approach is entirely different. Our proof instead relies on considering the time-reversal at quasi-stationarity of an absorbed Markov process, which to the authors' knowledge has not previously been examined (whereas there is a large literature on the time-reversal of Markov processes without absorption).

Convergence to a QSD is typically established in total variation norm, but other notions of convergence have been considered. For example, setwise convergence to a QSD is established in \cite{Cattiaux2010}, whilst \cite{Cattiaux2009} establishes convergence in $L^2$ of a measure with respect to which the kernel is symmetric. There are also a handful of results providing for Wasserstein-type convergence, for example \cite{Ocafrain2021}. As far as the author is aware, however, these are not complemented by any results providing for convergence of the Radon-Nikodym derivative $\frac{d\Law_{\mu}(X_t\lvert \tau_{\partial}>t)}{d\pi}$, that is the density with respect to the quasi-stationary distribution, under any metric. 

In this paper, we will establish that the following Dobrushin-type condition, Assumption \ref{assum:adjoint Dobrushin main results section}, along with an additional assumption on the QSD $\pi$, provides for $L^{\infty}(\pi)$ convergence of the density $\frac{d\Law_{\mu}(X_t\lvert \tau_{\partial}>t)}{d\pi}$ (Theorem \ref{theo:Linfty convergence main results section} on Page \pageref{theo:Linfty convergence main results section}). Under an additional ``anti-Dobrushin'' type criterion, Assumption \ref{assum:adjoint anti-Dobrushin main results section} (on Page \pageref{assum:adjoint anti-Dobrushin main results section}), we shall establish this for arbitrary initial condition, possibly mutually singular with respect to $\pi$ (Theorem \ref{theo:uniform Linfty convergence main results section} on Page \pageref{theo:uniform Linfty convergence main results section}). This therefore provides a notion of convergence to a QSD substantially stronger than any previously considered in the literature. In addition to strengthening existing notions of convergence, the criteria we introduce in this paper shall also allow us establish convergence to a QSD where this was previously unknown for any notion of convergence (see Section \ref{section:PDMPs}, commencing on Page \pageref{section:PDMPs}).

We assume throughout that $\chi$ is a metric space equipped with its Borel $\sigma$-algebra, and that $\partial$ is a one-point set distinct from $\chi$. We consider, in discrete or continuous time, a killed Markov chain $(X_t)_{0\leq t<\tau_{\partial}}$ on the state space $\chi$, with cemetery state $\partial$ and absorption time $\tau_{\partial}:=\inf\{t\geq 0:X_t\in \partial\}$, after which $X_t$ remains in $\partial$. Corresponding to $(X_t)_{0\leq t<\tau_{\partial}}$ is the submarkovian transition semigroup $(P_t)_{t\geq 0}$. Throughout, we impose the following standing assumption.
\renewcommand{\theAssumletter}{S}
\begin{Assumletter}[Standing assumption]\label{assum:standing assumption}
The killed Markov chain $(X_t)_{0\leq t<\tau_{\partial}}$ has a (not necessarily unique) quasi-stationary distribution $\pi$.
\end{Assumletter} 

This paper is concerned with consequences of the following Dobrushin-type criterion upon the killed Markov process $(X_t)_{0\leq t<\tau_{\partial}}$ and quasi-stationary distribution $\pi$. 

We suppose that we have some distinguished $\sigma$-finite Borel measure on $\chi$ with full support, which we denote $\Lambda$ (typically this corresponds to Lebesgue measure). $\calB_{b,\gg}(\chi)$ is the set of positive Borel functions on $\chi$ which are both bounded and bounded away from $0$ (see \eqref{eq:Bounded and bded away from 0 borel and cts functions}).

In this paper, we shall consider the implications of the following assumption.
\renewcommand{\theAssumletter}{AD}
\begin{Assumletter}[Adjoint Dobrushin condition]\label{assum:adjoint Dobrushin main results section}
There exists $\psi\in \calB_{b,\gg}(\chi)$, a time $t_0>0$, a constant $a>0$ and a submarkovian kernel $\tilde{P}$ on $\chi$ such that
\begin{equation}\label{eq:psi adjoint for verifying reverse Dobrushin main results section}
\psi(x)\Lambda(dx)P_{t_0}(x,dy)=a\psi(y)\Lambda(dy)\tilde{P}(y,dx).
\end{equation}
We assume that $\tilde{P}1(y)>0$ for $\Lambda$-almost every $y\in \chi$.

We further assume that there exists $c_0'>0$ and $\nu\in\calP(\chi)$, the latter of which is not mutually singular with respect to $\pi$, such that
\begin{equation}\label{eq:Dobrushin for P tilde kernel crit for reverse dobrushin main results section}
\frac{\tilde{P}(y,\cdot)}{\tilde{P}1(y)}\geq c_0'\nu(\cdot)\quad \text{for $\Lambda$-almost every $y\in\chi$}.
\end{equation}
\end{Assumletter}

This is referred to as an ``adjoint Dobrushin condition'' as it is a Dobrushin condition on a kernel $\tilde{P}$, which may be thought of as adjoint to $P_{t_0}$.

We will establish in Theorem \ref{theo:Linfty convergence main results section} that Assumption \ref{assum:adjoint Dobrushin main results section}, along with the additional assumption that $\pi$ has an (essentially) bounded density with respect to $\Lambda$, provides for $L^{\infty}(\pi)$ convergence of the density $\frac{d\Law_{\mu}(X_t\lvert \tau_{\partial}>t)}{d\pi}$ when the initial condition $\mu$ has an (essentially) bounded density with respect to $\pi$. We then extend this in Theorem \ref{theo:uniform Linfty convergence main results section} to all initial conditions (which may be mutually singular with respect to $\pi$) under an additional ``adjoint anti-Dobrushin'' condition, Assumption \ref{assum:adjoint anti-Dobrushin main results section}. In particular, Theorem \ref{theo:uniform Linfty convergence main results section} shall provide conditions under which we have uniform (over all initial distributions) exponential convergence in $L^{\infty}(\pi)$ of $\frac{d\Law_{\mu}(X_t\lvert \tau_{\partial}>t)}{d\pi}$.

The latter extension is possible since we show in Theorem \ref{theo:dominated by pi theorem general results main results section} (on Page \pageref{theo:dominated by pi theorem general results main results section}) that the adjoint Dobrushin and adjoint anti-Dobrushin conditons, \ref{assum:adjoint Dobrushin main results section} and \ref{assum:adjoint anti-Dobrushin main results section}, combine to ensure that the distribution of $X_t$ at a given time is dominated by a multiple of a QSD $\pi$, uniformly over all initial conditions. We shall similarly establish in Theorem \ref{theo:DAD lower bounds density main results section} (on Page \pageref{theo:DAD lower bounds density main results section}) that the combination of a Dobrushin condition on $(X_t)_{0\leq t<\tau_{\partial}}$ itself with an adjoint Dobrushin condition, Assumption \ref{assum:combined Dobrushin adjoint Dobrushin main results section} (on Page \pageref{assum:combined Dobrushin adjoint Dobrushin main results section}), along with the assumption that $\pi$ has an (essentially) bounded density with respect to $\Lambda$, provides for the reverse inequality. To be more precise, Theorem \ref{theo:dominated by pi theorem general results main results section} and Theorem \ref{theo:DAD lower bounds density main results section} provide, for a given time $T>0$, explicit $0<c\leq C<\infty$ such that
\begin{equation}\label{eq:comparison in intro}
c\pi(\cdot)\leq \Pm_{\mu}(X_{T}\in \cdot\lvert\tau_{\partial}>{T})\leq C\pi(\cdot)\quad\text{for all initial conditions $\mu$.}
\end{equation}

This has a clear PDE interpretation, which we formally describe as follows. We suppose that we have an arbitrary solution of Kolmogorov's forward equation, $\partial_tu=L^{\ast}u$, corresponding to the distribution of $X_t$ for some initial condition. We also suppose that we have the principal eigenfunction $\rho$, a non-negative solution of $L^{\ast}\rho=-\lambda \rho$, corresponding to the QSD $\pi$. We can then phrase \eqref{eq:comparison in intro} as
\begin{equation}\label{eq:PDE interpretation of comparison inequalities}
\begin{split}
ce^{-\gamma(t-{T})}\Big(\int_{\chi}u(x',{T})dx'\Big) \rho(x)\leq u(x,{t})\leq Ce^{-\gamma(t-{T})}\Big(\int_{\chi}u(x',{T})dx'\Big)\rho(x) \quad\text{for all}\quad x\in\chi,\quad t\geq {T},
\end{split}
\end{equation}
for all such solutions $u$ of Kolmogorov's forward equation, with the constants $0<c\leq C<\infty$  not depending upon $u$. Since this inequality allows us to compare any such solution of Kolmogorov's forward equation with the principal eigenfunction $\rho$, it then allows us to compare any two such solutions of Kolmorogov's forward equation with each other. Moreover this comparison is valid up to the boundary. We therefore obtain a parabolic boundary Harnack inequality, with the added caveat that we may only compare solutions of Kolmogorov's forward equation corresponding globally to our given absorbed Markov process for some initial condition. In particular, for any such solutions $u_1$ and $u_2$ of Kolmogorov's forward equation, we then have that
\begin{equation}
\frac{\Big(\frac{u_1(x,t_1)}{u_2(x,t_2)}\Big)}{\Big(\frac{u_1(x',t_1)}{u_2(x',t_2)}\Big)}\geq \frac{c^2}{C^2}>0\quad\text{for all}\quad t_1,t_2\geq T,\quad x,x'\in \chi,
\end{equation}
where the constants $0<c\leq C<\infty$ depend only upon $T$ (in particular, they don't depend upon $u_1,u_2$). This comparison inequality is valid up to the boundary.

We will apply these criteria to obtain such parabolic boundary Harnack inequalities for: degenerate diffusions satisfying parabolic H\"ormander conditions (Theorem \ref{theo:comparison inequality for degenerate diffusions} on Page \pageref{theo:comparison inequality for degenerate diffusions}), $1+1$-dimensional Langevin dynamics (Theorem \ref{theo:comparison inequality for 1D Langevin} on Page \pageref{theo:comparison inequality for 1D Langevin}), and piecewise-deterministic Markov processes both in dimension $1$ and in arbitrary dimension with constant drifts (Theorem \ref{theo:bdy Harnack for PDMPs} on Page \pageref{theo:bdy Harnack for PDMPs}). Whilst the author is not aware of previously established boundary Harnack inequalities under H\"ormander-type conditions, the literature is rather large (the classical work of Bony \cite[Section 7]{Bony1969} provides interior Harnack inequalities under H\"ormander conditions). In contrast to the first two settings, which are hypoelliptic, in the latter setting the associated Kolmogorov forward equation is a system of first-order PDEs, with the possibility of discontinuous solutions for smooth initial conditions.

We then obtain the following two corollaries of Theorem \ref{theo:Linfty convergence main results section}. We firstly obtain in Theorem \ref{theo:criterion for cty of QSD} (on Page \pageref{theo:criterion for cty of QSD}) a criterion, Assumption \ref{assum:cty assum for cty cor}, providing for the continuity of the density of a QSD with respect to $\Lambda$. This is applied in Section \ref{section:PDMPs} to the QSDs of piecewise deterministic Markov processes in dimensions $1$ and $2$. Then in Theorem \ref{theo:criterion for right efn main results section} (on Page \pageref{theo:criterion for right efn main results section}) we shall establish a criterion, Assumption \ref{assum:assum for cor right efn main results section}, providing for the existence of a bounded, strictly positive right eigenfunction for $(P_{t})_{t\geq 0}$. In contrast to the Krein-Rutman theorem, this criterion does not require that any operator is compact. The application of this to absorbed Markov processes is as follows.

The aforementioned criterion for establishing uniform exponential convergence in total variation for killed Markov processes, \cite[Assumption A]{Champagnat2014}, consists of a Dobrushin-type condition \cite[Assumption (A1)]{Champagnat2014}, and a second condition which is usually more difficult to establish, \cite[Assumption (A2)]{Champagnat2014}. In fact, we shall see in Proposition \ref{prop:Lp malthusian by interpolation} that \cite[Assumption (A)]{Champagnat2014} also provides for non-uniform exponential convergence in $L^p(\pi)$ for all $1\leq p<\infty$, which is well-known in the context of Markov processes without killing (see, for instance, \cite[p.114]{Cattiaux2014}). 

A typical strategy for verifying \cite[Assumption (A2)]{Champagnat2014} is to establish the existence of a strictly positive, bounded right eigenfunction for the submarkovian transition semigroup $(P_t)_{0\leq t<\infty}$, which along with \cite[Assumption (A1)]{Champagnat2014} suffices to give \cite[Assumption (A2)]{Champagnat2014} (we will provide a more precise statement of this fact in Proposition \ref{prop:right efn gives A2}). This is then accomplished by applying the Krein-Rutman theorem to an appropriate positive compact operator. If the process is sufficiently degenerate, however, the Krein-Rutman theorem is not available as one does not have operator compactness, so an alternative criterion is needed. Assumption \ref{assum:assum for cor right efn main results section} provides such an alternative.

Consequentially, in Section \ref{section:PDMPs} we shall be able to verify \cite[Assumption (A)]{Champagnat2014} for certain piecewise-deterministic Markov processes (PDMPs) (in dimension at most $2$ with arbitrary drifts or any dimension with constant drifts), for which it was not previously known that they satisfy \cite[Assumption (A2)]{Champagnat2014} (nor was convergence to a quasi-stationary distribution known, for any notion of convergence).

If we have \cite[Assumption (A)]{Champagnat2014}, then we can define the so-called ``$Q$-process'' - the process $(X_t)_{0\leq t<\tau_{\partial}}$ conditioned \textit{never} to be killed. The reader is directed towards \cite[Theorem 3.1]{Champagnat2014} for a more precise definition. Corollary \ref{cor:L infty convergence of Q-process main results section} (on Page \pageref{cor:L infty convergence of Q-process main results section}) provides for uniform exponential convergence in $L^{\infty}$ of the density of the $Q$-process with respect to its stationary distribution, under the same conditions as Theorem \ref{theo:uniform Linfty convergence main results section}. This follows from the observation that an absorbed Markov process at quasi-stationarity and the corresponding $Q$-process at stationarity have the same time-reversal (Observation \ref{observation: Reverse dobrushin implies reverse dobrushin for Q-process} on Page \pageref{observation: Reverse dobrushin implies reverse dobrushin for Q-process}).

We summarise our results as follows:
\begin{enumerate}
\item\label{enum:summary of results conv for bded ic}
Theorem \ref{theo:Linfty convergence main results section} (on Page \pageref{theo:Linfty convergence main results section}) provides for (non-uniform) exponential convergence in $L^{\infty}(\pi)$ for initial conditions having a bounded density with respect to $\pi$, under the adjoint Dobrushin condition, Assumption \ref{assum:adjoint Dobrushin main results section}, and the assumption that $\pi$ has an (essentially) bounded density with respect to $\Lambda$.
\item\label{enum:summary of results conv for arbitrary ic}
Theorem \ref{theo:uniform Linfty convergence main results section} (on Page \pageref{theo:uniform Linfty convergence main results section}) allows us to extend this to all possible initial conditions (which may be mutually singular with respect to $\pi$), under the additional assumption of the adjoint anti-Dobrushin condition, Assumption \ref{assum:adjoint anti-Dobrushin main results section}. If we also have \cite[Assumption (A)]{Champagnat2014} or Assumption \ref{assum:combined Dobrushin adjoint Dobrushin main results section} (the ``combined Dobrushin and adjoint Dobrushin condition''), the convergence becomes uniform over all initial conditions.
\item\label{enum:summary of results dominated by pi}
Theorems \ref{theo:dominated by pi theorem general results main results section} (on Page \pageref{theo:dominated by pi theorem general results main results section}) and \ref{theo:DAD lower bounds density main results section} (on Page \pageref{theo:DAD lower bounds density main results section}) provide Dobrushin and ``anti-Dobrushin'' type criteria allowing us to compare the distribution of an absorbed Markov process at a given time with its QSD, over the whole space and uniformly over all initial conditions. On a PDE level, these criteria provide for a parabolic boundary Harnack inequality, with the added caveat that they only allow us to compare solutions of Kolmogorov's forward equation which correspond globally to our given absorbed Markov process for some initial condition.
\item\label{enum:summary of results cts QSD}
Theorem \ref{theo:criterion for cty of QSD} (on Page \pageref{theo:criterion for cty of QSD}) provides a sufficient condition for establishing the continuity of the density of a QSD with respect to the distinguished measure $\Lambda$ (this typically being Lebesgue measure).
\item\label{enum:summary of results right eigenfunction}
Theorem \ref{theo:criterion for right efn main results section} (on Page \pageref{theo:criterion for right efn main results section}) provides a sufficient condition for the existence of a strictly positive, bounded right eigenfunction for $(P_t)_{t\geq 0}$, without requiring any operator to be compact (as is required by the Krein-Rutman theorem).
\item
Corollary \ref{cor:L infty convergence of Q-process main results section} (on Page \pageref{cor:L infty convergence of Q-process main results section}) provides for $L^{\infty}$-type convergence of the $Q$-process to its stationary distribution, analogously to theorems \ref{theo:Linfty convergence main results section} and \ref{theo:uniform Linfty convergence main results section}, and under the same conditions. 
\end{enumerate}

We apply the above to the following examples: absorbed Markov processes on finite state spaces, degenerate diffusions satisfying parabolic H\"{o}rmander conditions, killed at the boundary of their domain, $1+1$-dimensional Langevin dynamics killed at the boundary of their domain, absorbed random diffeomorphisms, $2$-dimensional neutron transport processes absorbed at the boundary of their domain, and piecewise-deterministic Markov processes (PDMPs) absorbed at the boundary of their domain, in both dimension at most $2$ and in arbitrary dimension with constant drifts.

Whereas the assumptions of the above theorems are stated in terms of the adjoint kernel $\tilde{P}$ with respect to an unspecified distinguished measure $\Lambda$ (typically corresponding to Lebesgue measure), the proofs of our results shall hinge on consideration of a special adjoint: the adjoint with respect to our given QSD $\pi$. This adjoint, which we denote as $R$, may be thought of as the time-reversal at quasi-stationarity of our absorbed Markov process. For this reason, we shall often refer to it as the ``reverse kernel''. For some given time $t_0>0$, it is a solution of
\[
\pi(dx)P_{t_0}(x,dy)=\pi(dy)R(y,dx).
\]
It has the following special property on which our results hinge: it corresponds to a submarkovian kernel without absorption (after rescaling). The adjoint kernel $\tilde{P}$ in \eqref{eq:psi adjoint for verifying reverse Dobrushin main results section} does not, in general, share this property. Whereas there exists an enormous literature on the time-reversal of Markov processes without absorpotion, to the authors' knowledge the time-reversal of absorbed Markov processes has not previously been considered. In Section \ref{section:Reverse Dobrushin condition}, we shall consider the implications of imposing Dobrushin-type criteria upon $R$.

Since we typically do not have an explicit expression for $\pi$, we typically do not have an explicit expression for $R$. On the other hand, for a reasonable choice of distinguished measure $\Lambda$, we will often have an explicit expression for $\tilde{P}$ (often $\Lambda$ and $\psi$ may be chosen so that $\tilde{P}$ corresponds to a process of the same form as our original process). Moreover, we can compare the adjoint kernel $\tilde{P}$ and reverse kernel $R$ under various assumptions - see Theorem \ref{theo:theo for reverse Dobrushin in Euclidean space} (on Page \pageref{theo:theo for reverse Dobrushin in Euclidean space}) and Theorem \ref{theo:theorem for bounded by pi in Euclidean space} (on Page \pageref{theo:theorem for bounded by pi in Euclidean space}). We will thereby be able to transfer the results of Section \ref{section:Reverse Dobrushin condition} - which require conditions imposed upon the reverse kernel $R$ - to results requiring more tractable conditions upon the adjoint kernel $\tilde{P}$. These are our main results, which may be found in the following section.

\subsection*{Structure of the paper}

In the following section we shall state the main results of this paper, beginning with the necessary notation, definitions and background. In Section \ref{section:Reverse Dobrushin condition}, we shall introduce the ``reverse Dobrushin condition'', Assumption \ref{assum:Dobrushin reverse time}, and related conditions. We state in this section the results these assumptions provide for, before proving these results in Section \ref{section:general state space proof}. We then use the results of Section \ref{section:Reverse Dobrushin condition} to prove our main results, those of Section \ref{section:main results}, in Section \ref{section:Euclidean state space proof}. The sections which then follow, \ref{section:finite state space}-\ref{section:PDMPs}, are dedicated to the various examples. We finally collect the proofs of various technical propositions and lemmas in the appendix.

\section{Statement of results}\label{section:main results}

\subsection*{Notation, definitions and background}

The space of functions to which the path $(X_t)_{0\leq t<\infty}$ belongs will play no role in our analysis, so will not be specified. In particular, these results apply in both discrete and continuous time. 

Given a QSD $\pi$, we write $\lambda(\pi)$ for the corresponding eigenvalue over time $1$:
\begin{equation}\label{eq:eigenvalue corresponding to QSD}
\lambda(\pi):=\Pm_{\pi}(\tau_{\partial}>1)>0.
\end{equation}

We recall that the state space $\chi$ is assumed to be a metric space, equipped with its Borel $\sigma$-algebra $\mathscr{B}(\chi)$, and with the cemetery set $\partial$ a one-point set distinct from $\chi$. We also write $\calB(\chi)$ and $\calB_b(\chi)$ for the set of Borel-measurable real-valued functions and bounded Borel-measurable real-valued functions on $\chi$ respectively. We equip $\calB_b(\chi)$ with the supremum norm, denoted $\lvert\lvert \cdot\rvert\rvert_{\infty}$. We then have the semigroup $(P_t)_{t\geq 0}$ on $\calB_b(\chi)$ defined by
\begin{equation}\label{eq:Pt semigroup bounded Borel measurable}
P_t:\calB_b(\chi)\ni f\mapsto (x\mapsto \expE_x[f(X_t)\Ind(\tau_{\partial}>t)])\in \calB_b(\chi).
\end{equation}

We further define
\begin{equation}\label{eq:Bounded and bded away from 0 borel and cts functions}
\calB_{b,\gg}(\chi):=\{f\in \calB_b(\chi):\inf_{x\in \chi}f(x)>0\},\quad C_{b,\gg}(\chi):=\{f\in C_b(\chi):\inf_{x\in \chi}f(x)>0\}.
\end{equation}

We define $\calP(\chi)$, $\calM_{\geq 0}(\chi)$, $\calM(\chi)$ and $\calM_{\sigma}(\chi)$ to be the space of Borel probability measures; bounded, non-negative Borel measures; bounded signed Borel measures and $\sigma$-finite Borel measures on $\chi$ respectively. 

We recall that we assume there to be some distinguished $\sigma$-finite measure on $\chi$ with full support, which we denote $\Lambda$. Later, in sections \ref{section:Reverse Dobrushin condition} and \ref{section:general state space proof}, we shall no longer assume there to exist some distinguished measure $\Lambda$. In our examples, $\Lambda$ will typically correspond to Lebesgue measure. 

Given our distinguished measure $\Lambda$, a kernel $P$, some choice of $\psi\in \calB_{b,\gg}(\chi)$ (we shall sometimes specify $\psi\in \calC_{b,\gg}(\chi)$), and some $a>0$, we are therefore able to define an adjoint kernel as a submarkovian kernel $\tilde{P}$ such that
\begin{equation}
\psi(x)\Lambda(dx)P(x,dy)=a\psi(y)\Lambda(dy)\tilde{P}(y,dx).
\end{equation}

We note that the role of $a>0$ is simply to ensure that $\tilde{P}$ is submarkovian, so can be thought of as corresponding to an absorbed Markov process.

We write $\Pm_x(A)$ for the probability of the event $A$ given $X_0=x$, whilst $\Pm_{\mu}(A)$ is the probability of the event $A$ given the initial distribution $X_0\sim \mu$. This is well-defined for all $\mu\in \calM(\chi)$ by
\[
\Pm_{\mu}(A)=\int_{\chi}\Pm_x(A)\mu(dx).
\]
Given a sub-probability measure $\mu\in \calM_{\geq 0}(\chi)$, we say that a random variable is distributed like $\mu$, $X\sim \mu$, if $\Pm(X\in A)=\mu(A)$ for all $A\in \calB(\chi)$. In particular, this specifies $\Pm(X\in \partial)=1-\mu(\chi)$.

\begin{defin}\label{defin:lower semicts kernel}
We recall that a function $f$ on $\chi$ is lower semicontinuous if $f(x_0)\leq \liminf_{x\ra x_0}f(x)$ for all $x_0\in \chi$. We define $LC_b(\chi;\Rm_{\geq 0})$ to be the set of all bounded, non-negative lower semicontinuous functions on $\chi$. We say that a submarkovian kernel $P$ is lower semicontinuous if
\begin{equation}\label{eq:lower semicty of kernel}
Pf\in LC_b(\chi;\Rm_{\geq 0})\quad\text{for all}\quad f\in C_b(\chi;\Rm_{\geq 0}).
\end{equation}
This is equivalent to $Pf\in LC_b(\chi;\Rm_{\geq 0})$ for all $f\in LC_b(\chi;\Rm_{\geq 0})$, which may be seen by using the fact that a function is lower semicontinuous if and only if it is the pointwise limit of a non-decreasing sequence of continuous functions, and applying the monotone convergence theorem. We say that a submarkovian transition semigroup $(P_t)_{t\geq 0}$ is lower semicontinuous if $P_t$ is lower semicontinuous for all $t\geq 0$. We say that an absorbed Markov process $(X_t)_{0\leq t<\tau_{\partial}}$ is lower semicontinuous if its associated submarkovian transition semigroup $(P_t)_{t\geq 0}$ is lower semicontinuous.
\end{defin}

We write $\lvert\lvert \cdot\rvert\rvert_{\TV}$ for the total variation norm on $\calM(\chi)$ (or on any subset thereof). For any linear map $T:X\ra Y$ between normed spaces, we write $\lvert\lvert T\rvert\rvert_{\text{op}}$ for the operator norm if the identities of $X$ and $Y$ are unambiguous, writing instead $\lvert\lvert T\rvert\rvert_{X\ra Y}$ if necessary to avoid ambiguity.

For any $\mu\in\calM_{\sigma}(\chi)\setminus \{0\}$ and $1\leq p\leq \infty$, $L^p(\mu)$ is the usual $L^p$-space while $L^p_{\geq 0}(\mu)$ and $L^p_{>0}(\mu)$ are $L^p(\mu)$ restricted to $\mu$-almost everywhere non-negative (respectively strictly positive) functions. 

Given measures $\nu\in \calM(\chi)\cup \calM_{\sigma}(\chi)$ and $\mu\in\calM_{\sigma}(\chi)\setminus \{0\}$ we write $\nu\ll \mu$ if $\nu$ is absolutely continuous with respect to $\mu$. For $1\leq p\leq \infty$ we write $\nu\ll_{L^{p}}\mu$ if, in addition, the Radon-Nikodym derivative belongs to $L^p(\mu)$, $\frac{d\nu}{d\mu}\in L^{p}(\mu)$ (note that the $p=1$ case is automatic if $\nu\in \calM(\chi)$). For any $\mu\in\calM_{\sigma}(\chi)\setminus \{0\}$, we define $\calP(\mu)$, $\calM_{\geq 0}(\mu)$, $\calM(\mu)$ and $\calM_{\sigma}(\chi)$ to be those measures belonging to $\calP(\chi)$, $\calM_{\geq 0}(\chi)$, $\calM(\chi)$ and $\calM_{\sigma}(\chi)$ respectively which are absolutely continuous with respect to $\mu$. We further define 
\[
\calP_p(\mu):=\{\nu\in\calP(\mu):\nu\ll_{L^p} \mu\}\quad\text{and}\quad\calM_p(\mu):=\{\nu\in\calM(\mu):\nu\ll_{L^p} \mu\}.
\]
For $\mu\in\calM_{\sigma}(\chi)\setminus \{0\}$ and $f\in L^{\infty}(\mu)$, we write $\osc_{\mu}(f)$ for the essential oscillation
\[
\osc_{\mu}(f)=\esssup_{\mu}(f)-\essinf_{\mu}(f).
\]
For $\nu\ll_{\infty} \mu$ we define \[
\osc_{\mu}(\nu):=\osc_{\mu}\Big(\frac{d\nu}{d\mu}\Big).
\]
For $\mu\in\calM_{\sigma}(\chi)\setminus\{0\}$ and $f\in L^1(\mu)$ we write $f\mu$ for the unique measure $\nu\in\calM(\mu)$ such that $\frac{d\nu}{d\mu}=f$, that is we define
\begin{equation}\label{eq:f mu notation}
(f\mu)(A):=\int_Af(x)\mu(dx),\quad A\in\mathscr{B}(\chi).
\end{equation}

The following simple proposition, proven in the appendix, shall be used implicitly throughout.
\begin{prop}\label{prop:density wrt pi-> density}
Assume that $\pi$ is a quasi-stationary distribution for $(X_t)_{0\leq t<\tau_{\partial}}$. If $\mu\ll \pi$, then $\mu P_t\ll \pi$ while if $\mu\ll_{\infty} \pi$, then $\mu P_t\ll_{\infty} \pi$, for all $0\leq t<\infty$. 
\end{prop}

The semigroup $(P_t)_{t\geq 0}$ is typically defined on $\calB_b(\chi)$, as in \eqref{eq:Pt semigroup bounded Borel measurable}. However in Theorem \ref{theo:Linfty convergence main results section} we shall obtain a non-negative right eigenfunction of $P_t$ belonging to $L^1(\pi)$, which makes sense only if $P_t$ can be defined on $L^1(\pi)$. In the following proposition, which shall be proven in the appendix, we establish that $(P_t)_{t\geq 0}$ may indeed be defined on $L^1(\pi)$. 

\begin{prop}\label{prop:Pt well-defined on L1}
For all $f\in\calB(\pi)$ and $0\leq t<\infty$, $P_t\lvert f\rvert(x)<\infty$ for $\pi$-almost every $x\in\chi$, with $(x\mapsto \Ind_{P_t\lvert f\rvert(x)<\infty}P_tf(x))\in\calB(\pi)$. Furthermore for $f,g\in\calB(\pi)$ and $0\leq t<\infty$, if $f=g$ $\pi$-almost everywhere then $P_tf=P_tg$ $\pi$-almost everywhere, so that
\begin{equation}
P_t:L^1(\pi)\ni f\mapsto (x\mapsto \expE_x[f(X_t)\Ind(\tau_{\partial}>t)])\in L^1(\pi)
\end{equation}
is a well-defined map for all $0\leq t<\infty$. Moreover we have that $(\mu P_t)(f)=\mu(P_tf)$ for all $\mu\in\calP_{\infty}(\pi)$, $f\in L^1(\pi)$ and $t\geq 0$, so that $\mu P_tf$ is well-defined. Furthermore, $(P_t)_{t\geq 0}$ is a semigroup of bounded linear maps on $L^1(\pi)$, with $\lvert\lvert P_t\rvert\rvert_{\text{op}}=\lambda^t$ for all $t\geq 0$, such that $P_t(L^1_{\geq 0}(\pi))\subseteq L^1_{\geq 0}(\pi)$ for all $t\geq 0$.
\end{prop}

We shall abuse notation by writing $(P_t)_{t\geq 0}$ for both the semigroup defined on $\calB_b(\chi)$ by \eqref{eq:Pt semigroup bounded Borel measurable} and the semigroup defined on $L^1(\pi)$ by Proposition \ref{prop:Pt well-defined on L1}.

Since elements of $L^1(\pi)$ are defined only up to $\pi$-null sets, whereas elements of $\calB_b(\chi)$ which differ on a $\pi$-null set are considered distinct, we have the following different notions of an eigenfunction for $P_t$. 
\begin{defin}[Pointwise and $L^1(\pi)$ right eigenfunctions]
Fix $t\geq 0$. A pointwise eigenfunction for $P_t$ is defined to be some $h\in \calB_b(\chi)$ such that, for some $c\in \Rm$, $P_th(x)=ch(x)$ for all $x\in \chi$. Given a QSD $\pi$ for $(X_t)_{t\geq 0}$, an $L^1(\pi)$-right eigenfunction is some $\phi\in L^1(\pi)$ such that, for some $c\in \Rm$, $P_t\phi=c\phi$ in the sense of $L^1(\pi)$. We say that $h$ (respectively $\phi$) is a pointwise (respectively $L^1(\pi)$) right eigenfunction for $(P_t)_{t\geq 0}$ if it is a pointwise (respectively $L^1(\pi)$) right eigenfunction for $P_t$, for all $t\geq 0$.
\end{defin}

\begin{rmk}
Throughout this paper, for clarity we shall use $\phi$ to denote $L^1(\pi)$ right eigenfunctions and $h$ to denote pointwise right eigenfunctions.
\end{rmk}

Note that if $\phi\in L^{\infty}(\pi)$ is an $L^1(\pi)$-eigenfunction for $P_t$ with eigenvalue $c$, and $h$ is a bounded version of $\phi$, it does not follow that $h$ is a pointwise eigenfunction for $P_t$, as $P_th$ and $ch$ may differ on a $\pi$-null but non-empty set.

The results of \cite{Champagnat2014} shall be used often in this paper. In addition to \cite[Assumption (A)]{Champagnat2014}, they also assume the following technical assumption, which can be found on \cite[Page 244]{Champagnat2014}.
\renewcommand{\theAssumletter}{TA}
\begin{Assumletter}\label{assum:technical assumption for Assum (A)}[Technical condition required by \cite[Assumption (A)]{Champagnat2014}, found on Page 244 of \cite{Champagnat2014}]
For every $x\in \chi$ and time $t<\infty$ we have $\Pm_x(\tau_{\partial}>t)>0$ and $\Pm_x(\tau_{\partial}<\infty)=1$.
\end{Assumletter}
We note that the first requirement is necessary for \cite[Assumption (A1)]{Champagnat2014} to be a well-defined condition. On the other hand, the second requirement can be weakened according to the following remark.
\begin{rmk}\label{rmk:remark for checking technical assum for Assum (A)}
We suppose that for every $x\in \chi$ and time $t<\infty$ we have $\Pm_x(\tau_{\partial}>t)>0$ and $\Pm_x(\tau_{\partial}<\infty)>0$. We assume, in addition, that $(X_t)_{0\leq t<\tau_{\partial}}$ satisfies \cite[Assumption (A1)]{Champagnat2014}. Then it may readily be checked that $\Pm_x(\tau_{\partial}<\infty)=1$.
\end{rmk}

Whilst Assumption \ref{assum:technical assumption for Assum (A)} shall be imposed at times in this paper, it shall not be imposed throughout.

The following well-known proposition provides a common strategy for verifying \cite[Assumption (A2)]{Champagnat2014}. We will supply a proof in the appendix.

\begin{prop}\label{prop:right efn gives A2}
Suppose that $(X_t)_{0\leq t<\tau_{\partial}}$ satisfies Assumption \ref{assum:technical assumption for Assum (A)} and \cite[Assumption (A1)]{Champagnat2014}, and that for some time $t_1>0$ there exists a pointwise right eigenfunction for $P_{t_1}:\calB_b(\chi)\ra \calB_b(\chi)$ belonging to $\calB_b(\chi;\Rm_{>0})$ (in particular, bounded and everywhere strictly positive). Then $(X_t)_{0\leq t<\tau_{\partial}}$ satisfies \cite[Assumption (A2)]{Champagnat2014} and hence \cite[Assumption A]{Champagnat2014}. 
\end{prop}

Conversely, we suppose that $(X_t)_{0\leq t<\tau_{\partial}}$ satisfies Assumption \ref{assum:technical assumption for Assum (A)} and \cite[Assumption A]{Champagnat2014}, so there exists a unique QSD $\pi$ by \cite[Theorem 1.1]{Champagnat2014}. We write $\lambda:=\lambda(\pi)>0=\Pm_{\pi}(\tau_{\partial}>1)$. We then have by \cite[Proposition 2.3]{Champagnat2014} that there exists $h\in \calB_b(\chi;\Rm_{> 0})$ which for all $t>0$ is an everywhere strictly positive pointwise right eigenfunction for $P_t$ of eigenvalue $\lambda^t$.

Combining \cite[Theorem 2.1]{Champagnat2014} and \cite[Theorem 2.1]{Champagnat2017}, we see that \cite[Assumption A]{Champagnat2014} implies that there exists $C<\infty$ and $\gamma>0$ ($\gamma>0$ being the constant given in the statement of \cite[Theorem 2.1]{Champagnat2014}) such that
\begin{equation}\label{eq:total variation Malthusian behaviour defin}
\sup_{\mu\in\calP(\chi)}\Big\lvert\Big\lvert\lambda^{-t}\Pm_{\mu}(X_t\in\cdot)-\mu(h)\pi(\cdot)\Big\rvert\Big\rvert_{\TV}\leq Ce^{-\gamma t}\quad\text{for all}\quad 0\leq t<\infty,
\end{equation}
where $h$ is the pointwise right eigenfunction whose existence is provided for by \cite[Proposition 2.3]{Champagnat2014}. This is what is typically referred to as ``Perron-Frobenius behaviour''. We refer to it as ``total variation-Perron-Frobenius'' behaviour to prevent confusion with the following.

We suppose that $\pi$ is a QSD of $(X_t)_{0\leq t<\tau_{\partial}}$, and denote $\lambda:=\lambda(\pi)$. We say that $(X_t)_{0\leq t<\tau_{\partial}}$ exhibits ``$L^{\infty}(\pi)$-Perron-Frobenius behaviour'' if there exists an $L^1(\pi)$-right eigenfunction $\phi$ for $(P_t)_{t\geq 0}$ such that, for some $C<\infty$ and $\gamma>0$, we have 
\begin{equation}\label{eq:Linfty malthusian defin}
\Big\lvert\Big\lvert \lambda^{-t}\frac{d\Pm_{\mu}(X_t\in \cdot)}{d\pi(\cdot)}-\mu(\phi)\Big\rvert\Big\rvert_{L^{\infty}(\pi)}\leq Ce^{-\gamma t}\Big\lvert\Big\lvert \frac{d\mu}{d\pi}\Big\rvert\Big\rvert_{L^{\infty}(\pi)}\quad\text{for all}\quad t\geq 0\quad\text{and}\quad \mu\in \calP_{\infty}(\pi).
\end{equation}

In \cite[Theorem 3.1]{Champagnat2014}, they established that \cite[Assumption (A)]{Champagnat2014} (along with Assumption \ref{assum:technical assumption for Assum (A)}) provides for the existence of the ``$Q$-process'' - the process $(X_t)_{0\leq t<\tau_{\partial}}$ conditioned never to be killed. A precise definition is given in \cite[Theorem 3.1]{Champagnat2014}. We shall call this $Q$-process $(Z_t)_{0\leq t<\infty}$. It corresponds to the limit
\[
\Law_x(Z_t)=\lim_{s\ra \infty}\Law_x(X_t\lvert \tau_{\partial}>t+s)\quad\text{for all}\quad t\geq 0,\quad x\in \chi.
\]

We write $\pi$ for the QSD of $(X_t)_{0\leq t<\tau_{\partial}}$, $\lambda:=\lambda(\pi)=\Pm_{\pi}(\tau_{\partial}>1)$, and $h$ for the strictly positive, bounded, pointwise right eigenfunction provided for by \cite[Proposition 2.3]{Champagnat2014}, normalised so that $\pi(h)=1$. The $Q$-process, which we shall call $(Z_t)_{0\leq t<\infty}$, is then an exponentially ergodic time-homogeneous Markov process on $\chi$ with Markovian transition kernel
\begin{equation}\label{eq:Q process Markov kernel}
Q_t(x,dy)=\frac{h(y)\lambda^{-t}}{h(x)}P_t(x,dy)
\end{equation}
and stationary distribution
\begin{equation}\label{eq:stationary dist Q-process}
\beta(dx)=h(x)\pi(dx).
\end{equation}

\subsection*{$L^{\infty}(\pi)$ convergence for $L^{\infty}(\pi)$ initial condition}

Our first theorem, Theorem \ref{theo:Linfty convergence main results section}, shall require the adjoint Dobrushin condition, Assumption \ref{assum:adjoint Dobrushin main results section} (found on page \pageref{assum:adjoint Dobrushin main results section}). If $(X_t)_{0\leq t<\tau_{\partial}}$ satisfies Assumption \ref{assum:adjoint Dobrushin main results section} and $\pi\in\calP_{\infty}(\Lambda)$, then we may also consider the following assumption.
\begin{assum}\label{Assum:assum for essentially positive right efn main results section}
We have that $(X_t)_{0\leq t<\tau_{\partial}}$ satisfies Assumption \ref{assum:adjoint Dobrushin main results section} and $\pi\in\calP_{\infty}(\Lambda)$, with $\nu\in\calP(\chi)$ being the probability measure assumed to satisfy \eqref{eq:Dobrushin for P tilde kernel crit for reverse dobrushin main results section}. We assume that for all $\mu\in\calP_{\infty}(\pi)$ there exists $t=t(\mu)<\infty$ (dependent upon $\mu$) such that $\mu P_t\big(\frac{d\nu}{d\Lambda}\big)>0$ (this is well-defined since $\mu P_t\in \calP_{\infty}(\pi)\subseteq \calP_{\infty}(\Lambda)$ for all $\mu \in\calP_{\infty}(\pi)$, and $\nu\ll\Lambda$ necessarily).
\end{assum}

\begin{theo}\label{theo:Linfty convergence main results section}
Suppose that the killed Markov process $(X_t)_{0\leq t<\tau_{\partial}}$ has a QSD $\pi$ which has an essentially bounded density with respect to $\Lambda$, $\pi\in\calP_{\infty}(\Lambda)$. We further assume that $(X_t)_{0\leq t<\tau_{\partial}}$ satisfies Assumption \ref{assum:adjoint Dobrushin main results section}. Then we have the following.

The constant $c_0'>0$, time $t_0>0$ and probability measure $\nu$ are those given by Assumption \ref{assum:adjoint Dobrushin main results section}, while $\lambda:=\lambda(\pi)$. We have that $\nu\in \calP(\Lambda)$ so that the following constant is unambiguous
\begin{equation}\label{eq:formula for c0 main results section}
c_0:=\frac{c_0'\nu(\frac{d\pi}{d\Lambda})}{\lvert\lvert \psi\rvert\rvert_{\infty}\lvert\lvert \frac{1}{\psi}\rvert\rvert_{\infty}\lvert\lvert \frac{d\pi}{d\Lambda}\rvert\rvert_{L^{\infty}(\Lambda)}}\in (0,1].
\end{equation}

Then there exists $\phi\in L^1_{\geq 0}(\pi)$ with $\lvert\lvert  \phi\rvert\rvert_{L^1(\pi)}=1$ such that $P_t\phi=\lambda^{t}\phi$ for all $0\leq t<\infty$. For all $0\leq t<\infty$, $\phi$ is both the unique non-negative $L^1(\pi)$-right eigenfunction of $P_t$ and the unique $L^1(\pi)$-right eigenfunction of eigenvalue $\lambda^t$, up to rescaling. Moreover we have the following $L^{\infty}$-Perron-Frobenius behaviour,
\begin{equation}\label{eq:Linfty malthusian theo statement main results section}
\Big\lvert\Big\lvert \lambda^{-t}\frac{d\Pm_{\mu}(X_t\in \cdot)}{d\pi(\cdot)}-\mu(\phi)\Big\rvert\Big\rvert_{L^{\infty}(\pi)}\leq (1-c_0)^{\lfloor \frac{t}{t_0}\rfloor} \osc_{\pi}(\mu),\quad 0\leq t<\infty,\quad \mu\in\calP_{\infty}(\pi).
\end{equation}

Consequentially we have for all $\mu\in\calP_{\infty}(\pi)$:
\begin{align}\label{eq:convergence of prob of killing main results section}
\lvert\lambda^{-t}\Pm_{\mu}(\tau_{\partial}>t)-\mu(\phi)\rvert&\leq (1-c_0)^{\lfloor \frac{t}{t_0}\rfloor}\osc_{\pi}(\mu),\quad 0\leq t<\infty,\\
\Big\lvert\Big\lvert \frac{d\Law_{\mu}(X_t\lvert \tau_{\partial}>t)}{d\pi}-1\Big\rvert\Big\rvert_{L^{\infty}(\pi)}&\leq \frac{2(1-c_0)^{\lfloor \frac{t}{t_0}\rfloor} \osc_{\pi}(\mu)}{\mu(\phi)-(1-c_0)^{\lfloor \frac{t}{t_0}\rfloor} \osc_{\pi}(\mu)},\quad 0\leq t<\infty,
\label{eq:Linfty convergence of dist cond of survival main results section}
\end{align}
\eqref{eq:Linfty convergence of dist cond of survival main results section} being understood to apply only when the denominator on the right is positive. 

If, in addition to Assumption \ref{assum:adjoint Dobrushin main results section}, we have Assumption \ref{Assum:assum for essentially positive right efn main results section}, then $\phi\in L^{\infty}_{>0}(\pi)$. In particular, for all $\mu\in\calP_{\infty}(\pi)$, \eqref{eq:Linfty convergence of dist cond of survival main results section} then holds for all $t$ sufficiently large.

On the other hand, if Assumption \ref{assum:adjoint Dobrushin main results section}, Assumption \ref{assum:technical assumption for Assum (A)} and \cite[Assumption (A)]{Champagnat2014} are satisfied (but we no longer assume Assumption \ref{Assum:assum for essentially positive right efn main results section}), then there exists constants $C,T<\infty$ and $\gamma>0$ such that
\begin{equation}\label{eq:Linfty conv when also have Assum (A) main results section}
\Big\lvert\Big\lvert \frac{d\Law_{\mu}(X_t\lvert \tau_{\partial}>t)}{d\pi}-1\Big\rvert\Big\rvert_{L^{\infty}(\pi)}\leq \frac{C}{\mu(h)}e^{-\gamma t}\Big\lvert\Big\lvert \frac{d\mu}{d\pi}\Big\rvert\Big\rvert_{L^{\infty}(\pi)}\quad\ \text{for all}\quad t\geq T\quad\text{and all}\quad  \mu\in\calP_{\infty}(\pi),
\end{equation}
where $h\in \calB_b(\chi;\Rm_{>0})$ is the bounded and strictly positive pointwise right eigenfunction provided for by \cite[Proposition 2.3]{Champagnat2014} (which must be a version of the $L^1(\pi)$-right eigenfunction $\phi$). In \eqref{eq:Linfty conv when also have Assum (A) main results section}, $\gamma>0$ is the minimum of the $\gamma>0$ given by \cite[Theorem 2.1]{Champagnat2014} and $\frac{-\ln(1-c_0)}{t_0}$ (we define $\frac{-\ln(1-c_0)}{t_0}:=+\infty$ when $c_0=1$), where $0<c_0\leq 1$ is the constant and $t_0>0$ the time given by Assumption \ref{assum:adjoint Dobrushin main results section}.
\end{theo}

We have not explicitly assumed in the assumptions for \eqref{eq:Linfty convergence of dist cond of survival main results section} that $\Pm_{\mu}(\tau_{\partial}>t)>0$ for all $\mu\in\calP_{\infty}(\pi)$ and $t<\infty$, so may wonder if $\frac{d\Law_{\mu}(X_t\lvert\tau_{\partial}>t)}{d\pi}$ is necessarily well-defined. However, if \eqref{eq:Linfty convergence of dist cond of survival main results section} applies for some $t<\infty$, then necessarily $\mu(\phi)>0$ so that $\Pm_{\mu}(\tau_{\partial}>t)>0$ for all $t<\infty$, hence $\Law_{\mu}(X_t\lvert\tau_{\partial}>t)$ is well-defined. If $\mu(\phi)=0$, then \eqref{eq:Linfty convergence of dist cond of survival main results section} does not apply for any $t<\infty$. For such initial conditions, \eqref{eq:convergence of prob of killing main results section} indicates that $\Pm_{\mu}(\tau_{\partial}>t)$ decays at a faster exponential rate than $\Pm_{\pi}(\tau_{\partial}>t)$.

One may ask whether the essential boundedness of $\frac{d\pi}{d\Lambda}$ is necessary, or whether $\pi\ll\Lambda$ might suffice for Theorem \ref{theo:Linfty convergence main results section}. The following proposition demonstrates that it is necessary.

\begin{prop}\label{prop:boundedness from reverse Dobrushin main results section}
We assume that there exists $\psi\in \calB_{b,\gg}(\chi)$, a time $t_0>0$, a constant $a>0$ and a submarkovian kernel $\tilde{P}$ on $\chi$ satisfying \eqref{eq:psi adjoint for verifying reverse Dobrushin main results section}. We assume that $(X_t)_{0\leq t<\tau_{\partial}}$ has a QSD $\pi$ which is absolutely continuous with respect to $\Lambda$, $\pi\ll\Lambda$, and that $(P_t)_{t\geq 0}$ has a non-negative $L^1_{\geq 0}(\pi)$ right eigenfunction such that $\pi(\phi)=1$. Finally, we assume that
\begin{equation}\label{eq:conv to QSD for propn that it implies QSD has a bounded density}
\Big\lvert\Big\lvert \frac{d\Law_{\mu}(X_t\lvert\tau_{\partial}>t)}{d\pi}-1\Big\rvert\Big\rvert_{L^{\infty}(\pi)}\ra 0\quad\text{as}\quad t\ra\infty\quad\text{whenever}\quad \mu\in\calP_{\infty}(\pi)\quad\text{and}\quad \mu(\phi)>0.
\end{equation}
Then $\pi$ has an essentially bounded density with respect to $\Lambda$, $\pi\in\calP_{\infty}(\Lambda)$.
\end{prop}

\subsection*{Inequalities relating the distribution of an absorbed Markov process at a fixed time with its QSD}

Theorem \ref{theo:Linfty convergence main results section} may only be applied when the initial condition belongs to $\calP_{\infty}(\pi)$. In Theorem \ref{theo:uniform Linfty convergence main results section}, we shall extend this to arbitrary initial conditions (which may be mutually singular with respect to $\pi$), under additional conditions. This is possible due to the following theorems, which provide probabilistic criteria allowing us to compare from from above and below the distribution of an absorbed Markov process at a fixed time with its QSD. Whilst the original motivation for these was to establish $L^{\infty}(\pi)$ convergence for arbitrary initial conditions (Theorem \ref{theo:uniform Linfty convergence main results section}), they may be of independent interest.

The first theorem, \ref{theo:dominated by pi theorem general results main results section}, provides a sufficient condition for the distribution of $X_t$ to be dominated by the QSD $\pi$ after a given time horizon, so that the distribution necessarily belongs to $\calP_{\infty}(\pi)$ after a given time horizon, in particular.

We consider the following assumption.
\renewcommand{\theAssumletter}{AaD}
\begin{Assumletter}[Adjoint anti-Dobrushin condition]\label{assum:adjoint anti-Dobrushin main results section}
There exists a time $t_1>0$, a constant $a_1>0$, $\psi_1\in\calB_{b,\gg}(\chi)$, and a submarkovian kernel $\tilde{P}^{(1)}$ on $\chi$ such that
\begin{equation}\label{eq:psi adjoint for verifying bounded pi euclidean condition main results section}
\psi_1(x)\Lambda(dx)P_{t_1}(x,dy)=a_1\psi_1(y)\Lambda(dy)\tilde{P}^{(1)}(y,dx).
\end{equation}
We assume that $\tilde{P}^{(1)}1(y)>0$ for $\Lambda$-almost every $y\in \chi$. We further assume that there exists $C_1<\infty$ such that
\begin{equation}\label{eq:bounded Lebesgue main results section}
\tilde{P}^{(1)}(y,\cdot)\leq C_1\Lambda(\cdot)\quad \text{for $\Lambda$-almost every $y\in\chi$}.
\end{equation}
\end{Assumletter}

\begin{theo}\label{theo:dominated by pi theorem general results main results section}
We suppose that $(X_t)_{0\leq t<\tau_{\partial}}$ has a QSD $\pi$ which is absolutely continuous with respect to $\Lambda$, $\pi\ll\Lambda$, and which has full support, $\text{spt}(\pi)=\chi$. We assume that $(X_t)_{0\leq t<\tau_{\partial}}$ satisfies assumptions \ref{assum:adjoint Dobrushin main results section} and \ref{assum:adjoint anti-Dobrushin main results section}. 

We let $t_0>0$ and $t_1>0$ respectively be the times, and $\psi_0$ and $\psi_1$ respectively be the functions, for which Assumption \ref{assum:adjoint Dobrushin main results section} and Assumption \ref{assum:adjoint anti-Dobrushin main results section} are satisfied. We define $\lambda:=\lambda(\pi)=\Pm_{\pi}(\tau_{\partial}>1)$ and $t_2:=t_0+t_1$. The constants $c_0'>0$, $a_1>0$ and $C_1<\infty$ are respectively the constants for which we have \eqref{eq:Dobrushin for P tilde kernel crit for reverse dobrushin main results section}, \eqref{eq:psi adjoint for verifying bounded pi euclidean condition main results section} and \eqref{eq:bounded Lebesgue main results section}. Finally $\nu$ is the probability measure for which we have \eqref{eq:Dobrushin for P tilde kernel crit for reverse dobrushin main results section}. It follows from Assumption \ref{assum:adjoint Dobrushin main results section} that $\nu\ll \Lambda$, so that $\nu(\frac{d\pi}{d\Lambda})$ is unambiguous, and strictly positive. We define
\begin{equation}\label{eq:formula for C2 dominated by pi theorem main results section}
C_2:=\frac{\lvert\lvert \frac{\psi_0}{\psi_1}\rvert\rvert_{\infty}\lvert\lvert \frac{\psi_1}{\psi_0}\rvert\rvert_{\infty}\lvert\lvert\psi_1\rvert\rvert_{\infty}\lvert\lvert \frac{1}{\psi_1}\rvert\rvert_{\infty}a_1C_1\lambda^{t_0}}{c_0'\nu(\frac{d\pi}{d\Lambda})}.
\end{equation} 
We finally assume that $P_{t_2}$ is lower semicontinuous in the sense of Definition \ref{defin:lower semicts kernel}.

Then we have that
\begin{equation}\label{eq:dominated by pi equation general results main results section}
P_{t_2}(x,\cdot)\leq C_2\pi(\cdot)\quad\text{for all}\quad x\in\chi.
\end{equation}
It follows, in particular, that if 
\begin{equation}
\Pm_x(\tau_{\partial}>t_2+h\lvert \tau_{\partial}>t_2)\geq \epsilon\quad\text{for all}\quad x\in \chi
\end{equation}
(which follows, for instance, if \cite[Assumption (A1)]{Champagnat2014} is satisfied over the time interval $h>0$), then
\begin{equation}\label{eq:upper bound dist cond on survival by pi if also A1 main results section}
\frac{P_{t_2+h}(x,\cdot)}{P_{t_2+h}1(x)}\leq \frac{C_2}{\epsilon}\pi(\cdot)\quad \text{for all}\quad x\in \chi.
\end{equation}
\end{theo}

Theorem \ref{theo:dominated by pi theorem general results main results section} is complemented by the following theorem, providing for the reverse inequality.

\renewcommand{\theAssumletter}{DAD}
\begin{Assumletter}[Combined Dobrushin and adjoint Dobrushin condition]\label{assum:combined Dobrushin adjoint Dobrushin main results section}
We assume that $\pi\in\calP_{\infty}(\Lambda)$. We assume that, for some measurable set $A\in \mathscr{B}(\chi)$, there exists $\nu_1\in \calP(A)$, a constant $c_1>0$ and time $t_1>0$ such that $P_{t_1}1(x)>0$ for all $x\in \chi$ and
\begin{equation}\label{eq:Dobrushin for DAD condition main results section}
\frac{P_{t_1}(x,\cdot)}{P_{t_1}1(x)}\geq c_1\nu_1(\cdot)\quad\text{for all}\quad x\in \chi.
\end{equation}
It follows that $\pi(A)>0$ so that $\Lambda(A)>0$. We further assume that, for this same measurable set $A$, Assumption \ref{assum:adjoint Dobrushin main results section} is satisfied, with $\nu:=\frac{\Lambda_{\lvert_A}}{\Lambda(A)}$, some constant $c_0'>0$ and time $t_0>0$.
\end{Assumletter}

\begin{theo}\label{theo:DAD lower bounds density main results section}
Suppose that $(X_t)_{0\leq t<\tau_{\partial}}$ is an absorbed Markov process for which $\pi$ is a (not necessarily unique) QSD. We denote $\lambda:=\lambda(\pi)$. We assume that $(X_t)_{0\leq t<\tau_{\partial}}$ and $\pi$ satisfy Assumption \ref{assum:combined Dobrushin adjoint Dobrushin main results section}. We write $c_0',c_1>0$ for the constants and $t_0,t_1>0$ for the times for which Assumption \ref{assum:combined Dobrushin adjoint Dobrushin main results section} is satisfied. We define
\begin{equation}\label{eq:time and constant for DAD condition consequence main results section}
t_3:=t_0+t_1,\quad c_3:=\frac{\lambda^{t_0}c_0'c_1}{\Lambda(A)\lvert\lvert\psi\rvert\rvert_{\infty}\lvert\lvert\frac{1}{\psi}\rvert\rvert_{\infty}\lvert\lvert \frac{d\pi}{d\Lambda}\rvert\rvert_{L^{\infty}(\Lambda)}}.
\end{equation}

Then we have that
\begin{equation}\label{eq:minorised by pi equation from DAD main results section}
\frac{P_{t_3}(x,\cdot)}{P_{t_3}1(x)}\geq c_3\pi(\cdot)\quad\text{for all}\quad x\in\chi.
\end{equation}
\end{theo}

We note, in particular, that the inequalities \eqref{eq:dominated by pi equation general results main results section}, \eqref{eq:upper bound dist cond on survival by pi if also A1 main results section} and \eqref{eq:minorised by pi equation from DAD main results section} are valid over the \textit{entire} domain. These have a PDE interpretation described by \eqref{eq:PDE interpretation of comparison inequalities}.

\subsection*{$L^{\infty}(\pi)$-convergence for arbitrary initial condition}

We can now state our theorem on $L^{\infty}$-convergence to a QSD for arbitrary initial condition.

\begin{theo}\label{theo:uniform Linfty convergence main results section}
Suppose that the killed Markov process $(X_t)_{0\leq t<\tau_{\partial}}$ has a QSD $\pi$ which has an essentially bounded density with respect to $\Lambda$, $\pi\ll_{\infty}\Lambda$, and which has full support, $\text{spt}(\pi)=\chi$. We assume that $(X_t)_{0\leq t<\tau_{\partial}}$ satisfies assumptions \ref{assum:adjoint Dobrushin main results section} and \ref{assum:adjoint anti-Dobrushin main results section}. We let $t_0>0$ and $t_1>0$ respectively be the times for which assumptions \ref{assum:adjoint Dobrushin main results section} and \ref{assum:adjoint anti-Dobrushin main results section} are satisfied, with $c_0>0$ the constant given by \eqref{eq:formula for c0 main results section} and $C_2<\infty$ the constant given \eqref{eq:formula for C2 dominated by pi theorem main results section}. We define $t_2:=t_0+t_1$ and assume that $P_{t_2}$ is lower semicontinuous in the sense of Definition \ref{defin:lower semicts kernel}. Then we have the following. 

There exists $h\in \calB_b(\chi;\Rm_{\geq 0})$ with $\pi(h)=1$ such that $P_th(x)=\lambda^th(x)$ for all $x\in \chi$ and $0\leq t<\infty$; $h$ is a bounded, non-negative, pointwise right eigenfunction for $(P_t)_{t\geq 0}$. We  have that $\Pm_{\mu}(X_{t}\in\cdot)\in \calP_{\infty}(\Lambda)$ for all $\mu\in\calP(\chi)$ and $t\geq t_2$. Then we have that
\begin{equation}\label{eq:Linfty malthusian theo statement arbitrary ic main results section}
\Big\lvert\Big\lvert \lambda^{-t}\frac{d\Pm_{\mu}(X_t\in \cdot)}{d\pi(\cdot)}-\mu(h)\Big\rvert\Big\rvert_{L^{\infty}(\pi)}\leq C_2(1-c_0)^{\lfloor \frac{t-t_2}{t_0}\rfloor} ,\quad t_2\leq t<\infty,\quad\text{for all}\quad \mu\in\calP(\chi).
\end{equation}
Consequentially we have for all $\mu\in\calP(\chi)$:
\begin{align}\label{eq:convergence of prob of killing arbitrary ics main results section}
\lvert\lambda^{-t}\Pm_{\mu}(\tau_{\partial}>t)-\mu(\phi)\rvert&\leq C_2(1-c_0)^{\lfloor \frac{t-t_2}{t_0}\rfloor},\quad t_2\leq t<\infty,\\
\Big\lvert\Big\lvert \frac{d\Law_{\mu}(X_t\lvert \tau_{\partial}>t)}{d\pi}-1\Big\rvert\Big\rvert_{L^{\infty}(\pi)}&\leq \frac{2C_2(1-c_0)^{\lfloor \frac{t-t_2}{t_0}\rfloor}}{\mu(h)-C_2(1-c_0)^{\lfloor \frac{t-t_2}{t_0}\rfloor}},\quad t_2\leq t<\infty,
\label{eq:Linfty convergence of dist cond of survival arbitrary ics main results section}
\end{align}
\eqref{eq:Linfty convergence of dist cond of survival arbitrary ics main results section} being understood to apply only when the denominator on the right is positive. 

If, in addition, either Assumption \ref{assum:technical assumption for Assum (A)} and \cite[Assumption (A)]{Champagnat2014} are satisfied or Assumption \ref{assum:combined Dobrushin adjoint Dobrushin main results section} is satisfied, then there exists $T<\infty$ (which does not depend on the initial condition $\mu\in\calP(\chi)$) such that $\Law_{\mu}(X_t\lvert \tau_{\partial}>t)\in \calP_{\infty}(\pi)$ for all $t\geq T$, with the density $\frac{d\Law_{\mu}(X_t\lvert \tau_{\partial}>t)}{d\pi}$ satisfying
\begin{equation}\label{eq:uniform exponential L infty convergence for arbitrary initial cond main results section}
\Big\lvert\Big\lvert \frac{d\Law_{\mu}(X_t\lvert \tau_{\partial}>t)}{d\pi}-1\Big\rvert\Big\rvert_{L^{\infty}(\pi)}\leq (1-c_0)^{\lfloor \frac{t-T}{t_0}\rfloor},\quad T\leq t<\infty,\quad\text{for all}\quad \mu\in\calP(\chi).
\end{equation}
In the latter case, that Assumption \ref{assum:combined Dobrushin adjoint Dobrushin main results section} is satisfied, we write $t_3>0$ for the time and $c_3>0$ for the constant given by \eqref{eq:time and constant for DAD condition consequence main results section}. We then have the quantitative estimate
\begin{equation}\label{eq:quantitative uniform exponential convergence for arbitrary initial cond using DAD main results section}
\Big\lvert\Big\lvert \frac{d\Law_{\mu}(X_t\lvert \tau_{\partial}>t)}{d\pi}-1\Big\rvert\Big\rvert_{L^{\infty}(\pi)}\leq \frac{2C_2(1-c_0)^{\lfloor \frac{t-(t_2+t_3)}{t_0}\rfloor}}{c_3-C_2(1-c_0)^{\lfloor \frac{t-(t_2+t_3)}{t_0}\rfloor}},\quad t_2+t_3\leq t<\infty,
\end{equation}
\eqref{eq:quantitative uniform exponential convergence for arbitrary initial cond using DAD main results section} being understood to apply only when the denominator on the right is positive. 
\end{theo}

\subsection*{Continuity of the quasi-stationary density}

We now introduce a criterion for the QSD $\pi$ to have a continuous density with respect to $\Lambda$. 

We consider a killed Markov chain $(X_t)_{0\leq t<\tau_{\partial}}$ on $\chi$ with quasi-stationary distribution $\pi$ belonging to $\calP_{\infty}(\Lambda)$ and submarkovian transition semigroup $(P_t)_{t\geq 0}$. We consider the following assumption on submarkovian transition semigroups $(\tilde{P}_t)_{t\geq 0}$.
\renewcommand{\theAssumletter}{C}
\begin{Assumletter}\label{assum:cty assum for cty cor}
There exists $\psi\in C_{b,\gg}(\chi)$ and a constant $a>0$ such that
\begin{equation}\label{eq:adjoint semigroup eqn from cty of pi criterion}
\psi(x)\Lambda(dx)P_{t}(x,dy)=a^t\psi(y)\Lambda(dy)\tilde{P}_t(y,dx),\quad t\geq 0.
\end{equation}
We assume that for all $t>0$, $\tilde{P}_t1(y)>0$ for $\Lambda$-almost every $y\in \chi$. We further assume that there exists a constant $c_0>0$, time $t_0>0$ and $\nu\in\calP(\chi)$ such that
\begin{equation}
\frac{\tilde{P}_{t_0}(y,\cdot)}{\tilde{P}_{t_0}1(y)}\geq c_0\nu(\cdot)\quad \text{for $\Lambda$-almost every}\quad y\in\chi.
\end{equation}

We finally assume that there exists an open set $\calO$, positive constant $c_1>0$ and a time $t_1>0$ such that $\pi\geq c_1\Lambda_{\lvert_{\calO}}$ and $\tilde{P}_{t_1}(\nu,\calO)>0$.
\end{Assumletter}

\begin{theo}\label{theo:criterion for cty of QSD}
We assume that $(X_t)_{0\leq t<\tau_{\partial}}$ has a QSD $\pi$ which has an essentially bounded density with respect to $\Lambda$, $\pi\in\calP_{\infty}(\Lambda)$. We further assume that the submarkovian transition semigroup $(\tilde{P}_t)_{t\geq 0}$ satisfies Assumption \ref{assum:cty assum for cty cor}. 

Then we have that:
\begin{enumerate}
\item\label{enum:QSD thm cty on V main results section}
For any open set $V$ satisfying
\begin{equation}\label{eq:cond for open set to have cts QSD density main results section}
\tilde{P}_t(C_b(\chi))\subseteq \calB_b(\chi)\cap C_b(V), \quad t\geq 0,
\end{equation}
(a version of) $\frac{d\pi}{d\Lambda}_{\lvert_V}$ belongs to $C_b(V)$.
\item\label{enum:QSD thm lower semi cty main results section}
We suppose, on the other hand, that $(\tilde{P}_t)_{t\geq 0}$ is lower semicontinuous in the sense of Definition \ref{defin:lower semicts kernel}, which we recall means that we have
\begin{equation}\label{eq:weak lower semicty of tilde P main results section}
\tilde{P}_{t}(C_b(\chi;\Rm_{\geq 0}))\subseteq LC_b(\chi;\Rm_{\geq 0}),\quad t\geq 0.
\end{equation}
We further assume that $\chi$ is separable. Then $\frac{d\pi}{d\Lambda}$ has a bounded, non-negative, lower semicontinuous version $\rho\in LC_b(\chi;\Rm_{\geq 0})$ which is maximal in the sense that any other bounded, non-negative lower semicontinuous version of $\frac{d\pi}{d\Lambda}$, $\tilde{\rho}\in LC_b(\chi;\Rm_{\geq 0})$, is everywhere dominated by $\rho$: $\tilde{\rho}(x)\leq \rho(x)$ for all $x\in \chi$.
\end{enumerate}
\end{theo}

\begin{rmk}\label{rmk:largest open set on which QSD cts main results section}
If there is a non-empty open set $V$ satisfying \eqref{eq:cond for open set to have cts QSD density main results section}, then there is a largest such open set, namely $\cup\{\text{$V$ open: $V$ satisfies }\eqref{eq:cond for open set to have cts QSD density main results section}\}$.
\end{rmk}

\begin{rmk}\label{rmk:discrete time semigroup satisfying cty assumption}
We observe that if the kernels $P$ and $\tilde{P}$ satisfy \eqref{eq:adjoint semigroup eqn from cty of pi criterion} for $t=1$ then we may obtain inductively that
\begin{equation}\label{eq:adjoint semigroup eqn from cty of pi criterion by induction}
\begin{split}
\psi(x)\Lambda(dx)P^{(n+1)}(x,dy)=\int_{z\in \chi}P_{1}(z,dy)\psi(x)\Lambda(dx)P^{n}(x,dz)=\int_{z\in \chi}P^{1}(z,dy)a^n\tilde{P}^{n}(z,dx)\psi(z)\Lambda(dz)\\
=\int_{z\in \chi}a^{n+1}\tilde{P}^n(z,dx)\tilde{P}(y,dz)\psi(y)\Lambda(dy)=\psi(y)\Lambda(dy)a^{n+1}\tilde{P}^{n+1}(y,dx)\quad\text{for all}\quad n\geq 1.
\end{split}
\end{equation}
It follows that the discrete time semigroups $(P^n)_{n\geq 0}$ and $(\tilde{P}^n)_{n\geq 0}$ satisfy \eqref{eq:adjoint semigroup eqn from cty of pi criterion}. We say that a kernel $P$ satisfies Assumption \ref{assum:cty assum for cty cor} if the discrete-time semigroup $(P^n)_{n\geq 0}$ satisfies Assumption \ref{assum:cty assum for cty cor}.
\end{rmk}

\subsection*{Existence of a positive and bounded pointwise right eigenfunction}

We now introduce a criterion for the existence of a bounded, everywhere strictly positive, pointwise right eigenfunction for $(P_t)_{t\geq 0}$. In contrast to the Krein-Rutman theorem, this criterion does not require that any operator is compact.

\renewcommand{\theAssumletter}{E}
\begin{Assumletter}\label{assum:assum for cor right efn main results section}
We assume that $\chi$ is separable and that $\Pm_x(\tau_{\partial}>t)>0$ for all $x\in \chi$ and $t<\infty$. We further assume that $(X_t)_{0\leq t<\tau_{\partial}}$ satisfies Assumption \cite[Assumption (A1)]{Champagnat2014}, so that there exists a constant $c_0>0$, time $t_0>0$ and $\nu\in\calP(\chi)$ such that
\begin{equation}\label{eq:A1 for right efn existence theorem main results section}
\Law_x(X_{t_0}\lvert \tau_{\partial}>t_0)\geq c_0\nu(\cdot)\quad\text{for all}\quad x\in \chi.
\end{equation}

We assume that $(P_t)_{t\geq 0}$ is lower semicontinuous in the sense of Definition \ref{defin:lower semicts kernel}. We further assume that there exists $\psi\in C_{b,\gg}(\chi)$, a constant $a>0$ and a submarkovian transition semigroup $(\tilde{P}_t)_{t\geq 0}$ such that
\begin{equation}\label{eq:adjoint eqn for right efn existence theorem main results section}
\psi(x)\Lambda(dx)P_{t}(x,dy)=a^t\psi(y)\Lambda(dy)\tilde{P}_t(y,dx),\quad t\geq 0.
\end{equation}
We assume that $(\tilde{P}_t)_{t\geq 0}$ has a QSD, $\tilde{\pi}$, which has an essentially bounded density with respect to $\Lambda$, $\tilde{\pi}\in\calP_{\infty}(\Lambda)$. Finally, we assume that there exists an open set $\calO\subseteq \chi$, positive constant $c_1>0$ and time $t_1>0$ such that $\tilde{\pi}\geq c_1\Lambda_{\calO}$ and $\Pm_{\nu}(X_{t_1}\in \calO)>0$.
\end{Assumletter}

\begin{defin}\label{defin:discrete time semigroup satisfying right efn assumption}
As with Remark \ref{rmk:discrete time semigroup satisfying cty assumption} we say that a kernel $P$ satisfies Assumption \ref{assum:assum for cor right efn main results section} if the discrete-time semigroup $(P^n)_{n\geq 0}$ satisfies Assumption \ref{assum:assum for cor right efn main results section}.
\end{defin}

\begin{theo}\label{theo:criterion for right efn main results section}
We consider a killed Markov process $(X_t)_{0\leq t<\tau_{\partial}}$ satisfying Assumption \ref{assum:assum for cor right efn main results section}. We define
\begin{equation}\label{eq:defin of evalue lambda pf of right efn existence theorem main results section}
\lambda:=a\tilde{\pi}\tilde{P}_11>0,
\end{equation}
whereby $\tilde{\pi}$ is the QSD, $(\tilde{P}_t)_{t\geq 0}$ the submarkovian transition semigroup and $a$ the strictly positive constant assumed to exist in Assumption \ref{assum:assum for cor right efn main results section}. Then there exists
\begin{equation}
h\in \calB_b(\chi;\Rm_{>0})(\chi),
\end{equation} 
which for all $t\geq 0$ is a pointwise right eigenfunction for $P_t$ of eigenvalue $\lambda^t$,
\begin{equation}
P_th(x)=\lambda^th(x)\quad \text{for all}\quad x\in\chi,\quad t\geq 0.
\end{equation}
If, in addition, there exists an open set $V\subseteq \chi$ such that we have
\begin{equation}\label{eq:cty preserved on V condition for efn}
P_t(LC_b(\chi;\Rm_{\geq 0})\cap C_b(V))\subseteq LC_b(\chi;\Rm_{\geq 0})\cap C_b(V),
\end{equation}
then we have that $h\in \calB_b(\chi;\Rm_{\geq 0})\cap C_b(V)$.
\end{theo}

\subsection*{Application to the $Q$-process}

We suppose that $(X_t)_{0\leq t<\tau_{\partial}}$ satisfies Assumption \ref{assum:technical assumption for Assum (A)} and \cite[Assumption (A)]{Champagnat2014}. These imply both that $(X_t)_{0\leq t<\tau_{\partial}}$ has a unique QSD, $\pi$, and that the $Q$-process defined in \cite[Theorem 3.1]{Champagnat2014} - the process $(X_t)_{0\leq t<\tau_{\partial}}$ conditioned never to be killed - exists, is unique, and is exponentially ergodic. We call the $Q$-process $(Z_t)_{0\leq t<\infty}$, and write $\beta$ for its stationary distribution. Then we have the following corollary.

\begin{cor}\label{cor:L infty convergence of Q-process main results section}
We suppose that $(X_t)_{0\leq t<\tau_{\partial}}$ satisfies Assumption \ref{assum:technical assumption for Assum (A)} and \cite[Assumption (A)]{Champagnat2014}. We also assume that $(X_t)_{0\leq t<\tau_{\partial}}$ satisfies Assumption \ref{assum:adjoint Dobrushin main results section}, and that $\pi\in \calP_{\infty}(\Lambda)$. The time $t_0>0$ is the time given by Assumption \ref{assum:adjoint Dobrushin main results section}, while $0<c_0\leq 1$ is the constant given by \eqref{eq:formula for c0 main results section}. Then for all $\mu\in\calP_{\infty}(\beta)$ we have that

\begin{equation}\label{eq:uniform exponential L infty for Q process bded initial cond main results section}
\Big\lvert\Big\lvert \frac{d\Law_{\mu}(Z_t)}{d\beta}-1\Big\rvert\Big\rvert_{L^{\infty}(\beta)}\leq (1-c_0)^{\lfloor \frac{t}{t_0}\rfloor} \osc_{\beta}(\mu),\quad 0\leq t<\infty.
\end{equation}

We now assume, in addition, that $(X_t)_{0\leq t<\tau_{\partial}}$ satisfies Assumption \ref{assum:adjoint anti-Dobrushin main results section}. Then there exists a time $T<\infty$ (which does not depend on the initial condition $\mu\in\calP(\chi)$) such that $\Law_{\mu}(Z_t)\in\calP_{\infty}(\beta)$ for all $\mu\in\calP(\chi)$ and $t\geq T$, with the density $\frac{d\Law_{\mu}(Z_t)}{d\beta}$ satisfying
\begin{equation}\label{eq:uniform exponential L infty for Q process convergence for arbitrary initial cond main results section}
\Big\lvert\Big\lvert \frac{d\Law_{\mu}(Z_t)}{d\beta}-1\Big\rvert\Big\rvert_{L^{\infty}(\beta)}\leq (1-c_0)^{\lfloor \frac{t-T}{t_0}\rfloor},\quad T\leq t<\infty,\quad\text{for all}\quad \mu\in\calP(\chi).
\end{equation}
\end{cor}

\subsection*{$L^p$ convergence}

It is a simple consequence of the Riesz-Thorin interpolation theorem that \cite[Assumption (A)]{Champagnat2014} also provides for non-uniform exponential convergence in $L^p(\pi)$ for all $1\leq p<\infty$, according to the following proposition. An analogous statement is well-known in the context of Markov processes without killing (see \cite[p.114]{Cattiaux2014}), for instance).

\begin{prop}\label{prop:Lp malthusian by interpolation}
We assume that $(X_t)_{0\leq t<\tau_{\partial}}$ is a killed Markov process satisfying Assumption \ref{assum:technical assumption for Assum (A)} and \cite[Assumption (A)]{Champagnat2014}. We let $h\in \calB_b(\chi;\Rm_{> 0})$ be the bounded and strictly positive pointwise right eigenfunction whose existence is provided for by \cite[Proposition 2.3]{Champagnat2014}. Then there exists constants $C<\infty$, $\gamma>0$ ($\gamma>0$ being the constant given in the statement of \cite[Theorem 2.1]{Champagnat2014}), and a time $T<\infty$, such that for all $1\leq p<\infty$ we have
\begin{equation}
\begin{split}
\Big\lvert\Big\lvert \frac{d\Law_{\mu}(X_t\lvert \tau_{\partial}>t)}{d\pi}-1\Big\rvert\Big\rvert_{L^p(\pi)}\leq \frac{Ce^{-\frac{\gamma}{p}t}\Big\lvert\Big\lvert \frac{d\mu}{d\pi}\Big\rvert\Big\rvert_{L^p(\pi)}}{\mu(h)}\quad
\text{for all}\quad T\leq t<\infty\quad\text{and}\quad \mu\in \calP_{p}(\pi).
\end{split}
\end{equation}
\end{prop}

\section{The reverse Dobrushin condition and its consequences}\label{section:Reverse Dobrushin condition}

In this section, we do not assume there to be some distinguished measure $\Lambda$. The results of this section shall be self-contained, and will be proven in Section \ref{section:general state space proof}. They will later be applied in Section \ref{section:Euclidean state space proof} to establish the results of Section \ref{section:main results}.

This section is concerned with consequences of the following Dobrushin-type criterion upon the killed Markov process $(X_t)_{0\leq t<\tau_{\partial}}$ and quasi-stationary distribution $\pi$. Whilst the author is not familiar with an analogous condition for unabsorbed Markov processes, it seems quite plausible that such a condition exists but it unknown to the author.

\renewcommand{\theAssumletter}{RD}
\begin{Assumletter}[Reverse Dobrushin condition]\label{assum:Dobrushin reverse time}
There exists $t_0>0$ such that:
\begin{enumerate}
\item
There exists a submarkovian kernel $R$ on $\chi$ such that
\begin{equation}\label{eq:exis of reverse R}
\pi(dx)P_{t_0}(x,dy)=\pi(dy)R(y,dx).
\end{equation}
\item
There exists $c_0> 0$ and $\nu\in\calP(\chi)$ such that
\begin{equation}\label{eq:reverse time R minorised by nu}
\frac{R(y,\cdot)}{R1(y)}\geq c_0\nu(\cdot)\quad\text{for $\pi$-almost every}\quad x\in\chi.
\end{equation}
\end{enumerate}
\end{Assumletter}

\subsection*{Discussion of the reverse Dobrushin condition, Assumption \ref{assum:Dobrushin reverse time}}

If we take a discrete-time absorbed Markov process $(\tilde{X}_t)_{0\leq t<\tilde{\tau}_{\partial}}$ with submarkovian transition kernel $R(x,\cdot)$, we can restate \eqref{eq:exis of reverse R} as
\[
\Pm(X_0\in dx,X_{t_1}\in dy\lvert X_0\sim \pi)=\Pm(\tilde{X}_0\in dy,\tilde{X}_1\in dx\lvert \tilde{X}_0\sim \pi)
\]
and \eqref{eq:reverse time R minorised by nu} as
\[
\Law_{x}(\tilde{X}_1\lvert \tilde{\tau}_{\partial}>1)\geq c_0\nu\quad\text{for $\pi$-almost every}\quad x\in\chi.
\]
Thus $R$ can be thought of as the time-reversal at quasi-stationarity of $P_{t_0}$, with \eqref{eq:reverse time R minorised by nu} a Dobrushin-type condition on this time-reversal. This time-reversal is only needed to hold over a fixed time horizon, however, but not for the paths of $(X_t)_{0\leq t<\tau_{\partial}}$. In particular, we are free to adjust the definition of $R(y,\cdot)$ on arbitrary $\pi$-null sets of $y\in\chi$, which would potentially be problematic if we were seeking to time-reverse a continuous-time process.

One might notice that we have not assumed that $R1(y)>0$ for $\pi$-almost every $y$, so wonder whether the left hand side of \eqref{eq:reverse time R minorised by nu} is necessarily well-defined. In fact, we shall establish in Proposition \ref{prop:R constant mass} that $R1(y)$ is necessarily $\pi$-almost everywhere constant and strictly positive; in particular it can be renormalised to be a Markov kernel.

We may observe that Assumption \ref{assum:Dobrushin reverse time} is of the same form as \cite[Assumption (A1)]{Champagnat2014}, except that it is a condition on a different kernel. On the other hand, we do not require any analogue of \cite[Assumption (A2)]{Champagnat2014}. The fundamental reason for this is that the reverse kernel $R$ can necessarily be renormalised to be a Markov kernel, so a Dobrushin condition alone suffices.

\subsection*{The Reverse kernel has constant mass}

We have not imposed directly in Assumption \ref{assum:Dobrushin reverse time} that $R1(y)>0$ for $\pi$-almost every $y\in \chi$, so it is not immediately clear that the condition \eqref{eq:reverse time R minorised by nu} is well-defined. The following proposition establishes, in particular, that this follows from \eqref{eq:exis of reverse R}, so that the condition \eqref{eq:reverse time R minorised by nu} is well-defined.

\begin{prop}[The Reverse kernel has constant mass]\label{prop:R constant mass}
We fix $0<t_0<\infty$ and suppose that for some QSD $\pi$ of $(X_t)_{0\leq t<\tau_{\partial}}$, $R$ is a non-negative kernel satisfying \eqref{eq:exis of reverse R}. We define $\lambda:=\lambda(\pi)$. Then we have that
\begin{equation}\label{eq:mass of R constant}
R1(y)=\lambda^{t_0}\quad\text{for $\pi$-almost every $y\in\chi$}.
\end{equation}
In particular, $R$ is a submarkovian kernel (perhaps after adjusting the definition of $R(y,\cdot)$ on a $\pi$-null set of $y\in \chi$). 
\end{prop}

Proposition \ref{prop:R constant mass} establishes that the time-reversal at quasi-stationarity of a submarkovian kernel has constant mass, so may be rescaled to be a Markov kernel. This is the reason that a Dobrushin condition alone shall suffices to establish convergence to a QSD in the following theorem, \ref{theo:Linfty convergence}, without any analogue of \cite[Assumption (A2)]{Champagnat2014}.

\subsection*{$L^{\infty}(\pi)$ convergence for $L^{\infty}(\pi)$ initial condition}

If we have Assumption \ref{assum:Dobrushin reverse time}, then we may also consider the following assumption.
\begin{assum}\label{assum:assumption for positivity of L1 right efn}
We have that $(X_t)_{0\leq t<\tau_{\partial}}$ satisfies Assumption \ref{assum:Dobrushin reverse time}, with $\nu\in\calP(\chi)$ being the probability measure assumed to exist in Assumption \ref{assum:Dobrushin reverse time}. It readily follows that $\nu\ll \pi$. We assume that for all $\mu\in\calP_{\infty}(\pi)$ there exists $t=t(\mu)<\infty$ (dependent upon $\mu$) such that $\mu P_t\frac{d\nu}{d\pi}>0$.
\end{assum}

We now state the main theorem of this paper.

\begin{theo}\label{theo:Linfty convergence}
We assume that the killed Markov process $(X_t)_{0\leq t<\tau_{\partial}}$ has a (not necessarily unique) quasi-stationary distribution $\pi$, with which it satisfies Assumption \ref{assum:Dobrushin reverse time}. The constant $0<c_0\leq 1$, time $t_0>0$ and probability measure $\nu$ are those given by Assumption \ref{assum:Dobrushin reverse time}, while $\lambda:=\lambda(\pi)$. Then there exists $\phi\in L^1_{\geq 0}(\pi)$ with $\lvert\lvert  \phi\rvert\rvert_{L^1(\pi)}=1$ such that $P_t\phi=\lambda^{t}\phi$ for all $0\leq t<\infty$. For all $0\leq t<\infty$, $\phi$ is both the unique non-negative $L^1(\pi)$-right eigenfunction of $P_t$ and the unique $L^1(\pi)$-right eigenfunction of eigenvalue $\lambda^t$, up to rescaling. Moreover we have the following $L^{\infty}$-Perron-Frobenius behaviour,
\begin{equation}\label{eq:Linfty malthusian theo statement}
\Big\lvert\Big\lvert \lambda^{-t}\frac{d\Pm_{\mu}(X_t\in \cdot)}{d\pi(\cdot)}-\mu(\phi)\Big\rvert\Big\rvert_{L^{\infty}(\pi)}\leq (1-c_0)^{\lfloor \frac{t}{t_0}\rfloor} \osc_{\pi}(\mu),\quad 0\leq t<\infty,\quad \mu\in\calP_{\infty}(\pi).
\end{equation}

Consequentially we have for all $\mu\in\calP_{\infty}(\pi)$:
\begin{align}\label{eq:convergence of prob of killing}
\lvert\lambda^{-t}\Pm_{\mu}(\tau_{\partial}>t)-\mu(\phi)\rvert&\leq (1-c_0)^{\lfloor \frac{t}{t_0}\rfloor}\osc_{\pi}(\mu),\quad 0\leq t<\infty,\\
\Big\lvert\Big\lvert \frac{d\Law_{\mu}(X_t\lvert \tau_{\partial}>t)}{d\pi}-1\Big\rvert\Big\rvert_{L^{\infty}(\pi)}&\leq \frac{2(1-c_0)^{\lfloor \frac{t}{t_0}\rfloor} \osc_{\pi}(\mu)}{\mu(\phi)-(1-c_0)^{\lfloor \frac{t}{t_0}\rfloor} \osc_{\pi}(\mu)},\quad 0\leq t<\infty,
\label{eq:Linfty convergence of dist cond of survival}
\end{align}
\eqref{eq:Linfty convergence of dist cond of survival} being understood to apply only when the denominator on the right is positive. 

If, in addition to Assumption \ref{assum:Dobrushin reverse time}, we have Assumption \ref{assum:assumption for positivity of L1 right efn}, then $\phi\in L^{\infty}_{>0}(\pi)$. In particular, for all $\mu\in\calP_{\infty}(\pi)$, \eqref{eq:Linfty convergence of dist cond of survival} then holds for all $t$ sufficiently large.

On the other hand, if Assumption \ref{assum:Dobrushin reverse time}, Assumption \ref{assum:technical assumption for Assum (A)} and \cite[Assumption (A)]{Champagnat2014} are satisfied (but we no longer assume Assumption \ref{assum:assumption for positivity of L1 right efn}), then there exists constants $C,T<\infty$ and $\gamma>0$ such that
\begin{equation}\label{eq:Linfty conv when also have Assum (A)}
\Big\lvert\Big\lvert \frac{d\Law_{\mu}(X_t\lvert \tau_{\partial}>t)}{d\pi}-1\Big\rvert\Big\rvert_{L^{\infty}(\pi)}\leq \frac{C}{\mu(h)}e^{-\gamma t}\Big\lvert\Big\lvert \frac{d\mu}{d\pi}\Big\rvert\Big\rvert_{L^{\infty}(\pi)}\quad\ \text{for all}\quad t\geq T\quad\text{and all}\quad  \mu\in\calP_{\infty}(\pi),
\end{equation}
where $h\in \calB_b(\chi;\Rm_{>0})$ is the bounded and strictly positive pointwise right eigenfunction provided for by \cite[Proposition 2.3]{Champagnat2014} (which must be a version of the $L^1(\pi)$-right eigenfunction $\phi$). In \eqref{eq:Linfty conv when also have Assum (A)}, $\gamma>0$ is the minimum of the $\gamma>0$ given by \cite[Theorem 2.1]{Champagnat2014} and $\frac{-\ln(1-c_0)}{t_0}$ (we define $\frac{-\ln(1-c_0)}{t_0}:=+\infty$ when $c_0=1$), where $0<c_0\leq 1$ is the constant and $t_0>0$ the time given by Assumption \ref{assum:Dobrushin reverse time}. 
\end{theo}

\subsection*{Inequalities relating the distribution of an absorbed Markov process at a fixed time with its QSD}

We consider the following assumption.

\renewcommand{\theAssumletter}{RaD}
\begin{Assumletter}[Reverse anti-Dobrushin condition]\label{assum:assumption for dominated by pi after certain time}
There exists a time $t_2>0$ and submarkovian kernel $R^{(2)}$ for which \eqref{eq:exis of reverse R} is satisfied by $(X_t)_{0\leq t<\tau_{\partial}}$ and $\pi$. Moreover, there exists $C_2'<\infty$ such that
\begin{equation}\label{eq:reverse kernel majorised for dominated by pi theorem}
\frac{R^{(2)}(y,\cdot)}{R^{(2)}1(y)}\leq C_2'\pi(\cdot)\quad\text{for $\pi$-almost every $y\in\chi$}.
\end{equation}
Finally, $\text{spt}(\pi)= \chi$ and $P_{t_2}$ is lower semicontinuous in the sense of Definition \ref{defin:lower semicts kernel}.
\end{Assumletter}

\begin{theo}\label{theo:dominated by pi theorem general results}
Suppose that $(X_t)_{0\leq t<\tau_{\partial}}$ is an absorbed Markov process for which $\pi$ is a (not necessarily unique) QSD. We denote $\lambda:=\lambda(\pi)=\Pm_{\pi}(\tau_{\partial}>1)$. We assume that $(X_t)_{0\leq t<\tau_{\partial}}$ and $\pi$ satisfy Assumption \ref{assum:assumption for dominated by pi after certain time}. The time $t_2>1$ and constant $C_2'<\infty$ are those for which Assumption \ref{assum:assumption for dominated by pi after certain time} is satisfied. We define
\begin{equation}\label{eq:formula for C2 RaD cond}
C_2:=\lambda^{t_2}C_2'
\end{equation}
Then we have that
\begin{equation}\label{eq:dominated by pi equation general results}
P_{t_2}(x,\cdot)\leq C_2\pi(\cdot)\quad\text{for all}\quad x\in\chi.
\end{equation}
\end{theo}

We similarly obtain the reverse inequality under the following condition.

\renewcommand{\theAssumletter}{DRD}
\begin{Assumletter}[Combined Dobrushin and reverse Dobrushin condition]\label{assum:combined Dobrushin reverse Dobrushin}
We assume that, for some measurable set $A\in \mathscr{B}(\chi)$, there exists $\nu_1\in \calP(A)$, a constant $c_1>0$ and time $t_1>0$ such that $P_{t_1}1(x)>0$ for all $x\in \chi$ and
\begin{equation}\label{eq:Dobrushin for DRD condition}
\frac{P_{t_1}(x,\cdot)}{P_{t_1}1(x)}\geq c_1\nu_1(\cdot)\quad\text{for all}\quad x\in \chi.
\end{equation}
It follows that $\pi(A)>0$. We further assume that, for this same measurable set $A$, Assumption \ref{assum:Dobrushin reverse time} is satisfied, with $\nu:=\frac{\pi_{\lvert_A}}{\pi(A)}$, and some constant $c_0>0$ and time $t_0>0$. 
\end{Assumletter}

\begin{theo}\label{theo:DRD lower bounds density}
Suppose that $(X_t)_{0\leq t<\tau_{\partial}}$ is an absorbed Markov process for which $\pi$ is a (not necessarily unique) QSD. We denote $\lambda:=\lambda(\pi)=\Pm_{\pi}(\tau_{\partial}>1)$. We assume that $(X_t)_{0\leq t<\tau_{\partial}}$ and $\pi$ satisfy Assumption \ref{assum:combined Dobrushin reverse Dobrushin}. We write $c_0,c_1>0$ for the constants and $t_0,t_1>0$ for the times for which Assumption \ref{assum:combined Dobrushin reverse Dobrushin} is satisfied. We define
\begin{equation}\label{eq:t3 and c3 from DRD}
t_3:=t_0+t_1,\quad c_3:=\frac{\lambda^{t_0}c_0c_1}{\pi(A)}.
\end{equation}

Then we have that
\begin{equation}\label{eq:minorised by pi equation from DRD}
\frac{P_{t_3}(x,\cdot)}{P_{t_3}1(x)}\geq c_3\pi(\cdot)\quad\text{for all}\quad x\in\chi.
\end{equation}
\end{theo}

\subsection*{$L^{\infty}(\pi)$-convergence for arbitrary initial condition}

The following theorem, \ref{theo:consequences of dominated assumption for whole space general results}, provides in particular for uniform exponential convergence in $L^{\infty}(\pi)$ of $\frac{d\Law_{\mu}(X_t\lvert\tau_{\partial}>t)}{d\pi}$, for arbitrary initial condition $\mu\in\calP(\chi)$. 

\begin{theo}\label{theo:consequences of dominated assumption for whole space general results}
We assume that the killed Markov process $(X_t)_{0\leq t<\tau_{\partial}}$ has a (not necessarily unique) quasi-stationary distribution $\pi$, with which it satisfies Assumption \ref{assum:Dobrushin reverse time}. The constant $0<c_0\leq 1$, time $t_0>0$ and probability measure $\nu$ are those given by Assumption \ref{assum:Dobrushin reverse time}, while $\lambda:=\lambda(\pi)$. We further assume that $(X_t)_{0\leq t<\tau_{\partial}}$ and $\pi$ satisfy Assumption \ref{assum:assumption for dominated by pi after certain time}. The time $t_2>0$ is the time for which Assumption \ref{assum:assumption for dominated by pi after certain time} is satisfied, while  $C_2<\infty$ is the constant given by \eqref{eq:formula for C2 RaD cond}. Then there exists $h\in \calB_b(\chi;\Rm_{\geq 0})$ with $\pi(h)=1$ such that $P_th(x)=\lambda^th(x)$ for all $x\in \chi$ and $0\leq t<\infty$; $h$ is a bounded, non-negative, pointwise right eigenfunction for $(P_t)_{t\geq 0}$. Moreover we have that
\begin{equation}\label{eq:Linfty malthusian theo statement arbitrary ic}
\Big\lvert\Big\lvert \lambda^{-t}\frac{d\Pm_{\mu}(X_t\in \cdot)}{d\pi(\cdot)}-\mu(h)\Big\rvert\Big\rvert_{L^{\infty}(\pi)}\leq C_2(1-c_0)^{\lfloor \frac{t-t_2}{t_0}\rfloor} ,\quad t_2\leq t<\infty,\quad\text{for all}\quad \mu\in\calP(\chi).
\end{equation}
Consequentially we have for all $\mu\in\calP(\chi)$:
\begin{align}\label{eq:convergence of prob of killing arbitrary ics}
\lvert\lambda^{-t}\Pm_{\mu}(\tau_{\partial}>t)-\mu(\phi)\rvert&\leq C_2(1-c_0)^{\lfloor \frac{t-t_2}{t_0}\rfloor},\quad t_2\leq t<\infty,\\
\Big\lvert\Big\lvert \frac{d\Law_{\mu}(X_t\lvert \tau_{\partial}>t)}{d\pi}-1\Big\rvert\Big\rvert_{L^{\infty}(\pi)}&\leq \frac{2C_2(1-c_0)^{\lfloor \frac{t-t_2}{t_0}\rfloor}}{\mu(h)-C_2(1-c_0)^{\lfloor \frac{t-t_2}{t_0}\rfloor}},\quad t_2\leq t<\infty,
\label{eq:Linfty convergence of dist cond of survival arbitrary ics}
\end{align}
\eqref{eq:Linfty convergence of dist cond of survival arbitrary ics} being understood to apply only when the denominator on the right is positive. 

If, in addition, either Assumption \ref{assum:technical assumption for Assum (A)} and \cite[Assumption (A)]{Champagnat2014} are satisfied or Assumption \ref{assum:combined Dobrushin reverse Dobrushin} is satisfied, then there exists $T<\infty$ (which does not depend on the initial condition $\mu\in\calP(\chi)$) such that $\Law_{\mu}(X_t\lvert \tau_{\partial}>t)\in \calP_{\infty}(\pi)$ for all $t\geq T$, with the density $\frac{d\Law_{\mu}(X_t\lvert \tau_{\partial}>t)}{d\pi}$ satisfying
\begin{equation}\label{eq:uniform exponential L infty convergence for arbitrary initial cond}
\Big\lvert\Big\lvert \frac{d\Law_{\mu}(X_t\lvert \tau_{\partial}>t)}{d\pi}-1\Big\rvert\Big\rvert_{L^{\infty}(\pi)}\leq (1-c_0)^{\lfloor \frac{t-T}{t_0}\rfloor},\quad T\leq t<\infty,\quad\text{for all}\quad \mu\in\calP(\chi).
\end{equation}
In the latter case, that Assumption \ref{assum:combined Dobrushin reverse Dobrushin} is satisfied, we write $t_3>0$ for the time and $c_3>0$ for the constant given by \eqref{eq:t3 and c3 from DRD}. We then have the quantitative estimate
\begin{equation}\label{eq:quantitative uniform exponential convergence for arbitrary initial cond using DRD}
\Big\lvert\Big\lvert \frac{d\Law_{\mu}(X_t\lvert \tau_{\partial}>t)}{d\pi}-1\Big\rvert\Big\rvert_{L^{\infty}(\pi)}\leq \frac{2C_2(1-c_0)^{\lfloor \frac{t-(t_2+t_3)}{t_0}\rfloor}}{c_3-C_2(1-c_0)^{\lfloor \frac{t-(t_2+t_3)}{t_0}\rfloor}},\quad t_2+t_3\leq t<\infty,
\end{equation}
\eqref{eq:quantitative uniform exponential convergence for arbitrary initial cond using DRD} being understood to apply only when the denominator on the right is positive. 
\end{theo}

\subsection*{Application to the $Q$-process}

We suppose that $(X_t)_{0\leq t<\tau_{\partial}}$ satisfies Assumption \ref{assum:technical assumption for Assum (A)} and \cite[Assumption (A)]{Champagnat2014}. We recall that these imply that the $Q$-process defined in \cite[Theorem 3.1]{Champagnat2014} exists, is unique, and is exponentially ergodic, and that $(X_t)_{0\leq t<\tau_{\partial}}$ has a unique QSD, $\pi$. The $Q$-process and its stationary distribution are denoted $(Z_t)_{0\leq t<\infty}$ and $\beta$ respectively. We recall that its submarkovian transition kernel, $(Q_t)_{t\geq 0}$, is given by \eqref{eq:Q process Markov kernel}. We make the following observation.

\begin{observ}\label{observation: Reverse dobrushin implies reverse dobrushin for Q-process}
If $(X_t)_{0\leq t<\tau_{\partial}}$ and $\pi$ satisfy Assumption \ref{assum:Dobrushin reverse time}, then we observe that
\[
\begin{split}
\beta(dx)Q_{t_0}(x,dy)=h(x)\pi(dx)\frac{h(y)\lambda^{-t_0}}{h(x)}P_{t_0}(x,dy)=\pi(dx)\lambda^{-t_0}\pi(dx)h(y)P_{t_0}(x,dy)\\
=\lambda^{-t_0}h(y)\pi(dy)R(y,dx)=\beta(dy)\lambda^{-t_0}R(y,dx),
\end{split}
\]
where $t_0>0$ is the time given by Assumption \ref{assum:Dobrushin reverse time}. It is then immediate to see that the $Q$-process must itself satisfy Assumption \ref{assum:Dobrushin reverse time} (its QSD being the stationary distribution $\beta$) with the reverse kernel $\lambda^{-t_0}R$.
\end{observ}

We now use Observation \ref{observation: Reverse dobrushin implies reverse dobrushin for Q-process} to obtain results analogous to theorems \ref{theo:Linfty convergence} and \ref{theo:consequences of dominated assumption for whole space general results} for the $Q$-process. 

\begin{cor}\label{cor:L infty convergence of Q-process}
We suppose that $(X_t)_{0\leq t<\tau_{\partial}}$ satisfies Assumption \ref{assum:technical assumption for Assum (A)} and \cite[Assumption (A)]{Champagnat2014}. We also assume that $(X_t)_{0\leq t<\tau_{\partial}}$ and its unique QSD $\pi$ satisfy Assumption \ref{assum:Dobrushin reverse time}. The constant $0<c_0\leq 1$ and time $t_0>0$ are those given by Assumption \ref{assum:Dobrushin reverse time}. Then for all $\mu\in\calP_{\infty}(\beta)$ we have that
\begin{equation}\label{eq:uniform exponential L infty for Q process bded initial cond}
\Big\lvert\Big\lvert \frac{d\Law_{\mu}(Z_t)}{d\beta}-1\Big\rvert\Big\rvert_{L^{\infty}(\beta)}\leq (1-c_0)^{\lfloor \frac{t}{t_0}\rfloor} \osc_{\beta}(\mu),\quad 0\leq t<\infty.
\end{equation}

We now assume, in addition, that $(X_t)_{0\leq t<\tau_{\partial}}$ and $\pi$ satisfy Assumption \ref{assum:assumption for dominated by pi after certain time}. Then there exists a time $T<\infty$ (which does not depend on the initial condition $\mu\in\calP(\chi)$) such that $\Law_{\mu}(Z_t)\in\calP_{\infty}(\beta)$ for all $\mu\in\calP(\chi)$, with the density $\frac{d\Law_{\mu}(Z_t)}{d\beta}$ satisfying
\begin{equation}\label{eq:uniform exponential L infty for Q process convergence for arbitrary initial cond}
\Big\lvert\Big\lvert \frac{d\Law_{\mu}(Z_t)}{d\beta}-1\Big\rvert\Big\rvert_{L^{\infty}(\beta)}\leq (1-c_0)^{\lfloor \frac{t-T}{t_0}\rfloor},\quad T\leq t<\infty,\quad\text{for all}\quad \mu\in\calP(\chi).
\end{equation}
\end{cor}

\section{Proof of the results of Section \ref{section:Reverse Dobrushin condition}}\label{section:general state space proof}

The proof of Theorem \ref{theo:Linfty convergence} shall hinge on consideration of a semigroup $(T_t)_{t\geq 0}$ of bounded operators on $L^{\infty}(\pi)$, and the adjoint semigroup $(T^{\dag}_t)_{t\geq 0}$ of bounded operators on $L^1(\pi)$. 

In general, given an absorbed Markov process, it is not automatic that it should have some time-reversal at quasi-stationarity, even if there exists a time-reverse kernel over some specific time interval. The semigroups $(T_t)_{t\geq 0}$ and $(T^{\dag}_t)_{t\geq 0}$, on the other hand, do always exist. If there exists a time-reversal over any given time horizon, $t>0$, this time-reversal corresponds to $T^{\dag}_t$ (see Proposition \ref{prop:time-reversal implies T Dobrushin} for a precise statement of this). The semigroup $T^{\dag}_t(\mu)$ is only defined for $\mu\ll \pi$, however, so this doesn't necessarily represent a Markov process. It is similar enough to a Markov process, however, for the classical proof of Dobrushin's criterion to be applied (see Proposition \ref{prop:Tt conv implies Pt conv}).

We firstly summarise a number of propositions, \ref{prop:Tt well-defined}-\ref{prop:time-reversal implies T Dobrushin}, concerned with these semigroups. These shall be proven in turn, before concluding the proof of Proposition \ref{prop:R constant mass} and Theorem \ref{theo:Linfty convergence}. We shall then establish theorems \ref{theo:dominated by pi theorem general results} and \ref{theo:consequences of dominated assumption for whole space general results}, followed by proving Theorem \ref{theo:DRD lower bounds density}, before concluding with a proof of Corollary \ref{cor:L infty convergence of Q-process}. The first proposition defines the semigroups $(T_t)_{t\geq 0}$ and $(T^{\dag}_t)_{t\geq 0}$, establishing they are well-defined.

We recall the notation \eqref{eq:f mu notation}, which shall be used extensively in this section: for $\mu\in\calM_{\geq 0}(\chi)$ and $f\in L^1(\chi)$ we write $f\mu$ for the unique measure $\nu\in\calM(\mu)$ such that $\frac{d\nu}{d\mu}=f$.

\begin{prop}\label{prop:Tt well-defined}
We have the following:
\begin{enumerate}
\item\label{enum:Tt semigroup prop well-def}
The family of maps,
\begin{equation}
T_t:L^{\infty}(\pi)\ni f\mapsto \frac{dP_t(f\pi,\cdot)}{dP_t(\pi,\cdot)}\in L^{\infty}(\pi),\quad 0\leq t<\infty,
\end{equation}
defines a semigroup of bounded operators on $L^{\infty}(\pi)$ with operator norm $\lvert\lvert T_t\rvert\rvert_{\text{op}}=1$ for all $t\geq 0$, such that $T_t(L^{\infty}_{\geq 0}(\pi))\subseteq L^{\infty}_{\geq 0}(\pi)$ with $T_t(1)=1$.
\item\label{enum:Tt dag semigroup prop well-def}
The family of maps,
\begin{equation}
T^{\dag}_t:\calM(\pi)\ni \mu\mapsto \Big(\calB(\chi)\ni A\mapsto\mu\Big(\frac{dP_t(\Ind_A\pi,\cdot)}{dP_t(\pi,\cdot)}\Big)\Big)\in \calM(\pi),\quad 0\leq t<\infty,
\end{equation}
defines a semigroup of bounded operators on $\calM(\pi)$ with operator norm $\lvert\lvert T^{\dag}_t\rvert\rvert_{\text{op}}=1$ for all $t\geq 0$, such that $T^{\dag}_t(\calP(\pi))\subseteq \calP(\pi)$.
\item\label{enum:consist semigroup prop well-def}
For all $\mu\in \calM(\pi)$ and $f\in L^{\infty}(\pi)$ we have
\begin{equation}\label{eq:consistency semigroup}
(T^{\dag}_t\mu)(f)=\mu(T_t(f)),\quad 0\leq t<\infty.
\end{equation}
\end{enumerate}
\end{prop}
Thus for $\mu\in\calM(\pi)$ we can write $\mu T_t$ for the measure $T_t^{\dag}\mu$, $\mu T_t(A)$ being understood to mean $\mu(T_t(\Ind_A))=(T_t^{\dag}\mu)(A)$ for $A\in\mathscr{B}(\chi)$. Note that $\mu T_t f$ is unambiguous for $\mu\in\calM(\pi)$ and $f\in L^{\infty}(\pi)$ by \eqref{eq:consistency semigroup}. 

The following proposition demonstrates that $T_t^{\dag}$ can be expressed in terms of $P_t$.
\begin{prop}\label{prop:formula for Tt in terms of Pt}
The semigroup $(T_t^{\dag})_{0\leq t<\infty}$ satisfies
\begin{equation}
T_t^{\dag}(\mu)=\lambda^{-t}P_t\Big(\frac{d\mu}{d\pi}\Big)\pi\quad \text{for all}\quad \mu\in\calM(\pi)\quad\text{and}\quad 0\leq t<\infty.
\end{equation}

\end{prop}

For $0< t<\infty$, we say $\beta\in\calP(\pi)$ is stationary for $T_t^{\dagger}$ if
\begin{equation}\label{eq:stationary for Ttdag}
T_t^{\dag}\beta=\beta.
\end{equation}

We say $\beta\in\calP(\pi)$ is stationary for $(T^{\dag}_t)_{t\geq 0}$ if it is stationary for $T^{\dag}_t$, for all $t>0$.

The following corollary, which is an immediate consequence of propositions \ref{prop:Tt well-defined} and \ref{prop:formula for Tt in terms of Pt}, establishes a correspondence between non-negative $L^1(\pi)$-right eigenfunctions for $P_t$ (normalised to have $\lvert\lvert \cdot\rvert\rvert_{L^1(\pi)}$-norm $1$) and stationary distributions for $T_t^{\dag}$ which are absolutely continuous with respect to $\pi$.

\begin{cor}\label{cor:stat dist for Ttdag right efn for P correspondence}
We fix $0\leq t<\infty$. If $\beta\in\calP(\pi)$ is stationary for $T_t^{\dagger}$, then the Radon-Nikodym derivative $\frac{d\beta}{d\pi}$, which satisfies $\frac{d\beta}{d\pi}\in L^1_{\geq 0}(\pi)$ with $\lvert\lvert \frac{d\beta}{d\pi}\rvert\rvert_{L^1(\pi)}=1$, is a non-negative $L^1(\pi)$-right eigenfunction of $P_t$ with eigenvalue $\lambda^t$,
\begin{equation}
P_t\Big(\frac{d\beta}{d\pi}\Big)=\lambda^t\frac{d\beta}{d\pi}.
\end{equation}

Conversely, we suppose that $\phi\in L^1_{\geq 0}(\pi)$ with $\lvert\lvert \phi\rvert\rvert_{L^1(\pi)}=1$ is a unit non-negative $L^1(\pi)$-right eigenfunction of $P_t$,
\begin{equation}
P_t\phi=k\phi\quad \text{for some}\quad k\in \Rm.
\end{equation}

Then the eigenvalue must be $\lambda^t$, $k=\lambda^t$, and $\beta$ defined by $\beta:=\phi\pi\in \calP(\pi)$ is stationary for $T^{\dag}_t$, that is $\beta$ satisfies \eqref{eq:stationary for Ttdag}.
\end{cor}

We now establish a correspondence between convergence in total variation for $T_t^{\dagger}$ and $L^{\infty}(\pi)$-Perron-Frobenius behaviour for $P_t$.
\begin{prop}\label{prop:Tt conv implies Pt conv}
We suppose that for some $\beta\in \calP(\pi)$ and some non-negative function of time $\epsilon(t)\in \Rm_{\geq 0}$ we have
\begin{equation}\label{eq:assum for prop control on Tt}
\sup_{\mu\in\calP(\pi)}\lvert\lvert \mu T_t-\beta\rvert\rvert_{\TV}\leq \epsilon(t).
\end{equation}
Then defining $\phi:=\frac{d\beta}{d\pi}\in L^1(\pi)$ we have
\begin{equation}\label{eq:conv in Linfty as conseq of T conv optimised}
\Big\lvert\Big\lvert \lambda^{-t}\frac{d(\mu P_t)}{d\pi}-\mu(\phi)\Big\rvert\Big\rvert_{L^{\infty}(\pi)}\leq \epsilon(t)\frac{\osc_{\pi}(\mu)}{2}\quad\text{for all}\quad \mu\in\calP_{\infty}(\pi).
\end{equation}
\end{prop}

We would therefore like a criterion for uniform exponential convergence to a unique stationary distribution for $T_t^{\dagger}$. Since $T_t^{\dagger}\mu$ is undefined for probability measures $\mu$ which are not absolutely continuous with respect to $\pi$, it does not necessarily correspond to a Markov process, so we cannot apply Dobrushin's criterion directly. The following proposition establishes that Dobrushin's criterion may be applied to $T_t^{\dag}$, the proof of which is essentially the same as the classical coupling proof.
\begin{prop}[Dobrushin criterion]\label{prop:Dobrushin for Tt}
We suppose that there exists $\nu\in \calP(\chi)$, $t_0>0$ and $c_0>0$ such that
\begin{equation}\label{eq:Tt minorised by nu}
T^{\dag}_{t_0}\mu\geq c_0\nu\quad\text{for all}\quad  \mu\in\calP(\pi).
\end{equation}
Then there exists $\beta\in \calP(\pi)$ which is stationary for $(T_t^{\dag})_{t\geq 0}$, and which is the unique stationary distribution for $T_t^{\dag}$, for all $t>0$. Moreover we have uniform exponential convergence to this stationary distribution,
\begin{equation}
\lvert\lvert T_t^{\dag}\mu-\beta\rvert\rvert_{\TV}\leq 2(1-c_0)^{\lfloor \frac{t}{t_0}\rfloor}\quad\text{for all}\quad \mu\in\calP(\pi).
\end{equation}
\end{prop}

The following proposition gives a criterion providing for \eqref{eq:Tt minorised by nu}.

\begin{prop}\label{prop:time-reversal implies T Dobrushin}
We suppose that we have a time $t_0>0$ and non-negative kernel $R$  satisfying \eqref{eq:exis of reverse R}. Then for all Borel sets $A\in\mathscr{B}(\chi)$ we have
\begin{equation}\label{eq:Tt in terms of R}
R(y,A)=\lambda^{t_0}T_{t_0}\Ind_A(y)\quad\text{for $\pi$-almost every $y\in\chi$.}
\end{equation}

If in addition there exists $\nu\in\calP(\chi)$ and $c_0>0$ such that Assumption \ref{assum:Dobrushin reverse time} is satisfied, then \eqref{eq:Tt minorised by nu} is satisfied for this same constant $c_0>0$, time $t_0>0$ and probability measure $\nu$.
\end{prop}

\subsection*{Proof of Proposition \ref{prop:Tt well-defined}}

We begin by establishing Part \ref{enum:Tt semigroup prop well-def}. It is immediate that $T_t$ is a well-defined bounded linear endomorphism on $L^{\infty}(\pi)$, with operator norm at most $1$. It is also immediate that $T_t(1)=1$, so that $\lvert\lvert T_t\rvert\rvert_{\text{op}}=1$, and that $T_t(L^{\infty}_{\geq 0}(\pi))\subseteq L^{\infty}_{\geq 0}(\pi)$ for all $t\geq 0$. 

We now show that $(T_t)_{t\geq 0}$ is a semigroup. It is immediate that $T_0$ is the identity. We have left to establish $T_{t+s}=T_t T_s$ for all $0\leq t,s<\infty$. We firstly observe that
\[
T_s(f)\pi=\lambda^{-s}(f\pi)P_s.
\]
Therefore we have for all $f\in L^{\infty}(\pi)$,
\[
T_t(T_s(f))=\frac{dP_t(T_s(f)\pi,\cdot)}{dP_t(\pi,\cdot)}=\frac{dP_t(\lambda^{-s}(f\pi)P_s,\cdot)}{dP_t(\pi,\cdot)}=\frac{d[(f\pi)P_sP_{t}](\cdot)}{dP_{t+s}(\pi,\cdot)}=\frac{dP_{t+s}(f\pi,\cdot)}{dP_{t+s}(\pi,\cdot)}=T_{t+s}(f).
\]
This concludes  the proof of Part \ref{enum:Tt semigroup prop well-def}, so we now turn to Part \ref{enum:Tt dag semigroup prop well-def}.

We fix $0\leq t<\infty$. It is immediate that for all $\mu\in\calM_{\geq 0}(\pi)$, $T_t^{\dag}\mu$ is a finitely additive (non-negative) measure. 

We now prove countable additivity of $T_t^{\dag}\mu$ for $\mu\in\calM_{\geq 0}(\pi)$. We take disjoint $(A_k)_{k=1}^{\infty}$ and define
\[
A^n=\cup_{k=1}^nA_k,\quad A=\cup_{k=1}^{\infty}A_k,\quad f_n=\frac{dP_t(\Ind_{A^n}\pi,\cdot)}{dP_t(\pi,\cdot)},\quad f=\frac{dP_t(\Ind_{A}\pi,\cdot)}{dP_t(\pi,\cdot)}.
\]
Clearly for $0\leq n\leq m<\infty$ we have
\[
0\leq f_n\leq f_m\leq f\leq 1.
\]
Moreover by the monotone convergence theorem we have
\[
\begin{split}
\pi(f_n)=\lambda^{-t}P_t(\Ind_{A^n}\pi,\chi)=\lambda^{-t}\int_{\chi}\int_{\chi}\Ind(x\in A^n)P_t(x,dy)\pi(dx)\\
\nearrow \lambda^{-t}\int_{\chi}\int_{\chi}\Ind(x\in A)P_t(x,dy)\pi(dx)
=\lambda^{-t}P_t(\Ind_{A}\pi,\chi)=\pi(f)\quad\text{as}\quad n\ra\infty.
\end{split}
\]
Thus we have that $f_n\nearrow f$ $\pi$-almost everywhere as $n\ra\infty$, hence $\mu$-almost everywhere. Therefore by the monotone convergence theorem we have
\[
T_t^{\dag}\mu(A^n)=\mu(f_n)\ra \mu(f)=T_t^{\dag}\mu(A)\quad\text{as}\quad n\ra\infty.
\]
Thus $T_t^{\dag}\mu$ is countably additive for $\mu\in \calM_{\geq 0}(\pi)$. For general $\mu\in\calM(\pi)$, we can write $\mu=\mu_+-\mu_-$ for $\mu_+,\mu_-\in\calM_{\geq 0}(\pi)$. Then we can identify $T_t^{\dag}\mu=T_t^{\dag}\mu_+-T_t^{\dag}\mu_-$, so the countable additivity of $T_t\mu$ follows from that of $T_t^{\dag}\mu_+$ and $T_t^{\dag}\mu_-$.

The linearity of $T_t^{\dag}$ is immediate, as is the fact that $T_t^{\dag}\mu(1)=\mu(1)$ (so that $T_t^{\dag}(\calP(\pi))\subseteq \calP(\pi)$). Thus $\lvert\lvert T_t^{\dag}\rvert\rvert_{\text{op}}\geq 1$.

For general $\mu\in\calM(\pi)$, the Hahn decomposition theorem gives us disjoint $\mu_+,\mu_-\in \calM_{\geq 0}(\pi)$ such that $\mu=\mu_+-\mu_-$, with $\lvert \mu\rvert=\mu_++\mu_-$. Therefore we have
\[
\lvert\lvert T_t^{\dag}\mu\rvert\rvert_{\TV}\leq \lvert\lvert T_t^{\dag}\mu_+\rvert\rvert_{\TV}+\lvert\lvert T_t^{\dag}\mu_-\rvert\rvert_{\TV}=\lvert\lvert  \mu_+\rvert\rvert_{\TV}+\lvert\lvert \mu_-\rvert\rvert_{\TV}=\lvert\lvert \mu\rvert\rvert_{\TV},
\]
so that $\lvert\lvert T_t^{\dag}\rvert\rvert_{\text{op}}=1$.

To complete the proof of Part \ref{enum:Tt dag semigroup prop well-def}, it is left only to prove that $(T_t^{\dag})_{0\leq t<\infty}$ constitutes a semigroup. Prior to doing this, we prove Part \ref{enum:consist semigroup prop well-def}.

We fix $\mu\in\calM(\pi)$ and $t\geq 0$. We observe that \eqref{eq:consistency semigroup} is immediate if $f=\Ind_A$ for some $A\in\mathscr{B}(\chi)$, hence if $f$ is a simple function. We now take arbitrary $f\in L^{\infty}(\pi)$. We take a sequence of simple functions $f_n$ converging in $L^{\infty}(\pi)$ to $f$, so that $f_n$ converges to $f$ in $L^{\infty}(T_t^{\dag}\mu)$ (since $T_t^{\dag}\mu\ll \pi$) and $T_tf_n$ converges to $T_tf$ in $L^{\infty}(\pi)$, hence in $L^{\infty}(\mu)$ (since $\mu\ll \pi$). Therefore we have
\[
(T^{\dag}_t\mu)(f)=\lim_{n\ra\infty}(T^{\dag}_t\mu)(f_n)=\lim_{n\ra\infty}\mu(T_t(f_n))=\mu(T_t(f)).
\]
We have therefore established Part \ref{enum:consist semigroup prop well-def}.

It is left only to prove that $(T_t^{\dag})_{\geq 0}$ is a semigroup. It is immediate that $T^{\dag}_0$ is the identity. We fix $t,s\geq 0$. We may use Part \ref{enum:consist semigroup prop well-def} to calculate for arbitrary $\mu\in\calM(\pi)$ and $A\in\mathscr{B}(\chi)$ that
\[
(T^{\dag}_{t+s}\mu)(A)=\mu(T_{t+s}(\Ind_A))=\mu(T_s(T_t(\Ind_A)))=(T^{\dag}_s\mu)(T_t(\Ind_A))=(T^{\dag}_tT^{\dag}_s\mu)(A).
\]
Since $\mu$ and $A\in\mathscr{B}(\chi)$ were arbitrary, $T^{\dag}_{t+s}=T^{\dag}_tT^{\dag}_s$ for all $t,s\geq 0$.
\qed

\subsection*{Proof of Proposition \ref{prop:formula for Tt in terms of Pt}}

We fix $\mu\in\calM(\pi)$, $A\in \mathscr{B}(\chi)$ and $t\geq 0$. Then we calculate
\[
\begin{split}
(T_t^{\dag}\mu)(A)=\mu (T_t(\Ind_A))=\int_{\chi}\frac{d P_t(\Ind_A\pi,\cdot)}{dP_t(\pi,\cdot)}(y)\mu(dy)=\lambda^{-t}\int_{\chi}\frac{d P_t(\Ind_A\pi,\cdot)}{d\pi(\cdot)}(y)\frac{d\mu}{d\pi}(y)\pi(dy)
\\=\lambda^{-t}\int_{\chi}\frac{d\mu}{d\pi}(y)d P_t(\Ind_A\pi,dy)=\lambda^{-t}\int_{\chi}\int_{\chi}\Ind_A(x)\frac{d\mu}{d\pi}(y) P_t(x,dy)\pi(dx)
=\int_{A} \lambda^{-t}P_t\Big(\frac{d\mu}{d\pi}\Big)(x)\pi(dx).
\end{split}
\]
Therefore we have Proposition \ref{prop:formula for Tt in terms of Pt}. \qed

\subsection*{Proof of Proposition \ref{prop:Tt conv implies Pt conv}}

We fix arbitrary $A\in \mathscr{B}({\chi})$ and $\mu\in\calP_{\infty}(\pi)$. We use Proposition \ref{prop:formula for Tt in terms of Pt} to calculate
\[
\begin{split}
\lambda^{-t}\mu P_t(A)=\int_{\chi}\lambda^{-t}(P_t\Ind_A)(x)\frac{d\mu}{d\pi}(x)\pi(dx)=[\lambda^{-t}(P_t\Ind_A)\pi]\Big(\frac{d\mu}{d\pi}\Big)
\overset{\text{Proposition }\ref{prop:formula for Tt in terms of Pt}}{=}[(\Ind_A\pi)T_t]\Big(\frac{d\mu}{d\pi}\Big).
\end{split}
\]
We also have that $\mu(\phi)=\mu(\frac{d\beta}{d\pi})=\beta(\frac{d\mu}{d\pi})$ so that $\mu(\phi)\pi(A)=\pi(A)\beta(\frac{d\mu}{d\pi})$ hence
\[
[\lambda^{-t}\mu P_t-\mu(\phi)\pi](A)=[(\Ind_A\pi)T_t-\pi(A)\beta]\Big(\frac{d\mu}{d\pi}\Big).
\]

We take $a\in\Rm$ an arbitrary constant. We have that $[(\Ind_A\pi)T_t-\pi(A)\beta](a)=0$ so that
\[
[\lambda^{-t}\mu P_t-\mu(\phi)\pi](A)=[(\Ind_A\pi)T_t-\pi(A)\beta]\Big(\frac{d\mu}{d\pi}-a\Big).
\]

Therefore by Holder's inequality and \eqref{eq:assum for prop control on Tt} we have
\[
\lvert (\lambda^{-t}\mu P_t-\mu(\phi)\pi)(A)\rvert\leq \pi(A)\Big\lvert \Big\lvert \frac{(\Ind_A\pi)}{\pi(A)}T_t-\beta\Big\rvert\Big\rvert_{\TV}\Big\lvert\Big\lvert \frac{d\mu}{d\pi}-a\Big\rvert\Big\rvert_{L^{\infty}(\pi)}\leq \epsilon(t)\pi(A)\Big\lvert\Big\lvert \frac{d\mu}{d\pi}-a\Big\rvert\Big\rvert_{L^{\infty}(\pi)}.
\]
Optimising over $a$, we see that the right hand side is minimised by taking
\[
a=\frac{\esssup_{\pi}\frac{d\mu}{d\pi}+\essinf_{\pi}\frac{d\mu}{d\pi}}{2}.
\]

Since $A\in\mathscr{B}(\chi)$ is arbitrary we have \eqref{eq:conv in Linfty as conseq of T conv optimised}.
\qed

\subsection*{Proof of Proposition \ref{prop:Dobrushin for Tt}}

The following proof is essentially the same as the classical coupling-based proof of Dobrushin's criterion, rewritten so as not to make reference to a Markov process as we cannot assume $T_t^{\dag}$ corresponds to a Markov process (in particular, it is not defined for all initial probability measures).

We firstly observe that for any $\mu\in\calP(\pi)$ we have $\nu \leq \frac{1}{c_0}T^{\dag}_{t_0}\mu\in \calM_{\geq 0}(\pi)$ so that $\nu \in\calP(\pi)$. We define
\[
A\mu:=c_0\mu(1) \nu\quad\text{and}\quad D\mu:= T^{\dag}_{t_0}\mu-A\mu\quad\text{for}\quad  \mu\in \calM(\pi).
\]
We note that $A\mu,D\mu\in \calM_{\geq 0}(\pi)$ for $\mu \in \calM_{\geq 0}(\pi)$, with
\[
A\mu=c_0\mu(1)\nu\quad\text{and}\quad  (D\mu)(1)=(1-c_0)\mu(1)\quad\text{for}\quad \mu\in \calM_{\geq 0}(\pi).
\]
Thus we see that $AD^k$ is constant on $\calP(\pi)$, and that
\[
\lvert\lvert D\mu\rvert\rvert_{\TV}\leq (1-c_0)\lvert\lvert \mu\rvert\rvert_{\TV}\quad\text{for all}\quad \mu\in\calM(\pi).
\]
We can write 
\[
(T^{\dag}_{t_0})^n=\sum_{k=0}^{n-1}(T^{\dag}_{t_0})^{n-k-1}AD^k+D^n,
\]
so that 
\[
(T^{\dag}_{t_0})^n\mu_1-(T^{\dag}_{t_0})^n\mu_2=D^n(\mu_1-\mu_2)\quad\text{for}\quad \mu_1,\mu_2\in \calP(\pi).
\]
Therefore we may conclude that
\begin{equation}
\lvert\lvert (T^{\dag}_{t_0})^n\mu_1-(T^{\dag}_{t_0})^n\mu_2\rvert\rvert_{\TV}\leq (1-c_0)^n\lvert\lvert \mu_1-\mu_2\rvert\rvert_{\TV}\quad\text{for}\quad \mu_1,\mu_2\in\calP(\pi).
\end{equation}
Banach's fixed point theorem therefore implies the existence of a unique fixed point, $\beta\in \calP(\pi)$, of $T^{\dagger}_{t_0}$. We fix $t'\geq 0$. Since
\[
T^{\dag}_{t_0}T^{\dag}_{t'}\beta = T^{\dag}_{t'}T^{\dag}_{t_0}\beta =T^{\dag}_{t'}\beta,
\]
$T^{\dag}_{t'}\beta$ is a fixed point of $T^{\dag}_{t_0}$ so that $T^{\dag}_{t'}\beta=\beta$ by uniqueness. Therefore $T^{\dag}_t\beta=\beta$ for all $t\geq 0$.

Thus if $t=kt_0+\delta$ for $\delta\geq 0$ we have
\[
\lvert\lvert T^{\dag}_{t}\mu-\beta\rvert\rvert_{\TV}=\lvert\lvert T^{\dag}_{kt_0}T^{\dag}_{\delta}\mu-T^{\dag}_{kt_0}\beta\rvert\rvert_{\TV}\leq (1-c_0)^k\lvert\lvert T^{\dag}_{\delta}\mu-\beta\rvert\rvert_{\TV}\leq 2(1-c_0)^k\quad\text{for all}\quad \mu\in\calP(\pi).
\]
\qed

\subsection*{Proof of Proposition \ref{prop:time-reversal implies T Dobrushin}}

We fix arbitrary $A,B\in\mathscr{B}(\chi)$, and write $f=\Ind_A\in L^{\infty}(\pi)$. Using Proposition \ref{prop:formula for Tt in terms of Pt} we have
\[
\begin{split}
\int_B(T_{t_0}f)(y)\pi(dy)=(\Ind_B\pi)T_{t_0}(f)=\lambda^{-t}((P_{t_0}\Ind_B)\pi)(f)=\lambda^{-t}\int_{\chi}\int_{\chi}\Ind_B(y)f(x)P_{t_0}(x,dy)\pi(dx)\\
=\lambda^{-t}\int_{\chi}\int_{\chi}\Ind_B(y)f(x)R(y,dx)\pi(dy)=\lambda^{-t}\int_B(Rf)(y)\pi(dy).
\end{split}
\]
Therefore, since $B$ is arbitrary, we have \eqref{eq:Tt in terms of R}. In follows in particular that
\begin{equation}\label{eq:semigroup Tt in terms of R in proof}
T_{t_0}f(y)=\frac{T_{t_0}f(y)}{T_{t_0}1(y)}=\frac{Rf(y)}{R1(y)}\quad\text{$\pi$-almost everywhere.}
\end{equation}

For $A\in \mathscr{B}(\chi)$ we write
\[
\mathscr{E}_A:=\{y:\frac{R(y,A)}{R1(y)}=T_{t_0}\Ind_A(y)\}\quad\text{and}\quad \mathscr{G}_A:=\{y:\frac{R(y,A)}{R1(y)}\geq c_0\nu(A)\}.
\]

Whilst \eqref{eq:semigroup Tt in terms of R in proof} gives that $\mathscr{E}_A$ is a $\pi$-null set for all $A\in\mathscr{B}(\chi)$, we must be careful with respect to the fact that it may be a different $\pi$-null set for different $A\in \mathscr{B}(\chi)$.

We now fix arbitrary $A\in\mathscr{B}(\chi)$ and $\mu\in\calP(\pi)$, and set $f=\Ind_A$. Since \eqref{eq:Tt in terms of R} gives that $\pi(\mathscr{E}_A^c)=0$, $\mu(\mathscr{E}_A^c)=0$. Similarly $\pi(\mathscr{G}_A^c)=0$ implies that $\mu(\mathscr{G}_A^c)=0$. Therefore we have
\[
(T_{t_0}^{\dag}\mu)(A)=\mu T_{t_0} \Ind_A=\int_{\chi}\frac{R(y,A)}{R1(y)}\mu(dy)\geq \int_{\chi}c_0\nu(A)\mu(dy)=c_0\nu(A).
\]
Since $A\in\calB(\chi)$ is arbitrary, we have
\[
(T_{t_0}^{\dag}\mu)(\cdot)\geq c_0\nu(\cdot).
\]
\qed

\subsection*{Conclusion of the proof of Proposition \ref{prop:R constant mass} and Theorem \ref{theo:Linfty convergence}}

We begin by establishing Proposition \ref{prop:R constant mass}. We integrate \eqref{eq:exis of reverse R} to obtain
\[
\lambda^{t_0}\pi(dy)=R1(y)\pi(dy),
\]
whence we conclude \eqref{eq:mass of R constant}. \qed

We now turn to proving Theorem \ref{theo:Linfty convergence}. Using proposition \ref{prop:Dobrushin for Tt} and \ref{prop:time-reversal implies T Dobrushin}, we have the existence of a stationary distribution for $(T_t^{\dagger})_{t\geq 0}$, $\beta\in\calP(\pi)$, which for all $t>0$ is the unique stationary distribution for $T^{\dag}_t$. These also provide for \eqref{eq:assum for prop control on Tt} with $\epsilon(t)=2(1-c_0)^{\lfloor \frac{t}{t_0}\rfloor}$ and $\beta\in \calP(\pi)$ the aforementioned stationary distribution for $(T_t^{\dagger})_{t\geq 0}$.

Corollary \ref{cor:stat dist for Ttdag right efn for P correspondence} then implies that $\phi:=\frac{d\beta}{d\pi}\in L^1(\pi)$ is a non-negative $L^1(\pi)$-right eigenfunction for $P_t$ with eigenvalue $\lambda^t$, for all $t\geq 0$, such that $\lvert\lvert \phi\rvert\rvert_{\pi}=1$. 

We now fix $t>0$. Corollary \ref{cor:stat dist for Ttdag right efn for P correspondence} implies that $\phi$ is the unique unit non-negative $L^1(\pi)$-right eigenfunction for $P_t$ as another such $L^1(\pi)$-right eigenfunction, $\phi'\in L^1(\pi)$, would give rise to a different stationary distribution for $T_t^{\dagger}$, $\beta'=\phi'\pi$, contradicting the uniqueness of stationary distributions for $T_t^{\dagger}$.
We prove the following lemma in the appendix.
\begin{lem}\label{lem:eigenfunctions of eigenvalue lambda t difference of non-negative eigenfunctions}
We fix $0\leq t<\infty$. Any $L^1(\pi)$-right eigenfunction $\phi\in L^1(\pi)$ of $P_t$ of eigenvalue $\lambda^t$ must be the difference of two non-negative $L^1(\pi)$-right eigenfunctions $\phi_1,\phi_2\in L^1_{\geq 0}(\pi)$ of $P_t$ of eigenvalue $\lambda^t$, such that $\phi_1\wedge \phi_2=0$ $\pi$-almost everywhere.
\end{lem}
Lemma \ref{lem:eigenfunctions of eigenvalue lambda t difference of non-negative eigenfunctions} then implies that $\phi$ is the only $L^1(\pi)$-right eigenfunction of $P_t$ of eigenvalue $\lambda^t$ up to rescaling, as any other $L^1(\pi)$-right eigenfunction of eigenvalue $\lambda^t$ must be the difference of two non-negative $L^1(\pi)$-right eigenfunctions, hence the difference of two multiples of $\phi$.

Since we have \eqref{eq:assum for prop control on Tt} with $\epsilon(t)=2(1-c_0)^{\lfloor \frac{t}{t_0}\rfloor}$ and $\beta\in \calP(\pi)$ the aforementioned stationary distribution for $(T_t^{\dagger})_{t\geq 0}$, Proposition \ref{prop:Tt conv implies Pt conv} then implies \eqref{eq:Linfty malthusian theo statement}. We then obtain \eqref{eq:convergence of prob of killing} by integration.

The following formula, which holds for any $\mu\in\calP_{\infty}(\pi)$ and $0\leq t<\infty$, can be derived by simple algebraic manipulation,

\begin{equation}\label{eq:conditional distribution from killing prob and malthusian formula Linfty}
\Big\lvert\Big\lvert \frac{d\Law_{\mu}(X_t\lvert \tau_{\partial}>t)}{d\pi}-1\Big\rvert\Big\rvert_{L^{\infty}(\pi)}\leq \frac{\lvert\lvert \lambda^{-t}\frac{d\Pm_{\mu}(X_t\in \cdot)}{d\pi(\cdot)}-\mu(\phi)\rvert\rvert_{L^{\infty}(\pi)}+\lvert \mu(\phi)-\lambda^{-t}\Pm_{\mu}(\tau_{\partial}>t)\rvert}{\mu(\phi)-\lvert \mu(\phi)-\lambda^{-t}\Pm_{\mu}(\tau_{\partial}>t)\rvert}.
\end{equation}
We immediately obtain \eqref{eq:Linfty convergence of dist cond of survival} by using  \eqref{eq:conditional distribution from killing prob and malthusian formula Linfty} to combine \eqref{eq:Linfty malthusian theo statement} and \eqref{eq:convergence of prob of killing}.

We now also assume Assumption \ref{assum:assumption for positivity of L1 right efn}. We have that $\beta:=\phi \pi$ satisfies 
\[
\beta(dy) R(y,dx) =P_{t_0}(x,dy)\phi(y)\pi(dx),
\]
which by integrating implies that
\begin{equation}\label{eq:beta stationary for R}
\beta R =\lambda^{t_0}\beta.
\end{equation}
Using also Proposition \ref{prop:R constant mass}, it follows that
\[
\beta =\lambda^{-t_0}\beta R\geq c_0\nu,
\]
so that $\phi=\frac{d\beta}{d\pi}\geq c_0\frac{d\nu}{d\pi}$. We recall that $t(\mu)<\infty$ is the time, dependent upon $\mu\in\calP_{\infty}(\pi)$, assumed to exist in Assumption \ref{assum:assumption for positivity of L1 right efn}. For all $\mu\in \calP_{\infty}(\pi)$ we have that
\[
\mu(\phi)=\lambda^{-t(\mu)}\mu P_{t(\mu)}\phi\geq \frac{\lambda^{-t(\mu)}}{c_0}\mu P_{t(\mu)}\frac{d\nu}{d\pi}>0,
\]
whence we conclude that $\mu(\phi)>0$ for all $\mu\in\calP_{\infty}(\pi)$, so that $\phi\in L^1_{>0}(\pi)$.

We now assume Assumption \ref{assum:Dobrushin reverse time}, Assumption \ref{assum:technical assumption for Assum (A)} and \cite[Assumption A]{Champagnat2014}, but not Assumption \ref{assum:assumption for positivity of L1 right efn}. We let $h$ be the everywhere strictly positive pointwise right eigenfunction provided for by \cite[Proposition 2.3]{Champagnat2014}, which we observe must be a version of $\phi$. We use \cite[Theorem 2.1]{Champagnat2017} to see that there exists $C<\infty$, $\gamma'>0$ and $T<\infty$ such that
\begin{equation}\label{eq:killing probability if Assumption (A) satisfied}
\lvert \lambda^{-t}\Pm_{\mu}(\tau_{\partial}>t)-\mu(h)\rvert\leq Ce^{-\gamma't}\mu(h)\quad\text{for all}\quad t\geq T\quad\text{and all}\quad \mu\in\calP(\chi),
\end{equation}

where $\gamma'>0$ is the ``$\gamma>0$'' provided for by \cite[Theorem 2.1]{Champagnat2014}. We therefore obtain \eqref{eq:Linfty conv when also have Assum (A)} by using \eqref{eq:conditional distribution from killing prob and malthusian formula Linfty} to combine \eqref{eq:Linfty malthusian theo statement} and \eqref{eq:killing probability if Assumption (A) satisfied}, concluding the proof of Theorem \ref{theo:Linfty convergence}.
\qed

\subsection*{Proof of Theorem \ref{theo:dominated by pi theorem general results}}

We impose the assumptions of Theorem \ref{theo:dominated by pi theorem general results}. It follows from \eqref{eq:reverse kernel majorised for dominated by pi theorem} and Proposition \ref{prop:R constant mass} that $R(y,\cdot)\leq C_2'\lambda^{t_2}\pi(\cdot)$ for $\pi$-almost every $y\in\chi$. Since \eqref{eq:exis of reverse R} is satisfied, we have that
\[
\pi(dx)P_{t_2}(x,dy)\leq \pi(dy)C_2'\lambda^{t_2}\pi(dx).
\]

We now take arbitrary $f\in C_b(\chi;\Rm_{\geq 0})$. We have that $P_{t_2}f(x)\pi(dx)\leq \pi(f)C_2\pi(dx)$, where $C_2$ is given by \eqref{eq:formula for C2 RaD cond}, $C_2:=\lambda^{t_2}C_2'$. We therefore deduce that
\begin{equation}\label{eq:proof of dominated by pi theorem general inequality for f test fn}
P_{t_2}f(x)\leq \pi(f)C_2\quad\text{for $\pi$-almost every $x\in \chi$.}
\end{equation}
Since $\text{spt}(\pi)=\chi$, \eqref{eq:proof of dominated by pi theorem general inequality for f test fn} holds on a dense set of $x\in\chi$. It follows from the lower semicontinuity of $P_{t_2}$ that $P_{t_2}f$ is lower semicontinuous, whence \eqref{eq:proof of dominated by pi theorem general inequality for f test fn} must hold for every $x\in\chi$. Since $\chi$ is a metric space and $f\in C_b(\chi;\Rm_{\geq 0})$ is arbitrary, we have that $P_{t_1}(x,U)\leq C_1\lambda^{t_1}\pi(U)$ for all $x\in\chi$ and $U$ open in $\chi$. It follows from \cite[Theorem 1.2, p. 27]{Parthasarathy1968} that 
\[
P_{t_2}(x,\cdot)\leq C_2\pi(\cdot)\quad\text{for all}\quad x\in \chi.
\]
\qed

\subsection*{Proof of Theorem \ref{theo:DRD lower bounds density}}

We now impose the assumptions of Theorem \ref{theo:DRD lower bounds density}. Using Proposition \ref{prop:R constant mass}, we have that
\[
\pi(dx)P_{t_0}(x,dy)=\pi(dy)R(y,dx)\geq \pi(dy)\lambda^{t_0}c_0\nu(dx)=\pi_{\lvert_A}(dx)\frac{\lambda^{t_0}c_0}{\pi(A)}\pi(dy)
\]
We now fix $f\in \calB_b(\chi;\Rm_{\geq 0})$. It follows that
\[
P_{t_0}f(x)\geq \frac{\lambda^{t_0}c_0}{\pi(A)}\pi(f)\quad\text{for $\pi$-almost every $x\in A$}.
\]
We have from \eqref{eq:Dobrushin for DRD condition} that $P_{t_1}(x,\cdot)\geq c_1P_{t_1}1(x)\nu(\cdot)$ for all $x\in \chi$, implying that $\nu\ll \pi$. Since $\nu(A^c)=0$, it follows that $\nu\ll \pi_{\lvert_A}$, implying that $\nu P_{t_0}f\geq \frac{\lambda^{t_0}c_0}{\pi(A)}\pi(f)$. It then follows that
\[
P_{t_0+t_1}f(x)\geq c_1P_{t_1}1(x)\nu P_{t_0}f\geq P_{t_0+t_1}1(x)\frac{\lambda^{t_0}c_0c_1}{\pi(A)}\pi(f)\quad\text{for all}\quad x\in\chi.
\]
Since $f\in \calB_b(\chi;\Rm_{\geq 0})$ is arbitrary, we obtain \eqref{eq:minorised by pi equation from DRD}.
\qed

\subsection*{Proof of Theorem \ref{theo:consequences of dominated assumption for whole space general results}}

We now impose the assumptions of Theorem \ref{theo:consequences of dominated assumption for whole space general results}.

We define $\tilde{h}$ to be some fixed version of $\phi$, the non-negative $L^1(\pi)$-right eigenfunction shown to exist in Theorem \ref{theo:Linfty convergence}. For all $t\geq 0$, we have that
\[
\lambda^{-t}P_{t}\tilde{h}(x)=\tilde{h}(x)\quad\text{for $\pi$-almost every $x\in\chi$}.
\]
Since, by Theorem \ref{theo:dominated by pi theorem general results}, $P_{t_2}(x,\cdot)\leq C_2\lambda^{t_2}\pi(\cdot)$ for all $x\in \chi$, it follows that for all $t\geq 0$ we have
\[
\lambda^{-t}P_{t_2+t}\tilde{h}(x)=P_{t_2}\tilde{h}(x)\quad\text{for every $x\in\chi$}.
\]
We now define $h\in\calB(\chi;\Rm_{\geq 0})$ by $h(x):=\lambda^{-t_2}P_{t_2}\tilde{h}(x)$ for $x\in \chi$, which we observe must be a non-negative pointwise right eigenfunction for $(P_t)_{t\geq 0}$. Since $h$ must be a version of $\phi$, $\pi(h)=1$. Moreover we see that 
\[
h(x)\leq C_2\lambda^{-t_2}\pi(\tilde{h})=C_2\lambda^{-t_2}<\infty, 
\]
so that $h$ must be bounded.

We have from Theorem \ref{theo:dominated by pi theorem general results} that $\Pm_{\mu}(X_t\in \cdot)\leq C_2 \pi(\cdot)$ for all $\mu\in\calP(\chi)$. We note that \eqref{eq:Linfty malthusian theo statement} actually holds for all $\mu\in\calM_{\geq 0}(\pi)$ such that $\mu\ll_{\infty}\pi$, which may be seen by rescaling both sides. We may therefore obtain \eqref{eq:Linfty malthusian theo statement arbitrary ic} by applying \eqref{eq:Linfty malthusian theo statement} to the initial condition $\Pm_{\mu}(X_{t_1}\in \cdot)$. We then obtain \eqref{eq:convergence of prob of killing arbitrary ics} by integration, whence we obtain \eqref{eq:Linfty convergence of dist cond of survival arbitrary ics} by applying \eqref{eq:conditional distribution from killing prob and malthusian formula Linfty} (with $\phi$ replaced by $h$ in \eqref{eq:conditional distribution from killing prob and malthusian formula Linfty}).

We now assume either that Assumption \ref{assum:technical assumption for Assum (A)} and \cite[Assumption (A)]{Champagnat2014} are satisfied, or that Assumption \ref{assum:combined Dobrushin reverse Dobrushin} is satisfied. We begin by considering the former case. It follows from \cite[Theorem 2.1]{Champagnat2014} that
\[
\Law_{\mu}(X_t\lvert\tau_{\partial}>t)(h)\geq \pi(h)-\lvert\lvert \Law_{\mu}(X_t\lvert\tau_{\partial}>t)-\pi\rvert\rvert_{\TV}\lvert\lvert h\rvert\rvert_{\infty}\ra 1\text{ as $t\ra\infty$ uniformly in $\mu\in \calP(\chi)$}.
\]
In the latter case, Theorem \ref{theo:DRD lower bounds density} implies that
\begin{equation}\label{eq:conditional expectation of right efn pf of quantitative convergence from DRD}
\Law_{\mu}(X_{t_3}\lvert\tau_{\partial}>t_3)(h)\geq c_3\pi(h)=c_3\quad \text{for all}\quad \mu\in\calP(\chi),
\end{equation}
$c_3>0$ being the constant and $t_3>0$ the time given by \eqref{eq:t3 and c3 from DRD}.

In either case, we may choose $T'<\infty$ and $\epsilon'>0$ such that $\Law_{\mu}(X_{T'}\lvert\tau_{\partial}>{T'})(h)\geq \epsilon'$ for all $\mu\in\calP(\chi)$. We then obtain \eqref{eq:uniform exponential L infty convergence for arbitrary initial cond} by applying \eqref{eq:Linfty convergence of dist cond of survival arbitrary ics} from time $T'<\infty$ onwards.

In the latter case - that Assumption \ref{assum:combined Dobrushin reverse Dobrushin} is satisfied - using \eqref{eq:conditional expectation of right efn pf of quantitative convergence from DRD} we apply \eqref{eq:Linfty convergence of dist cond of survival arbitrary ics} from time $t_3$ onwards to obtain \eqref{eq:quantitative uniform exponential convergence for arbitrary initial cond using DRD}.
\qed

\subsection*{Proof of Corollary \ref{cor:L infty convergence of Q-process}}

We define $\pi$ to be the unique QSD provided for by \cite[Theorem 2.1]{Champagnat2014}, associated to which is the time-$1$ eigenvalue $\lambda:=\lambda(\pi)=\Pm_{\pi}(\tau_{\partial}>1)$. Moreover, \cite[Proposition 2.3]{Champagnat2014} then provides for a unique (up to renormalisation) everywhere strictly positive, bounded, pointwise right eigenfunction for $(P_t)_{t\geq 0}$, which we denote as $h$. We normalise $h$ so that $\pi(h)=1$. The stationary distribution of the $Q$-process, which we denote as $\beta$, is then given by $\beta(dx) =h(x)\pi(dx)$, by \cite[Theorem 3.1 (iii)]{Champagnat2014}.

We immediately obtain \eqref{eq:uniform exponential L infty for Q process bded initial cond} by combining Observation \ref{observation: Reverse dobrushin implies reverse dobrushin for Q-process} with Theorem \ref{theo:Linfty convergence}.

We now turn to establishing \eqref{eq:uniform exponential L infty for Q process convergence for arbitrary initial cond}. In the following, $x\in \chi$ is arbitrary. Using \eqref{eq:Q process Markov kernel}, we calculate that the transition kernel $Q_t$ of the $Q$-process satisfies
\begin{equation}\label{eq:equation for Q process kernel dominated by beta calculation}
\begin{split}
Q_t(x,dy)=\frac{\lambda^{-t}h(y)P_t(x,dy)}{h(x)}=\frac{h(y)\Pm_x(X_t\in dy)}{\expE_x[h(X_t)]}\\
=\frac{h(y)\Pm_x(X_t\in dy\lvert \tau_{\partial}>s)}{\lambda^{t-s}\expE_x[h(X_s)\lvert \tau_{\partial}>s]}\quad \text{for all}\quad 0\leq s\leq t<\infty.
\end{split}
\end{equation}

We let $\bar{t}_0>0$, $\bar c_0>0$ and $\bar \nu\in\calP(\chi)$ be the time, constant and probability measure respectively for which $(X_t)_{0\leq t<\infty}$ satisfies \cite[Assumption (A1)]{Champagnat2014}. We define $t_2>0$ to be the time for which $(X_t)_{0\leq t<\infty}$ and $\pi$ satisfy Assumption \ref{assum:assumption for dominated by pi after certain time}. Then we have from Theorem \ref{theo:dominated by pi theorem general results} that
\[
\Pm_x(X_{\bar{t}_0+t_2}\in \cdot\lvert \tau_{\partial}>\bar{t}_0)\leq C_2\pi(\cdot),
\]
where $C_2<\infty$ is the constant given by \eqref{eq:formula for C2 RaD cond}. We also have that
\[
\expE_x[h(X_{\bar{t}_0})\lvert \tau_{\partial}>\bar{t}_0]\geq \bar c_0\bar\nu(h)>0.
\]
Combining these with \eqref{eq:equation for Q process kernel dominated by beta calculation}, we obtain that
\[
Q_{\bar{t}_0+t_2}(x,dy)\leq \frac{h(y)C_2\pi(dy)}{\bar c_0\lambda^{t_2}\bar \nu(h)}=\frac{C_2}{\bar c_0\lambda^{t_2}\bar\nu(h)}\beta(dy)\quad\text{for all}\quad x\in\chi.
\]

Combining this with \eqref{eq:uniform exponential L infty for Q process bded initial cond}, we immediately obtain \eqref{eq:uniform exponential L infty for Q process convergence for arbitrary initial cond}.
\qed

\section{Proof of the results of Section \ref{section:main results}}\label{section:Euclidean state space proof}

We shall firstly prove the following theorem.

\begin{theo}\label{theo:theo for reverse Dobrushin in Euclidean space}
We suppose that the killed Markov process $(X_t)_{0\leq t<\tau_{\partial}}$ has a QSD $\pi$ which has an essentially bounded density with respect to $\Lambda$, $\pi\in\calP_{\infty}(\Lambda)$. We further assume that $(X_t)_{0\leq t<\tau_{\partial}}$ satisfies Assumption \ref{assum:adjoint Dobrushin main results section}. We let $c_0'>0$ be the constant, $t_0>0$ be the time, $\psi\in \calB_{b,\gg}(\chi)$ be the function, and $\nu'\in \calP(\chi)$ be the probability measure for which $(X_t)_{0\leq t<\tau_{\partial}}$ satisfies Assumption \ref{assum:adjoint Dobrushin main results section} (the latter denoted ``$\nu$'' in the statement of that assumption). Then $\nu'\in \calP(\Lambda)$ so that $c_0>0$ given by \eqref{eq:formula for c0 main results section} is unambiguous, as is the probability measure
\begin{equation}\label{eq:prob measure coming from adjoint dobrushin}
\nu:=\frac{\frac{d\pi}{d\Lambda}\nu'}{\nu'(\frac{d\pi}{d\Lambda})}\in \calP(\Lambda).
\end{equation}
We recall from \eqref{eq:formula for c0 main results section} that $c_0$ is given by 
\[
c_0:=\frac{c_0'\nu'(\frac{d\pi}{d\Lambda})}{\lvert\lvert \psi\rvert\rvert_{\infty}\lvert\lvert \frac{1}{\psi}\rvert\rvert_{\infty}\lvert\lvert \frac{d\pi}{d\Lambda}\rvert\rvert_{L^{\infty}(\Lambda)}}\in (0,1].
\]

Then $(X_t)_{0\leq t<\tau_{\partial}}$ and $\pi$ satisfy Assumption \ref{assum:Dobrushin reverse time}, with the above constant $c_0>0$ given by \eqref{eq:formula for c0 main results section}, probability measure $\nu$ given by \eqref{eq:prob measure coming from adjoint dobrushin}, and the same time $t_0>0$.
\end{theo}

Theorem \ref{theo:Linfty convergence main results section} then immediately follows from Theorem \ref{theo:Linfty convergence} and Theorem \ref{theo:theo for reverse Dobrushin in Euclidean space}. 

We shall then prove proposition \ref{prop:boundedness from reverse Dobrushin main results section}, after which we shall prove the following theorem.
\begin{theo}\label{theo:theorem for bounded by pi in Euclidean space}
We suppose that $(X_t)_{0\leq t<\tau_{\partial}}$ has a QSD $\pi$ which is absolutely continuous with respect to $\Lambda$, $\pi\ll\Lambda$, and which has full support, $\text{spt}(\pi)=\chi$. We assume that $(X_t)_{0\leq t<\tau_{\partial}}$ satisfies assumptions \ref{assum:adjoint Dobrushin main results section} and \ref{assum:adjoint anti-Dobrushin main results section}. 

We let $t_0>0$ and $t_1>0$ respectively be the times, and $\psi_0$ and $\psi_1$ respectively be the functions, for which Assumption \ref{assum:adjoint Dobrushin main results section} and Assumption \ref{assum:adjoint anti-Dobrushin main results section} are satisfied. We define $\lambda:=\lambda(\pi)=\Pm_{\pi}(\tau_{\partial}>1)$ and $t_2:=t_0+t_1$. The constants $c_0'>0$, $a_1>0$ and $C_1<\infty$ are respectively the constants for which we have \eqref{eq:Dobrushin for P tilde kernel crit for reverse dobrushin main results section}, \eqref{eq:psi adjoint for verifying bounded pi euclidean condition main results section} and \eqref{eq:bounded Lebesgue main results section}. Finally $\nu'$ is the probability measure for which we have \eqref{eq:Dobrushin for P tilde kernel crit for reverse dobrushin main results section}. It follows from Assumption \ref{assum:adjoint Dobrushin main results section} that $\nu'(\frac{d\pi}{d\Lambda})$ is unambiguous and strictly positive.

We finally assume that $P_{t_2}$ is lower semicontinuous in the sense of Definition \ref{defin:lower semicts kernel}. Then $(X_t)_{0\leq t<\tau_{\partial}}$ and $\pi$ satisfy Assumption \ref{assum:assumption for dominated by pi after certain time}, for the time $t_2:=t_0+t_1>0$ and $C_2'<\infty$ given by
\begin{equation}\label{eq:formula for C2' AD and AaD implies RaD}
C_2':=\frac{\lvert\lvert \frac{\psi_0}{\psi_1}\rvert\rvert_{\infty}\lvert\lvert \frac{\psi_1}{\psi_0}\rvert\rvert_{\infty}\lvert\lvert\psi_1\rvert\rvert_{\infty}\lvert\lvert \frac{1}{\psi_1}\rvert\rvert_{\infty}a_1C_1}{c_0'\lambda^{t_1}\nu'(\frac{d\pi}{d\Lambda})}.
\end{equation}
\end{theo}

Having established theorems \ref{theo:theo for reverse Dobrushin in Euclidean space} and \ref{theo:theorem for bounded by pi in Euclidean space}, Theorem \ref{theo:dominated by pi theorem general results main results section}, Theorem \ref{theo:DAD lower bounds density main results section}, Theorem \ref{theo:uniform Linfty convergence main results section}, and Corollary \ref{cor:L infty convergence of Q-process main results section} then immediately follow from Theorem \ref{theo:dominated by pi theorem general results}, Theorem \ref{theo:DRD lower bounds density}, Theorem \ref{theo:consequences of dominated assumption for whole space general results} and Corollary \ref{cor:L infty convergence of Q-process} respectively.

We shall then establish Theorem \ref{theo:criterion for cty of QSD}, Theorem \ref{theo:criterion for right efn main results section} and Proposition \ref{prop:Lp malthusian by interpolation}.

\subsection*{Proof of Theorem \ref{theo:theo for reverse Dobrushin in Euclidean space}}

We let $\rho\in\calB_b(\chi;\Rm_{\geq 0})$ be a version of $\frac{d\pi}{d\Lambda}$ which is everywhere bounded by $\lvert\lvert \rho\rvert\rvert_{\infty}=\lvert\lvert \frac{d\pi}{d\Lambda}\rvert\rvert_{L^{\infty}(\Lambda)}$. It follows from \eqref{eq:psi adjoint for verifying reverse Dobrushin main results section} and \eqref{eq:Dobrushin for P tilde kernel crit for reverse dobrushin main results section} that $\nu'\ll \Lambda$, so that $\nu'(\rho)>0$ since $\nu'$ and $\pi$ are not mutually singular. We define the non-negative kernel
\begin{equation}\label{eq:reverse kernel from Euclidean criterion proof}
R(y,dx):=\begin{cases}
0,\quad \rho(y)=0\\
\frac{\rho(x)\psi(y)a}{\rho(y)\psi(x)}\tilde{P}(y,dx),\quad \rho(y)>0
\end{cases}.
\end{equation}
We now fix $y\in \chi$ such that $\rho(y)>0$ (which is $\pi$-almost every $y\in\chi$). We take $X\sim \tilde{P}(y,\cdot)$ and calculate that
\[
\begin{split}
\frac{R(y,dx)}{R1(y)}=\frac{\frac{\psi(y)a}{\rho(y)}\expE\big[\big(\frac{\rho}{\psi}\big)(X)\Ind(X\in dx)\big]}{\frac{\psi(y)a}{\rho(y)}\expE\big[\big(\frac{\rho}{\psi}\big)(X)\big]}=\frac{\expE\big[\big(\frac{\rho}{\psi}\big)(X)\Ind(X\in dx)\big\lvert X\in \chi\big]}{\expE\big[\big(\frac{\rho}{\psi}\big)(X)\big\lvert X\in \chi\big]}\geq \frac{c_0'\nu'(\rho)}{\lvert\lvert \psi\rvert\rvert_{\infty}\lvert\lvert \frac{\rho}{\psi}\rvert\rvert_{\infty}}\frac{\rho(x)\nu'(dx)}{\nu'(\rho)}.
\end{split}
\]
It follows that $R$ satisfies \eqref{eq:reverse time R minorised by nu} with $c_0$ given by \eqref{eq:formula for c0 main results section} and $\nu$ given by \eqref{eq:prob measure coming from adjoint dobrushin}.

We now calculate (using that $\rho(y)>0$ $\pi$-almost everywhere) that
\[
\pi(dy)R(y,dx)=\frac{\rho(x)}{\psi(x)}a\psi(y)\Lambda(dy)\tilde{P}(y,dx)=\rho(x)\Lambda(dx)P(x,dy)=\pi(dx)P(x,dy).
\]
Thus $R$ is a non-negative kernel satisfying \eqref{eq:exis of reverse R}, so by Proposition \ref{prop:R constant mass} it is a submarkovian kernel satisfying \eqref{eq:exis of reverse R}, possibly after adjusting the defintion of $R(y,\cdot)$ on a $\pi$-null set of $y\in\chi$. 
\qed

\subsection*{Proof of Proposition \ref{prop:boundedness from reverse Dobrushin main results section}}

We define the probability measure
\[
\mu:=\frac{\Big(\frac{d\pi}{d\Lambda}\wedge 1\Big)\Lambda}{\Lambda\Big(\frac{d\pi}{d\Lambda}\wedge 1\Big)}.
\]
We observe that $\mu\in\calP_{\infty}(\pi)\cap \calP_{\infty}(\Lambda)$. We let $\phi\in L^1_{\geq 0}(\pi)$ be the non-negative $L^1(\pi)$-right eigenfunction whose existence is assumed by the proposition. Since $\pi(\phi)=1$, it follows that $\mu(\phi)>0$. We have from \eqref{eq:conv to QSD for propn that it implies QSD has a bounded density} that $\Law_{\mu}(X_t\lvert \tau_{\partial}>t)\geq \frac{\pi}{2}$ for all $t$ large enough. 

It therefore follows from \eqref{eq:psi adjoint for verifying reverse Dobrushin main results section} that
\[
(\Lambda P_{t_0})(dy)= a\psi(y)\tilde{P}\Big(\frac{1}{\psi}\Big)(y)\Lambda(dy)\leq a\lvert\lvert\psi\rvert\rvert_{\infty}\lvert\lvert\frac{1}{\psi}\rvert\rvert_{\infty}\Lambda(dy).
\]
It follows that $\Law_{\mu}(X_{nt_0}\lvert \tau_{\partial}>nt_0)\ll_{\infty}\Lambda$ for any $n\in \Nm$. Thus we may choose $n\in \Nm$ such that
\[
\pi\ll_{\infty}\Law_{\mu}(X_{nt_0}\lvert \tau_{\partial}>nt_0)\ll_{\infty}\Lambda.
\]
\qed

\subsection*{Proof of Theorem \ref{theo:theorem for bounded by pi in Euclidean space}}

We write $\tilde{P}$ for the submarkovian kernel, $\psi_0\in \calB_{b,\gg}(\chi)$ for the function, $t_0>0$ for the time, $\nu'\in\calP(\chi)$ for the probability measure, and $a,c_0'>0$ for the constants for which Assumption \ref{assum:adjoint Dobrushin main results section} is satisfied ($c_0'>0$ being the constant for which \eqref{eq:Dobrushin for P tilde kernel crit for reverse dobrushin main results section} is satisfied). It follows from Assumption \ref{assum:adjoint Dobrushin main results section} that $\nu'\ll \Lambda$, so that $\nu'(\frac{d\pi}{d\Lambda})$ is unambiguous. We further write $\tilde{P}^{(1)}$ for the submarkovian kernel, $\psi_1\in \calB_{b,\gg}(\chi)$ for the function, $t_1>0$ for the time, and $a_1>0$, $C_1<\infty$ for the constants for which Assumption \ref{assum:adjoint anti-Dobrushin main results section} is satisfied ($C_1<\infty$ being the constant for which \eqref{eq:bounded Lebesgue main results section} is satisfied).

It follows from \eqref{eq:psi adjoint for verifying reverse Dobrushin main results section} that $\tilde{P}^{(0)}$ defined to be
\[
\tilde{P}^{(0)}(y,dx):=\frac{\psi_0(y)\psi_1(x)a}{\psi_0(x)\psi_1(y)A}\tilde{P}(y,dx),
\]
for some $A>0$, satisfies
\begin{align}
\psi_1(x)\Lambda(dx)P_{t_0}(x,dy)=A\psi_1(y)\Lambda(dy)\tilde{P}^{(0)}(y,dx),\\
\frac{\tilde{P}^{(0)}(y,dx)}{\tilde{P}^{(0)}1(y)}\geq \frac{c_0'}{\lvert\lvert\frac{\psi_0}{\psi_1}\rvert\rvert_{\infty}\lvert\lvert \frac{\psi_1}{\psi_0}\rvert\rvert_{\infty}}\nu'(dx).\label{eq:lower bound for tilde P (0) by nu' C2 calculation}
\end{align}
For $A>0$ sufficiently large, $\tilde{P}^{(0)}$ is submarkovian. We fix such an $A$, thereby fixing $\tilde{P}^{(0)}$. We then define
\[
\tilde{P}^{(2)}:=\tilde{P}^{(0)}\tilde{P}^{(1)},\quad a_2:=Aa_1
\]
It follows that
\begin{equation}\label{eq:adjoint equation for proof of sufficient for RaD condition}
\begin{split}
a_2\Lambda(dy)\psi_1(y)\tilde{P}^{(2)}(y,dx)=\int_{z\in \chi}a_1\tilde{P}^{(1)}(z,dx)A\psi_1(y)\Lambda(dy)\tilde{P}^{(0)}(y,dz)\\
=\int_{z\in \chi}a_1\tilde{P}^{(1)}(z,dx)\psi_1(z)P_{t_0}(z,dy)\Lambda(dz)=\int_{z\in \chi}\psi_1(x)P_{t_0}(z,dy)\Lambda(dx)P_{t_1}(x,dz)\\
=\psi_1(x)\Lambda(dx)P_{t_2}(x,dy).
\end{split}
\end{equation}
We also have that
\begin{equation}\label{eq:bound for P(2) calculation of C2 constant}
\begin{split}
\tilde{P}^{(2)}(y,dx)\leq C_1\Lambda(dx)\tilde{P}^{(0)}1(y).
\end{split}
\end{equation}

We now let $\rho\in\calB(\chi;\Rm_{\geq 0})$ be a non-negative version of $\frac{d\pi}{d\Lambda}$ (which we do not assume to be bounded). As in the proof of Theorem \ref{theo:theo for reverse Dobrushin in Euclidean space} we define the non-negative reverse kernel $R$ by 
\begin{equation}\label{eq:formula for R calculation of C2 constant}
R(y,dx):=\begin{cases}
0,\quad \rho(y)=0\\
\frac{\rho(x)\psi_1(y)a_2}{\rho(y)\psi_1(x)}\tilde{P}^{(2)}(y,dx),\quad \rho(y)>0
\end{cases}.
\end{equation}

Using \eqref{eq:adjoint equation for proof of sufficient for RaD condition}, this clearly satisfies
\[
\pi(dx)P_{t_2}(x,dy)=\pi(dy)R(y,dx).
\]

Moreover, using \eqref{eq:lower bound for tilde P (0) by nu' C2 calculation}, \eqref{eq:bound for P(2) calculation of C2 constant} and \eqref{eq:formula for R calculation of C2 constant}, we have whenever $\rho(y)>0$ (which is $\pi$-almost every $y\in\chi$) that
\begin{align}
R(y,dx)\leq \frac{a_2\psi_1(y)C_1}{\rho(y)\psi_1(x)}\pi(dx)\tilde{P}^{(0)}1(y),\label{eq:upper bound 1 for R pf of suff for Rad condition}\\
R1(y)\geq \tilde{P}^{(0)}1(y)\frac{a_2\psi_1(y)c_0'}{\rho(y)\lvert\lvert \frac{\psi_0}{\psi_1}\rvert\rvert_{\infty}\lvert\lvert \frac{\psi_1}{\psi_0}\rvert\rvert_{\infty}}\nu'\tilde{P}^{(1)}\Big(\frac{\rho}{\psi_1}\Big).\label{eq:lower bound 1 for R1 pf of suff for Rad condition}
\end{align}

Integrating $\frac{\rho}{\psi_1}(x)$ against both sides of \eqref{eq:psi adjoint for verifying bounded pi euclidean condition main results section}, we have that
\[
a_1\Big(\tilde{P}^{(1)}\Big(\frac{\rho}{\psi_1}\Big)\Big)(y)\psi_1(y)\Lambda(dy)=\int_{x\in \chi}P_{t_1}(x,dy)\pi(dx)=\lambda^{t_1}\pi(dy)=\lambda^{t_1}\rho(y)\Lambda(dy).
\]
We conclude that
\begin{equation}\label{eq:rho over psi1 given efn for adjoint RaD suff pf}
\tilde{P}^{(1)}\Big(\frac{\rho}{\psi_1}\Big)(y)=\frac{\lambda^{t_1}}{a_1}\Big(\frac{\rho}{\psi_1}\Big)(y)\quad\text{for $\Lambda$-almost every $y\in \chi$.}
\end{equation}
It follows from \eqref{eq:lower bound 1 for R1 pf of suff for Rad condition} and \eqref{eq:rho over psi1 given efn for adjoint RaD suff pf} that for $\pi$-almost every $y\in \chi$ we have
\begin{equation}\label{eq:lower bound for R kernel mass proof sufficient for RaD}
R1(y)\geq \tilde{P}^{(0)}1(y)\frac{a_2c_0'\psi_1(y)\lambda^{t_1}}{\rho(y)\lvert\lvert \frac{\psi_0}{\psi_1}\rvert\rvert_{\infty}\lvert\lvert \frac{\psi_1}{\psi_0}\rvert\rvert_{\infty}a_1}\nu'\Big(\frac{\rho}{\psi_1}\Big)\geq\tilde{P}^{(0)}1(y)\frac{a_2c_0'\psi_1(y)\lambda^{t_1}}{\rho(y)\lvert\lvert \frac{\psi_0}{\psi_1}\rvert\rvert_{\infty}\lvert\lvert \frac{\psi_1}{\psi_0}\rvert\rvert_{\infty}a_1\lvert\lvert \psi_1\rvert\rvert_{\infty}}\nu'(\rho).
\end{equation}
Combining \eqref{eq:upper bound 1 for R pf of suff for Rad condition} and \eqref{eq:lower bound for R kernel mass proof sufficient for RaD}, we obtain that for $\pi$-almost every $y\in \chi$ we have
\[
\begin{split}
\frac{R(y,dx)}{R1(y)}\leq \frac{\rho(y)\lvert\lvert \frac{\psi_0}{\psi_1}\rvert\rvert_{\infty}\lvert\lvert \frac{\psi_1}{\psi_0}\rvert\rvert_{\infty}\lvert\lvert \psi_1\rvert\rvert_{\infty}a_1a_2\psi_1(y)C_1\tilde{P}^{(0)}1(y)\pi(dx)}{\tilde{P}^{(0)}1(y)a_2c_0'\psi_1(y)\lambda^{t_1}\nu'(\rho)\rho(y)\psi_1(x)}\\
\leq \frac{\lvert\lvert \frac{\psi_0}{\psi_1}\rvert\rvert_{\infty}\lvert\lvert \frac{\psi_1}{\psi_0}\rvert\rvert_{\infty}\lvert\lvert \psi_1\rvert\rvert_{\infty}a_1C_1\pi(dx)}{c_0'\lambda^{t_1}\nu'(\rho)\psi_1(x)}.
\end{split}
\]
Therefore $R$ satisfies \eqref{eq:reverse kernel majorised for dominated by pi theorem} with
\[
\begin{split}
C_2':=\frac{\lvert\lvert \frac{\psi_0}{\psi_1}\rvert\rvert_{\infty}\lvert\lvert \frac{\psi_1}{\psi_0}\rvert\rvert_{\infty}\lvert\lvert\psi_1\rvert\rvert_{\infty}\lvert\lvert \frac{1}{\psi_1}\rvert\rvert_{\infty}a_1C_1}{c_0'\lambda^{t_1}\nu'(\rho)}.
\end{split}
\]

Applying Proposition \ref{prop:R constant mass}, we see that $R$ is submarkovian after adjusting the defintion of $R(y,\cdot)$ on a $\pi$-null set of $y\in \chi$. 

We already have by assumption that $\text{spt}(\pi)=\chi$ and that $P_{t_2}$ is lower semicontinuous in the sense of Definition \ref{defin:lower semicts kernel}.
\qed

\subsection*{Proof of Theorem \ref{theo:criterion for cty of QSD}}

\underline{Step $1$}

We note that $\frac{\tilde{P}_{t_0+t_1}(y,\cdot)}{\tilde{P}_{t_0+t_1}1(y)}\geq c_0\nu\tilde{P}_{t_1}(\cdot)$ for $\Lambda$-almost every $y\in\chi$, and that $\nu\tilde{P}_{t_1}(\calO)>0$. Since we may therefore replace $t_0$, $t_1$ and $\nu$ with $t_0+t_1$, $0$ and $\frac{\nu \tilde{P}_{t_1}(\cdot)}{\nu \tilde{P}_{t_1}1}$ respectively, we may assume without loss of generality that 
\[
\nu(\calO)>0\quad\text{so that}\quad \tilde{P}_{t_0}(y,\calO)>0\quad\text{for $\Lambda$-almost every $y\in\chi$.}
\]
We also have from \eqref{eq:adjoint semigroup eqn from cty of pi criterion} that if $A\in\mathscr{B}(\chi)$ with $\Lambda(A)=0$, then $\tilde{P}_{t_0}\Ind_A(y)=0$ for $\Lambda$-almost every $y\in\chi$, so that $\nu(A)=0$. Therefore $\nu\ll\Lambda$. Thus it follows that $\nu$ and $\pi$ are not mutually singular. 
It follows that $(X_t)_{0\leq t<\tau_{\partial}}$ satisfies Assumption \ref{assum:adjoint Dobrushin main results section} with this time $t_0$ and probability measure $\nu$. We already have that $\pi\in\calP_{\infty}(\Lambda)$. It therefore follows that $(X_t)_{0\leq t<\tau_{\partial}}$ and $\pi$ satisfy the assumptions required by Theorem \ref{theo:Linfty convergence main results section} to have \eqref{eq:Linfty convergence of dist cond of survival main results section}.

\underline{Step $2$}

We let $\phi\in L^1(\pi)$ be the non-negative $L^1(\pi)$-right eigenfunction whose existence and uniqueness is provided for by Theorem \ref{theo:Linfty convergence main results section}. Our goal is to establish that 
\begin{equation}\label{eq:phi has pve pi integral on O cty criterion}
\pi(\phi\Ind_{\calO})>0.
\end{equation}

We recall from the proof of Theorem \ref{theo:theo for reverse Dobrushin in Euclidean space} that the reverse kernel $R$ providing for $\pi$ and $(X_t)_{0\leq t<\tau_{\partial}}$ to satisfy Assumption \ref{assum:Dobrushin reverse time} is given by \eqref{eq:reverse kernel from Euclidean criterion proof} (except that the definition of $R(y,\cdot)$ may be adjusted on a $\pi$-null set of $y\in \chi$). Thus, fixing some bounded non-negative version $\rho$ of $\frac{d\pi}{d\Lambda}$ with $\lvert\lvert \rho\rvert\rvert_{\infty}=\lvert\lvert \frac{d\pi}{d\Lambda}\rvert\rvert_{L^{\infty}(\Lambda)}$, $R(y,\cdot)$ given for $\pi$-almost every $y\in \chi$ by
\[
R(y,dx)=\frac{\rho(x)\psi(y)a^{t_0}}{\rho(y)\psi(x)}\tilde{P}_{t_0}(y,dx)
\]
satisfies \eqref{eq:psi adjoint for verifying reverse Dobrushin main results section}. It follows that $R(y,\calO)>0$ for $\pi$-almost every $y\in \chi$.

We now define $\beta:=\phi \pi$. Since $\beta\ll \pi$, we have
\[
R(y,\calO)>0\quad\text{for}\quad \text{$\beta$-almost every $y\in\chi$.}
\]

As in \eqref{eq:beta stationary for R}, we may see that $\beta R=\lambda^{t_0}\beta$, so that
\[
\beta(\calO)= \lambda^{-t_0}\int_{\chi}R(y,\calO)\beta(dy)>0.
\]

Thus we have \eqref{eq:phi has pve pi integral on O cty criterion}.

\underline{Step $3$}

Our goal is to construct $\mu\in \calP_{\infty}(\pi)\cap \calP_{\infty}(\Lambda)$ such that $\frac{d\mu}{d\Lambda}$ has a $C_b(\chi;\Rm_{\geq 0})$ version and $\mu(\phi)>0$.

Since $c_1\Lambda_{\lvert_{\calO}}\leq \pi\ll_{\infty} \Lambda$ and $\pi(\phi\Ind_{\calO})>0$, $\phi\Ind_{\calO}\in L^1(\Lambda)$ with $\Lambda(\phi\Ind_{\calO})>0$. We now take an ascending sequence of functions $(f_n)_{n\geq 1}$, each of which is non-negative, continuous, bounded, $\Lambda$-integrable, and vanishes on $\calO^c$, such that $f_n$ converges pointwise to $\Ind_{\calO}$. The monotone convergence theorem then implies that $\Lambda(\phi f_n)\ra \Lambda(\phi \Ind_{\calO})>0$ as $n\ra \infty$. Thus by taking $f:=f_n$ for some $n$ sufficiently large, we have $f\in C_b(\chi;\Rm_{\geq 0})$ such that $\Lambda(\phi f)>0$, $0<\Lambda(f)<1$ and $f_{\lvert_{\calO^c}}\equiv 0$. Rescaling $f$, we have that $\Lambda(f)=1$. We now define the probability measure
\[
\mu:=f\Lambda,
\]
which we observe has the desired properties. We henceforth define $\mu$ to be this measure and $f$ the above-defined function.

It follows from \eqref{eq:adjoint semigroup eqn from cty of pi criterion} that $\mu P_t$ has a density with respect to $\Lambda$ given by
\begin{equation}\label{eq:density of mu Pt in terms of adjoint pf of pi cty}
\frac{d(\mu P_t)}{d\Lambda}(y)=a^t\psi(y)\tilde{P}_t\Big(\frac{f}{\psi}\Big).
\end{equation}

\underline{Step $4$}

Since $\pi\ll_{\infty}\Lambda$, Theorem \ref{theo:Linfty convergence main results section} implies that
\begin{equation}
\label{eq:L infty lebesgue convergence of density pf of cty criterion}
\Big\lvert\Big\lvert \frac{d\Law_{\mu}(X_t\lvert \tau_{\partial}>t)}{d\Lambda}-\frac{d\pi}{d\Lambda}\Big\rvert\Big\rvert_{L^{\infty}(\Lambda)}\leq \Big\lvert\Big\lvert \frac{d\Law_{\mu}(X_t\lvert \tau_{\partial}>t)}{d\pi}-1\Big\rvert\Big\rvert_{L^{\infty}(\pi)}\Big\lvert\Big\lvert \frac{d\pi}{d\Lambda}\Big\rvert\Big\rvert_{L^{\infty}(\Lambda)}\ra 0\quad \text{as}\quad t\ra \infty.
\end{equation}

\underline{Step $5$}

We now focus on Part \ref{enum:QSD thm cty on V main results section} of Theorem \ref{theo:criterion for cty of QSD}. We therefore take some open set $V$ satisfying \eqref{eq:cond for open set to have cts QSD density main results section}. It follows from \eqref{eq:density of mu Pt in terms of adjoint pf of pi cty} that (a version of) $\frac{d\Law_{\mu}(X_t\lvert \tau_{\partial}>t)}{d\Lambda}$ is continuous on $V$, for all $t\geq 0$.

Since $V$ is an open subset of $\chi$ and $\text{spt}(\Lambda)=\chi$, $\text{spt}(\Lambda_{\lvert_V})\supseteq V$, which suffices to ensure that every $L^{\infty}(\Lambda)$ limit of continuous functions on $V$ is continuous on $V$. It then follows from \eqref{eq:L infty lebesgue convergence of density pf of cty criterion} that (a version of) $\frac{d\pi}{d\Lambda}$ is continuous on $V$.

\underline{Step $6$}

We now turn our attention to Part \ref{enum:QSD thm lower semi cty main results section} of Theorem \ref{theo:criterion for cty of QSD}. We therefore no longer assume there to exist $V$ satisfying \eqref{eq:cond for open set to have cts QSD density main results section}, instead assuming that $(\tilde{P}_t)_{t\geq 0}$ is lower semicontinuous - it satisfies \eqref{eq:weak lower semicty of tilde P main results section}.

It follows from \eqref{eq:weak lower semicty of tilde P main results section} and \eqref{eq:density of mu Pt in terms of adjoint pf of pi cty} that $\Law_{\mu}(X_t\lvert \tau_{\partial}>t)$ has a version belonging to $LC_b(\chi;\Rm_{\geq 0})$, for all $t\geq 0$. We may then conclude Part \ref{enum:QSD thm lower semi cty main results section} of Theorem \ref{theo:criterion for cty of QSD} from \eqref{eq:L infty lebesgue convergence of density pf of cty criterion}, by application of the following lemma, which we shall prove in the appendix.

\begin{lem}\label{lem:l infty limit of lower semi cts functions}
We assume that $\chi$ is a seperable metric space on which is defined the distinguished $\sigma$-finite Borel measure $\Lambda$, with full support $\text{spt}(\Lambda)=\chi$. We suppose that, defined on $\chi$, is a sequence $(f_n)_{n\geq 1}$ in $L^{\infty}(\Lambda)$, converging in $L^{\infty}(\Lambda)$ to $f\in L^{\infty}(\Lambda)$. We further assume that each $f_n$ has a version, $u_n$, which is bounded, non-negative, and lower semicontinuous, $u_n\in LC_b(\chi;\Rm_{\geq 0})$ for all $n\geq 1$. Then $f$ has a version $u\in LC_b(\chi;\Rm_{\geq 0})$ which is bounded, non-negative and lower semicontinuous, $u\in LC_b(\chi;\Rm_{\geq 0})$, and which is maximal in the sense that any other bounded, non-negative lower semicontinuous version of $f$, $\tilde{u}\in LC_b(\chi;\Rm_{\geq 0})$, is everywhere dominated by $u$: $\tilde{u}(x)\leq u(x)$ for all $x\in \chi$.
\end{lem}
\qed

\subsection*{Proof of Theorem \ref{theo:criterion for right efn main results section}}

All of the aforestated results are given as results for a killed Markov process $(X_t)_{0\leq t<\tau_{\partial}}$ with quasi-stationary distribution $\pi$. However, all that is ever used about $(X_t)_{0\leq t<\tau_{\partial}}$ is that it corresponds to the submarkovian transition semigroup $(P_t)_{t\geq 0}$. In particular, Theorem \ref{theo:criterion for cty of QSD} can be applied to $(\tilde{P}_t)_{t\geq 0}$ (that is, with $(\tilde{P}_t)_{t\geq 0}$ substituted for $(P_t)_{t\geq 0}$ in the statement of Theorem \ref{theo:criterion for cty of QSD} and vice-versa), without requiring that $(\tilde{P}_t)_{t\geq 0}$ correspond to some killed Markov process.

Using Part \ref{enum:QSD thm lower semi cty main results section} of Theorem \ref{theo:criterion for cty of QSD}, we may take a non-negative, bounded, lower semicontinuous version of $\frac{d\tilde{\pi}}{d\Lambda}$, $\tilde{\rho}\in LC_b(\chi;\Rm_{\geq 0})$, which is maximal in the sense that any other such non-negative, bounded, lower semicontinuous version $\tilde{\rho}'\in LC_b(\chi;\Rm_{\geq 0})$ of $\frac{d\pi}{d\Lambda}$ is everywhere dominated by $\tilde{\rho}$, $\tilde{\rho}'(x)\leq \tilde{\rho}(x)$ for all $x\in\chi$. We then define 
\begin{equation}\label{eq:maximal version of right efn h tilde in terms of QSD}
\tilde{h}:=\frac{\tilde{\rho}}{\psi}\in LC_b(\chi;\Rm_{\geq 0}).
\end{equation}
It follows from \eqref{eq:adjoint eqn for right efn existence theorem main results section} that
\[
a^t\frac{\tilde{\rho}(y)}{\psi(x)}\Lambda(dy)\tilde{P}_t(y,dx)=\tilde{h}(y)\Lambda(dx)P_{t}(x,dy),\quad t\geq 0.
\]
We recall from \eqref{eq:defin of evalue lambda pf of right efn existence theorem main results section} that $\lambda:=a\tilde{\pi}\tilde{P}_11>0$. By integrating over $y$, we obtain that
\[
\lambda^t\tilde{h}(x)\Lambda(dx)=P_{t}\tilde{h}(x)\Lambda(dx),\quad t\geq 0.
\]

We now define $Q_t:=\lambda^{-t}P_t$. It follows that $\tilde{h}$ satisfies $Q_t\tilde{h}=\tilde{h}$ $\Lambda$-almost everywhere, for all $t\geq 0$. The maximality of $\tilde{\rho}$ implies that $\tilde{h}$ is also maximal, in the sense that if some other lower semicontinuous $\tilde{h}'\in LC_b(\chi;\Rm_{\geq 0})$ is equal to $\tilde{h}$ $\Lambda$-almost everywhere, then $\tilde{h}'(x)\leq \tilde{h}(x)$ for all $x\in \chi$.

We now fix $s\geq 0$. We have that 
\[
Q_s\tilde{h}=\tilde{h}\quad\text{$\Lambda$-almost everywhere.}
\]
Moreover, since $(P_t)_{t\geq 0}$ is lower semicontinuous (in the sense of Definition \ref{defin:lower semicts kernel}), $(Q_t)_{t\geq 0}$ is also lower semicontinuous, so that $Q_sh\in LC_b(\chi;\Rm_{\geq 0})$. It follows from the maximality of $\tilde{h}$ that $Q_s\tilde{h}(x)\leq \tilde{h}(x)$ for all $x\in\chi$. We therefore have that
\[
Q_{t_2}\tilde{h}(x)\leq Q_{t_1}\tilde{h}(x)\quad\text{for all}\quad x\in \chi,\quad 0\leq t_1\leq t_2<\infty.
\]

We may therefore define the pointwise limit
\[
h(x):=\lim_{t\ra \infty}Q_t\tilde{h}(x)\in \calB_b(\chi;\Rm_{\geq 0}).
\]
It follows from the dominated convergence theorem that $Q_th(x)=h(x)$ for all $t\geq 0$ and $x\in \chi$, and that $h=\tilde{h}$ almost everywhere (so that $h$ is non-trivial, in particular). Thus $h\in \calB_b(\chi;\Rm_{\geq 0})$ is our desired pointwise right eigenfunction for $(P_t)_{t\geq 0}$. All that remains is to prove that $h(x)>0$ for all $x\in \chi$.

It is an immediate consequence of \eqref{eq:A1 for right efn existence theorem main results section} that
\[
P_{t_0+t_1}(x,\cdot)\geq P_{t_0}1(x)c_0(\nu P_{t_1})_{\lvert_{\calO}}(\cdot)\quad\text{for all}\quad x\in\chi.
\]

It follows from \eqref{eq:adjoint eqn for right efn existence theorem main results section} that if $A$ is a $\Lambda$-null Borel subset of $\calO$, then $P_{t_0+t_1}(x,A)=0$ for $\Lambda$ almost-every $x\in\chi$. It therefore follows that $(\nu P_{t_1})_{\lvert_{\calO}}(A)=0$. Thus $(\nu P_{t_1})_{\lvert_{\calO}}\ll \Lambda_{\calO}$. Since $\psi h$ is a version of $\frac{d\tilde{\pi}}{d\Lambda}\geq c_1\Ind_{\calO}$, $h\geq \frac{c_1}{\lvert\lvert \psi\rvert\rvert_{\infty}}$ $\Lambda$-almost everywhere on $\calO$. Therefore, $\nu P_{t_1}(h)>0$. It therefore follows that
\[
h(x)=\lambda^{-(t_0+t_1)}P_{t_0+t_1}h(x)>0\quad \text{for all}\quad x\in\chi.
\]

We now assume \eqref{eq:cty preserved on V condition for efn} is satisfied for some open set $V\subseteq \chi$. It follows from Part \ref{enum:QSD thm cty on V main results section} of Theorem \ref{theo:criterion for cty of QSD} and \eqref{eq:maximal version of right efn h tilde in terms of QSD} that there exists $h'\in C_b(V)$ such that $\tilde{h}_{\lvert_V}=h'$ $\Lambda$-almost everywhere on $V$. Therefore by the maximality of $\tilde{h}$, we have that $\tilde{h}=h'$ everywhere on $V$, so that $\tilde{h}\in C_b(V)$. Therefore $Q_t\tilde{h}\in LC_b(\chi;\Rm_{\geq 0})\cap C_b(V)$ for all $t\geq 0$, so that $\tilde{h}=Q_t\tilde{h}$ everywhere on $V$, for all $t\geq 0$. It follows that $h=\tilde{h}$ everywhere on $V$, so that $h\in \calB_b(\chi;\Rm_{\geq 0})\cap C_b(V)$.
\qed

\subsection*{Proof of Proposition \ref{prop:Lp malthusian by interpolation}}

We proceed by applying the Riesz-Thorin interpolation theorem, as in the proof of the well-known analogous statement in the context of Markov processes without killing (see \cite[p.114]{Cattiaux2014}), for instance).

We consider for all $0\leq t<\infty$ the linear operator
\[
A_t:f\mapsto \lambda^{-t}\frac{d\Pm_{f\pi}(X_t\in \cdot)}{d\pi(\cdot)}-\pi(fh).
\]
We have from \eqref{eq:total variation Malthusian behaviour defin} that there exists $C_1<\infty$ and $\gamma>0$ ($\gamma>0$ being the same constant as is given in the statement of \cite[Theorem 2.1]{Champagnat2014}) such that $\lvert\lvert A_t\rvert\rvert_{L^1(\pi)\ra L^1(\pi)}\leq C_1e^{-\gamma t}$ for all$0\leq t<\infty$. Moreover we have that
$\lvert\lvert A_t\rvert\rvert_{L^{\infty}(\pi)\ra L^{\infty}(\pi)}\leq 2$. It follows from the Riesz-Thorin interpolation theorem that there exists $C_2<\infty$ such that
\[
\lvert\lvert A_t\rvert\rvert_{L^p(\pi)\ra L^p(\pi)}\leq C_2e^{-\frac{\gamma}{p}t}\quad \text{for all}\quad 0\leq t<\infty\quad\text{and}\quad 1\leq p\leq \infty.
\]
Furthermore, \cite[Theorem 2.1]{Champagnat2017} implies the existence of $C_3<\infty$ such that $\lvert \mu(h)-\lambda^{-t}\Pm_{\mu}(\tau_{\partial}>t)\rvert\leq \mu(h)C_3e^{-\gamma t}$ ($\gamma>0$ also being the same constant as is given in the statement of \cite[Theorem 2.1]{Champagnat2014}). We now proceed as in the conclusion of the proof of Theorem \ref{theo:Linfty convergence}. We may replace $L^{\infty}(\pi)$ with $L^p(\pi)$ in \eqref{eq:conditional distribution from killing prob and malthusian formula Linfty} to obtain \eqref{prop:Lp malthusian by interpolation}.
\qed

\section{Irreducible Markov chains on finite state spaces}\label{section:finite state space}
We assume that $\chi=\{1,\ldots,n\}$ is a finite state space. We define the distinguished measure $\Lambda$ to be the counting measure on $\chi$. For probability measures $\mu$ on $\chi$ we abuse notation by writing $\mu(x)$ for $\mu(\{x\})=\frac{d\mu}{d\Lambda}(x)$, for all $x\in\chi$. 

We assume that $(X_t)_{0\leq t<\tau_{\partial}}$ is an irreducible killed Markov chain on $\chi\sqcup \partial$ in discrete or continuous time, with time $t$ transition matrix $P_t(x,y)$. If time is discrete, we assume in addition that the killed Markov chain is aperiodic. Then \cite{Darroch1965,Darroch1967} imply that $(X_t)_{0\leq t<\tau_{\partial}}$ has a unique QSD $\pi$, which has full support. We write $0<\lambda=\lambda(\pi)\leq 1$ for the corresponding eigenvalue over time $1$. Then we have the following theorem, which provides a quantitative rate of convergence in $L^{\infty}(\pi)$ of $\frac{d\Law_{\mu}(X_t\lvert\tau_{\partial}>t)}{d\pi}$.

\begin{theo}\label{theo:assum satisfied finite stat space}
We take any time $t_0>0$ such that, for some $x\in \chi$, $p_{t_0}(x,y)>0$ for all $y\in \chi$. We define
\begin{align}
c_0:=\frac{\inf_{x,y\in \chi}P_{t_0}(x,y)}{\sup_x\pi(x)\sup_{y\in \chi}(\sum_xP_{t_0}(x,y))},\\
C_2:=\frac{\sup_{x,y\in \chi}P_{t_0}(x,y)\sup_{y\in \chi}(\sum_xP_{t_0}(x,y))}{\inf_{x,y\in \chi}P_{t_0}(x,y)},\\
c_3:=\frac{n(\inf_{x,y\in \chi}P_{t_0}(x,y)\big)^2}{\sup_x\pi(x)\sup_{y\in \chi}(\sum_xP_{t_0}(x,y))\sup_{x\in \chi}(\sum_yP_{t_0}(x,y))}.
\end{align}
Then we have that 
\begin{equation}\label{eq:quantitative convergence finite state space}
\Big\lvert\Big\lvert \frac{d\Law_{\mu}(X_t\lvert \tau_{\partial}>t)}{d\pi}-1\Big\rvert\Big\rvert_{L^{\infty}(\pi)}\leq \frac{2C_2(1-c_0)^{\lfloor \frac{t-4t_0}{t_0}\rfloor}}{c_3-C_2(1-c_0)^{\lfloor \frac{t-4t_0}{t_0}\rfloor}},\quad 4t_0\leq t<\infty.
\end{equation}
\end{theo}

\subsection*{Proof of Theorem \ref{theo:assum satisfied finite stat space}}

We fix $a>0$ sufficiently large such that $\tilde{P}$ given by
\[
\tilde{P}(y,x):=\frac{1}{a}P_{t_0}(x,y)
\]
is submarkovian. It therefore satisfies \eqref{eq:psi adjoint for verifying reverse Dobrushin main results section} with $\psi\equiv 1$, this $a>0$ and the time $t_1>0$. We define the probability measure 
\begin{equation}\label{eq:nu for finite state space}
\nu:=\frac{\Lambda}{\lvert \chi\rvert},
\end{equation}
We see that $\tilde{P}$ satisfies \eqref{eq:Dobrushin for P tilde kernel crit for reverse dobrushin main results section} with $\nu$ given by \eqref{eq:nu for finite state space} and
\[
c_0':=\frac{n\inf_{x,y\in \chi}P_{t_0}(x,y)}{\sup_{y\in \chi}(\sum_xP_{t_0}(x,y))}.
\]
We may also see that $\tilde{P}$ satisfies \eqref{eq:bounded Lebesgue main results section} with
\[
C_1:=\frac{1}{a}\sup_{x,y\in \chi}P_{t_0}(x,y).
\]
Finally, we observe that $P_{t_0}$ satisfies \eqref{eq:Dobrushin for DAD condition main results section} with $\nu_1=\nu$ and
\[
c_1:=\frac{n\inf_{x,y\in \chi}P_{t_0}(x,y)}{\sup_{x\in \chi}(\sum_yP_{t_0}(x,y))}.
\]
Applying Theorem \ref{theo:uniform Linfty convergence main results section}, we obtain \eqref{eq:quantitative convergence finite state space} from \eqref{eq:quantitative uniform exponential convergence for arbitrary initial cond using DAD main results section}.
\qed

\section{Degenerate diffusions}

We shall consider degenerate diffusions killed instantaneously at the boundary of the domain $\chi\subseteq \Rm^d$, on which we impose the following assumption throughout this section.

\begin{assum}\label{assum:deg diff section domain assum}
The domain $\chi\subseteq \Rm^d$ is a bounded, open, connected, non-empty subdomain of $d$-dimensional Euclidean space with $C^{\infty}$ boundary $\partial \chi$, for arbitrary dimension $d<\infty$. 
\end{assum}
Throughout this section, the distinguished measure $\Lambda$ should be understood to be Lebesgue measure on $\chi$.

We consider on $\chi$ solutions $(X_t)_{0\leq t<\tau_{\partial}}$ of the stochastic differential equation
\begin{equation}\label{eq:degenerate diffusion}
dX_t=v^0(X_t)dt+\sum_{j=1}^mv^j(X_t)\circ dB^j_t,\quad 0\leq t<\tau_{\partial}:=\inf\{s>0:X_{s-}\in \partial \chi\},
\end{equation}
whereby $\circ$ refers to Stratanovich integration, $v^0,\ldots,v^m$ are $C^{\infty}(\Rm^d)$ vector fields (note that their definition on $(\bar \chi)^c$ is arbitrary), and $B^1,\ldots,B^m$ are independent Brownian motions.

Given smooth vector fields $v,w$ on $U$, we write $[v,w]$ for the Lie bracket of $v$ and $w$. Given a family $\calV$ of smooth vector fields on $\chi$ we inductively define
\[
[\calV]_1:=\calV,\quad [\calV]_{k+1}:=[\calV]_k\cup \{[v,w]:v,w\in [\calV]_k\},\quad [\calV]:=\cup_{k}[\calV]_k\quad \text{and}\quad [\calV](x):=\{v(x):v\in [\calV]\}.
\]
We impose the following parabolic H\"{o}rmander condition.
\begin{assum}[Parabolic H\"{o}rmander condition on $\chi$]\label{assum:strong Hormander}
The vector fields $v^1,\ldots,v^m$ are such that $\text{span}([\calV](x))=\Rm^d$ for all $x\in \chi$.
\end{assum}

For $x\in\partial \chi$, $\vec{n}(x)$ is the inward unit normal, which is well-defined since $\partial \chi$ is assumed to be $C^{\infty}$. We impose the following additional assumption on the vector fields at the boundary. 
\begin{assum}\label{assum:vector field degenerate diffusion traverses bdy}
For every $x\in \partial \chi$, there exists $1\leq i\leq m$ such that $\langle v^i(x),\vec{n}(x)\rangle\neq 0$.
\end{assum}
This is the classical boundary assumption under which Bony established well-posedness of the Dirichlet problem for operators satisfying H\"ormander's condition \cite[Theorem 5.2]{Bony1969}.

We assume that assumptions \ref{assum:deg diff section domain assum}, \ref{assum:strong Hormander} and \ref{assum:vector field degenerate diffusion traverses bdy} are satisfied, and consider solutions $(X_t)_{0\leq t<\tau_{\partial}}$ of \eqref{eq:degenerate diffusion}. The existence of a unique quasi-stationary distribution $\pi$, along with non-uniform exponential convergence in total variation to this QSD, has already been established in \cite[Corollary 1.9]{Benaim2021} (under the additional assumption that the Parabolic H\"{o}rmander condition also holds on the boundary $\partial \chi$). In particular they established that there exists $h\in C_0(\chi;\Rm_{>0})$, $C<\infty$ and $\gamma>0$ such that
\begin{equation}\label{eq:degenerate diffusion non-unif conv}
\lvert\lvert \Law_{\mu}(X_t\lvert \tau_{\partial}>t)-\mu\rvert\rvert_{\TV}\leq \frac{C}{\mu(h)}e^{-\gamma t}\quad\text{for all}\quad 0\leq t<\infty\quad\text{and all initial conditions}\quad \mu\in\calP(\chi).
\end{equation}

We establish the following.
\begin{theo}\label{theo:defenerate diffusions hard killing}
We assume that assumptions \ref{assum:deg diff section domain assum}, \ref{assum:strong Hormander} and \ref{assum:vector field degenerate diffusion traverses bdy} are satisfied, and consider solutions $(X_t)_{0\leq t<\tau_{\partial}}$ of \eqref{eq:degenerate diffusion}. Then $(X_t)_{0\leq t<\tau_{\partial}}$ satisfies Assumption \ref{assum:technical assumption for Assum (A)} and \cite[Assumption (A)]{Champagnat2014}, so that it has a unique QSD, $\pi$. This QSD has an essentially bounded density with respect to Lebesgue measure. Moreover Assumption \ref{assum:adjoint Dobrushin main results section} is satisfied by $(X_t)_{0\leq t<\tau_{\partial}}$.
\end{theo}

Since Assumption \ref{assum:technical assumption for Assum (A)} and \cite[Assumption (A)]{Champagnat2014} are satisfied, the convergence in \eqref{eq:degenerate diffusion non-unif conv} can be made uniform by \cite[Theorem 2.1]{Champagnat2014}: there exists some constants $C<\infty$ and $\gamma>0$ such that
\begin{equation}
\lvert \lvert \Law_{\mu}(X_t\lvert\tau_{\partial}>t)-\pi\rvert\rvert_{\TV}\leq C e^{-\gamma t}\quad\text{for all}\quad 0\leq t<\infty\quad\text{and all initial conditions}\quad \mu\in\calP(\chi).
\end{equation}
Moreover we can apply Theorem \ref{theo:Linfty convergence main results section} to see that there exists constants $C,T<\infty$ and $\gamma>0$ such that
\begin{equation}
\Big\lvert\Big\lvert \frac{d\Law_{\mu}(X_t\lvert \tau_{\partial}>t)}{d\pi}-1\Big\rvert\Big\rvert_{L^{\infty}(\pi)}\leq \frac{C}{\mu(h)}e^{-\gamma t}\Big\lvert\Big\lvert \frac{d\mu}{d\pi}\Big\rvert\Big\rvert_{L^{\infty}(\pi)}\quad \text{for all}\quad t\geq T\quad\text{and all}\quad  \mu\in\calP_{\infty}(\pi),
\end{equation}

where $h$ is the strictly positive, bounded pointwise right eigenfunction provided for by \cite[Proposition 2.3]{Champagnat2014}.

We may strengthen Assumption \ref{assum:strong Hormander} to the following.
\begin{assum}[Parabolic H\"{o}rmander condition on $\bar \chi$]\label{assum:strong Hormander on closure}
The vector fields $v^1,\ldots,v^m$ are such that $\text{span}([\calV](x))=\Rm^d$ for all $x\in \bar\chi$.
\end{assum}

We establish the following.
\begin{theo}\label{theo:degenerate diffusions hard killing bounded density wrt pi}
We assume that assumptions \ref{assum:deg diff section domain assum}, \ref{assum:vector field degenerate diffusion traverses bdy} and \ref{assum:strong Hormander on closure} are satisfied, and consider solutions $(X_t)_{0\leq t<\tau_{\partial}}$ of \eqref{eq:degenerate diffusion}. Then the QSD $\pi$ has full support. Moreover $(X_t)_{0\leq t<\tau_{\partial}}$ is lower semicontinuous (in the sense of Definition \ref{defin:lower semicts kernel}) satisfies Assumption \ref{assum:adjoint anti-Dobrushin main results section}.
\end{theo}

It follows from Theorem \ref{theo:uniform Linfty convergence main results section} that (under the assumptions of Theorem  \ref{theo:degenerate diffusions hard killing bounded density wrt pi}) there exists a time $T<\infty$ and constant $\gamma>0$ such that $\Law_{\mu}(X_t\lvert \tau_{\partial}>t)\ll_{\infty}\pi$ for all $t\geq T$ and $\mu\in\calP(\chi)$, with its density with respect to $\pi$ satisfying
\begin{equation}
\Big\lvert\Big\lvert \frac{d\Law_{\mu}(X_t\lvert \tau_{\partial}>t)}{d\pi}-1\Big\rvert\Big\rvert_{L^{\infty}(\pi)}\leq e^{-\gamma(t-T)}\quad\text{for all}\quad  T\leq t<\infty,\quad\mu\in\calP(\chi)..
\end{equation}

We may observe that, over the course of proving Theorem \ref{theo:defenerate diffusions hard killing}, we have established that Aassumption \ref{assum:combined Dobrushin adjoint Dobrushin main results section} (which includes \cite[Assumption (A1)]{Champagnat2014}) is satisfied by $(X_t)_{0\leq t<\tau_{\partial}}$ along the way. Moreover, it is also clear that the time horizon over which we establish assumptions \ref{assum:adjoint Dobrushin main results section}, \ref{assum:adjoint anti-Dobrushin main results section} and \ref{assum:combined Dobrushin adjoint Dobrushin main results section} can be made arbitrarily small, without any changes to the proof. We therefore obtain from theorems \ref{theo:dominated by pi theorem general results main results section} and \ref{theo:DAD lower bounds density main results section} the following.
\begin{theo}\label{theo:comparison inequality for degenerate diffusions}
We impose the assumptions of Theorem \ref{theo:degenerate diffusions hard killing bounded density wrt pi}. For all $t>0$ there exists $0<c_t\leq C_t<\infty$ such that
\begin{equation}\label{eq:comparison of dist and QSD in results for degenerate diffusion}
c_t\pi\leq \Law_{\mu}(X_t\lvert \tau_{\partial}>t)\leq C_t\pi\quad\text{for all}\quad \mu\in\calP(\chi).
\end{equation}
We may put this in the form of a parabolic boundary Harnack inequality as follows. For any initial conditions $\mu,\nu\in\calP(\chi)$, we let $u_1(x,t)$ and $u_2(x,t)$ be continuous versions (see Proposition \ref{prop:basic facts about degenerate diffusions} for a justification that this exists) of $\frac{d\Pm_{\mu}(X_t\in \cdot,\tau_{\partial}>t)}{dLeb(\cdot)}(x)$ and $\frac{d\Pm_{\nu}(X_t\in \cdot,\tau_{\partial}>t)}{dLeb(\cdot)}(x)$ for $x\in \chi$ and $t>0$, respectively. It follows that for all $t>0$ we have
\begin{equation}\label{eq:parabolic boundary Harnack diffusion}
\inf_{t_1,t_2\geq t}\frac{\inf_{x\in \chi}\Big(\frac{u_1(x,t_1)}{u_2(x,t_2)}\Big)}{\sup_{x'\in \chi}\Big(\frac{u_1(x',t_1)}{u_2(x',t_2)}\Big)}\geq \frac{c_t^2}{C_t^2}>0.
\end{equation} 
Note in particular that the constants $0<c_t<C_t<\infty$ do not depend upon $\mu$ and $\nu$, and that this comparison is valid up to the boundary. 
\end{theo}
We note that the lower bound in \eqref{eq:comparison of dist and QSD in results for degenerate diffusion} holds under the assumptions of Theorem \ref{theo:defenerate diffusions hard killing} (that is, the lower bound doesn't require that the parabolic H\"ormander condition holds on the boundary).

The classical work of Bony \cite[Section 7]{Bony1969} provides an interior Harnack inequality under H\"ormander conditions. Whilst the author is not familiar with a boundary Harnack inequality under H\"ormander-type conditions, the literature is rather large.

\begin{prob} 
Do theorems \ref{theo:degenerate diffusions hard killing bounded density wrt pi} and \ref{theo:comparison inequality for degenerate diffusions} remain true with Assumption \ref{assum:vector field degenerate diffusion traverses bdy} relaxed?
\end{prob}

\subsection*{Proof of Theorem \ref{theo:defenerate diffusions hard killing}}

For open domains $U\subseteq \chi$ it will be convenient to define 
\begin{equation}\label{eq:defin of Cb3}
C_b^3(U):=\{f\in C^3(U):\partial^{\alpha}f\in C_b(U)\text{ for all }\lvert\alpha\rvert\leq 3\}.
\end{equation}

We shall firstly prove the following three propositions, before using them to prove Theorem \ref{theo:defenerate diffusions hard killing}.

\begin{prop}\label{prop:degerenate diffusion prob of being killed goes to 1 as approach bdy}
For all $t>0$ we have $\Pm_x(\tau_{\partial}>t)\ra 0$ as $d(x,\partial \chi)\ra 0$.
\end{prop}

\subsubsection*{Proof of Proposition \ref{prop:degerenate diffusion prob of being killed goes to 1 as approach bdy}}

\begin{fact}\label{fact:distance to the bdy unit derivative}
We have from \cite[Theorem 8.1]{Delfour2011} that there exists $d_0>0$ such that $d_{\partial \chi}(x):=d(x,\partial \chi)$, belongs to $C^{3}_b(\chi\cap B(\partial \chi,d_0)) $ with $\lvert \nabla d_{\partial \chi}(x)\rvert\equiv 1$ on $\chi\cap B(\partial \chi,d_0)$.
\end{fact}

It then follows from Fact \ref{fact:distance to the bdy unit derivative} and Assumption \ref{assum:vector field degenerate diffusion traverses bdy} that $d(X_t,\partial \chi)$ is a $1$-dimensional diffusion with bounded coefficients and diffusivity bounded from below away from $0$ whenever $d(X_t,\partial \chi)<d_0$, whence we have Proposition \ref{prop:degerenate diffusion prob of being killed goes to 1 as approach bdy}.

\begin{prop}\label{prop:basic facts about degenerate diffusions}
We assume that assumptions \ref{assum:deg diff section domain assum}, \ref{assum:strong Hormander} and \ref{assum:vector field degenerate diffusion traverses bdy} are satisfied, and consider solutions $(X_t)_{0\leq t<\tau_{\partial}}$ of \eqref{eq:degenerate diffusion}. Moreover we have that:
\begin{enumerate}
\item\label{enum:degenerate diffusion hard killing smoothness transition density}
We write $P_t(x,dy)$ for the submarkovian transition kernel of $X_t$. There exists $p_t(x,y)\in C^{\infty}((0,\infty)\times\chi\times\chi;\Rm_{>0})$ such that $P_t(x,dy)=p_t(x,y)\text{Leb}(dy)$. Moreover $p_t(x,y)$ is a solution of 
\[
(L_x-\frac{\partial}{\partial t})p_t(x,y)=0,\quad (L_y^{\ast}-\frac{\partial}{\partial t})p_t(x,y)=0,
\]
whereby $L_x$ is the infinitesimal generator of $X_t$ and $L_y^{\ast}$ is its formal $\text{Leb}(dy)$ adjoint.
\item\label{enum:exist of Y and Y0 adjoint processes}
There exists a constant $0\leq A<\infty$ and killed processes $(\tilde{X}_t)_{0\leq t<\tilde{\tau}_{\partial}}$ and $(\tilde{X}^0_t)_{0\leq t<\tilde{\tau}_{\partial}^0}$, whose submarkovian transition kernels we call $\tilde{P}_t(y,dx)$ and $\tilde{P}^0_t(y,dx)$ respectively, such that:
\begin{enumerate}
\item\label{enum:adjoint transition density degenerate diffusion}
There exists $\tilde{p}_t(y,x)\in C^{\infty}((0,\infty)\times \chi\times \chi)$ such that $\tilde{P}_t(y,dx)=\tilde{p}_t(y,x)\Leb(dx)$, which is given by
\begin{equation}\label{eq:adjoint density eqn for degenerate diffusion}
\tilde{p}_t(y,x)=e^{-At}p_t(x,y)\quad\text{for all}\quad x,y\in \chi,\quad t> 0.
\end{equation}
\item
The process $(\tilde{X}^0_t)_{0\leq t<\tilde{\tau}_{\partial}}$ is a solution of \eqref{eq:degenerate diffusion} (with a possibly different drift vector but with $v^1,\ldots,v^n$ unchanged) satisfying assumptions \ref{assum:strong Hormander} and \ref{assum:vector field degenerate diffusion traverses bdy} (Assumption \ref{assum:deg diff section domain assum} is still satisfied as the domain is the same).
\item
We have that for all $t> 0$ and $y\in \chi$,
\begin{equation}\label{eq:degen diff section Q and Q0 relationship}
e^{-2At}\tilde{P}^0_t(y,\cdot)\leq \tilde{P}_t(y,\cdot)\leq \tilde{P}^0_t(y,\cdot).
\end{equation}
\end{enumerate}
\item\label{enum:positive Cb right e-fn/QSD}
There exists $\lambda>0$ and $h,\rho\in C_0(\chi;\Rm_{>0})$ such that
\begin{align}
\lambda h(x)=\int_{\chi}p_1(x,y)h(y)\Leb(dy)\quad \text{for all}\quad x\in \chi,\label{eq:deg diff basic properties pve efn}\\
\lambda\rho(y)=\int_{\chi}p_1(x,y)\rho(x)\Leb(dx)\quad \text{for all}\quad y\in \chi.\label{eq:deg diff basic properties pve qsd}
\end{align}
\end{enumerate}
\end{prop}

\begin{prop}\label{prop:degenerate diffusion satisfies A1}
We assume that assumptions \ref{assum:deg diff section domain assum}, \ref{assum:strong Hormander} and \ref{assum:vector field degenerate diffusion traverses bdy} are satisfied, and consider solutions $(X_t)_{0\leq t<\tau_{\partial}}$ of \eqref{eq:degenerate diffusion}. Then $(X_t)_{0\leq t<\tau_{\partial}}$ satisfies Assumption \ref{assum:technical assumption for Assum (A)}. Moreover there exists an open ball $B(x^{\ast},r)\subseteq \chi$ (for some $x^{\ast}\in \chi$ and $r>0$) and $c_1>0$ such that $\nu:=\frac{1}{\Leb(B(x^{\ast},r))}\Leb_{\lvert_{B(x^{\ast},r)}}$ satisfies $\Law_x(X_1\lvert \tau_{\partial}>1)\geq c_1\nu$ for all $x\in \chi$.
\end{prop}

\subsection*{Proof of Proposition \ref{prop:basic facts about degenerate diffusions}}

Aside from the strict positivity of $p_t(x,y)$, Part \ref{enum:degenerate diffusion hard killing smoothness transition density} follows from \cite[Theorem 3]{Ichihara1974}. We defer for the time being the proof that $p_t(x,y)>0$. We note here, however, that the weaker statement
\begin{equation}\label{eq:degenerate diffusions full support}
\overline{\text{spt}(P_t(x,\cdot))}=\bar \chi\quad\text{for all}\quad x\in\chi\quad\text{and}\quad t>0
\end{equation}
follows from the Stroock-Varadhan support theorem \cite[Corollary 4.1]{Stroock1972} and the Chow-Rashevskii theorem \cite[Theorem 3.31]{Agrachev2019}.

\subsubsection*{Proof of Part \ref{enum:exist of Y and Y0 adjoint processes} of Proposition \ref{prop:basic facts about degenerate diffusions}}

An expression for $L^{\ast}_y$ is given by \cite[Eq. $(4.4)$]{Ichihara1974}. In particular, for some sufficiently large $A>0$, we can construct a killed diffusion having infinitesimal generator $L^{\ast}_y-A$ as follows:
\begin{enumerate}
\item
We take a diffusion $(\tilde{X}^0_t)_{0\leq t<\tilde{\tau}_{\partial}^0}$ in $\chi$, killed upon contact with $\partial \chi$ but not in the interior of $\chi$, with drift and diffusivity given by the first and second order terms of \cite[Eq. $(4.4)$]{Ichihara1974}. This killed diffusion is a solution of \eqref{eq:degenerate diffusion} with a possibly different drift vector, but with the diffusivity given by the same vector fields $v^1,\ldots,v^n$. Thus it satisfies assumptions \ref{assum:strong Hormander} and \ref{assum:vector field degenerate diffusion traverses bdy} (Assumption \ref{assum:deg diff section domain assum} is still satisfied as the domain is the same).
\item
The expression for $L^{\ast}_y$ given by \cite[Eq. $(4.4)$]{Ichihara1974} features an additional ``$+au$'' term. We take a constant $A>\sup_{x\in \bar \chi}\lvert a(x)\rvert$, so that $\kappa(x):=A-a(x)$ is a $C^{\infty}(\Rm^d)$ function which is strictly positive on $\bar \chi$. 
\item
We then construct the process $(\tilde{X}_t)_{0\leq t<\tilde{\tau}_{\partial}}$ from $(\tilde{X}^0_t)_{0\leq t<\tilde{\tau}_{\partial}^0}$. In particular, $\tilde{X}_t$ is equal to $\tilde{X}^0_t$ up to the ringing time of a position-dependent Poisson clock of rate $\kappa(\tilde{X}_t)=\kappa(\tilde{X}^0_t)$, at which time $\tilde{X}_t$ is killed. Note that since $\tilde{X}^0_t$ is killed upon contact with the boundary, so is $\tilde{X}_t$. We observe that $\tilde{X}_t$ has infinitesimal generator $L^{\ast}-A$.
\end{enumerate}
We write $\tilde{P}_t(y,\cdot)$ and $\tilde{P}^0_t(y,\cdot)$ for the submarkovian transition kernels of $(\tilde{X}_t)_{0\leq t<\tilde{\tau}_{\partial}}$ and $(\tilde{X}^0_t)_{0\leq t<\tilde{\tau}_{\partial}^0}$ respectively. It follows from \cite[Theorem 3]{Ichihara1974} and its proof that there exists $\tilde{p}_t(y,x)\in C^{\infty}((0,\infty)\times \chi\times \chi;\Rm_{\geq 0})$ satisfying $\tilde{P}_t(y,dx)=\tilde{p}_t(y,x)\Leb(dx)$ such that
\[
\Big(L_y^{\ast}-A-\frac{\partial}{\partial t}\Big)\tilde{p}_t(y,x)=0,\quad \Big(L_x-A-\frac{\partial}{\partial t}\Big)\tilde{p}_t(y,x)=0.
\]
We claim that $\tilde{p}_t(y,x)$ satisfies \eqref{eq:adjoint density eqn for degenerate diffusion}.

We fix $f,g\in C_c^{\infty}(\chi;\Rm_{\geq 0})$ and $0<T<\infty$. We define:
\[
d(t):=\begin{cases}
\int_{\chi}\int_{\chi}\tilde{p}_T(y,x)f(x)g(y)dxdy,\quad t=0\\
\int_{\chi}\int_{\chi}e^{-AT}p_T(x,y)f(x)g(y)dxdy,\quad t=T\\
\int_{\chi}\int_{\chi}\int_{\chi}f(x)g(y)e^{-At}p_t(x,z)\tilde{p}_{T-t}(y,z)dxdydz,\quad 0< t<T
\end{cases}.
\]

We claim that 
\begin{equation}\label{eq:d inequality for adjoint degenerate diff proof}
d(t)\leq d(0)\quad \text{for}\quad 0\leq t<T\quad\text{and}\quad \liminf_{t\ra 0}d(t)\geq d(0).
\end{equation}

We assume without loss of generality that $\Leb(f)=1$, and define $\mu:=f(x)\Leb$. We observe that for $0\leq t<T$ we can write
\[
I_t:=\int_{\chi}e^{-At}\tilde{p}_{T-t}(y,X_t)g(y)\Leb(dy),\quad 0\leq t<T\wedge \tau_{\partial},\quad d(t)=\expE_{\mu}\Big[\Ind(\tau_{\partial}>t)I_t\Big],\quad 0\leq t<T.
\]

We observe that
\[
I_t-\int_0^t\int_{\chi}e^{-As}\underbrace{\big(L^{\ast}_x-A-(L^{\ast}_x-A)\big)}_{=0}\tilde{p}_{T-s}(y,X_s)g(y)\Leb(dy),\quad 0<t<\tau_{\partial}\wedge T,
\]
is a local martingale.

For any $n<\infty$, $\tilde{p}_s(y,x)$ is bounded uniformly over all $(s,x,y)$ such that $y\in \text{spt}(g)$, $d(x,\partial \chi)\geq \frac{1}{n}$ and $s\geq \frac{1}{n}$. Therefore, defining $\tau_n:=\inf\{t>0:d(X_t,\partial \chi)\leq \frac{1}{n}\}\wedge (T-\frac{1}{n})$, we have that $I_{t\wedge \tau_n}$ is a martingale. Furthermore, we observe that $I_t\Ind(\tau_{\partial}>t)\leq \liminf_{n\ra\infty}I_{t\wedge \tau_n}$ for all $0<t<T$. To see this, observe that if $\lim_{n\ra\infty}\tau_n\leq t$, then $\tau_{\partial}\leq t$ so that $I_t\Ind(\tau_{\partial}>t)= 0\leq  \liminf_{n\ra\infty}I_{t\wedge \tau_n}$. On the other hand, if $\lim_{n\ra\infty}\tau_n> t$, then $I_{t\wedge \tau_n}=I_t=I_{t}\Ind(\tau_{\partial}>t)$ for all $n$ large enough. Thus by Fatou's lemma, we have that
\[
d(t)=\expE_{\mu}[I_t\Ind(\tau_{\partial}>t)]\leq \expE_{\mu}[\liminf_{n\ra\infty}I_{t\wedge \tau_n}]\leq \liminf_{n\ra\infty}\expE_{\mu}[I_{t\wedge \tau_n}]=\expE_{\mu}[I_0]=d(0),\quad 0 \leq t<T.
\]
We now take $n$ sufficiently large such that $d(x,\partial \chi)>\frac{2}{n}$ for all $x\in \text{spt}(f)$. Since $\tilde{p}_{T-t}(y,x)$ is uniformly bounded for $d(x,\partial \chi)\geq \frac{1}{n}$, $y\in \text{spt}(g)$ and $t\leq \frac{T}{2}$, it follows from the dominated convergence theorem and the fact that $X_t\ra X_0$ almost surely that we have
\[
I_t\Ind(\tau_n>t)\ra I_0\quad\text{almost surely as}\quad t\ra 0.
\]
Again applying the Dominated convergence theorem, we have that
\[
d(0)=\expE[\lim_{t\ra 0}I_{t}\Ind(\tau_n>t)]=\lim_{t\ra 0}\expE[I_{t}\Ind(\tau_n>t)]\leq \liminf_{t\ra 0}\expE[I_t\Ind(\tau_{\partial}>t)]=\liminf_{t\ra 0}d(t).
\]

We have therefore established \eqref{eq:d inequality for adjoint degenerate diff proof}. We may reverse the above argument to see that
\[
d(T-t)\leq d(T)\quad \text{for}\quad 0\leq t<T\quad\text{and}\quad \liminf_{t\ra 0}d(T-t)\geq d(T).
\]

We therefore see that
\[
d(0)\leq \liminf_{t\ra 0}d(t)\leq d(T)\leq \liminf_{t\ra T}d(t)\leq d(0).
\]

Therefore $d(0)=d(T)$. Since $f,g\in C_c^{\infty}(\chi;\Rm_{\geq 0})$ are arbitrary, we obtain \eqref{eq:adjoint density eqn for degenerate diffusion}.

Finally, \eqref{eq:degen diff section Q and Q0 relationship} follows from the soft killing rate $\kappa$ being non-negative and bounded by $2A$.
\qed

\subsubsection*{Proof of positivity in Part \ref{enum:degenerate diffusion hard killing smoothness transition density} of Proposition \ref{prop:basic facts about degenerate diffusions}}

We now fix $x,y\in \chi$ and $t> 0$. It follows from \eqref{eq:degenerate diffusions full support} and the continuity of $p_t$ that $p_t(x,z)$ and $p_t(z,y)=e^{At}\tilde{p}_t(y,z)$ must both be strictly positive on an open, dense set of $z\in \chi$, hence $p_t(x,z)p_t(z,y)>0$ on an open, dense set of $z\in \chi$. It follows that $p_{2t}(x,y)>0$.
\qed 

\subsection{Proof of Part \ref{enum:positive Cb right e-fn/QSD} of Proposition \ref{prop:basic facts about degenerate diffusions}}

Our first goal is to establish that
\begin{equation}\label{eq:continuity in total variation degenerate diffusion}
\chi\ni x\mapsto P_1(x,\cdot)\in (\calM_{\geq 0}(\chi),\lvert\lvert .\rvert\rvert_{\TV})\quad\text{is continuous.} 
\end{equation}

We define $\chi_n:=\{x\in \chi:d(x,\partial \chi)> \frac{1}{n}\}$. We fix arbitrary $x_0\in \chi$ and take $0<t_0<1$ such that $\Pm_{x_0}(\tau_{\partial}=t_0)=0$. We claim that
\begin{equation}\label{eq:no mass on complement of chi}
\limsup_{r\ra\infty}\limsup_{n\ra\infty}\sup_{x\in \overline{B(x_0,r)}}\lvert\lvert P_{t_0}(x,\chi_n^c)\rvert\rvert_{\TV}=0.
\end{equation}

We define on the same probability space strong solutions $(X^x_t)_{0\leq t<\tau_{\partial}^x}$ to \eqref{eq:degenerate diffusion} with initial conditions $X^x_0=x$, for all $x\in\chi$, driven by the same Brownian motion. We adapt here the definition of $\tau_{\partial}^x$ so that $\tau_{\partial}^x:=\inf\{t>0:X_t^x\in (\bar\chi)^c\}$, which for any fixed $x\in \chi$ has no effect on the distribution (by Proposition \ref{prop:degerenate diffusion prob of being killed goes to 1 as approach bdy} and its proof).

We have that
\[
\begin{split}
\sup_{x\in \overline{B(x_0,r)}}\lvert\lvert P_{t_0}(x,\chi_n^c)\rvert\rvert_{\TV}\leq \Pm(\text{there exists}\quad  x\in \overline{B(x_0,r)}\quad\text{such that}\quad 
X_{t_0-}^x\in \bar \chi\setminus \chi_n).
\end{split}
\]
Now taking $\limsup_{n\ra\infty}$ of both sides we have that
\[
\begin{split}
\limsup_{n\ra\infty}\sup_{x\in \overline{B(x_0,r)}}\lvert\lvert P_{t_0}(x,\bar\chi\setminus\chi_n)\rvert\rvert_{\TV}\leq \Pm(\text{for all $n\in \Nm$ there exists $x\in \overline{B(x_0,r)}$}\\
\text{such that}\quad 
X_{t_0-}^x\in \bar \chi\setminus \chi_n)\leq \Pm(\text{there exists $x\in \overline{B(x_0,r)}$ such that}\quad X_{t_0-}^x\in \partial \chi).
\end{split}
\]
Now taking $\limsup_{r\ra 0}$ of both sides we have that
\[
\begin{split}
\limsup_{r\ra 0}\limsup_{n\ra\infty}\sup_{x\in \overline{B(x_0,r)}}\lvert\lvert P_{t_0}(x,\bar\chi\setminus\chi_n)\rvert\rvert_{\TV}\leq \Pm(X_{t_0}^{x_0}\in \partial \chi)=\Pm(\tau^{x_0}=t_0)=0.
\end{split}
\]

Therefore, for arbitrary $\epsilon>0$ we may take $n\in \Nm$ and $r>0$ such that $\sup_{x\in \overline{B(x,r)}}\lvert\lvert P_{t_0}(x,\chi_n^c)\rvert\rvert_{\TV}<\epsilon$. We have from Part \ref{enum:degenerate diffusion hard killing smoothness transition density} (without using the strict positivity of $p_t$, which we have not yet established) that $\lvert\lvert P_{t_0}(x,\cdot)_{\lvert_{\chi_n}}-P_{t_0}(x_0,\cdot)_{\lvert_{\chi_n}}\rvert\rvert_{\TV}\ra 0$ as $x\ra x_0$. It follows that $\limsup_{x\ra x_0}\lvert\lvert P_{t_0}(x,\cdot)-P_{t_0}(x_0,\cdot)\rvert\rvert_{\TV}\leq 2\epsilon$. Since $\epsilon>0$ is arbitrary and $\lvert\lvert P_1(x,\cdot)-P_1(x_0,\cdot)\rvert\rvert_{\TV}\leq \lvert\lvert P_{t_0}(x,\cdot)-P_{t_0}(x_0,\cdot)\rvert\rvert_{\TV}$, 
\[
\lvert\lvert P_{1}(x,\cdot)-P_{1}(x_0,\cdot)\rvert\rvert_{\TV}\ra 0\quad\text{as} \quad x\ra x_0. 
\]
Finally, since $x_0\in\chi$ is arbitrary, \eqref{eq:continuity in total variation degenerate diffusion} follows.

Since $\Pm_x(\tau_{\partial}>1)\ra 0$ as $x\ra \bar \chi$ by Proposition \ref{prop:degerenate diffusion prob of being killed goes to 1 as approach bdy}, this extends to a continuous function $\bar \chi\ra (\calM(\chi),\lvert\lvert .\rvert\rvert_{\TV})$, vanishing on $\partial \chi$. Since $\bar \chi$ is compact, this is uniformly continuous. Now for $f\in C_0(\chi)$ we therefore have that $P_1f\in C_0(\chi)$, and that
\[
\lvert P_1f(x)-P_1f(y)\rvert \leq \lvert\lvert f\rvert\rvert_{\infty}\lvert\lvert P_1(x,\cdot)-P_1(y,\cdot)\rvert\rvert_{\TV}.
\]
It therefore follows that $\{P_1f:f\in C_0(\chi),\quad \lvert\lvert f\rvert\rvert_{\infty}\leq 1\}$, considered as a subset of $C_b(\bar \chi)$ by extension, is equicontinuous. It is clearly also uniformly bounded. Therefore
\[
P_1:C_0(\chi)\ni f\mapsto P_1f\in C_0(\chi)
\]
is a compact operator by the Arzela-Ascoli theorem, which fixes the cone $C_0(\chi;\Rm_{\geq 0})$.

We now take $B(y,r)$ compactly contained in $\chi$. By \eqref{eq:degenerate diffusions full support} and the continuity of $p_t$, there exists $c>0$ such that $P_1(x,B(y,r))\geq c$ for all $x\in B(y,r)$. It follows from the spectral radius formula that the spectral radius of $P_1$ is strictly positive, $r(P_1)>0$.

The Krein-Rutman theorem therefore implies the existence of $h\in C_0(\chi;\Rm_{\geq 0})\setminus\{0\}$ and $\lambda=r(P_1)>0$ such that $P_1h=\lambda h$. Since $h$ must be strictly positive on some open subset of $\chi$, it follows from \eqref{eq:degenerate diffusions full support} that $h$ is everywhere strictly positive, $h\in C_0(\chi;\Rm_{>0})$. 

Repeating the above argument with $P_1$ replaced by $\tilde{P}_1$, and using \eqref{eq:adjoint density eqn for degenerate diffusion}, we obtain \eqref{eq:deg diff basic properties pve qsd} with $\lambda$ replaced by some $\tilde{\lambda}>0$. By considering $\int_{\chi}\int_{\chi}\rho(x)p_1(x,y)h(y)\Leb(dx)\Leb(dy)$ we see that $\tilde{\lambda}=\lambda$, giving \eqref{eq:deg diff basic properties pve qsd}.
\qed

This completes the proof of Proposition \ref{prop:basic facts about degenerate diffusions}.
\qed

\subsection*{Proof of Proposition \ref{prop:degenerate diffusion satisfies A1}}

It is an immediate consequence of propositions \ref{prop:degerenate diffusion prob of being killed goes to 1 as approach bdy} and \ref{prop:basic facts about degenerate diffusions} that
$\Pm_x(\tau_{\partial}>t)>0$ and $\Pm_x(\tau_{\partial}<\infty)>0$ for all $x\in\chi, t\geq 0$. We shall establish the following lemma.
\begin{lem}\label{lem:degenerate diffusion positive conditional prob of being in compact set}
There exists compact $K\subseteq \chi$ and $c_2>0$ such that $\Pm_x(X_{\frac{1}{2}}\in K\lvert \tau_{\partial}>\frac{1}{2})\geq c_2$ for all $x\in\chi$.
\end{lem}

This then immediately implies Proposition \ref{prop:degenerate diffusion satisfies A1} by Part \ref{enum:degenerate diffusion hard killing smoothness transition density} of Proposition \ref{prop:basic facts about degenerate diffusions} and Remark \ref{rmk:remark for checking technical assum for Assum (A)}, since
\[
\Pm_x\Big(X_1\in \cdot\Big\lvert \tau_{\partial}>1\Big)\geq \Pm_x\Big(X_1\in \cdot\Big\lvert \tau_{\partial}>\frac{1}{2}\Big)\geq \Pm_x\Big(X_{\frac{1}{2}}\in K\Big\lvert \tau_{\partial}>\frac{1}{2}\Big)\inf_{\substack{x\in K,\\ y\in B(x^{\ast},r)}}p_{\frac{1}{2}}(x,y)\Leb_{\lvert_{B(x^{\ast},r)}}(\cdot).
\]

\subsubsection*{Proof of Lemma \ref{lem:degenerate diffusion positive conditional prob of being in compact set}}

We fix an initial condition $x\in\chi$ for the time being. In the following, it is important that the bounds we shall obtain shall not be dependent upon the choice of $x$. 

We define the process $\hat{X}_t$ to have initial condition $\hat{X}_0=x$, to evolve as a solution to \eqref{eq:degenerate diffusion} between jumps, and to jump upon contact with the boundary (say at time $t$, $\hat{X}_{t-}\in \partial \chi$) according to $\hat{X}_t\sim \Law_{x}(X_t\lvert \tau_{\partial}>t)$. Then we have that $\Law_x(\hat{X}_t)=\Law_{x}(X_t\lvert \tau_{\partial}>t)$.

We recall Fact \ref{fact:distance to the bdy unit derivative}: $d_{\partial \chi}(x):=d(x,\partial \chi)$ belongs to $C^{3}_b(\chi\cap B(\partial \chi,5\epsilon))$ with $\lvert \nabla d_{\partial \chi}\rvert\equiv 1$ on $B(\partial \chi,5\epsilon)\cap \chi$, for some $\epsilon>0$ sufficiently small enough. Reducing $\epsilon>0$ if necessary, it follows from Assumption \ref{assum:vector field degenerate diffusion traverses bdy} that $\sum_i\lvert (\nabla d_{\partial \chi})\cdot v^i\rvert^2$ is bounded away from $0$ on $B(\partial \chi,5\epsilon)\cap\chi$. We further define $\varphi\in C^{\infty}(\Rm;\Rm_{\geq 0})$ such that
\[
\varphi(u)=1\quad\text{for}\quad u\leq 2\epsilon,\quad \varphi(u)\in (0,1)\quad\text{for}\quad 2\epsilon <u<3\epsilon\quad\text{and}\quad \varphi(u)=0\quad\text{for}\quad u\geq 3\epsilon.
\]
Then $(\varphi\circ d_{\partial \chi}) d_{\partial \chi}\in C^3(\chi)$.

In the following, for a given process $D_t$, $L^D_t$ refers to the local time of $D_t$ at $0$. Thus in an SDE, a term of the form $dL^D_t$ corresponds to reflection at a lower boundary at $0$, whilst $-dL^{\epsilon-D}_t$ corresponds to reflection an an upper boundary at $\epsilon$. We also note that $\circ$ shall be used both to denote Stratonovich integration and composition of functions.

We define $W_t$ to be another independent Brownian motion, and define $Z_t$ to be a solution of
\[
dZ_t=dL^{Z}_t-dL^{\epsilon-Z}_t+\nabla [(\varphi\circ d_{\partial \chi})d_{\partial \chi}](X_t)\circ dX_t+(1-\varphi\circ d_{\partial \chi})(X_t)dB_t,\quad Z_0=0.
\]
We note that $0\leq Z_t\leq \epsilon$ for all $t$, almost surely, and that in between the jumps of $X_t$,
\[
dd_{\partial \chi}(X_t)-dZ_t=dL_t^{\epsilon-Z}-dL_t^Z+\nabla[(1-(\varphi\circ d_{\partial \chi}))d_{\partial \chi}](X_t)\circ dX_t+(1-\varphi\circ d_{\partial \chi})(X_t)dB_t.
\]
We observe that the latter two terms on the right hand side are $0$ when $d_{\partial \chi}(X_t)\leq 2\epsilon$. We note also that $d_{\partial \chi}(X_t)$ strictly increases at jumps, which occur when $X_{t-}=0$. It follows that $d_{\partial \chi}(X_t)\geq Z_t$ for all $t$.

We now claim that
\begin{equation}\label{eq:expression for diffusivity of reflected process degenerate diff pf}
\sum_i\big(\nabla [(\varphi\circ d_{\partial \chi})d_{\partial \chi}](x)\cdot v^i(x)\big)^2+\big(1-\varphi\circ d_{\partial \chi}(x)\big)^2
\end{equation}
is uniformly bounded from below away from $0$.

We observe that for $d_{\partial \chi}(x)\leq 2\epsilon$, we have that $\sum_i\big(\nabla [(\varphi\circ d_{\partial \chi})d_{\partial \chi}](x)\cdot v^i(x)\big)^2=\sum_i(\nabla d_{\partial \chi}(x)\cdot v^i(x))^2$, which is bounded away from $0$, so this is true on $\{x:d_{\partial \chi}(x)\leq 2\epsilon\}$. For $d_{\partial \chi}>2\epsilon$, on the other hand, $(1-\varphi\circ d_{\partial \chi}(x))^2$ is positive, so \eqref{eq:expression for diffusivity of reflected process degenerate diff pf} is positive. Since it is also continuous and $\{x\in \chi:d_{\partial \chi}(x)\geq 2\epsilon\}$ is compact, it follows that \eqref{eq:expression for diffusivity of reflected process degenerate diff pf} is bounded from below away from $0$, so that the diffusivity of $Z_t$ is bounded from below away from $0$. 

It is clear, also, that the drift and diffusivity of the above SDE are both bounded. Therefore there exists a random time change $\tau(s)$ and process $A_s$ such that:
\begin{enumerate}
\item
There exists constants $0<c<C<\infty$ such that $cs\leq \tau(s)\leq Cs$ for all $s\geq 0$, almost surely.
\item
$A_s=Z_{\tau(s)}$ for all $s\geq 0$, almost surely.
\item
$A_s$ satisfies the SDE
\begin{equation}\label{eq:reflected SDE with non-constant drift degen diff pf}
dA_s=dL^A_s-dL^{\epsilon-A}_s+d\tilde{W}_s+b_sds,\quad A_0=0,
\end{equation}
whereby $\tilde{W}_s$ is a Brownian motion and $\lvert b_s\rvert\leq B$ for all $s$, almost surely.
\end{enumerate} 
This then dominates $R_s$, a strong solution (which exists and is unique by \cite[Theorem 3.1]{Lions1984a}) to 
\begin{equation}\label{eq:reflected SDE with constant drift degen diff pf}
dR_s=dL^R_s-dL^{\epsilon-R}_s+d\tilde{W}_s-B ds,\quad R_0=0,
\end{equation}
where $\tilde{W}_t$ is the same Brownian motion as in \eqref{eq:reflected SDE with non-constant drift degen diff pf}. In particular, we have that
\[
0\leq R_s\leq A_s= Z_{\tau(s)}\leq d_{\partial \chi}(X_{\tau(s)})\quad\text{for all}\quad s\geq 0.
\]
We may choose $0<s_1<s_2$ such that $\tau(s_1)<\frac{1}{2}<\tau(s_2)$ almost surely. Then we have that
\[
\begin{split}
\Pm_{\mu}\Big(d_{\partial \chi}(X_{\frac{1}{2}})>\frac{\epsilon}{2}\Big\lvert \tau_{\partial}>\frac{1}{2}\Big)=\Pm_{\mu}\Big(d_{\partial \chi}(\hat{X}_{\frac{1}{2}})>\frac{\epsilon}{2}\Big)\\
\geq \Pm\Big(\inf_{\tau(s_1)\leq t'\leq \tau(s_2)}d_{\partial \chi}(\hat{X}_{t'})>\frac{\epsilon}{2}\Big)\geq \Pm\Big(\inf_{s_1\leq s'\leq s_2}R_{s'}>\frac{\epsilon}{2}\Big).
\end{split}
\]
Since this last probability is positive, and not dependent on $x$, we are done. 
\qed

Having established Lemma \ref{lem:degenerate diffusion positive conditional prob of being in compact set}, we have completed the proof of Proposition \ref{prop:degenerate diffusion satisfies A1}. 
\qed

\subsubsection*{Conclusion of the proof of Theorem \ref{theo:defenerate diffusions hard killing}}

Proposition \ref{prop:degenerate diffusion satisfies A1} implies that $(X_t)_{0\leq t<\tau_{\partial}}$ satisfies Assumption \ref{assum:technical assumption for Assum (A)} and \cite[Assumption (A1)]{Champagnat2014}. Proposition \ref{prop:basic facts about degenerate diffusions} provides a pointwise right eigenfunction for $P_1$ belonging to $C_0(\chi;\Rm_{>0})$, whence $(X_t)_{0\leq t<\tau_{\partial}}$ satisfies \cite[Assumption (A2)]{Champagnat2014} by Proposition \ref{prop:right efn gives A2}. It follows, in particular, that $(X_t)_{0\leq t<\tau_{\partial}}$ has a unique QSD, $\pi$.

It is left to establish that $(X_t)_{0\leq t<\tau_{\partial}}$ and $\pi$ satisfy Assumption \ref{assum:adjoint Dobrushin main results section}.

We may repeat the above argument to the discrete-time killed Markov chain obtained by only considering integer times, to see that $P_1$ must have a unique QSD, which must be $\pi$, the QSD for $(P_t)_{0\leq t<\tau_{\partial}}$. On the other hand, this unique QSD for $P_1$ must also have a density with respect to Lebesgue belonging to $C_0(\chi;\Rm_{>0})$ by Part \ref{enum:positive Cb right e-fn/QSD} of Proposition \ref{prop:basic facts about degenerate diffusions}. Therefore $\pi\in \calP_{\infty}(\Leb)$ with (a version of) $\frac{d\pi}{d\Leb}$ belonging to $C_0(\chi;\Rm_{>0})$.

We have from Part \ref{enum:exist of Y and Y0 adjoint processes} of Proposition \ref{prop:basic facts about degenerate diffusions} that we have
\begin{equation}\label{eq:adjoint eqn for P1 tilde P1 conclusion of degenerate diffusion proof}
\Leb(dx)P_1(d,dy)=e^{-A}\Leb(dy)\tilde{P}_1(y,dx),
\end{equation}
where $0\leq A<\infty$ is the constant given by Part \ref{enum:exist of Y and Y0 adjoint processes} of Proposition \ref{prop:basic facts about degenerate diffusions}. It also follows from Proposition \ref{prop:basic facts about degenerate diffusions} that 
\begin{equation}\label{eq:conclusion of the degenerate diffusion proof positive mass}
\tilde{P}_11(y)>0\quad\text{for all}\quad y\in \chi.
\end{equation}

We may apply Proposition \ref{prop:degenerate diffusion satisfies A1} to $(\tilde{X}^0_t)_{0\leq t<\tilde{\tau}_{\partial}^0}$, the killed process provided for by Part \ref{enum:exist of Y and Y0 adjoint processes} of Proposition \ref{prop:basic facts about degenerate diffusions}; we write $c_1>0$ and $\nu$ respectively for the positive constant and probability measure given by doing so. Propositions \ref{prop:basic facts about degenerate diffusions} and \ref{prop:degenerate diffusion satisfies A1} therefore imply that
\[
\frac{\tilde{P}_1(y,dx)}{\tilde{P}_11(y)}\geq e^{-2A}\frac{\tilde{P}^0_1(y,dx)}{\tilde{P}^0_11(y)}\geq e^{-2A}c_1\nu.
\]
Since (a version of) $\frac{d\pi}{d\Leb}$ belongs to $C_0(\chi;\Rm_{>0})$, it follows from the description of $\nu$ given by Proposition \ref{prop:degenerate diffusion satisfies A1} that $\nu$ and $\pi$ are not mutually singular. 
\qed

\subsection*{Proof of Theorem \ref{theo:degenerate diffusions hard killing bounded density wrt pi}}

It is an immediate consequence of Part \ref{enum:degenerate diffusion hard killing smoothness transition density} of Proposition \ref{prop:basic facts about degenerate diffusions} that $\text{spt}(\pi)=\chi$.

We now seek to verify Assumption \ref{assum:adjoint anti-Dobrushin main results section}. We take $\tilde{P}_1$ constructed in Part \ref{enum:exist of Y and Y0 adjoint processes} of Proposition \ref{prop:basic facts about degenerate diffusions}, which we recall satisfies \eqref{eq:adjoint eqn for P1 tilde P1 conclusion of degenerate diffusion proof}. Therefore \eqref{eq:bounded Lebesgue main results section} is satisfied by the kernel $\tilde{P}_1$ at the time $1$. We already have from \eqref{eq:conclusion of the degenerate diffusion proof positive mass} that $\tilde{P}1(y)>0$ for all $y\in\chi$.

All that remains is to check that $\tilde{P}_1$ satisfies \eqref{eq:psi adjoint for verifying bounded pi euclidean condition main results section} and that $x\mapsto P_2(x,\cdot)$ is lower semicontinuous (in the sense of \eqref{eq:lower semicty of kernel}). We let $p_t(x,y)$ be the transition density of $X_t$ as given by Part \ref{enum:degenerate diffusion hard killing smoothness transition density} of Proposition \ref{prop:basic facts about degenerate diffusions}, with $\tilde{p}_t(y,x)$ the transition density and $A\geq 0$ the constant given by Part \ref{enum:adjoint transition density degenerate diffusion} of Proposition \ref{prop:basic facts about degenerate diffusions}. We now seek to show that $p_t$ and $\tilde{p}_t$ are bounded for any fixed $t>0$.

Since the parabolic H\"{o}rmander condition is satisfied on open sets, we can construct some connected, bounded, open set $\hat{\chi}\supset \bar \chi$ with $C^{\infty}$ boundary on which $v^1,\ldots,v^m$ satisfy Assumption \ref{assum:strong Hormander}. We now take $C^{\infty}(\Rm^d)$ functions $\hat{v}^1,\ldots,\hat{v}^r$ such that $\hat{v}^j\equiv 0$ on $\bar \chi$ and $\{v^1,\ldots,v^m,\hat{v}^1,\ldots,\hat{v}^r\}$ satisfy Assumption \ref{assum:vector field degenerate diffusion traverses bdy} on $\hat{\chi}$. We see that $\{v^1,\ldots,v^m,\hat{v}^1,\ldots,\hat{v}^r\}$ must also satisfy Assumption \ref{assum:strong Hormander} on $\hat{\chi}$. 

We then consider strong solutions to 
\begin{equation}\label{eq:degenerate diffusion on larger domain}
d\hat{X}_t=v^0(\hat{X}_t)dt+\sum_{j=1}^mv^j(\hat{X}_t)\circ dB^j_t+\sum_{j=1}^k\hat{v}^j(\hat{X}_t)\circ dB^j_t,\quad 0\leq t<\hat{\tau}_{\partial}:=\inf\{s>0:\hat{X}_{s-}\in \partial \chi\},
\end{equation}
whereby $B^1,\ldots,B^m$ are the Brownian motions driving $X_t$ and $\hat{B}^1,\ldots,\hat{B}^k$ are an additional $k$ independent Brownian motions.

We can apply Proposition \ref{prop:basic facts about degenerate diffusions} to see that $\hat{X}_t$ has a transition density given by some $\hat{p}_t\in C^{\infty}((0,\infty)\times \hat{\chi}\times \hat{\chi})$. We see that $\hat{X}_t=X_t$ up to the time, $\tau_{\partial}$, when they leave $\chi$. At this time $X_t$ is killed. Therefore, for all $t>0$, $e^A\tilde{p}_t(y,x)=p_t(x,y)\leq \hat{p}_t(x,y)$ for all $x,y\in \chi$, where $A$ is the constant provided for by Part \ref{enum:exist of Y and Y0 adjoint processes} of Proposition \ref{prop:basic facts about degenerate diffusions}. Since $\bar \chi\times \bar \chi$ is compact, $\hat{p}_t$ must be bounded on $\bar \chi\times \bar \chi$ for any fixed $t>0$, so that
\begin{equation}\label{eq:boundedness of p1 tilde p1 degenerate diffusion}
p_t\in C_b(\chi\times \chi),\quad\text{and similarly}\quad \tilde{p}_t\in C_b(\chi\times \chi).
\end{equation}
It immediately follows from \eqref{eq:boundedness of p1 tilde p1 degenerate diffusion} that there exists $C_1<\infty$ such that
\[
\tilde{P}_1(x,\cdot)\leq C_1\Leb(\cdot)\quad\text{for all}\quad x\in\chi.
\]
It is also immediate from \eqref{eq:boundedness of p1 tilde p1 degenerate diffusion} and the dominated convergence theorem that $P_tf$ is lower semi continuous for all $f\in C_b(\chi;\Rm_{\geq 0})$, so that $(X_t)_{0\leq t<\tau_{\partial}}$ is lower semicontinuous (in the sense of Definition \ref{defin:lower semicts kernel}).
\qed

\section{$1+1$-dimensional Langevin dynamics}

We define $\gamma\in\Rm$ and $\sigma>0$ to be constants; $\calO$ to be a bounded, open subinterval of $\Rm$ and $F\in C^{\infty}(\Rm)$ to be a smooth function on $\Rm$ (whose values on $(\bar \calO)^c$ are arbitrary). We define the state space $\chi:=\calO\times \Rm$. Throughout this section, the distinguished measure $\Lambda$ should be understood to be Lebesgue measure on $\chi=\calO\times \Rm$. 

We consider $\calO\times \Rm$-valued solutions $((q_t,p_t))_{0\leq t<\tau_{\partial}}$ of the absorbed Langevin SDE
\begin{equation}\label{eq:Langevin SDE}
\begin{cases}
dq_t=p_tdt\\
dp_t=F(q_t)dt-\gamma p_tdt+\sigma dB_t
\end{cases}
,\quad 0\leq t<\tau_{\partial}:=\inf\{s>0:q_{s-}\in \partial \calO\},\quad (q_0,p_0)=(q^0,p^0).
\end{equation}

We define $\chi:=\calO\times \Rm$ and $(X_t)_{0\leq t<\tau_{\partial}}:=((q_t,p_t))_{0\leq t<\tau_{\partial}}$. Moreover we write $P_t(x,\cdot)$ for the associated submarkovian transition kernel.

A detailed study of this process and its associated Fokker-Planck equation has been undertaken by Leli\`evre, Ramil and Reygner in \cite{Lelievre2021,Lelievre2022}, in the much more general setting where the dimension is arbitrary. In particular, they established in \cite[Theorem 2.13]{Lelievre2021} that there exists a unique quasi-stationary distribution for $(X_t)_{0\leq t<\tau_{\partial}}$, which we call $\pi$. We define $\lambda:=\lambda(\pi)=\Pm_{\pi}(\tau_{\partial}>1)$.

They established in \cite[theorems 2.12 and 2.13]{Lelievre2021} that there exists $h\in C_b(\chi;\Rm_{>0})$ which is the unique (up to a multiplicative constant) pointwise right eigenfunction of $P_t$ of eigenvalue $\lambda^t$ belonging to $C_b(\chi)$, for any $t>0$. We shall choose the normalisation $\pi(h)=1$. Note, in particular, that $h$ is bounded and everywhere strictly positive.

Furthermore, \cite[Theorem 2.21]{Lelievre2021} gives non-uniform exponential convergence in total variation of the distribution conditioned on survival to the QSD: there exists $C<\infty$ and $\gamma>0$ such that
\begin{equation}\label{eq:Langevin non-unif conv to a QSD}
\lvert\lvert \Law_{\mu}(X_t\lvert\tau_{\partial}>t)-\pi\rvert\rvert_{\TV}\leq \frac{C}{\mu(h)}e^{-\gamma t}.
\end{equation}

We prove the following
\begin{theo}\label{theo:1D Langevin satisfies Reverse Dobrushin}
The process $((q_t,p_t))_{0\leq t<\tau_{\partial}}$ satisfies Assumption \ref{assum:technical assumption for Assum (A)} and \cite[Assumption (A)]{Champagnat2014}. Its unique QSD, $\pi$, belongs to $\calP_{\infty}(\Leb)$ and has full support. Moreover $((q_t,p_t))_{0\leq t<\tau_{\partial}}$ is lower semicontinuous (in the sense of Definition \ref{defin:lower semicts kernel}) and satisfies assumptions \ref{assum:adjoint Dobrushin main results section} and \ref{assum:adjoint anti-Dobrushin main results section}.
\end{theo}

It follows from Theorem \ref{theo:uniform Linfty convergence main results section} that there exists a time $T<\infty$ and constant $\gamma>0$ such that $\Law_{\mu}(X_t\lvert \tau_{\partial}>t)\ll_{\infty}\pi$ for all $t\geq T$ and $\mu\in\calP(\chi)$, with its density with respect to $\pi$ satisfying
\begin{equation}
\Big\lvert\Big\lvert \frac{d\Law_{\mu}(X_t\lvert \tau_{\partial}>t)}{d\pi}-1\Big\rvert\Big\rvert_{L^{\infty}(\pi)}\leq e^{-\gamma(t-T)}\quad\text{for all}\quad  T\leq t<\infty,\quad\mu\in\calP(\chi).
\end{equation}

We may observe that, over the course of proving Theorem \ref{theo:1D Langevin satisfies Reverse Dobrushin}, we have established that Aassumption \ref{assum:combined Dobrushin adjoint Dobrushin main results section} (which includes \cite[Assumption (A1)]{Champagnat2014}) is satisfied by $((q_t,p_t))_{0\leq t<\tau_{\partial}}$ along the way. Moreover, it is also clear that the time horizon over which we establish assumptions \ref{assum:adjoint Dobrushin main results section}, \ref{assum:adjoint anti-Dobrushin main results section} and \ref{assum:combined Dobrushin adjoint Dobrushin main results section} can be made arbitrarily small, without any changes to the proof. We therefore obtain from theorems \ref{theo:dominated by pi theorem general results main results section} and \ref{theo:DAD lower bounds density main results section} the following.
\begin{theo}\label{theo:comparison inequality for 1D Langevin}
For all $t>0$ there exists $0<c_t\leq C_t<\infty$ such that
\begin{equation}
c_t\pi\leq \Law_{\mu}((q_t,p_t)\lvert \tau_{\partial}>t)\leq C_t\pi\quad\text{for all}\quad \mu\in\calP(\chi).
\end{equation}
\end{theo}
We put this in the form of a parabolic boundary Harnack inequality as follows. For any initial conditions $\mu,\nu\in\calP(\chi)$, we let $u_1((q,p),t)$ and $u_2((q,p),t)$ be continuous versions (see \cite[Theorem 2.20]{Lelievre2022} for a justification that this exists) of $\frac{d\Pm_{\mu}((q_t,p_t)\in \cdot,\tau_{\partial}>t)}{dLeb(\cdot)}$ and $\frac{d\Pm_{\nu}((q_t,p_t)\in \cdot,\tau_{\partial}>t)}{dLeb(\cdot)}$ for $(q,p)\in \chi$ and $t>0$, respectively. It follows that for all $t>0$ we have
\begin{equation}\label{eq:parabolic boundary Harnack Langevin}
\inf_{t_1,t_2\geq t}\frac{\inf_{(q,p)\in \chi}\Big(\frac{u_1((q,p),t_1)}{u_2((q,p),t_2)}\Big)}{\sup_{(q',p')\in \chi}\Big(\frac{u_1((q',p'),t_1)}{u_2((q',p'),t_2)}\Big)}\geq \frac{c_t^2}{C_t^2}>0.
\end{equation} 
Note in particular that the constants $0<c_t<C_t<\infty$ do not depend upon $\mu$ and $\nu$, and that this comparison is valid up to the boundary.

\subsection*{Proof of Theorem \ref{theo:1D Langevin satisfies Reverse Dobrushin}}

Without loss of generality we may assume that $\calO=(0,1)$.

We begin by recalling from \cite{Lelievre2021,Lelievre2022} some properties of $(X_t)_{0\leq t<\tau_{\partial}}$.

\begin{enumerate}
\item There exists by \cite[Theorem 2.20]{Lelievre2022}
\begin{equation}\label{eq:transition density 1D Langevin}
p_t(x,y)\in C((0,\infty)\times \chi\times \chi;\Rm_{> 0}) \quad\text{such that}\quad p_t\in C_b([T,\infty)\times\chi\times \chi;\Rm_{> 0})\quad\text{for all fixed $T>0$,}
\end{equation}
providing for the transition density of the submarkovian kernel $P_t(x,\cdot)$,
\[
P_t(x,dy)=p_t(x,y)\Leb(dy)\quad\text{for all}\quad x\in  \chi.
\]
Note that in the above statement, \eqref{eq:transition density 1D Langevin} includes the statement that the transition densities are everywhere strictly positive.
\item
The adjoint Langevin process $(\tilde{X}_t)_{0\leq t<\tilde{\tau}_{\partial}}=((\tilde{q}_t,\tilde{p}_t))_{0\leq t<\tilde{\tau}_{\partial}}$ corresponds to the solution of the SDE
\begin{equation}\label{eq:adjoint Langevin SDE}
\begin{split}
\begin{cases}
d\tilde{q}_t=-\tilde{p}_tdt\\
d\tilde{p}_t=-F(\tilde{q}_t)dt+\gamma \tilde{p}_tdt+\sigma dB_t
\end{cases}
,\quad 0\leq t<\tau_{\partial}:=\inf\{s>0:\tilde{q}_{s-}\in \partial \calO\},\\ (\tilde{q}_0,\tilde{p}_0)=(\tilde{q}^0,\tilde{p}^0).
\end{split}
\end{equation}
We write $\tilde{P}_t$ for the associated submarkovian kernel and $\tilde{p}_t$ for the transition densities as given by \eqref{eq:transition density 1D Langevin}. Then \cite[Theorem 2.7]{Lelievre2021} gives that
\begin{equation}\label{eq:transition density of Langevin adjoint Langevin relation}
p_t(x,y)=e^{\gamma t}\tilde{p}_t(y,x)\quad\text{for all}\quad x,y\in\chi.
\end{equation}
\end{enumerate}

It is immediate that $\Pm_x(\tau_{\partial}>t)>0$ and $\Pm_x(\tau_{\partial}<\infty)>0$ for all $x\in\chi$ and $0\leq t<\infty$.

The main ingredient in the proof of Theorem \ref{theo:1D Langevin satisfies Reverse Dobrushin} is the following proposition.
\begin{prop}\label{prop:1D Langevin satisfies (A1)}
There exists $t_1,c_1>0$ and $\nu\in\calP(\calO\times \Rm)$ such that
\begin{equation}\label{eq:(A1) for 1D Langevin}
\Law_{(q^0,p^0)}((q_{t_1},p_{t_1})\lvert \tau_{\partial}>t_1)(\cdot)\geq c_1\nu(\cdot)\quad\text{for all}\quad (q^0,p^0)\in \chi,
\end{equation}
whereby $\nu=4\Leb_{\lvert_{[\frac{1}{4},\frac{3}{4}]\times [\frac{1}{4},\frac{3}{4}]}}$.
\end{prop}

We defer for later the proof of Proposition \ref{prop:1D Langevin satisfies (A1)}.

There exists by \cite[theorems 2.12 and 2.13]{Lelievre2021} $h\in C_b(\chi;\Rm_{>0})$ which is an everywhere strictly positive, pointwise right eigenfunction for $P_1$. It therefore follows from Proposition \ref{prop:right efn gives A2} and Remark \ref{rmk:remark for checking technical assum for Assum (A)} that $(X_t)_{0\leq t<\tau_{\partial}}$ satisfies \cite[Assumption (A)]{Champagnat2014} and Assumption \ref{assum:technical assumption for Assum (A)}. We also have from \cite[theorems 2.12 and 2.13]{Lelievre2021} that $\pi\in \calP_{\infty}(\Leb)$ with (a version of) $\frac{d\pi}{d\Leb}$ belonging to $C_b(\chi;\Rm_{>0})$ (so everywhere strictly positive, in particular). In particular, $\text{spt}(\pi)=\chi$.

We now establish that $(X_t)_{0\leq t<\tau_{\partial}}$ and $\pi$ satisfies Assumption \ref{assum:adjoint Dobrushin main results section}. Since $(\tilde{X}_t)_{0\leq t<\tilde{\tau}_{\partial}}$ satisfies \eqref{eq:adjoint Langevin SDE}, which is of the same form as \eqref{eq:Langevin SDE}, it must be the case that $\tilde{P}_t1(y)>0$ for all $y\in \chi$ and $t>0$. It must also be the case that, for some constant $\tilde{c}_1>0$ and time $\tilde{t}_1>0$,
\[
\frac{\tilde{P}_{\tilde{t}_1}(y,\cdot)}{\tilde{P}_{\tilde{t}_1}1(y)}\geq \tilde{c}_1\nu(\cdot),
\]
where $\nu$ is the probability measure given in Proposition \ref{prop:1D Langevin satisfies (A1)}. Since $\frac{d\pi}{d\Leb}$ has a version which is everywhere strictly positive, $\pi$ and $\nu$ are not mutually singular.

We have from \eqref{eq:transition density of Langevin adjoint Langevin relation} that
\begin{equation}\label{eq:adjoint equation for P P tilde Langevin}
P_{{\tilde{t}_1}}(x,dy)=e^{\gamma {\tilde{t}_1}}\tilde{P}_{{\tilde{t}_1}}(y,dx).
\end{equation}

Therefore $(X_t)_{0\leq t<\tau_{\partial}}$ satisfies Assumption \ref{assum:adjoint Dobrushin main results section}.

We now seek to verify Assumption \ref{assum:adjoint anti-Dobrushin main results section}. We already have that $\tilde{P}_{\tilde{t}_1}$ satisfies \eqref{eq:adjoint equation for P P tilde Langevin}, and that $\tilde{P}_{\tilde{t}_1}1(y)>0$ for all $y\in\chi$. 
It immediately follows from \eqref{eq:transition density 1D Langevin} and \eqref{eq:transition density of Langevin adjoint Langevin relation} that there exists $C_1<\infty$ such that
\[
\tilde{P}_1(x,\cdot)\leq C_1\Leb(\cdot)\quad\text{for all}\quad x\in\chi.
\]

We have therefore established that $(X_t)_{0\leq t<\tau_{\partial}}$ satisfies assumptions \ref{assum:adjoint Dobrushin main results section} and \ref{assum:adjoint anti-Dobrushin main results section}.

It immediately follows from \eqref{eq:transition density 1D Langevin} that $P_tf$ is lower semicontinuous for all $f\in C_b(\chi;\Rm_{\geq 0})$ and $t>0$, by application of the dominated convergence theorem. Therefore $(X_t)_{0\leq t<\tau_{\partial}}$ is lower semicontinuous (in the sense of Definition \ref{defin:lower semicts kernel}).

We have left only to establish Proposition \ref{prop:1D Langevin satisfies (A1)}.

\subsubsection*{Proof of Proposition \ref{prop:1D Langevin satisfies (A1)}}

We recall that, without loss of generality, we have assumed that $\mathcal{O}=(0,1)$. 

We take $0<\delta<1$ such that $1000\delta (\lvert\lvert F\rvert\rvert_{\infty}+\lvert\gamma\rvert)<\sigma\wedge 1$, then take $K<\infty$ such that $K\delta>1000$ and $K>1000\delta (\lvert\lvert F\rvert\rvert_{\infty}+\lvert \gamma\rvert)$. We define the set $\kappa$ and stopping time $\tau_{\kappa}$ to be given by
\[
\kappa:=\{(q,p)\in\chi:\lvert p\rvert\leq K\},\quad \tau_{\kappa}:=\inf\{t>0:X_t\in\kappa\}=\inf\{t>0:\lvert p_t\rvert \leq K\}.
\]
We observe that 
\begin{equation}\label{eq:survive for time delta then hit kappa in time delta langevin pf}
\text{if} \quad\tau_{\partial}>\delta\quad \text{then}\quad \tau_{\kappa}<\delta.
\end{equation}
For $\epsilon>0$ to be determined we define
\[
V:=V_-\cup V_+\quad\text{whereby}\quad V_-:=\{(q,p)\in \kappa:q,p< \epsilon\}\quad\text{and}\quad V_+:=\{(q,p)\in \kappa:q> 1-\epsilon,p>-\epsilon\}.
\]
We observe by \eqref{eq:transition density 1D Langevin} that there exists $c_1(\epsilon)>0$, dependent upon $\epsilon>0$, such that
\begin{equation}\label{eq:minorisation of dist from kappa not V}
\Pm_{x}(X_t\in\cdot)\geq c_1\nu(\cdot)\quad\text{for all}\quad x\in \kappa\setminus V\quad\text{and}\quad\delta\leq t\leq 10\delta.
\end{equation}

\begin{lem}\label{lem:enter U from V}
For all $\epsilon>0$ small enough, there exists $c_2(\epsilon)>0$ dependent upon $\epsilon$ such that, if $(q,p)\in V_-\cup V_+$, then
\begin{equation}\label{eq:enter U from V}
\Pm_{(q,p)}((q_t,p_t)\in \kappa\setminus V\quad\text{for some}\quad 2\delta\leq t\leq 3\delta\lvert \tau_{\partial}>5\delta)>c_2.
\end{equation}
\end{lem}

Before proving Lemma \ref{lem:enter U from V}, we show how it provides for Proposition \ref{prop:1D Langevin satisfies (A1)}. For $x\in V_-$ we define \[
\tau_{\kappa\setminus V}:=\inf\{t\geq 2\delta:X_t\in \kappa\setminus V\}.
\]
For $5\delta \leq t\leq 7\delta$ we have by \eqref{eq:minorisation of dist from kappa not V} and Lemma \ref{lem:enter U from V} that
\[
\begin{split}
\Pm_{x}(X_t\in \cdot)\geq \expE_x[\Ind(2\delta\leq \tau_{\kappa\setminus V}\leq 3\delta)\Pm_{X_{\tau_{\kappa\setminus V}}}(X_{t-\tau_{\kappa\setminus V}}\in \cdot)]\geq c_1\nu\Pm_x(2\delta \leq \tau_{\kappa\setminus V}\leq 3\delta)\\
\geq c_1c_2\Pm_x(\tau_{\partial}\geq 5\delta)\nu\geq c_1c_2\Pm_x(\tau_{\partial}> t)\nu.
\end{split}
\]
This is also true for $x\in V_+$ by the same argument. Therefore there exists $c_3>0$ such that
\[
\Pm_x(X_t\in\cdot)\geq c_3\Pm_x(\tau_{\partial}>t)\nu(\cdot)\quad\text{for all}\quad x\in \kappa\quad\text{and}\quad 5\delta\leq t\leq 7\delta.
\]
Then, using \eqref{eq:survive for time delta then hit kappa in time delta langevin pf}, we have that
\[
\Pm_x(X_{7\delta}\in \cdot)= \expE_x[\Pm_{X_{\tau_{\kappa}}}(X_{7\delta-\tau_{\kappa}}\in \cdot)]\geq c_3\expE_x[\Pm_{X_{\tau_{\kappa}}}(\tau_{\partial}>7\delta-\tau_{\kappa})]\nu=c_3\Pm_x(\tau_{\partial}>7\delta)\nu.
\]
This gives Proposition \ref{prop:1D Langevin satisfies (A1)}. We have left only to prove Lemma \ref{lem:enter U from V}.

\begin{proof}[Proof of Lemma \ref{lem:enter U from V}]
We assume that $x=(q,p)\in V_-$ (so that $q_0,p_0<\epsilon$), the argument for $x\in V_+$ being identical. We assume that $\epsilon>0$ is sufficiently small such that $3\epsilon<K$. We define
\[
\tau_1:=\inf\{t>0:p_t=2\epsilon\}.
\]
We observe that if $\epsilon+15\delta\epsilon<1$, then $p_0,q_0< \epsilon$, $\tau_1<\tau_{\partial}$ and $\epsilon\leq p_t\leq 3\epsilon$ for all $\tau_1\leq t\leq 5\delta$ guarantees that $\tau_{\partial}>5\delta$. Moreover, once we have $p_t=2\epsilon$, the probability that $p_t$ remains in $(\epsilon,3\epsilon)$ for time $5\delta-\tau_1$ is bounded away from $0$ (with the lower bound dependent upon $\epsilon$). Therefore by taking $\epsilon>0$ sufficiently small, we have $c'(\epsilon)>0$ (dependent upon $\epsilon>0$) such that
\[
\Pm_x(\tau_{\partial}>5\delta,p_t\leq 3\epsilon\quad\text{for all}\quad 0\leq t\leq 5\delta)\geq c'(\epsilon)\Pm_x(\tau_{\partial}>5\delta,\tau_1\leq 5\delta)\quad\text{for all}\quad x\in V_-.
\]
On the other hand, if $\tau_1>5\delta$ and $\tau_{\partial}>5\delta$, then $p_t\leq 2\epsilon\leq 3\epsilon$ for all $0\leq t\leq 5\delta$. Therefore we have
\[
\Pm_x(\tau_{\partial}>5\delta,p_t\leq 3\epsilon\quad\text{for all}\quad 0\leq t\leq 5\delta)\geq \Pm(\tau_{\partial}>5\delta,\tau_1>5\delta)\quad\text{for all}\quad x\in V_-.
\]
Therefore, defining $c''(\epsilon):=\frac{c'(\epsilon)}{1+c'(\epsilon)}>0$, we have that
\[
\Pm_x(\tau_{\partial}>5\delta,p_t\leq 3\epsilon\quad\text{for all}\quad 0\leq t\leq 5\delta)\geq c''(\epsilon)\Pm(\tau_{\partial}>5\delta)\quad\text{for all}\quad x\in V_-.
\]

Thus for all $\epsilon>0$ small enough there exists $c''(\epsilon)>0$ such that
\[
\Pm_x(p_t\leq 3\epsilon\quad\text{for all}\quad 0\leq t\leq 5\delta\lvert \tau_{\partial}>5\delta)\geq c''(\epsilon)>0\quad\text{for all}\quad x\in V_-.
\]
Note that on this event,
\[
0<q_t\leq m(\epsilon):=\epsilon+15\delta \epsilon\quad\text{for all}\quad 0\leq t\leq 5\delta.
\]

Since we can't have $p_t<-K$ for all $t$ in an interval of length $\delta$ without hitting the lower boundary (since $K\delta>1000$), we have that
\begin{equation}\label{eq:prob of pt at most epsilon, qt at most m, pt at least -K sometime cond on survival}
\begin{split}
\Pm_x(p_t\leq 3\epsilon\quad\text{and}\quad q_t\leq m(\epsilon)\quad\text{for all}\quad 0\leq t\leq 5\delta,\quad\text{and}\quad p_t\geq -K\quad\text{both for some}\\
 0\leq t\leq \delta\quad\text{and for some}\quad 2\delta\leq t\leq 3\delta\lvert \tau_{\partial}>5\delta)
\geq c''(\epsilon)>0\quad\text{for all}\quad x\in V_-.
\end{split}
\end{equation}

We now define $A:=2(\lvert\lvert F\rvert\rvert_{\infty}+\lvert \gamma\rvert)$ and take $b=\frac{2A}{\sigma}$. We recall that $\tau_{\kappa}:=\inf\{t>0:\lvert p_t\rvert\leq K\}$. We define $\tilde{B}_t=B_t-b((t-\tau_{\kappa})\vee 0)$ and consider a strong solution $\tilde{X}_t=(\tilde{q}_t,\tilde{p}_t)$ of
\[
\begin{cases}
d\tilde{q}_t=\tilde{p}_tdt\\
d\tilde{p}_t=F(\tilde{q}_t)dt-\gamma \tilde{p}_tdt+\sigma d\tilde{B}_t
\end{cases}
,\quad 0\leq t<\tilde{\tau}_{\partial}:=\inf\{s>0:\tilde{q}_{s-}\in \partial \calO\},\quad \tilde{X}_0=X_0.
\]
By Girsanov's theorem, we obtain a probability measure $\tilde{\Pm}$ under which $\tilde{B}_t$ is a Brownian motion, so that $(\tilde{X}_t)_{0\leq t<\tilde{\tau}_{\partial}}$ under $\tilde{\Pm}$ is equal in law to $(X_t)_{0\leq t<\tau_{\partial}}$ under $\Pm$. Over the time interval $5\delta$, this probability measure is given by

\begin{equation}\label{eq:Girsanov trnasform Langevin lemma}
\frac{d\tilde{\Pm}_{\lvert_{\mathcal{F}_{5\delta}}}}{d\Pm_{\lvert_{\mathcal{F}_{5\delta}}}}=\exp(b(B_{5\delta}-B_{\tau_{\kappa}\wedge 5\delta})-\frac{b^2}{2}(5\delta-\tau_{\kappa}\wedge 5\delta))\leq \exp(\frac{b}{\sigma}(p_{5\delta}-p_{\tau_{\kappa}\wedge 5\delta}+5\delta A)),
\end{equation}
where $\mathcal{F}_{\lvert_{\mathcal{F}_{5\delta}}}$ is the filtration at time $5\delta$.

We observe that
\[
(q_t,p_t)=(\tilde{q}_t,\tilde{p}_t)\quad\text{for all}\quad 0\leq t\leq \tau_{\kappa}\quad \text{and}\quad Adt \leq dp_t-d\tilde{p}_t\leq 3A dt\quad \text{for all}\quad \tau_{\kappa}\leq t\leq 5\delta
\]
so that
\begin{equation}\label{eq:bounds for pt-ptilde t and qt-qtilde t Langevin lemma}
A((t-\tau_{\kappa})\vee 0)\leq p_t-\tilde{p}_t\leq 3At\quad\text{and}\quad \frac{A((t-\tau_{\kappa})\vee 0)^2}{2}\leq q_t-\tilde{q}_t\leq \frac{3At^2}{2}\quad \text{for all}\quad 0\leq t\leq 5\delta.
\end{equation}
We also have that
\begin{equation}\label{tilde tau kappa = tau kappa}
\tilde{\tau}_{\kappa}:=\inf\{t>0:\tilde{X}_t\in \kappa\}=\tau_{\kappa},
\end{equation}
since $p_t=\tilde{p}_t$ for $t\leq \tau_{\kappa}$.

We recall that $0<\delta<1$ was chosen so that $1000\delta (\lvert\lvert F\rvert\rvert_{\infty}+\gamma)<1$, so that in particular $500 A\delta<1$. From this, we conclude that for all $\epsilon>0$ sufficiently small,
\[
\tilde{\tau}_{\partial}>5\delta\quad\text{and}\quad \tilde{q}_t\leq m(\epsilon)\quad\text{for all}\quad 0\leq t\leq 5\delta
\]
implies that
\begin{equation}\label{eq:lower and upper bound for qt after tilting}
0\leq \frac{A((t-\tau_{\kappa})\vee 0)^2}{2}< q_t\leq m(\epsilon)+\frac{3At^2}{2}<\frac{1}{2}\quad\text{for all}\quad 0\leq t\leq 5\delta,
\end{equation}
so that $\tau_{\partial}>5\delta$ in particular. Therefore for all $\epsilon>0$ small enough we have by \eqref{eq:survive for time delta then hit kappa in time delta langevin pf}, \eqref{eq:prob of pt at most epsilon, qt at most m, pt at least -K sometime cond on survival}, \eqref{eq:bounds for pt-ptilde t and qt-qtilde t Langevin lemma}, \eqref{tilde tau kappa = tau kappa} and \eqref{eq:lower and upper bound for qt after tilting} that for all $x\in V_-$,
\[
\begin{split}
c''(\epsilon)\Pm_x(\tau_{\partial}>5\delta)\\
\leq \tilde{\Pm}_x(\tilde{\tau}_{\partial}>5\delta,\quad 0<\tilde{q}_t\leq m(\epsilon)\quad\text{and}\quad \tilde{p}_t\leq 3\epsilon\quad\text{for all}\quad 0\leq t\leq 5\delta,\\
\text{and}\quad \tilde{p}_t\geq -K\quad\text{both for some}\quad 0\leq t\leq \delta\quad\text{and for some}\quad 2\delta\leq t\leq 3\delta)\\
\leq \tilde{\Pm}_x(\tau_{\partial}>5\delta,\quad \tau_{\kappa}\leq \delta,\quad \frac{A(t-\delta)^2}{2}< q_t\leq m(\epsilon)+\frac{3At^2}{2}<\frac{1}{2}\quad\text{for all}\quad \delta\leq t\leq 5\delta,\\ p_t\leq 3At+3\epsilon\quad \text{for all}\quad 0\leq t\leq 5\delta\quad\text{and}\quad p_t\geq A(t-\delta)-K\quad\text{for some}\quad 2\delta \leq t\leq 3\delta).
\end{split}
\]
Since $500 A\delta<K$, for all $\epsilon>0$ small enough we have
\[
c''(\epsilon)\Pm_x(\tau_{\partial}>5\delta)\leq \tilde{\Pm}_x(\tau_{\partial}>5\delta,\quad \tau_{\kappa}\leq \delta,\quad X_t\in \kappa\setminus V\quad\text{for some}\quad 2\delta \leq t\leq 3\delta,\quad p_{5\delta}\leq K).
\]
On the event that $p_{5\delta}\leq K$ and $\tau_{\kappa}\leq \delta$, \eqref{eq:Girsanov trnasform Langevin lemma} implies that $\frac{d\tilde{\Pm}_{\lvert_{\mathcal{F}_{5\delta}}}}{d\Pm_{\lvert_{\mathcal{F}_{5\delta}}}}$ is uniformly bounded from above, say by $C<\infty$. Therefore
\[
\begin{split}
c''(\epsilon)\Pm_x(\tau_{\partial}>5\delta) \leq \expE_x\Big[\frac{d\tilde{\Pm}}{d\Pm}\Ind(\tau_{\partial}>5\delta,\quad\tau_{\kappa}\leq \delta,\quad X_t\in \kappa\setminus V\quad\text{for some}\quad 2\delta 
\leq t\leq 3\delta,\quad p_{5\delta}\leq K)\Big]\\\leq C\Pm_x(\tau_{\partial}>5\delta,\quad X_t\in \kappa\setminus V\quad\text{for some}\quad 2\delta \leq t\leq 3\delta).
\end{split}
\]

This completes the proof of Lemma \ref{lem:enter U from V} and hence of Proposition \ref{prop:1D Langevin satisfies (A1)}.
\end{proof}
Having established Proposition \ref{prop:1D Langevin satisfies (A1)}, we have completed the proof of Theorem \ref{theo:1D Langevin satisfies Reverse Dobrushin}.
\qed

\section{Random diffeomorphisms}

We let $\chi$ be a (non-empty) open subset of $d$-dimensional Euclidean space $\Rm^d$ (or the torus or cylinder), for any given $1\leq d<\infty$. Throughout this section, the distinguished measure $\Lambda$ should be understood to be Lebesgue measure on $\chi$. We consider in discrete or continuous time a killed Markov process $(X_t)_{0\leq t<\tau_{\partial}}$ on $\chi$.

A diffeomorphism is defined to be a continuously differentiable bijection with continuously differentiable inverse between open subsets of $\chi$. We shall also consider the unique function from the empty set to the empty set to be a diffeomorphism, so that a random diffeomorphism from a random domain to a random codomain may have an empty domain and codomain.

We assume that $(X_t)_{0\leq t<\tau_{\partial}}$ has a (not necessarily unique) QSD $\pi$. We then consider the following assumption on $(X_t)_{0\leq t<\tau_{\partial}}$ and $\pi$. 

\begin{assum}\label{assum:random diffeo assum}
We firstly assume that $\pi\in \calP_{\infty}(\Leb)$. We assume that there exists $t_0>0$, a probability space $(\Theta,\vartheta)$ and a measurable function
\begin{equation}\label{eq:function F in Assumption for random diffeo}
F:\chi\times \Theta\ra \chi\sqcup \partial
\end{equation}
such that for every $x\in \chi$, $\Pm_{x}(X_{t_0}\in \cdot)$ is given by the law of the random variable
\[
(\Theta,\vartheta)\ni \theta \mapsto F(x,\theta)\in \chi\sqcup \partial,
\]
that is $\Pm_x(X_{t_0}\in \cdot)=\Pm(F(x,\theta)\in\cdot)$ for every $x\in\chi$. We define
\begin{equation}\label{eq:ftheta Utheta V theta defin}
f_{\theta}:\chi\ni x\mapsto F(x,\theta)\in \chi\sqcup \partial,\quad U_{\theta}=f_{\theta}^{-1}(\chi)\quad\text{and}\quad V_{\theta}=\text{Im}(f_{\theta})\cap \chi\quad\text{for all}\quad \theta\in \Theta.
\end{equation}
We assume that 
\[
{f_{\theta}}_{\lvert_{U_{\theta}}}:U_{\theta}\ra V_{\theta}
\]
is $\vartheta$-almost surely a diffeomorphism. We assume in addition that there exists $1\leq M<\infty$ such that
\begin{equation}\label{divergence of random diffeo bounded and bded away from 0}
M^{-1}\leq \lvert\det(D f_{\theta})(x)\rvert\leq M\quad\text{for all}\quad x\in U_{\theta},\quad\vartheta-\text{almost surely.}
\end{equation}
Thus 
\begin{equation}\label{eq:gtheta defin}
g_{\theta}:\chi\ni y\mapsto \begin{cases}
f_{\theta}^{-1}(y),\quad y\in V_{\theta}\\
\partial,\quad y\notin V_{\theta}
\end{cases},
\end{equation}
is such that ${g_{\theta}}_{\lvert_{V_{\theta}}}$ is a well-defined diffeomorphism $V_{\theta}\ra U_{\theta}$, $\vartheta$-almost surely. We may therefore define the discrete-time absorbed Markov process $(\tilde{X}^0_n)_{0\leq n<\tilde{\tau}^0_{\partial}}$ such that
\begin{equation}\label{eq:discrete time Y process for random diffeo theorem}
\Pm_y(\tilde{X}^0_1\in\cdot)=\Pm(g_{\theta}(y)\in \cdot)\quad\text{for all}\quad y\in \chi.
\end{equation}
We assume that $\Pm_y(\tilde{\tau}_{\partial}^0>1)>0$ for every $y\in\chi$, and that $(\tilde{X}^0_n)_{0\leq t<\tilde{\tau}^0_{\partial}}$ satisfies \cite[Assumption (A1)]{Champagnat2014} (note that the former assumption ensures the latter makes sense), for some probability measure $\nu$ not mutually singular with respect to $\pi$. 
\end{assum}

\begin{rmk}
Note that in the above assumption, we include the possibility that $U_{\theta}$ and $V_{\theta}$ are empty for some $\vartheta$-positive collection of $\theta\in\Theta$, since the unique function from the empty set to the empty set is considered to be a diffeomorphism.
\end{rmk}

We then have the following theorem.

\begin{theo}\label{theo:random diffeo}
If $(X_t)_{0\leq t<\tau_{\partial}}$ has a (not necessarily unique) QSD $\pi$, with which it satisfies Assumption \ref{assum:random diffeo assum}, then $(X_t)_{0\leq t<\tau_{\partial}}$ satisfies Assumption \ref{assum:adjoint Dobrushin main results section}.
\end{theo}

We shall firstly prove this theorem, before considering an application of it to $2$-dimensional neutron transport dynamics.

\subsection*{Proof of Theorem \ref{theo:random diffeo}}

We write $(P_t)_{t\geq 0}$ for the submarkovian transition semigroup corresponding to $(X_t)_{0\leq t<\tau_{\partial}}$. We further define the submarkovian kernels $\tilde{P}$ and $\tilde{P}^0$ to be given by (here $1\leq M<\infty$ is the constant in \eqref{divergence of random diffeo bounded and bded away from 0})
\[
\begin{split}
\tilde{P}^0(y,\cdot):=M^{-1}\Pm_y(g_{\theta}(y)\in \cdot\cap \chi),\quad
\tilde{P}(y,dx):=M^{-1}\expE_y[\Ind(g_{\theta}(y)\in \cdot)\det(Dg_{\theta}(y))].
\end{split}
\]
We observe that $M\tilde{P}^0$ is the submarkovian kernel for the time-steps of $(\tilde{X}^0_n)_{0\leq n<\tilde{\tau}^0_{\partial}}$. We now prove that
\begin{equation}\label{eq:adjoint equation for random diffeo}
\Leb(dx)P_{nt_0}(x,dy)=M^n\Leb(dy)\tilde{P}^n(y,dx)\quad\text{for all}\quad n\geq 1.
\end{equation}

\begin{proof}[Proof of \eqref{eq:adjoint equation for random diffeo}]
We fix $A,B\in\mathscr{B}(\chi)$ such that $\Leb(A),\Leb(B)<\infty$. Then using Tonelli's theorem and the change of variables formula we have that
\[
\begin{split}
\int_AP_{t_0}(x,B)\Leb(dx)=\int_{\Theta}\int_{U_{\theta}}\Ind(x\in A)\Ind(f_{\theta}(x)\in B)\Leb(dx)\vartheta(\theta)\\
\underbrace{=}_{\substack{\text{substitute}\\
y=f_{\theta}(x)}}\int_{\theta}\int_{V_{\theta}}\Ind(g_{\theta}(y)\in A)\Ind(y\in B)\det(Dg_{\theta}(y))\Leb(dy)\vartheta(d\theta)=\int_BM\tilde{P}(y,A)\Leb(dy).
\end{split}
\]

Thus we have \eqref{eq:adjoint equation for random diffeo} for $n=1$. We then obtain\eqref{eq:adjoint equation for random diffeo} for all $n\geq 1$ by \eqref{eq:adjoint semigroup eqn from cty of pi criterion by induction}.
\end{proof}

It follows from \eqref{divergence of random diffeo bounded and bded away from 0} that
\begin{equation}\label{eq:proof of random diffeo theorem tilde P tilde P0 comparison}
M^{-n}(\tilde{P}^0)^n(y,\cdot)\leq \tilde{P}^n(y,\cdot)\leq M^n(\tilde{P}^0)^n(y,\cdot)\quad\text{for all}\quad y\in\chi,\;n\in \Nm.
\end{equation}

Since we have that $\Pm_y(\tilde{\tau}_{\partial}^0>1)>0$ for every $y\in\chi$, using \eqref{eq:proof of random diffeo theorem tilde P tilde P0 comparison} we have that $\tilde{P}1(y)>0$ for every $y\in\chi$.

We now let $c_0>0$, $n_1>0$ and $\nu\in\calP(\chi)$ be the constant, discrete time and probability measure respectively for which $(\tilde{X}_t)_{0\leq t<\tilde{\tau}_{\partial}^0}$ satisfies \cite[Assumption (A1)]{Champagnat2014}. Using \eqref{eq:proof of random diffeo theorem tilde P tilde P0 comparison} we have that
\begin{equation}\label{eq:tilde P in random diffeo pf satisfies A1 calculation}
\frac{\tilde{P}^{n_1}(y,\cdot)}{P^{n_1}1(y)}\geq M^{-2n_1}\frac{(P^0)^{n_1}(y,\cdot)}{(P^0)^{n_1}1(y)}= M^{-2n_1}\Law(\tilde{X}_{n_1}\lvert \tilde{\tau}_{\partial}>n_1)(\cdot)\geq c_0M^{-2n_1}\nu(\cdot)\quad\text{for every}\quad y\in\chi.
\end{equation}

Finally, we have from \eqref{eq:adjoint equation for random diffeo} with $n=n_1$ and \eqref{eq:tilde P in random diffeo pf satisfies A1 calculation} that \eqref{eq:psi adjoint for verifying reverse Dobrushin main results section} and \eqref{eq:Dobrushin for P tilde kernel crit for reverse dobrushin main results section} are satisfied, the latter being with the constant $c_0M^{-2n_1}>0$ and probability measure $\nu$. We already have by assumption that $\pi$ is not mutually singular with respect to $\nu$.

We have therefore verified that $(X_t)_{0\leq t<\tau_{\partial}}$ satisfies Assumption \ref{assum:adjoint Dobrushin main results section}.
\qed

\subsection{$2$-dimensional neutron transport dynamics}\label{subsection:neutron transport}

The neutron transport equation models the propagation of neutrons in a fissile medium. It corresponds to the expectation semigroup of a neutron transport process, which mimics the dynamics of a typical neutron. This neutron transport process is an absorbed Markov process, with absorption corresponding to the absorption of neutrons at the physical spatial boundary. We consider the simple $2$-dimensional system considered in \cite[Section 4.2]{Champagnat2014}. A more general, $3$-dimensional system has been extensively studied by Horton and Kyprianou et al. in the sequence of papers \cite{Horton2020,Harris2020,Cox2021}.

This neutron transport process we consider is defined as follows. We take $U$ to be a non-empty open, connected, bounded subdomain of $\Rm^2$ with $C^2$ boundary $\partial U$, corresponding to the physical space. The state space $\chi$ is given by $\chi:=U\times (\Rm/(2\pi\Zm))$. We then consider the $\chi$-valued absorbed Markov process $(X_t)_{0\leq t<\tau_{\partial}}$ whose dynamics are as follows. Prior to absorption at time $\tau_{\partial}$, $X_t$ consists of a spatial position $q_t$ and a direction $\varphi_t$. The particles moves at constant speed $1$ in the direction given by $\varphi_t$, 
\[
\dot{q}_t=v(\varphi_t)\quad\text{whereby}\quad v(\varphi):=\begin{pmatrix}
\cos(\varphi)\\
\sin(\varphi)
\end{pmatrix}.
\]
At constant Poisson rate $\mu>0$, $\varphi_t$ jumps to a new angle chosen uniformly from $[0,2\pi)$, corresponding to the neutron scattering upon collision with an atomic nucleus. The direction is constant in between jump times. The particle is absorbed upon contact of its spatial position with the boundary of $U$, $\tau_{\partial}:=\inf\{t>0:q_{t-}\in \partial U\}$.

It is obvious that $(X_t)_{0\leq t<\tau_{\partial}}$ satisfies Assumption \ref{assum:technical assumption for Assum (A)}. We have from \cite[Theorem 4.3]{Champagnat2014} that \cite[Assumption (A)]{Champagnat2014} is satisfied by $(X_t)_{0\leq t<\tau_{\partial}}$. It therefore has a unique QSD, which we call $\pi$.

We prove the following.
\begin{theo}\label{theo:simple neutron transport reverse dobrushin satisfied}
The simple $2$-dimensional neutron transport process $(X_t)_{0\leq t<\tau_{\partial}}$ satisfies Assumption \ref{assum:Dobrushin reverse time}, with its QSD $\pi$ belonging to $\calP_{\infty}(\Leb)$.
\end{theo}

It therefore follows from Theorem \ref{theo:Linfty convergence} that there exists $C<\infty$ and $\gamma>0$ such that
\begin{equation}
\Big\lvert\Big\lvert \frac{d\Law_{\mu}(X_t\lvert \tau_{\partial}>t)}{d\pi}-1\Big\rvert\Big\rvert_{L^{\infty}(\pi)}\leq \frac{C}{\mu(h)}e^{-\gamma t}\Big\lvert\Big\lvert \frac{d\mu}{d\pi}\Big\rvert\Big\rvert_{L^{\infty}(\pi)}\quad\ \text{for all}\quad t\geq T\quad\text{and all}\quad  \mu\in\calP_{\infty}(\pi),
\end{equation}
where $h\in \calB_b(\chi;\Rm_{>0})$ is the bounded and strictly positive pointwise right eigenfunction provided for by \cite[Proposition 2.3]{Champagnat2014}.

\subsubsection*{Proof of Theorem \ref{theo:simple neutron transport reverse dobrushin satisfied}}

We shall proceed by applying Theorem \ref{theo:random diffeo}.

Since we have from \cite[Theorem 4.3]{Champagnat2014} that $(X_t)_{0\leq t<\tau_{\partial}}$ satisfies \cite[Assumption (A)]{Champagnat2014}, we may take a time $t_0>0$, constant $c_1>0$ and probability measure $\nu$ for which \cite[Assumption (A)]{Champagnat2014} is satisfied by $(X_t)_{0\leq t<\tau_{\partial}}$. We henceforth fix this to be the definition of $t_0$, $c_1$ and $\nu$. We write $(\hat{X}_n)_{0\leq n<\hat{\tau}_{\partial}}$ for the discrete time process obtained by only considering $X_t$ over time-steps of $t_0$, so that $\hat{X}_n=X_{nt_0}$ for $n\in \Nm$. It follows that $(\hat{X}_n)_{0\leq n<\hat{\tau}_{\partial}}$ also satisfies Assumption \ref{assum:technical assumption for Assum (A)} and \cite[Assumption (A)]{Champagnat2014}, with the discrete time $1$, constant $c_1>0$ and probability measure $\nu$. 

We now define $F$ as in \eqref{eq:function F in Assumption for random diffeo}. We define $\Theta=D([0,t_0];\Rm/(2\pi\Zm))$, which we equip with its Borel $\sigma$-algebra $\mathscr{B}(\Theta)$. We then define on $(\Theta,\mathscr{B}(\Theta))$ the probability measure $\vartheta$ corresponding to the law of a jump process on $\Rm/(2\pi\Zm)$, which jumps at Poisson rate $\mu$ to a new position chosen uniformly from $[0,2\pi)$. The measurable function $F$ in \eqref{eq:function F in Assumption for random diffeo} is then defined to be
\[
F:\chi\times \Theta\ni ((q^0,\varphi^0),(\theta_t)_{0\leq t\leq t_0})\mapsto \begin{cases}
\partial,\quad q^0+\int_0^{t}v(\varphi^0+\theta_s)ds\in \partial U\quad\text{for some}\quad 0\leq t\leq t_0,\\
(q^0+\int_0^{t_0}v(\varphi^0+\theta_s)ds,\varphi^0+\theta_{t_0}),\quad\text{otherwise.}
\end{cases}
\]

This defines as in \eqref{eq:ftheta Utheta V theta defin} the map $f^{\theta}:\chi\mapsto \chi\sqcup \partial$, along with the (possibly empty) random open subsets of $\chi$, $U^{\theta}$ and $V^{\theta}$. We observe that $f^{\theta}_{\lvert_{U_{\theta}}}:U_{\theta}\ra V_{\theta}$ is $\vartheta$-almost surely a diffeomorphism, such that
\[
D f_{\lvert_{U_{\theta}}}\equiv \begin{pmatrix}
\text{Id}_{2\times 2} && \text{something}\\
\underline{0}_{1\times 2} && 1
\end{pmatrix},
\]
whence we conclude that $\det(Df_{\lvert_{U_{\theta}}})\equiv 1$. Therefore \eqref{divergence of random diffeo bounded and bded away from 0} is satisfied.

We now consider $g^{\theta}$, the map given by \eqref{eq:gtheta defin}. We define $r:\chi\sqcup \partial \ra \chi\sqcup \partial$ by $r((q,\varphi)):=(q,\varphi+\pi)$ for $(q,\varphi)\in \chi$ and $r(\partial):=\partial$. We further define $\iota(\theta)=(\iota(\theta)_s)_{0\leq s\leq t_0}\in \Theta$ to be the right-continuous version of $(\theta_{t_0-s}-\theta_{t_0})_{0\leq s\leq t_0}$. We observe that $\iota_{\#}\vartheta=\vartheta$. It is straightforward to check that $g^{\theta}$ satisfies
\[
g^{\theta}=r\circ f^{\iota(\theta)}\circ r\overset{d}{=} r\circ f^{\theta}\circ r.
\]

Therefore, if we take independent $\theta_1,\ldots,\theta_n\sim \vartheta$ (for any $n\in \Nm$), since $r\circ r$ is the identity we have that
\begin{equation}\label{eq:n step g theta in terms of n step f theta 2d neutron transport}
g^{\theta_n}\circ\ldots\circ g^{\theta_1}\overset{d}{=}(r\circ f^{\theta_n}\circ r)\circ (r \circ f^{\theta_{n-1}}\circ r)\circ \ldots \circ (r\circ f^{\theta_1}\circ r) =r\circ (f^{\theta_n}\circ \ldots \circ f^{\theta_1})\circ r.
\end{equation}

We now take the discrete time absorbed Markov process $(\tilde{X}^0_n)_{0\leq n<\tilde{\tau}^0_{\partial}}$ defined in \eqref{eq:discrete time Y process for random diffeo theorem}. We have from \eqref{eq:n step g theta in terms of n step f theta 2d neutron transport} that
\begin{equation}\label{eq:Y process in terms of X process 2d neutron transport}
\Pm_y(\tilde{X}^0_1\in \cdot)=\Pm_{r(y)}(X_{t_0}\in  r(\cdot))\quad\text{for all}\quad y\in \chi\sqcup\partial,\quad n\in \Nm.
\end{equation}

It follows from \eqref{eq:Y process in terms of X process 2d neutron transport} that $(\tilde{X}^0_n)_{0\leq n<\tilde{\tau}^0_{\partial}}$ also satisfies Assumption \ref{assum:technical assumption for Assum (A)} and \cite[Assumption (A)]{Champagnat2014}, with the discrete time $1$, constant $c_1>0$ and probability measure $r_{\#}\nu$. We now establish that $\pi\in\calP_{\infty}(\Leb)$, and that $\pi$ and $r_{\#}\nu$ are not mutually singular. Unfortunately, we cannot apply \cite[Theorem 3.1]{Horton2020}, as there the new velocity is chosen according to a distribution which has a density with respect to Lebesgue measure on the annulus, whereas here the speed is deterministically $1$ with only the direction being random.

Since $\det(Df_{\theta})\equiv 1$ on $U_{\theta}$, we can apply \eqref{eq:adjoint equation for random diffeo} to see that
\begin{equation}\label{eq:adjoint eqn for neutron transport pf}
\Leb(dx)P_{t_0}(x,dy)=\Leb(dy)\tilde{P}^0(y,dx).
\end{equation}

Since $(\tilde{X}^0_n)_{0\leq n<\tilde{\tau}^0_{\partial}}$ satisfies Assumption \ref{assum:technical assumption for Assum (A)} and \cite[Assumption (A)]{Champagnat2014}, $\tilde{P}^0$ must have a pointwise right eigenfunction belonging to $\calB_b(\chi;\Rm_{>0})$, which we call $\tilde{h}$. Rescaling if necessary, we have that $\Leb(\tilde{h})=1$. It then follows from \eqref{eq:adjoint eqn for neutron transport pf} that $\tilde{h}\Leb$ is a QSD for $(\hat{X}_n)_{0\leq n<\hat{\tau}_{\partial}}$. Since $(\hat{X}_n)_{0\leq n<\hat{\tau}_{\partial}}$ satisfies Assumption \ref{assum:technical assumption for Assum (A)} and \cite[Assumption (A)]{Champagnat2014}, this QSD must be unique. Since $\pi$ is a QSD for $(\hat{X}_n)_{0\leq n<\hat{\tau}_{\partial}}$, it follows that $\pi=\tilde{h}\Leb$. Therefore $\pi\in \calP_{\infty}(\Leb)$, with (a version of) $\frac{d\pi}{d\Leb}$ being everywhere positive. 

For any Lebesgue-null Borel set $A\in\mathscr{B}(\chi)$, we may integrate \eqref{eq:adjoint eqn for neutron transport pf} over $x\in A$ to see that $\tilde{P}^0\Ind_A(y)=0$ for Lebesgue-almost every $y\in\chi$, implying that $(r_{\#}\nu)(A)=0$. We may therefore conclude that $r_{\#}\nu\ll \Leb$. Since (a version of) $\frac{d\pi}{d\Leb}$ is everywhere positive, it follows that $r_{\#}\nu$ and $\pi$ cannot be mutually singular.

It follows from Theorem \ref{theo:simple neutron transport reverse dobrushin satisfied} that $(\hat{X}_n)_{0\leq n<\hat{\tau}_{\partial}}$ satisfies Assumption \ref{assum:adjoint Dobrushin main results section}, implying that $(X_t)_{0\leq t<\tau_{\partial}}$ does also. 
\qed

\section{Piecewide-deterministic Markov processes}\label{section:PDMPs}

Piecewise deterministic Markov processes (PDMPs) were introduced by Davis in \cite{Davis1984}, and have since been widely studied. They are characterised by the property that they undergo deterministic motion in between random times, as opposed to diffusions whose motion is always random. For the PDMPs we consider, the deterministic motion shall correspond to the flows generated by a family of vector fields. The long-term behaviour of such PDMPs without absorption is now well understood (see \cite{Benaim2015}), it being possible to apply the classical Dobrushin condition under H\"ormander-type conditions on the vector fields. The long-term behaviour of absorbed PDMPs, however, is not. We shall consider absorbed PDMPs defined as follows.

We take the space $U$, assumed to be a (non-empty) open, bounded, connected subdomain of $d$-dimensional Euclidean space $\Rm^d$, with $C^{\infty}$ boundary $\partial U$. The set $E$ is a finite set, $E=\{1,\ldots,n\}$. We define $(Q_{ij})_{i,j\in E}$ to be a fixed rate matrix, with $Q_{ij}$ being the $i\mapsto j$ jump rate for $i\neq j$. The state space is given by $\chi:=U\times E$ (with a separate one-point cemetery state, $\partial$). Corresponding to each element of $E$, $i\in E$, there is a $C^{\infty}(\Rm^d)$ vector fields $v^i$. We note that the definition of $v^1,\ldots,v^n$ on $(\bar U)^c$ is arbitrary. The following standing assumption shall be imposed throughout this section.
\renewcommand{\theAssumletter}{PDMPS}
\begin{Assumletter}[Standing Assumption for results on PDMPs]\label{assum:standing assum PDMP}
Since $\partial U$ is $C^{\infty}$, we may define the unit inward-normal $\hat{n}(x)$ for all $x\in U$. We assume that for all $x\in \partial U$ there exists $i,j\in E$ such that
\begin{equation}\label{eq:PDMP vi and vj pointing outwards and inwards everywhere on the boundary}
\langle \hat{n}(x),v^i(x)\rangle >0,\quad \langle \hat{n}(x),v^j(x)\rangle <0.
\end{equation}
We further assume the vector fields $v^1,\ldots,v^n$ are nowhere-zero on $\bar U$. Finally we assume that $Q_{ij}>0$ for all $i\neq j$.
\end{Assumletter}

The absorbed PDMP $((X_t,I_t))_{0\leq t<\tau_{\partial}}$ on the state space $\chi=U\times E$ consists of a spatial position $X_t\in U$ and a state $I_t\in E$, prior to absorption at time $\tau_{\partial}$. The component $I_t$ evolves as a continuous-time jump process on $E$ with rate matrix $Q$. The position component $X_t$ then evolves according to the ODE
\begin{equation}
\dot{X}_t=v^{I_t}(X_t).
\end{equation}
The process is absorbed upon contact of the spatial position with the boundary, 
\begin{equation}
\tau_{\partial}:=\inf\{t>0:X_{t-}\in \partial U\}.
\end{equation}

PDMPs can be considered in much greater generality than the above definition. For instance, one can consider deterministic motion which does not correspond to the flow of a vector field, transition rates which depend upon the spatial position, or random jumps in the spatial position at the jump times. One can also consider killing mechanisms other than killing at the boundary. We will not consider these possibilities.

As a result of the piecewise deterministic dynamics, PDMPs are not, in general, strong Feller. In fact, in dimension greater than $1$, absorbed PDMPs will often not even be Feller, due to the effect of the boundary. On the other hand, much of the QSD literature is reliant on spectral arguments, which seem to be rather difficult to apply in this context. Moreover, on a PDE level, the corresponding Fokker-Planck equation is first order, so PDE theory doesn't give us the sort of controls we might obtain for uniformly elliptic or hypoelliptic diffusions. The successful approach pursued in the setting without absorption employs a probabilistic argument (which can be found in \cite{Benaim2015}), avoiding these analytic difficulties. 

In this section, we will obtain convergence to a QSD for three classes of absorbed PDMP: in dimension $1$, in dimension $2$, and in arbitrary dimension with constant drift vectors (all under suitable assumptions on the drift vector fields). We will do this by verifying the various criteria provided by this paper, using probabilistic arguments. Whereas our results shall not require analytic controls, they shall furnish analytic controls. In Theorem \ref{theo:bdy Harnack for PDMPs} we shall obtain a parabolic boundary Harnack-type inequality for PDMPs either in dimension $1$, or in arbitrary dimension with constant drift vectors. This is notable, in particular, since the corresponding Fokker-Planck equation is a system of first-order PDEs, and in the latter case may have discontinuous solutions for smooth initial conditions. We shall also obtain continuity of the density of the QSD with respect to Lebesgue in dimensions $1$ and $2$ (in dimension $2$, we obtain continuity on an open set we describe).

In \cite{Benaim2021} the authors developed criteria to establish convergence to a QSD for degenerate Feller processes. Whilst the main application in that paper concerned degenerate diffusions, in \cite[Section 4]{Benaim2021} they also considered perhaps the simplest possible absorbed PDMP of the above form. This involves switching between two vector fields on the unit interval $(0,1)$ at constant rate, with the two vector fields having constant velocities $+1$ and $-1$. They established non-uniform exponential convergence in total variation. This relied on being able to write down a (simple) expression providing for the principal right eigenfunction of the transition semigroup, allowing one to apply the criteria they developed in that paper. In \cite{Cloez2022} the authors considered the Crump-Young model from biology. This is a one-dimensional model in which the (randomly evolving) number of bacteria affects the (deterministically evolving) nutrient concentration, and which is considered to go extinct when there are no more bacteria. It differs from the PDMPs we consider here in that the nutrient concentration and bacterial number are unbounded, and the deterministic dynamics affects the jump rate of the random dynamics. They obtained non-uniform exponential convergence in total variation to the unique QSD. Existence of a QSD for the same model had been established earlier in \cite{Collet2013}. The analysis in \cite{Cloez2022} proceeded by obtaining sharp estimates on this model, allowing them to apply earlier results on convergence to a QSD. For absorbed PDMPs involving switching between a finite or countable number of vector fields, these two processes are the only ones for which convergence to a QSD was previously known, to the authors' knowledge. Uniform exponential convergence in total variation to a QSD has also been established for the neutron transport process, firstly in \cite{Champagnat2014} (this is the process we consider in Subsection \ref{subsection:neutron transport}), then later in much more generality in \cite{Horton2020}. Both of these involve the process moving at constant velocity in between the jump times. They differ from the PDMPs we consider in this section, however, in that the velocity is not chosen from a finite or countably infinite set. In the former the speed is fixed with the direction chosen from the uniform distribution, whilst for the latter the new velocity (i.e. both the speed and the direction) is chosen from a bounded density. Finally, in \cite{Villemonais2022}, they considered a one-dimensional absorbed PDMP involving only a single deterministic flow. The randomness instead comes from random jumps of the position at the random times. They established convergence to a QSD in weighted total variation norm. To the authors' knowledge, this constitutes the extent of previously known results on absorbed PDMPs. 

The author is only aware of two earlier results on the regularity of the density of a QSD with respect to Lebesgue measure. For the aforementioned Crump-Young model, it was established in \cite[Theorem 5.1]{Collet2013} that the quasi-stationary density is smooth. For the aforementioned simple one-dimensional absorbed PDMP with drifts $\pm 1$ and constant jump rate, \cite[Lemma 4.1]{Benaim2021} gives a formula for the quasi-stationary density (which is analytic). In the case without absorption, more is known. In dimension one, it was established in \cite{Bakhtin2015} that the stationary density of a PDMP is smooth away from the critical points. In arbitrary dimension, in a collaboration of the present author with Michel Bena{\"i}m, we shall establish in a forthcoming paper that the stationary density is $C^k$ whenever the jump rate is sufficiently fast, for any finite $k$.

\subsection*{Results}

Throughout this section, the distinguished measure $\Lambda$ should be understood to be Lebesgue measure on $U$ times the counting measure on $E$. We will often refer to this simply as Lebesgue measure, as it is Lebesgue measure on $n$ copies of Euclidean space.

We firstly consider the case of $d=1$. We note that in one dimension, there must be at least one everywhere positive vector field and one everywhere negative vector field by \eqref{eq:PDMP vi and vj pointing outwards and inwards everywhere on the boundary} and the fact that the vector fields are non-vanishing.

\begin{theo}[Convergence to a quasi-stationary distribution for one-dimensional absorbed PDMPs]\label{theo:1D PDMPs}
In addition to the standing assumption, \ref{assum:standing assum PDMP}, we suppose that $d=1$. Then Assumption \ref{assum:technical assumption for Assum (A)} and \cite[Assumption (A)]{Champagnat2014} are satisfied. In particular, $((X_t,I_t))_{0\leq t<\tau_{\partial}}$ has a unique QSD $\pi$. This QSD has full support, and has a density with respect to Lebesgue measure (a version of) which is continuous and bounded. Moreover, $((X_t,I_t))_{0\leq t<\tau_{\partial}}$ satisfies assumptions \ref{assum:adjoint Dobrushin main results section}, \ref{assum:combined Dobrushin adjoint Dobrushin main results section} and \ref{assum:adjoint anti-Dobrushin main results section}.
\end{theo}

It therefore follows from Theorem \ref{theo:uniform Linfty convergence main results section} that for one-dimensional absorbed PDMPs $((X_t,I_t))_{0\leq t<\tau_{\partial}}$ satisfying the conditions of Theorem \ref{theo:1D PDMPs}, there exists a time $T<\infty$ and constant $\gamma>0$ such that $\Law_{\mu}(X_t\lvert \tau_{\partial}>t)\ll_{\infty}\pi$ for all $t\geq T$ and $\mu\in\calP(\chi)$, with its density with respect to $\pi$ satisfying
\begin{equation}\label{eq:exp conv to pi for 1D PDMP result}
\Big\lvert\Big\lvert \frac{d\Law_{\mu}(X_t\lvert \tau_{\partial}>t)}{d\pi}-1\Big\rvert\Big\rvert_{L^{\infty}(\pi)}\leq e^{-\gamma(t-T)}\quad\text{for all}\quad  T\leq t<\infty,\quad\mu\in\calP(\chi).
\end{equation}

We now turn our attention to absorbed PDMPs in dimension greater that $1$. We firstly introduce some definitions. Given vectors $v_1,\ldots,v_n\in \Rm^d$ we define the open convex hull and closed convex hull respectively to be
\begin{equation}
\begin{split}
\conv(v_1,\ldots,v_n):=\{a_1v_1+\ldots+a_nv_n:a_1,\ldots,a_n>0:\sum_{k=1}^na_k=1\},\\
\overline{conv}(v_1,\ldots,v_n):=\{a_1v_1+\ldots+a_nv_n:a_1,\ldots,a_n\geq 0:\sum_{k=1}^na_k=1\}.
\end{split}
\end{equation}
We note that the open convex hull isn't necessarily an open set, but its closure is the closed convex hull.

We define for each $i\in E$ the flow map $\varphi^i_t(x)$, corresponding to the flow of solutions to $\dot{x}_s=v^i(x_s)$. The flow of solutions to $\dot{y}_s=-v^i(y_s)$ is then given by $\varphi^i_{-t}(y)$. 

We consider absorbed PDMPs in two-dimensions satisfying the following assumption.
\renewcommand{\theAssumletter}{2DPDMP}
\begin{Assumletter}[Assumption for $2$-dimensional absorbed PDMPs]\label{assum:assum for 2d PDMPs}
We assume that for all $(x,i)\in U\times E$ there exists $t_+,t_-<\infty$ (dependent upon $(x,i)$) such that $\varphi^i_{t_-}(x),\varphi^i_{t_+}(x)\in \inte(U^c)$.  We assume that $v^i(x)$ and $v^j(x)$ are transversal for all $x\in \bar U$ and $i\neq j$. We finally assume that $0\in \overline{\conv}(v_1(x),\ldots,v_n(x))$ for all $x\in U$. 
\end{Assumletter}

Given that Assumption \ref{assum:assum for 2d PDMPs} is satisfied, the following function is necessarily everywhere finite
\begin{equation}\label{eq:time to hit partial U}
T_{\partial}^i(x):=\inf\{t>0:\varphi^i_{-t}(x)\in \partial U\}.
\end{equation}

We may then define the following set.
\begin{equation}
\calC:=\inte(\{(x,i)\in U\times E:x'\mapsto T_{\partial}^i(x')\quad\text{is continuous at $x$}\}).
\end{equation}

\begin{theo}[Convergence to a quasi-stationary distribution for two-dimensional absorbed PDMPs]\label{theo:2D PDMPs}
We assume that $d=2$. In addition to the standing assumption, \ref{assum:standing assum PDMP}, we assume that Assumption \ref{assum:assum for 2d PDMPs} is satisfied. Then $((X_t,I_t))_{0\leq t<\tau_{\partial}}$ satisfies Assumption \ref{assum:technical assumption for Assum (A)} and \cite[Assumption (A)]{Champagnat2014}. In particular, $((X_t,I_t))_{0\leq t<\tau_{\partial}}$ has a unique QSD $\pi$. This QSD must have an (essentially) bounded density with respect to Lebesgue measure. Moreover, $((X_t,I_t))_{0\leq t<\tau_{\partial}}$ satisfies Assumption \ref{assum:adjoint Dobrushin main results section}. If, in addition, $\{x\in \partial U:\hat{n}(x)\cdot v^i(x)=0\}$ has finitely many connected components for all $i\in E$, then $\pi$ has a density with respect to Lebesgue measure (a version of) which is continuous on $\calC$.
\end{theo}

One may observe that the QSD $\pi$ shouldn't be expected to be continuous on $\calC^c$.

It follows from \cite[Theorem 2.1]{Champagnat2014} that for two-dimensional absorbed PDMPs $((X_t,I_t))_{0\leq t<\tau_{\partial}}$ satisfying the conditions of Theorem \ref{theo:2D PDMPs}, there must exist constants $C<\infty$ and $\gamma>0$ such that
\begin{equation}
\lvert\lvert \Law_{\mu}(X_t\lvert\tau_{\partial}>t)-\pi\rvert\rvert_{\TV}\leq Ce^{-\gamma t}.
\end{equation}

It also follows from Theorem \ref{theo:Linfty convergence main results section} that there exists a time $T<\infty$, and (possibly different) constants $C<\infty$ and $\gamma>0$, such that
\begin{equation}
\Big\lvert\Big\lvert \frac{d\Law_{\mu}(X_t\lvert\tau_{\partial}>t)}{d\pi}-1\Big\rvert\Big\rvert_{L^{\infty}(\pi)}\leq \frac{C}{\mu(h)}e^{-\gamma t}\Big\lvert\Big\lvert \frac{d\mu}{d\pi}\Big\rvert\Big\rvert_{L^{\infty}(\pi)}\quad\ \text{for all}\quad t\geq T\quad\text{and all}\quad  \mu\in\calP_{\infty}(\pi),
\end{equation}
where $h$ is the everywhere strictly positive, bounded pointwise right eigenfunction provided for by \cite[Proposition 2.3]{Champagnat2014}.

We finally turn our attention to the case of arbitrary dimension. For this, we must restrict our attention to vector fields with constant drift vectors. We call these absorbed piecewise constant Markov processes (absorbed PCMPs). In particular, we consider the following assumption.
\renewcommand{\theAssumletter}{PCMP}
\begin{Assumletter}[Assumption for $d$-dimensional absorbed PCMPs]\label{assum:assumption for absorbed PCMPs}
We suppose that the drift vectors are constant, so that $v^i(x)\equiv v_i$ ($1\leq i\leq n$), and there are at least $d+1$ of them ($n>d$). We assume that:
\begin{enumerate}
\item\label{enum:any d of constant vectors PDMP l.i.}
for any $1\leq i_1<\ldots<i_d\leq n$, $v_{i_1},\ldots,v_{i_d}$ are linearly independent;
\item\label{enum:sum of vectors returning to 0 constant vector PDMP}
$0\in \overline{\conv}(v_1,\ldots,v_n)$.
\end{enumerate}
\end{Assumletter}

We note that Part \ref{enum:any d of constant vectors PDMP l.i.} of Assumption \ref{assum:assumption for absorbed PCMPs} is generic, in that Lebesgue-almost every choice of $(v_1,\ldots,v_n)\in (\Rm^d)^n$ will satisfy it. On the other hand, given Part \ref{enum:any d of constant vectors PDMP l.i.} of Assumption \ref{assum:assumption for absorbed PCMPs} is satisfied, one can show that Part \ref{enum:sum of vectors returning to 0 constant vector PDMP} is necessary and sufficient for it to be possible for the corresponding PCMP to be able to survive for arbitrarily long times.

\begin{theo}[Convergence to a quasi-stationary distribution for PCMPs]\label{theo:PCMPs}
In addition to the standing assumption, \ref{assum:standing assum PDMP}, we assume that Assumption \ref{assum:assumption for absorbed PCMPs} is satisfied. Then $((X_t,I_t))_{0\leq t<\tau_{\partial}}$ satisfies Assumption \ref{assum:technical assumption for Assum (A)} and \cite[Assumption (A)]{Champagnat2014}. In particular, there exists a unique QSD $\pi$. This QSD has an (essentially) bounded density with respect to Lebesgue measure, and full support. Moreover, $(X_t)_{0\leq t<\tau_{\partial}}$ is lower semicontinuous (in the sense of Definition \ref{defin:lower semicts kernel}) and satisfies assumptions \ref{assum:adjoint Dobrushin main results section}, \ref{assum:combined Dobrushin adjoint Dobrushin main results section} and \ref{assum:adjoint anti-Dobrushin main results section}.
\end{theo}

It therefore follows from Theorem \ref{theo:uniform Linfty convergence main results section} that, for PCMPs satisfying that assumptions of Theorem \ref{theo:PCMPs}, there exists a time $T<\infty$ and constant $\gamma>0$ such that $\Law_{\mu}(X_t\lvert \tau_{\partial}>t)\ll_{\infty}\pi$ for all $t\geq T$ and $\mu\in\calP(\chi)$, with its density with respect to $\pi$ satisfying
\begin{equation}\label{eq:exp conv to pi for PCMP result} 
\Big\lvert\Big\lvert \frac{d\Law_{\mu}(X_t\lvert \tau_{\partial}>t)}{d\pi}-1\Big\rvert\Big\rvert_{L^{\infty}(\pi)}\leq e^{-\gamma(t-T)}\quad\text{for all}\quad  T\leq t<\infty,\quad\mu\in\calP(\chi).
\end{equation}

The following parabolic boundary Harnack-type inequality may be obtained either straight from \eqref{eq:exp conv to pi for 1D PDMP result} and \eqref{eq:exp conv to pi for PCMP result} or by applying theorems \ref{theo:dominated by pi theorem general results main results section} and \ref{theo:DAD lower bounds density main results section} with theorems \ref{theo:1D PDMPs} and \ref{theo:PCMPs}.
\begin{theo}\label{theo:bdy Harnack for PDMPs}
In addition to Assumption \ref{assum:standing assum PDMP}, we assume either that $d=1$ or that $d$ is arbitrary and Assumption \ref{assum:assumption for absorbed PCMPs} is satisfied. Then there exists a time $T<\infty$ and constants $0<c\leq C<\infty$ such that
\begin{equation}\label{eq:comparison inequality for PDMPs}
c\pi\leq \Law_{\mu}((X_t,I_t)\lvert \tau_{\partial}>t)\leq C\pi\quad\text{for all}\quad \mu\in\calP(\chi),\quad t\geq T.
\end{equation}
We put this in the form of a parabolic boundary Harnack inequality as follows. For any initial conditions $\mu,\nu\in\calP(\chi)$, we let $u_1((x,i),t)$ and $u_2((x,i),t)$ be versions of $\frac{d\Pm_{\mu}((X_t,I_t)\in \cdot,\tau_{\partial}>t)}{dLeb(\cdot)}((x,i))$ and $\frac{d\Pm_{\nu}((X_t,I_t)\in \cdot,\tau_{\partial}>t)}{dLeb(\cdot)}((x,i))$ for $(x,i)\in U\times E$ and $t\geq T$, respectively. It follows that
\begin{equation}\label{eq:parabolic boundary Harnack PDMP}
\inf_{t_1,t_2\geq T}\frac{\essinf_{(x,i)\in U\times E}\frac{u_1((x,i),t_1)}{u_2((x,i),t_2)}}{\esssup_{(x',i')\in U\times E}\frac{u_1((x',i'),t_1)}{u_2((x',i'),t_2)}}\geq \frac{c^2}{C^2}>0.
\end{equation} 
Note in particular that $c$, $C$ and $T$ do not depend upon $\mu$ and $\nu$, and that this comparison is valid up to the boundary.
\end{theo}
Since the comparison in \eqref{eq:parabolic boundary Harnack PDMP} is valid up to the boundary, this provides for a parabolic boundary Harnack-type inequality (with the caveat that it only allows us to compare those $u_1,u_2$ corresponding globally to the absorbed PDMP). On the other hand, the corresponding Fokker-Planck equation is first order. The author is not aware of boundary comparison inequalities of this type having previously been established for first-order PDEs. 

We may observe that, over the course of proving Theorem \ref{theo:2D PDMPs}, we have established that Assumption \ref{assum:combined Dobrushin adjoint Dobrushin main results section} is satisfied, so that the lower bound in \eqref{eq:comparison inequality for PDMPs} is satisfied under the assumptions of that theorem.

In contrast to theorems \ref{theo:comparison inequality for degenerate diffusions} and \ref{theo:comparison inequality for 1D Langevin}, it is straightforward to see that the time horizon $T>0$ in Theorem \ref{theo:bdy Harnack for PDMPs} cannot be made arbitrarily small. Moreover $u_1$ and $u_2$ in Theorem \ref{theo:bdy Harnack for PDMPs} may not have continuous versions, so it doesn't necessarily make sense to talk about pointwise infimums and supremums, justifying the use of the essential infimum and supremum in \eqref{eq:parabolic boundary Harnack PDMP}.

We shall prove these theorems by application of the following theorem. We shall state and prove this theorem in the following subsection, before applying it to obtain theorems \ref{theo:1D PDMPs}, \ref{theo:2D PDMPs} and \ref{theo:PCMPs} in the subsections thereafter.

\subsection*{Theorem \ref{theo:general theorem PDMPs}}

We shall establish here a theorem, Theorem \ref{theo:general theorem PDMPs}, containing theorems \ref{theo:1D PDMPs}, \ref{theo:2D PDMPs} and \ref{theo:PCMPs} as particular cases. We begin with some necessary definitions, before stating Theorem \ref{theo:general theorem PDMPs}. We will then establish theorems \ref{theo:1D PDMPs}, \ref{theo:2D PDMPs} and \ref{theo:PCMPs} by verifying that their assumptions imply the assumptions of Theorem \ref{theo:general theorem PDMPs}.

We write $\omega$ for the stationary distribution of $I_t$. We write $\hat{I}_t$ for its time-reversal at stationarity, and $\hat{Q}$ for the rate matrix of this time-reversal. This is the rate matrix $\hat{Q}$ satisfying
\[
\omega_iQ_{ij}=\omega_j\hat{Q}_{ji},\quad i\neq j.
\]
We note that $\hat{Q}_{ij}>0$ for all $i\neq j$. We may then define the reversed absorbed PDMP $((\hat{X}_t,\hat{I}_t))_{0\leq t<\hat{\tau}_{\partial}}$ as follows.

\begin{defin}[Reversed absorbed PDMP]\label{defin:reversed absorbed PDMP}
Prior to the absorption time $\hat{\tau}_{\partial}$, the component $\hat{I}_t$ evolves as a continuous-time jump process on $E$ with rate matrix $\hat{Q}$. The position component $\hat{X}_t$ then evolves according to the ODE
\[
\dot{\hat{X}}_t=-v^{\hat{I}_t}(\hat{X}_t).
\]
The process is absorbed upon contact of the spatial position with the boundary, 
\[
\hat{\tau}_{\partial}:=\inf\{t>0:\hat{X}_{t-}\in \partial U\}.
\]
\end{defin}

We write $(P_t)_{t\geq 0}$ and $(\hat{P}_t)_{t\geq 0}$ respectively for the submarkovian transition semigroups associated to $((X_t,I_t))_{0\leq t<\tau_{\partial}}$ and $((\hat{X}_t,I_t))_{0\leq t<\hat{\tau}_{\partial}}$. We then define the following Green-type kernels.
\begin{equation}\label{eq:PDMP Green kernels}
G_{T_0,T_1}:=\int_{T_0}^{{T_1}}P_sds,\quad \hat{G}_{{T_0},{T_1}}:=\int_{T_0}^{{T_1}}\hat{P}_sds,\quad 0\leq {T_0}<{T_1} <\infty.
\end{equation}

For $x\in U$, $1\leq m<\infty$, $\bft =(t_1,\ldots,t_m)\in \Rm_{>0}^m$ and $\bfi=(i_1,\ldots,i_m)\in E^m$, we define the composite flows
\begin{equation}\label{eq:composite flows}
\Phi^\bfi_\bft=\varphi^{i_m}_{t_m}\circ\ldots\circ \varphi^{i_1}_{t_1}\quad\text{and}\quad \Phi^\bfi_{-\bft}=\varphi^{i_m}_{-t_m}\circ\ldots\circ \varphi^{i_1}_{-t_1}.
\end{equation}

\begin{defin}[Accessibility and reverse-accessibility]
For $x\in U$, an open set $V\subseteq U$ we say that $V$ is accessible (respectively reverse-accessible) from $x$ if there exists $1\leq m<\infty$, $\bft\in \Rm_{>0}^m$ and $\bfi\in E^m$ such that $\Phi^\bfi_{\bft}(x)\in V$ (respectively $\Phi^{\bfi}_{-\bft}(x)\in V$). For $x,y\in \text{int}(E)$ we say $y$ is accessible (respectively reverse-accessible) from $y$, $x\ra y$ (respectively $x\ra_- y$), if for all open neighbourhoods $V\ni y$ we have $x \ra V$ (respectively $x\ra_- V$). 
\end{defin}

We define the following sets.
\begin{defin}\label{defin:D+ and D- sets PDMP proof}
We define the following functions,
\begin{equation}
T^i_{\partial,+}(x):=\inf\{t>0:\varphi^i_t(x)\in \partial U\},\quad T^i_{\partial,-}(x):=\inf\{t>0:\varphi^i_{-t}(x)\in \partial U\},\quad x\in U,
\end{equation}
the latter of which is the function defined in \eqref{eq:time to hit partial U}. We define $\calD_+$ and $\calD_-$ respectively to be
\begin{equation}
\begin{split}
\calD_+:=\{(x,i)\in U\times E:x'\mapsto T^i_{\partial,+}(x')\text{ is discontinuous at $x$}\},\\
\calD_-:=\{(x,i)\in U\times E:x'\mapsto T^i_{\partial,-}(x')\text{ is discontinuous at $x$}\}.
\end{split}
\end{equation}
\end{defin}

We now consider the following assumption
\begin{assum}\label{assum:assumption for killed PDMP}
We assume that we have the following:
\begin{enumerate}
\item \label{enum:PDMP accessibility condition}
For any two points $x,y\in U$, $y$ is both accessible and reverse-accessible from $x$.
\item\label{enum:assum for minorisation when start from compact set}
There exists $x^{\ast}\in \text{int}(E)$ satisfying:
\begin{enumerate}
\item\label{enum:linear independence of choices of d vectors at x ast}
for any $1\leq i_1<\ldots<i_d\leq n$, $v^{i_1}(x^{\ast}),\ldots,v^{i_d}(x^{\ast})$ are linearly independent;
\item\label{enum:x ast convex hull so accessible from itself condition}
$0\in \overline{\conv}(v^1(x^{\ast}),\ldots,v^n(x^{\ast}))$.

\end{enumerate}
\item\label{enum:assum Gt minorised}
There exists times $0\leq T_0<T_1<\infty$ and $0\leq \hat{T}_0<\hat{T}_1<\infty$, and a constant $C<\infty$, such that
\[
G_{T_0,T_1}((x,i),\cdot),\hat{G}_{\hat{T}_0,\hat{T}_1}((x,i),\cdot)\leq C\text{Leb}(\cdot) \quad\text{for all}\quad (x,i)\in U\times  E.
\]
\end{enumerate}
\end{assum}

\begin{theo}\label{theo:general theorem PDMPs}
In addition to the standing assumption, \ref{assum:standing assum PDMP}, we assume that Assumption \ref{assum:assumption for killed PDMP} is satisfied. Then the absorbed PDMP $((X_t,I_t))_{0\leq t<\tau_{\partial}}$ satisfies Assumption \ref{assum:technical assumption for Assum (A)} and \cite[Assumption (A)]{Champagnat2014}. In particular, there exists a unique QSD $\pi$. This QSD, $\pi$, has an essentially bounded density with respect to Lebesgue measure, and full support. Moreover, $((X_t,I_t))_{0\leq t<\tau_{\partial}}$ satisfies assumptions \ref{assum:adjoint Dobrushin main results section} and \ref{assum:combined Dobrushin adjoint Dobrushin main results section}, and is lower semicontinuous (in the sense of Definition \ref{defin:lower semicts kernel}). If, in addition, there exists $\hat{T}_2>0$ such that $\hat{P}_{\hat{T}_2}((x,i),\cdot)\leq C\Leb(\cdot)$ for all $(x,i)\in U\times E$, then $(X_t)_{0\leq t<\tau_{\partial}}$ satisfies Assumption \ref{assum:adjoint anti-Dobrushin main results section}. Moreover if $V=\calD_-^c\subseteq U\times E$ is an open set such that $\Pm_{(x,i)}((X_t,I_t)\in \calD_-)=0$ for all $(x,i)\in V$ and $t\geq 0$, then (a version of) $\frac{d\pi}{d\Leb}$ is continuous on $V$. 
\end{theo}

\subsubsection*{Overview of the Proof of Theorem \ref{theo:general theorem PDMPs}}

We henceforth assume that assumptions \ref{assum:standing assum PDMP} and \ref{assum:assumption for killed PDMP} are satisfied. We firstly note that it is trivial to see that $(X_t)_{0\leq t<\tau_{\partial}}$ satisfies Assumption \ref{assum:technical assumption for Assum (A)}. It is also immediate from the accessibility, Part \ref{enum:PDMP accessibility condition} of Assumption \ref{assum:assumption for killed PDMP}, that any QSD of $(X_t)_{0\leq t<\tau_{\partial}}$ (if it exists) must have full support.

We define the constant
\[
\bar d:=\sup_{\substack{x\in U\\ i\in E}}\lvert\text{tr}(Dv^i(x))\rvert. 
\]
We then define the submarkovian transition semigroup
\begin{equation}\label{eq:adjoint kernel for PDMP}
\tilde{P}_t((y,j),\cdot):=\expE_{(y,j)}[\Ind((\hat{X}_t,\hat{I}_t)\in \cdot)\Ind(\hat{\tau}_{\partial}>t)e^{-\int_0^t(Dv^{\hat{I}_s}(\hat{X}_s)+\bar d)ds}],\quad t\geq 0.
\end{equation}
We shall firstly prove the following proposition.
\begin{prop}\label{prop: reverse PDMP provides adjoint}
The submarkovian transition semigroup $(\tilde{P}_t)_{t\geq 0}$ satisfies
\begin{equation}
(\Leb(dx)\times \delta_i)\omega(i)P_t((x,i),dy\times \{j\})=e^{\bar d t}(\Leb(dy)\times \delta_j)\omega(j)\tilde{P}_t((y,j),dx\times \{i\}).
\end{equation}
\end{prop}

Corresponding to \eqref{eq:adjoint kernel for PDMP} we define
\begin{equation}
\tilde{G}_{T_0,T_1}:=\int_{T_0}^{T_1}e^{\bar d s}\tilde{P}_sds,\quad t\geq 0,
\end{equation}
which we observe by Proposition \ref{prop: reverse PDMP provides adjoint} satisfies
\begin{equation}\label{eq:adjoint equation for Green kernel}
\begin{split}
(\Leb(dx)\times \delta_i)\omega(i)G_{T_0,T_1}((x,i),dy\times \{j\})\\
=(\Leb(dy)\times \delta_j)\omega(j)\tilde{G}_{T_0,T_1}((y,j),dx\times \{i\}),\quad 0\leq T_0<T_1<\infty.
\end{split}
\end{equation}
\begin{rmk}
The proof of Proposition \ref{prop: reverse PDMP provides adjoint} does not use Assumption \ref{assum:assumption for killed PDMP}, hence it can (and will) be used when verifying that assumption.
\end{rmk}

We observe that we have the relationship
\begin{equation}\label{eq:PDMP proof relationship between hat and tilde kernels}
\begin{split}
e^{-2\bar{d}t}\hat{P}_t((y,j),\cdot)\leq \tilde{P}_t((y,j),\cdot)\leq \hat{P}_t((y,j),\cdot)\quad \text{for all}\quad (y,j)\in U\times E,\quad t\geq 0,\\
e^{-2\bar{d}T_1}\hat{G}_{T_0,T_1}((y,j),\cdot)\leq \tilde{G}_{T_0,T_1}((y,j),\cdot)\\
\leq e^{\bar{d}T_1}\hat{G}_{T_0,T_1}((y,j),\cdot)\quad\text{for all}\quad (y,j)\in U\times E,\quad 0\leq T_0\leq T_1<\infty.
\end{split}
\end{equation}

We have the following lemma.
\begin{lem}\label{lem:Lebesgue measure of set from flow is 0}
Suppose that $\cal{M}$ is a $d-1$-dimensional $C^{\infty}$ submanifold of $\Rm^d$. We define $v$ to be a $C^{\infty}$, bounded vector field in $\Rm^d$, corresponding to which is the flow $\varphi_t$. We define $\calS:=\{\varphi_t(x):x\in \calM,\; v(x)\in T_x\calM,\; t\in \Rm\}$. Then $\Leb(\calS)=0$.
\end{lem}

\begin{proof}[Proof of Lemma \ref{lem:Lebesgue measure of set from flow is 0}]

We define
\[
f:\calM\times \Rm\ni (x,t)\mapsto \varphi_t(x)\in \Rm^d.
\]
We see that $(x,t)$ is a critical point of $f$ if and only if $v(x)\in T_x\calM$, which is equivalent to $f((x,t))\in \calS$. The conclusion of Lemma \ref{lem:Lebesgue measure of set from flow is 0} then follows by Sard's theorem.
\end{proof}
\begin{observ}\label{observ:traversal of bdy then cts}
We observe that if $v^i$ traverses the boundary at $\varphi^i_{T^i_{\partial,+}}(x)$, that is $v^i(\varphi^i_{T^i_{\partial,+}}(x))\notin T_{\varphi^i_{T^i_{\partial,+}}(x)}\partial U$, then $x'\mapsto T^i_{\partial,+}(x')$ must be continuous at $x$. 
\end{observ}
It follows from Lemma \ref{lem:Lebesgue measure of set from flow is 0} and Observation \ref{observ:traversal of bdy then cts} that 
\begin{equation}\label{eq:lebesgue measure of D+ and D- = 0}
\Leb(\calD_+)=0,\quad\text{and similarly}\quad\Leb(\calD_-)=0.
\end{equation}

\begin{prop}\label{prop:PDMP lower semicts and Green kernel cts propn}
The semigroups $(P_t)_{t\geq 0}$ and $(\tilde{P}_t)_{t\geq 0}$ are lower semicontinuous (in the sense of Definition \ref{defin:lower semicts kernel}). Moreover, for all $0\leq T_0< T_1<\infty$, $G_{T_0,T_1}$ and $\tilde{G}_{T_0,T_1}$ satisfy the following:
\begin{enumerate}
\item
$G_{T_0,T_1}$ and $\tilde{G}_{T_0,T_1}$ are lower semicontinuous (in the sense of Definition \ref{defin:lower semicts kernel}).
\item
If $\mu_n$ are a sequence of probability measures converging weakly to $\mu\in\calP(\Leb)$, then $\mu_n G_{T_0,T_1}\ra \mu G_{T_0,T_1}$ and $\mu_n\tilde{G}_{T_0,T_1}\ra \mu \tilde{G}_{T_0,T_1}$ weakly. 
\item\label{enum:Greens function propn PDMP cty on open set}
If $\Pm_{x}((X_t,I_t)\in \calD_+)=0$ for all $t\geq 0$ (respectively $\Pm_{x}((\hat{X}_t,\hat{I}_t)\in \calD_-)=0$ for all $t\geq 0$) and $\calD_+$ is closed (respectively $\calD_-$ is closed, then $G_{T_0,T_1}f$ is continuous at $x$ (respectively $\tilde{G}_{T_0,T_1}f$ is continuous at $x$) for all $f\in C_b(\calD_+^c)\cap \calB_b(\chi)$ (respectively $f\in C_b(\calD_-^c)\cap \calB_b(\chi)$).
\end{enumerate}
\end{prop}

We note that it follows from parts \ref{enum:linear independence of choices of d vectors at x ast} and \ref{enum:x ast convex hull so accessible from itself condition} of Assumption \ref{assum:assumption for killed PDMP} that
\begin{equation}\label{eq:x ast strong differences of full rank assumption}
\text{span}(\{v^i(x^{\ast})-v^j(x^{\ast}):i\neq j\})=v^1(x^{\ast})-\overline{\conv}(v^1(x^{\ast}),\ldots,v^n(x^{\ast}))=\Rm^d.
\end{equation}

We now state the following linear algebra lemma, which shall be proven in the appendix.
\begin{lem}\label{lem:0 contained in open convex hull lemma}
We suppose that the vectors $v_1,\ldots,v_n\in \Rm^d$ have the following two properties:
\begin{enumerate}
\item any choice of $d$ vectors, $v_{i_1},\ldots,v_{i_d}$, are linearly independent;
\item $0\in \overline{\conv}(v_1,\ldots,v_n)$.
\end{enumerate}
Then $\conv(v_1,\ldots,v_n)$ is an open set containing $0$.
\end{lem}

\begin{prop}\label{prop: PDMP satisfies A1}
There exists a (non-empty) open set $\calO\subset\subset U\times E$, a time $t_1>0$ and a constant $c_1>0$ such that the probability measure $\nu:=\frac{\Leb_{\lvert_{\calO}}}{\Leb(\calO)}$ satisfies
\begin{equation}
\frac{P_{t_1}((x,i),\cdot)}{P_{t_1}1((x,i))}\geq c_1\nu(\cdot)\quad\text{for all}\quad (x,i)\in U\times E.
\end{equation}
\end{prop}

Note that, by applying Proposition \ref{prop: PDMP satisfies A1} to $((\hat{X}_t,\hat{I}_t))_{0\leq t<\tau_{\partial}}$ and using \eqref{eq:PDMP proof relationship between hat and tilde kernels}, we see that there exists a (non-empty) open set $\tilde{\calO}\subset\subset U$, a time $\tilde{t}_1>0$, and a constant $\tilde{c}_1>0$, such that the probability measure $\tilde{\nu}:=\frac{\Leb_{\lvert_{\tilde{\calO}}}}{\Leb(\tilde{\calO})}$ satisfies
\begin{equation}
\frac{\tilde{P}_{t_1}((x,i),\cdot)}{\tilde{P}_{t_1}1((x,i))}\geq \tilde{c}_1\tilde{\nu}(\cdot)\quad\text{for all}\quad (x,i)\in U\times E.
\end{equation}

We now have the following proposition.
\begin{prop}\label{prop:left emeasure for greens kernels}
We define $0\leq T_0<T_1<\infty$ and $0\leq \hat{T}_0<\hat{T}_1<\infty$ to be the times assumed to exist in Part \ref{enum:assum Gt minorised} of Assumption \ref{assum:assumption for killed PDMP}, and $t_1,\tilde{t}_1$ to be the times provided for in Proposition \ref{prop: PDMP satisfies A1}. We define $t_2:=(T_1\vee\hat{T}_1)+(t_1\vee \tilde{t}_1)>0$. Then for all $s>0$, $G_{t_2,t_2+s}$ and $\tilde{G}_{t_2,t_2+s}$ both have a left eigenmeasure (with strictly positive eigenvalue), $\pi_s$ and $\tilde{\pi}_s$ respectively, belonging to $\calP_{\infty}(\Leb)$. Moreover there exists $c_{2,s},\tilde{c}_{2,s}>0$ for all $s>0$ such that $\pi_s\geq c_{2,s}\Leb_{\calO}$ and $\tilde{\pi}_s\geq \tilde{c}_{2,s}\Leb_{\tilde{\calO}}$.
\end{prop}

We now use the above propositions to conclude Theorem \ref{theo:general theorem PDMPs}. It follows from propositions \ref{prop: reverse PDMP provides adjoint},  \ref{prop:PDMP lower semicts and Green kernel cts propn}, \ref{prop: PDMP satisfies A1} and \ref{prop:left emeasure for greens kernels}, and the accessibility assumption, that $\frac{1}{s}G_{T_2,T_2+s}$ satisfies Assumption \ref{assum:assum for cor right efn main results section} for some $s>0$ sufficiently large in the sense of Definition \ref{defin:discrete time semigroup satisfying right efn assumption}, so that the corresponding discrete-time absorbed Markov process with one-step transition kernel $\frac{1}{s}G_{T_2,T_2+s}$ satisfies Assumption \ref{assum:assum for cor right efn main results section}. We henceforth fix this $s>0$ and denote the corresponding discrete-time absorbed Markov process as $(Y_n)_{0\leq n<\tau^Y_{\partial}}$. It is trivial to see that it must satisfy Assumption \ref{assum:technical assumption for Assum (A)}. It follows from Theorem \ref{theo:criterion for right efn main results section} that there exists a strictly positive, bounded, pointwise right eigenfunction $h$ for $G_{T_2,T_2+s}$. We have, by Proposition \ref{prop: PDMP satisfies A1}, that $(Y_n)_{0\leq n<\tau^Y_{\partial}}$ satisfies \cite[Assumption (A1)]{Champagnat2014}, so that it must also satisfy \cite[Assumption A]{Champagnat2014}, by Proposition \ref{prop:right efn gives A2}. It therefore follows from \cite[Theorem 2.1 and Corollary 2.4]{Champagnat2014} that $\pi$ is the unique left eigenmeasure and $h$ the unique non-negative right eigenfunction for $G_{T_2,T_2+s}$. 

Since $G_{T_2,T_2+s}$ and $P_t$ commute for any $t>0$, it follows that $\pi$ is a QSD and $h$ is a bounded, strictly positive pointwise right eigenfunction for $(P_t)_{t\geq 0}$. By precisely the same argument, we obtain the unique QSD $\tilde{\pi}$ and strictly positive, bounded, pointwise right eigenfunction $\tilde{h}$ for $(\tilde{P}_t)_{t\geq 0}$. We have $\pi,\tilde{\pi}\in\calP_{\infty}(\Leb)$ by Proposition \ref{prop:left emeasure for greens kernels}. Using Proposition \ref{prop: PDMP satisfies A1}, the existence of $h$ and Proposition \ref{prop:right efn gives A2}, it follows that $((X_t,I_t))_{0\leq t<\tau_{\partial}}$ satisfies \cite[Assumption (A)]{Champagnat2014}. 

We have from \eqref{eq:PDMP proof relationship between hat and tilde kernels} and propositions \ref{prop: reverse PDMP provides adjoint} and \ref{prop: PDMP satisfies A1} that $((X_t,I_t))_{0\leq t<\tau_{\partial}}$ satisfies assumptions \ref{assum:adjoint Dobrushin main results section} and \ref{assum:combined Dobrushin adjoint Dobrushin main results section}. Proposition \ref{prop:PDMP lower semicts and Green kernel cts propn} provides for $((X_t,I_t))_{0\leq t<\tau_{\partial}}$ being lower semicontinuous. 

It is immediate from \eqref{eq:PDMP proof relationship between hat and tilde kernels} and Proposition \ref{prop: reverse PDMP provides adjoint} that if there exists $\hat{T}_2>0$ such that $\hat{P}_{\hat{T}_2}((x,i),\cdot)\leq C\Leb(\cdot)$ for all $(x,i)\in U\times E$, then $((X_t,I_t))_{0\leq t<\tau_{\partial}}$ satisfies Assumption \ref{assum:adjoint anti-Dobrushin main results section}.

Finally, we have established that $G_{0,1}$ satisfies Assumption \ref{assum:cty assum for cty cor} (in the sense of Remark \ref{rmk:discrete time semigroup satisfying cty assumption}). We have that $\pi$ is a QSD for $G_{0,1}$ belonging to $\calP_{\infty}(\Leb)$. It therefore follows from Part \ref{enum:QSD thm cty on V main results section} of Theorem \ref{theo:criterion for cty of QSD} that for any open set $V$, if
\begin{equation}\label{eq:PDMP proof cond for V for cty of pi summary of proof}
\tilde{G}_{0,1}(\calB_b(\chi)\cap C_b(V))\subseteq \calB_b(\chi)\cap C_b(V),
\end{equation}
then (a version of) $\frac{d\pi}{d\Leb}_{\lvert_V}$ belongs to $C_b(V)$. We now suppose that $V=\calD_-^c\subseteq U\times E$ is an open set such that $\Pm_{(x,i)}((\hat{X}_t,\hat{I}_t)\in \calD_-)=0$ for all $(x,i)\in V$ and $t\geq 0$. It follows that $\Pm_{(x,i)}((\hat{X}_t,\hat{I}_t)\in \calD_-\text{ for some }t\in \Qm_{\geq 0})=0$. If $(\hat{X}_t,\hat{I}_t)$ is in $\calD_-$ at any time, it has to be in $\calD_-$ at some rational time. Thus $\Pm_{(x,i)}((\hat{X}_t,\hat{I}_t)\in \calD_-\text{ for some }t\in \Rm_{\geq 0})=0$. It follows from Part \ref{enum:Greens function propn PDMP cty on open set} of Proposition \ref{prop:PDMP lower semicts and Green kernel cts propn} that $\tilde{G}_{0,1}f$ is continuous on $V$ for all $f\in \calB_b(\chi)\cap C_b(V)$, so that $V$ satisfies \eqref{eq:PDMP proof cond for V for cty of pi summary of proof}.

It is left to establish propositions \ref{prop: reverse PDMP provides adjoint} - \ref{prop: PDMP satisfies A1}.

\subsubsection*{Proof of Proposition \ref{prop: reverse PDMP provides adjoint}}

We fix $t>0$. We define $(\Theta,\vartheta)$ to be a probability space on which is supported a stationary copy of $(I_s)_{0\leq s\leq t}$, which we label $(I^{\theta}_s)_{0\leq s\leq t}$. Then $(\hat{I}^{\theta}_s)_{0\leq s\leq t}:=(\hat{I}^{\theta}_{t-s})_{0\leq s\leq t}$ is a stationary copy of $(\hat{I}_s)_{0\leq s\leq t}$. This defines a random diffeomorphism as follows.

We define for each $\theta\in\Theta$ and $x^0\in U$ the ODE
\begin{equation}\label{eq:ODE corresponding to fwd time jump process PDMP}
\dot{x}_s=v^{I^{\theta}_s}(x_s)ds,\quad x_0=x^0.
\end{equation}
Given a solution to this ODE, we set
\[
f^{\theta}(x^0):=\begin{cases}
x_t,\quad x_s\notin \partial U\quad\text{for all}\quad 0\leq s\leq t\\
\partial,\quad x_s\in \partial U\quad\text{for some}\quad 0\leq s\leq t
\end{cases}.
\]

We define $U^{\theta}:=\{x\in U:f^{\theta}(x)\in U\}$ and $V^{\theta}=f^{\theta}(U^{\theta})$, so that $f^{\theta}:U^{\theta}\ra V^{\theta}$ is a random diffeomorphism with
\[
\det(Df^{\theta}(x))= e^{\int_0^t\text{tr}(Dv^{I^{\theta}_s}(x_s))ds},
\]
which we may see by differentiating the ODE \eqref{eq:ODE corresponding to fwd time jump process PDMP} at time $t$ with respect to its initial condition.

We define $g^{\theta}:\hat{V}^{\theta}\ra \hat{U}^{\theta}$ similarly for $(\hat{I}^{\theta}_s)_{0\leq s\leq t}$ and the vector fields $-v^1,\ldots,-v^n$, with $(y_s)_{0\leq s\leq t}$ the corresponding ODE with initial condition $y_0=y^0$ as in \eqref{eq:ODE corresponding to fwd time jump process PDMP}. We observe that $V^{\theta}=\hat{V}^{\theta}$ and $U^{\theta}=\hat{U}^{\theta}$, so that $g^{\theta}$ gives the inverse diffeomorphism for $f^{\theta}$. It follows that
\[
\text{det}(Dg^{\theta}(y^0))=e^{-\int_0^t\text{tr}(Dv^{\hat{I}^{\theta}_s}(y_s))ds}.
\]

We now fix $A,B\in\mathscr{B}(U)$ and $i,j\in E$. Then using Tonelli's theorem and the change of variables formula we have that
\[
\begin{split}
\omega(i)\int_AP_{t}((x,i),B\times \{j\})\Leb(dx)\\
=\int_{\Theta}\int_{U_{\theta}}\Ind(I^{\theta}_0=i)\Ind(I^{\theta}_t=j)\Ind(x\in A)\Ind(f_{\theta}(x)\in B)\Leb(dx)\vartheta(\theta)\\
\underbrace{=}_{\substack{\text{substitute}\\
y=f_{\theta}(x)}}\int_{\theta}\int_{V_{\theta}}\Ind(\hat{I}^{\theta}_0=j)\Ind(\hat{I}^{\theta}_t=i)\Ind(g_{\theta}(y)\in A)\Ind(y\in B)\det(Dg_{\theta}(y))\Leb(dy)\vartheta(d\theta)\\
=\omega(j)e^{\bar d t}\int_B\tilde{P}_t((y,j),A\times \{i\})\Leb(dy).
\end{split}
\]
\qed

\subsubsection*{Proof of Proposition \ref{prop:PDMP lower semicts and Green kernel cts propn}}

We define $(\Theta,\vartheta)$ to be a probability space on which is defined a family of Poisson jump processes $(I^i_s)_{0\leq s<\infty}$, with rate matrix given by $Q_{ij}$, and initial conditions $I^i_0=i$. This defines, for all $(x,i)\in U\times E$ and all $\theta\in\Theta$,
\[
X^{(x,i)}_t(\theta):=\begin{cases}
x+\int_0^tv^{I^i_s(\theta)}(X_s^{(x,i)})ds,\quad t<\tau_{\partial}^{(x,i)}(\theta):=\inf\{t'>0:x+\int_0^{t'}v^{I^i_s(\theta)}(X_s^{(x,i)})ds\\
\partial,\quad \text{otherwise}
\end{cases}.
\]

We now take a sequence of probability measures $\mu_n\in \calP(U\times E)$ converging weakly to $\mu\in \calP(U\times E)$. By the Skorokhod representation theorem, we may define on the separate probability space $(\Omega,\Pm^{\Omega})$ the $U\times E$-valued random variables $\{(X^0_n,I^0_n):1\leq n<\infty\}$ and $(X^0,I^0)$ such that
\[
(X^0_n,I^0_n)\ra (X^0,I^0) \quad\Pm^{\Omega}\text{-almost surely,}
\]
with distributions $(X^0_n,I^0_n)\sim \mu_n$ ($1\leq n<\infty$) and $(X^0,I^0)\sim \mu$. We adjust the definitions of $\{(X^0_n,I^0_n):1\leq n<\infty\}$ and $(X^0,I^0)$ on a $\Pm^{\Omega}$-null set of $\omega\in \Omega$ to ensure that 
\[
(X^0_n,I^0_n)(\omega)\ra (X^0,I^0)(\omega) \quad\text{for every $\omega\in \Omega$.}
\]

We now fix $(\theta,\omega)\in \Theta\times \Omega$ and $t\geq 0$. We observe that if $\tau^{(X^0,I^0)(\omega)}_{\partial}>t$ then 
\[
\begin{split}
(X^{(X^0_n,I^0_n)(\omega)}_t,I^{(X^0_n,I^0_n)(\omega)}_t)(\theta)\ra (X^{(X^0,I^0)(\omega)}_t,I^{(X^0,I^0)(\omega)_t)}(\theta))\\\text{and}\quad \tau^{(X^0_n,I^0_n)(\omega)}_{\partial}>t\quad\text{for all $n$ sufficiently large}.
\end{split}
\]
It follows that for any $f\in C_b(U\times E;\Rm_{\geq 0})$ we have
\[
\begin{split}
\liminf_{n\ra \infty}f((X^{(X^0_n,I^0_n)(\omega)}_t,I^{(X^0_n,I^0_n)(\omega)}_t)(\theta))\Ind(\tau^{(X^0_n,I^0_n)(\omega)}_{\partial}>t)\\
\geq f((X^{(X^0,I^0)(\omega)}_t,I^{(X^0,I^0)(\omega)}_t)(\theta))\Ind(\tau^{(X^0,I^0)(\omega)}_{\partial}>t).
\end{split}
\]
Taking the expectation with respect to $\vartheta\otimes \Pm^{\Omega}$ and applying Fubini's theorem, we see that
\[
\liminf_{n\ra\infty}\mu_nP_tf\geq \mu P_tf.
\]
By considering the case where $\mu_n=\delta_{(x_n,i_n)}$, $\mu=\delta_{(x,i)}$ and $(x_n,i_n)\ra (x,i)$, we conclude that $P_tf$ must be lower semicontinuous. To conclude that $G_{T_0,T_1}$ is lower semicontinuous, we fix $f\in C_b(U\times E;\Rm_{\geq 0})$, take a sequence $x_n\ra x$, and use Fatou's lemma to see that
\[
\liminf_{n\ra\infty}G_{T_0,T_1}f(x_n)=\liminf_{n\ra\infty}\int_{T_0}^{T_1}P_{s}f(x_n)ds\geq \int_{T_0}^{T_1}\liminf_{n\ra\infty}P_sf(x_n)ds\geq \int_{T_0}^{T_1}P_sf(x)ds=G_{T_0,T_1}f(x).
\]
The proof that $\tilde{P}_t$ and $\tilde{G}_{T_0,T_1}$ must be lower semicontinuous is identical.

We now define the event
\[
A:=\{(\theta,\omega):(X^{(X^0,I^0)(\omega)}_t,I^{(X^0,I^0)(\omega)}_t)(\theta)\notin \calD_+\quad\text{for all}\quad t\geq 0\}.
\]
We fix $0\leq T_0\leq T_1<\infty$. We observe that on the event $A$,
\[
\tau^{(X^0_n,I^0_n)(\omega)}_{\partial}(\theta)\wedge T_1\ra\tau^{(X^0,I^0)(\omega)}_{\partial}(\theta)\wedge T_1.
\]
It follows that on the event $A$, for all $f\in C_b(\chi)$ we have
\begin{equation}\label{eq:convergence of integral of f cty of green kernel propn PDMP}
\begin{split}
\int_{T_0}^{T_1}f((X^{(X^0_n,I^0_n)(\omega)}_t,I^{(X^0_n,I^0_n)(\omega)}_t)(\theta))\Ind(\tau^{(X^0_n,I^0_n)(\omega)}_{\partial}>t)dt\\
\ra \int_{T_0}^{T_1}f((X^{(X^0,I^0)(\omega)}_t,I^{(X^0,I^0)(\omega)}_t)(\theta))\Ind(\tau^{(X^0,I^0)(\omega)}_{\partial}>t)dt.
\end{split}
\end{equation}
It follows by the dominated convergence theorem that
\begin{equation}
\text{if}\quad \Pm_{\mu}((X_t,I_t)\in \calD_+\text{ for some }t\geq 0)=0\quad\text{then}\quad\mu_nG_{T_0,T_1}\ra \mu G_{T_0,T_1}\quad\text{weakly as}\quad n\ra\infty.
\end{equation}
We observe that if $(X_t,I_t)\in \calD_+$ for some time $t\in \Rm_{\geq 0}$, then $(X_t,I_t)\in \calD_+$ for some rational time $t\in \Qm_{\geq 0}$. We therefore obtain, using the countability of $\Qm$, that
\begin{equation}\label{eq:condition for weak convergence of green function pdmp propn proof}
\text{if}\quad \Pm_{\mu}((X_t,I_t)\in \calD_+)=0\quad\text{for all}\quad t\geq 0\quad\text{then}\quad\mu_nG_{T_0,T_1}\ra \mu G_{T_0,T_1}\quad\text{weakly as}\quad n\ra\infty.
\end{equation}

We fix arbitrary $t\geq 0$. If $\mu\in \calP_{\infty}(\Leb)$, then Proposition \ref{prop: reverse PDMP provides adjoint} implies that $\mu P_t\in \calP_{\infty}(\Leb)$. It follows from the monotone convergence theorem that if $\mu\in \calP(\Leb)$, then $\mu P_t\in \calP(\Leb)$ so that $(\mu P_t)(\calD_+)=0$ by \eqref{eq:lebesgue measure of D+ and D- = 0}. Thus by \eqref{eq:condition for weak convergence of green function pdmp propn proof}, if $\mu\in \calP(\Leb)$ and $\mu_n\ra \mu$ weakly, it follows that $\mu_n G_t\ra \mu G_t$ weakly. The proof that $\mu_n\tilde{G}_{T_0,T_1}\ra \mu \tilde{G}_{T_0,T_1}$ is identical.

Part \ref{enum:Greens function propn PDMP cty on open set} of Proposition \ref{prop:PDMP lower semicts and Green kernel cts propn} follows by observing that if $\calD_+^c$ is open, \eqref{eq:convergence of integral of f cty of green kernel propn PDMP} remains true with $f\in C_b(\calD_+^c)\cap\calB_b(\chi)$, and following the above argument with $\mu=\delta_x$.
\qed 

\subsubsection*{Proof of Proposition \ref{prop: PDMP satisfies A1}}

We recall that for all $x\in \partial U$, $\hat{n}(x)$ is the inward unit normal vector. We define for all $x\in \partial U$ the set
\begin{equation}
\begin{split}
E_x^+:=\{i\in E:\langle v^i(x),\hat{n}(x)\rangle>0\}\neq \emptyset.
\end{split}
\end{equation}
This corresponds to the indices of those vector fields pointing inwards at $x$. 

Throughout, we shall define $\inf(\emptyset):=+\infty$, so that a stopping time defined as the infimum of those $t>0$ such that an event occurs is defined to be $+\infty$ if the event doesn't occur. For any compact set $K\subset\subset U\times E$, we define the stopping time
\[
\tau_K:=\inf\{t>0:(X_t,I_t)\in K\}.
\]

Since $0<Q_{ij}<\infty$ for all $i\neq j$, there exists $0<\bar q<\bar Q<\infty$ such that for all $0<t\leq 1$ and $i\neq j$ we have:
\begin{enumerate}
\item
given that $I_0=i$, the probability of $I_s$ jumping from $i\ra j$ in time $t$, and this being the first jump, is at least $\bar q t$;
\item
the probability of $I_s$ jumping at all in time $t$ is at most $\bar Q t$, for any initial condition $I_0=i$.
\end{enumerate}
It is clear that for every $x\in \partial U$ and $i\in E_x^+$ there exists a radius $\delta_x>0$, time $t_x>0$ and probability $p_x>0$ such that:
\begin{enumerate}
\item
If $(X_0,I_0)\in (B(x,\delta_x)\cap U)\times \{i\}$ and $I_s$ doesn't jump in time $t_x$, then $\tau_{\partial}>t_x$. The probability of this is as least $p_x$.
\item
For all $0<t\leq t_x$ there exists a compact set $K_{x,t}\subset\subset U\times \{i\}$ such that if $(X_0,I_0)\in (B(x,\delta_x)\cap U)\times \{i\}$ and $I_s$ doesn't jump in time $t$, then $(X_t,I_t)\in K_{x,t}$.
\end{enumerate}
We take $(x_1,i_1),\ldots,(x_m,i_m)$ such that $i_k\in I^+_{x_k}$ for all $1\leq k\leq m$ and $\{B(x_k,\frac{\delta_{x_k}}{10})\}$ cover $\partial U$. We then define
\[
\begin{split}
\bar \delta:=\min_{1\leq k\leq m}\Big(\frac{\delta_{x_k}}{10}\Big)>0,\quad t_1:=\frac{\bar \delta}{\sup_{i\in E,x\in \bar U}\lvert v^i(x)\rvert+1}\wedge 1\wedge \frac{1}{2\bar Q+1}>0,\\
t_2:=\min_{1\leq k\leq m}t_{x_k}>0,\quad p_2:=\min_{1\leq k\leq m}p_k>0,\quad \text{and}\quad K_2:=\cup_{1\leq k\leq m}K_{x_k,t_2}\subset\subset U\times E.
\end{split}
\]
We define the sets
\[
\begin{split}
S_0:=(B(\partial U,\bar \delta)\cap U)\times E,\quad S_1:=\cup_{1\leq k\leq m}((B(x_k,5\bar \delta)\cap U)\times \{i_k\},\\ S_2:=\cup_{1\leq k\leq m}((B(x_k,8\bar \delta)\cap U)\times \{i_k\}).
\end{split}
\]
We further define the stopping time
\[
\tau_1:=\inf\{t\geq 0:(X_t,I_t)\in S_1\}.
\]

It follows that we have:
\begin{enumerate}
\item
If $(X_0,I_0)\in S_1$ and $I_s$ doesn't jump in time $t_1$, then $\tau_{\partial}>t_1$ and $I_s\in S_2$, for all $0\leq s\leq t_1$. The probability of this is at least $1-\bar Q t_1\geq \frac{1}{2}$.
\item
If $(X_0,I_0)\in S_2$ and $I_s$ doesn't jump in time $t_2$, then $\tau_{\partial}>t_2$ and $X_{t_2}\in K_2$. The probability of this is as least $p_2$. That is we have
\begin{equation}\label{eq:prob of entering K2 from S2 bounded below PDMP propn}
\Pm_{(x,i)}((X_{t_2},I_{t_2})\in K_2,\tau_{\partial}>t_2)\geq p_2\quad\text{for all}\quad (x,i)\in S_2.
\end{equation}
\item
For every $x\in B(\partial U,\bar \delta)\cap U$ there exists $k\in \{1,\ldots,m\}$ such that $B(x,\bar \delta)\cap U\subseteq B(x_k,3\bar\delta)\cap U$. 
\end{enumerate}

We consider $(X_0,I_0)=(x,i)\in  S_0$. For every such initial condition $(x,i)\in S_0$ there are three possibilities:
\begin{enumerate}
\item
There exists $k\in \{1,\ldots,m\}$ such that $i=i_k$ and $x\in B(x_k,3\bar\delta)\cap U$, so that $\tau_1=0$ almost surely. The probability that it then doesn't jump in time $t_1$, hence remains in $S_2$, is at least $\frac{1}{2}$. Thus
\begin{equation}
\Pm_{(x,i)}((X_{t_1},I_{t_1})\in S_2,\tau_{\partial}>t_1)\geq \frac{1}{2}.
\end{equation} 
\item
If $\varphi_{s}^{i}(x)\in U$ for all $0\leq s\leq t_1$, then it follows that there exists some $k\in \{1,\ldots,m\}$ such that $\varphi_s^{i}(x)\in B(x_k,4\bar\delta)\cap U$ for all $0\leq s\leq t_1$. If $I_s$ jumps from $i$ to $i_k$ in time $t_1$, and this is the first jump, then $\tau_1< \tau_{\partial}\wedge t_1$ (the probability of jumping exactly at time $t_1$ being $0$). The probability of this is at least $\bar q t_1$. Thus 
\[
\Pm_{(x,i)}(\tau_1<\tau_{\partial}\wedge t_1)\geq \bar q t_1.
\]
If this occurs, the probability of not jumping again, so remaining in $S_2$, is at least $\frac{1}{2}$. It follows that
\begin{equation}
\Pm_{(x,i)}((X_{t_1},I_{t_1})\in S_2,\tau_{\partial}>t_1)\geq \frac{1}{2}\bar qt_1.
\end{equation}
\item
There exists some minimal $0<t(x,i)\leq t_1$ such that $\varphi^{i}_{t(x,i)}(x)\in \partial U$. In this case, there exists some $k\in \{1,\ldots,m\}$ such that $\varphi_s^{i}(x)\in B(x_k,4\bar \delta)\cap U$ for all $0\leq s<  t(x,i)$. The probability of $I_s$ jumping to $i_k$ in time $t(x,i)$, and this being the first jump, is at least $\bar q t(x,i)$. Thus $\Pm_{(x,i)}(\tau_1<t(x,i))\geq \bar qt(x,i)$. If this occurs, the probability of there being no more jumps of $I_s$ in the following time $t_1-\tau_1$ is at least $\frac{1}{2}$. On this event, $\tau_{\partial}>t_1$ and $(X_{t_1},I_{t_1})\in S_2$. Therefore 
\[
\Pm_{(x,i)}((X_{t_1},I_{t_1})\in S_2, \tau_{\partial}>t_1)\geq \frac{1}{2}\bar qt(x,i).
\]

On the other hand, if there is no jump in time $t(x,i)$, then $\tau_{\partial}=t(x,i)\leq t_1$. The probability of this is at least $1-\bar Q t(x,i)$. Thus $\Pm_{(x,i)}(\tau_{\partial}>t_1)\leq \bar Qt(x,i)$. It follows that 
\begin{equation}
\Pm_{(x,i)}((X_{t_1},I_{t_1})\in S_2\lvert \tau_{\partial}>t_1)\geq \frac{\bar q}{2\bar Q}.
\end{equation}
\end{enumerate}

We therefore obtain $c_2>0$ such that
\begin{equation}\label{eq:prob of entering S2 cond on survival from S0}
\begin{split}
\Pm_{(x,i)}((X_{t_1},I_{t_1})\in S_2\lvert \tau_{\partial}>t_1)\geq c_2\quad \text{for all}\quad (x,i)\in S_0.
\end{split}
\end{equation}

Therefore, by combining \eqref{eq:prob of entering K2 from S2 bounded below PDMP propn} with \eqref{eq:prob of entering S2 cond on survival from S0} and defining $c_3:=p_2c_2>0$, we see that
\begin{equation}
\begin{split}
\Pm_{(x,i)}((X_{t_1+t_2},I_{t_1+t_2})\in K_2\lvert \tau_{\partial}>t_1+t_2)\\
\geq \Pm_{(x,i)}((X_{t_1+t_2},I_{t_1+t_2})\in K_2\lvert \tau_{\partial}>t_1)\geq c_3\quad \text{for all}\quad (x,i)\in  S_0.
\end{split}
\end{equation}

On the other hand, since $S_0^c$ is compact, there exists $c_3'>0$ and a compact set $K_3'\subset\subset U\times E$ such that
\begin{equation}
\Pm_{(x,i)}((X_{t_1+t_2},I_{t_1+t_2})\in K_3'\lvert \tau_{\partial}>t_1+t_2)\geq c_3'\quad \text{for all}\quad (x,i)\in S_0^c.
\end{equation}

We define $K_3:=K_2\cup K_3'$, $t_3:=t_1+t_2$ and $p_3=c_3\wedge c_3'>0$. It follows that
\begin{equation}\label{eq:PDMP proof conditional prob of hitting compact set}
\Pm_{(x,i)}((X_{t_3},I_{t_3})\in K_3\lvert \tau_{\partial}>{t_3})\geq p_3\quad\text{for all}\quad (x,i)\in U\times E.
\end{equation}  

From \cite[Theorem 4.4]{Benaim2015}, we see that there exists a time $t_4>0$, a constant $c_4>0$, a (non-empty) open subset $U_4$ of $U$ such that $x^{\ast}\in U_4\subseteq U$, and a (non-empty) open subset $\calO$ of $U\times E$ such that 
\begin{equation}\label{eq:PDMP proof minorised by O from open neighbourhood of O} 
\Pm_{(x,i)}((X_{t_4},I_{t_4})\in \cdot)\geq c_4\Leb_{\calO}\quad\text{for all}\quad (x,i)\in U_4\times E. 
\end{equation}

It follows from Part \ref{enum:x ast convex hull so accessible from itself condition} of Assumption \ref{assum:assumption for killed PDMP} and Lemma \ref{lem:0 contained in open convex hull lemma} that there exists an open subset $U_5$ of $U_4$ containing $x^{\ast}$, $x^{\ast}\in U_5\subseteq U_4$, such that $0\in \conv(v_1(x),\ldots,v_n(x))$ for all $x\in U_5$. It follows that
\[
\Pm_{(x,i)}((X_s,I_s)\in U_5\times E\quad\text{for all}\quad 0\leq s\leq t)>0\quad \text{for all}\quad (x,i)\in U_5\times E,\quad t>0.
\]

It then follows by the accessibility of $x^{\ast}$, Part \ref{enum:PDMP accessibility condition} of Assumption  \ref{assum:assumption for killed PDMP}, that for all $(x,i)\in K_3$ there exists an open neighbourhood $V_{(x,i)}\ni (x,i)$ and a time $t_{(x,i)}>0$ such that 
\[
\Pm_{(x',i')}((X_t,I_t)\in U_5\times E)>0\quad \text{for all}\quad (x',i')\in V_{(x,i)},\quad t\geq {t_{(x,i)}}.
\]
Taking a finite subcover of $K_3$, we obtain $t_5>0$ such that
\[
\Pm_{(x,i)}((X_{t_5},I_{t_5})\in U_5\times E)>0 \quad\text{for all}\quad (x,i)\in K_3.
\]

Since $P_{t_5}$ is lower semicontinuous (in the sense of Definition \ref{defin:lower semicts kernel}), and $\Ind_{U_5\times E}$ is a lower semicontinuous function, it follows that $P_{t_5}\Ind_{U_5\times E}$ is a lower semicontinuous function which is everywhere positive on $K_3$, hence bounded away from $0$ on $K_3$. Thus there exists $c_5>0$ such that
\begin{equation}\label{eq:PDMP proof prob of hitting open nhood of x ast from compact set}
\Pm_{(x,i)}((X_{t_5},I_{t_5})\in U_5\times E)\geq c_5\quad\text{for all}\quad (x,i)\in K_3.
\end{equation}

We finally combine \eqref{eq:PDMP proof conditional prob of hitting compact set}, \eqref{eq:PDMP proof minorised by O from open neighbourhood of O} and \eqref{eq:PDMP proof prob of hitting open nhood of x ast from compact set} to see that
\begin{equation}
\Pm_{(x,i)}((X_{t_3+t_4+t_5},I_{t_3+t_4+t_5})\in \cdot\lvert \tau_{\partial}>t_3+t_4+t_5)\geq p_3c_4c_5\Leb_{\lvert_{\calO}}(\cdot).
\end{equation}
\qed

\subsubsection*{Proof of Proposition \ref{prop:left emeasure for greens kernels}}
We take $T_0,T_1$ as given by Part \ref{enum:assum Gt minorised} of Assumption \ref{assum:assumption for killed PDMP}. It follows from Proposition \ref{prop: reverse PDMP provides adjoint} that for all $s>0$ there exists a constant $\bar C_s<\infty$ such that $\lvert\lvert \frac{d\mu P_s}{d\Leb}\rvert\rvert_{L^{\infty}(\Leb)}\leq \bar C_s \lvert\lvert \frac{d\mu}{d\Leb}\rvert\rvert_{L^{\infty}(\Leb)}$ for all $\mu\in \calP_{\infty}(\Leb)$. For all $s>0$ we can choose $n$ sufficiently large such that $G_{t_2,t_2+s}\leq G_{T_0,T_1}(P_{T_1-T_0}+\ldots+P_{n(T_1-T_0)})$. It then follows from Part \ref{enum:assum Gt minorised} of Assumption \ref{assum:assumption for killed PDMP} that for all $s>0$ there exists $C'_s<\infty$ such that $G_{t_0,t_0+s}((x,i),\cdot)\leq C_s'\Leb(\cdot)$ for all $(x,i)\in U\times E$. Using Proposition \ref{prop: PDMP satisfies A1}, we therefore have for all $s>0$ that
\[
\begin{split}
G_{t_2,t_2+s}((x,i),\cdot)=\delta_{(x,i)}P_{t_1}G_{t_0}(\cdot)\leq (P_{t_1}1)((x,i))C'_s\text{Leb}(\cdot)\quad \text{and}\\
G_{t_2,t_2+s}1((x,i))\geq  c_1(P_{t_1}1)((x,i))(\nu G_{t_0,t_0+s} 1)\quad \text{for all $(x,i)\in U\times E$}.
\end{split}
\]
Therefore for all $s>0$ there exists $C_s<\infty$ such that
\[
\frac{\mu G_{t_2,t_2+s}}{\mu G_{t_2,t_2+s}1}\leq C_s\Leb(\cdot)\quad \text{for all}\quad \mu\in \calP(U\times E).
\]

We now define for all $s>0$ the convex set
\[
K_s=\{\mu\in \calP(\chi):\mu\leq C_s\Leb\},
\]
equipped with the topology of weak convergence of measures, which is compact by Prohorov's theorem. We may therefore define the map
\[
F:K_s\ni \mu\mapsto \frac{\mu G_{t_2,t_2+s}}{\mu G_{t_2,t_2+s}1}\in K_s,
\]
which by Proposition \ref{prop:PDMP lower semicts and Green kernel cts propn} is continuous. It follows by Schauder's fixed point theorem that $G_{t_2,t_2+s}$ has a fixed point belonging to $\calP_{\infty}(\Leb)$, $\pi_s$, for all $s>0$. 

Finally we use Proposition \ref{prop: PDMP satisfies A1} to see that
\[
\pi_sG_{t_2,t_2+s}=\pi_s G_{t_2-t_1,t_2-t_1+s}P_{t_1}\geq \pi_s G_{t_2-t_1,t_2-t_1+s}P_{t_1}1 c_1\nu.
\]
We may therefore conclude that there exists $c_{2,s}>0$ such that $\pi_s\geq c_{2,s}\Leb_{\calO}$, for all $s>0$. 

The proof for $\tilde{\pi}_s$ is identical.
\qed

This concludes the proof of Theorem \ref{theo:general theorem PDMPs}. 
\qed

\subsection*{Proof of Theorem \ref{theo:1D PDMPs}}
We proceed by applying Theorem \ref{theo:general theorem PDMPs}. The only non-trivial thing to check here is that there exists $\hat{T}_2>0$ and $C<\infty$ such that $\hat{P}_{\hat{T}_2}((x,i),\cdot)\leq C\Leb(\cdot)$ for all $(x,i)\in U\times E$.

We may assume without loss of generality that the drift vectors are everywhere non-zero on all of $\Rm$. Since the drift vectors are non-vanishing, we can partition $E$ into the non-empty sets 
\[
E_+:=\{i:-v^i\quad\text{is everywhere positive}\},\quad E_-:=\{i:-v^i\quad\text{is everywhere negative}\}.
\]
We label these the two ``classes''.

We claim that it suffices to take $\hat{T}_2$ such that, in time $\frac{\hat{T}_2}{2}$, $\hat{I}_s$ has to switch to the opposite class it started in if $(\hat{X}_t,\hat{I}_t)_{0\leq t<\tau_{\partial}}$ is to survive.

We recall the definition of the composite flow maps $\Phi^{\bfi}_{\bft}$ given in \eqref{eq:composite flows}. We take $(x,i_0)\in U\times E$, and define $i_2$ to be the index of the first state $I_s$ jumps to in the opposite class. The (possibly empty) set of states $I_s$ visits in between are given by $\bfi^1$. Corresponding to $i_0$, $\bfi^1$ and $i_2$ are the occupation times $t_0$, $\bft^1$ and $t_2$ respectively. We may therefore define for all $x\in U$, the sequence of states $\bfi_1$ with corresponding occupation times $\bft_1$, and $i_0,i_2\in E$, the map
\[
F^{x,i_0,\bfi,\bft,i_2}:\Big[0,\frac{\hat{T}_2}{2}\Big]^2\ni(t_0,t_2)\mapsto (t_0+\sum_k\bft^1_k+t_2,\varphi^{i_2}_{-t_2}\circ\Phi^{\bfi}_{\bft}\circ \varphi^{i_0}_{-t_0}(x))\in [0,\hat{T}_2]\times \Rm.
\]
We observe that $F^{x,i_0,\bfi_1,\bft_1,i_2}$ is a diffeomorphism onto its image for all $x,i_0,\bfi_1,\bft_1,i_2$, with the divergence $\lvert\det(DF^{x,i_0,\bfi_1,\bft_1,i_2})\rvert$ bounded away from $0$ uniformly in $x,i_0,\bfi_1,\bft_1,i_2$. It follows that there exists $C'<\infty$, not dependent upon $x,i_0,\bfi_1,\bft_1,i_2$, such that
\begin{equation}\label{eq:pushforward of starting time and first time in opposite direction}
F^{x,i_0,\bfi_1,\bft_1,i_2}_{\#}\Leb_{\lvert_{[0,\frac{\hat{T}_2}{2}]^2}}\leq C'\Leb_{[0,\hat{T}_2]\times U}.
\end{equation}
Conditional upon $x,i_0,\bfi_1,\bft_1,i_2$, $t_0$ and $t_2$ are exponentially distributed with parameters given by the rate matrix $\hat{Q}$. Using \eqref{eq:pushforward of starting time and first time in opposite direction}, we may therefore conclude that
\[
\Pm(F^{x,i_0,\bfi_1,\bft_1,i_t}(t_0,t_2)\in \cdot,t_0,t_2\leq \frac{\hat{T}_2}{2}\lvert \bfi_1,\bft_1,i_2)\leq C\Leb_{[0,\hat{T}_2]\times U}(\cdot),
\]
for some uniform constant $C<\infty$. We now write $\tau$ for the stopping time when $I_s$ jumps away from $i_2$, which must be at most $\hat{T}_2$ if $\tau_{\partial}>\hat{T}_2$. Taking the expectation, we therefore have that
\begin{equation}\label{eq:joint distribution of stopping time and location 1D proof}
\begin{split}
\Pm_{(x,i_0)}((\tau,X_{\tau})\in \cdot,\tau_{\partial}>\tau)\leq \Pm(F^{x,i_0,\bfi,\bft,i_t}(t_0,t_2)\in \cdot,t_0,t_2\leq \frac{\hat{T}_2}{2})\\
\leq C\Leb_{[0,\hat{T}_2]\times U}(\cdot)\quad\text{for all}\quad (x,i_0)\in U\times E.
\end{split}
\end{equation}
It follows from Proposition \ref{prop: reverse PDMP provides adjoint} that there exists a constant $\bar C<\infty$ such that $\lvert\lvert \frac{d\mu P_s}{d\Leb}\rvert\rvert_{L^{\infty}(\Leb)}\leq \bar C \lvert\lvert \frac{d\mu}{d\Leb}\rvert\rvert_{L^{\infty}(\Leb)}$ for all $\mu\in \calP_{\infty}(\Leb)$ and $0\leq s\leq \hat{T}_2$. Combining this with \eqref{eq:joint distribution of stopping time and location 1D proof} we are done.
\qed

\subsection*{Proof of Theorem \ref{theo:2D PDMPs}}

We proceed by applying Theorem \ref{theo:general theorem PDMPs}. Our goal is therefore to verify that Assumption \ref{assum:assumption for killed PDMP} is satisfied.

It follows from the assumptions that $0\in\overline{\conv}(v^1(x),\ldots,v^n(x))$, and that any two vector fields are transversal, and Lemma \ref{lem:0 contained in open convex hull lemma}, that every point in $U$ is both accessible and reverse-accessible from every other point, so that we have Part \ref{enum:PDMP accessibility condition} of Assumption \ref{assum:assumption for killed PDMP}. Part \ref{enum:assum for minorisation when start from compact set} of Assumption \ref{assum:assumption for killed PDMP} is immediate. 

We now turn to verifying Part \ref{enum:assum Gt minorised} of Assumption \ref{assum:assumption for killed PDMP}. We write $\tau_1$ and $\tau_2$ for the time of the first (respectively second) jump of $I_t$. We have by assumption that there exists $T<\infty$ such that if $I_s$ does not switch in time $T$, then $X_s$ must hit the boundary in that time. We now fix $i_0\neq i_1$ and set $\bfi:=(i_0,i_1)$. We define 
\[
\begin{split}
A:=\{(x,t_0,t_1)\in U\times [0,T]\times [0,T]:\varphi^{i_0}_s(x)\in U\quad\text{for all}\quad 0\leq s\leq t_0,\\ \varphi^{i_1}_{t_1}\circ\varphi^{i_0}_{t_0}(x)\in U\quad\text{for all}\quad 0\leq s\leq t_1\}.
\end{split}
\]
Since $v^{i_1}$ and $v^{i_2}$ are everywhere transversal on $\bar U$, $\lvert \det(D((t_0,t_1)\mapsto\Phi^{\bfi}_{(t_0,t_1)}(x))\rvert$ is bounded away from $0$ on $A$. Since the jump rates of $I_s$ are bounded, it follows that $\Pm(\Phi^{\bfi}_{(t_0,t_1)}(x)\in \cdot,\tau_{\partial}>\tau_2)$ corresponds to the pushforward of a measure with a bounded density (the joint distribution of the first two switching times of $I_t$) under a local diffeomorphism with divergence bounded uniformly away from $0$ (the map $(t_0,t_1)\mapsto\Phi^{\bfi}_{(t_0,t_1)}(x)$), hence has a bounded density. It follows that there exists a constant $C<\infty$ such that the kernel defined by
\[
K((x,i),\cdot):=\Pm_{(x,i)}((X_{\tau_2},I_{\tau_2})\in \cdot, \tau_{\partial}>\tau_2)
\]
satisfies $K((x,i),\cdot)\leq C\Leb(\cdot)$ for all $(x,i)\in U\times E$. We now observe that
\[
G_{2T,4T}\leq KG_{0,4T}.
\]
It follows from \eqref{eq:adjoint equation for Green kernel} that if $\mu\in\calP_{\infty}(\Leb)$ then $\mu G_{[0,4T]}\in \calP_{\infty}(\Leb)$, whence we conclude that Part \ref{enum:assum Gt minorised} of Assumption \ref{assum:assumption for killed PDMP} is satisfied. This concludes the proof that Assumption \ref{assum:assumption for killed PDMP} is satisfied.

We now assume that $\{x\in \partial U:\hat{n}(x)\cdot v^i(x)=0\}$ has finitely many connected components. It follows that the set $\calS$ defined in the statement of Lemma \ref{lem:Lebesgue measure of set from flow is 0} is the union of finitely many $C^{\infty}$ curves of the form $\{\varphi^i_t(x):t\in \Rm\}$, for some $x\in \partial U$, $i\in E$. It then follows from Observation \ref{observ:traversal of bdy then cts} that $U\setminus \calC$ is the union of finitely many curves of the form
\[
\{(\varphi^i_{-s}(x),i):0\leq s\leq t_-(x,i)\},
\]
for some $(x,i)\in \partial U\times E$, where $t_-(x,i):=\inf\{s>0:\varphi^i_{-s}(x)\notin U\}$. The probability of hitting such a set, given that we start outside of it, must be $0$ since all pairs of vector fields are traversal. Therefore 
\[
\Pm_{(x,i)}((X_t,I_t)\in \calD_-\quad\text{for some}\quad t\geq 0)=0\quad\text{for all}\quad (x,i)\in \calC.
\]
\qed

\subsection*{Proof of Theorem \ref{theo:PCMPs}}

We proceed by applying Theorem \ref{theo:general theorem PDMPs}. Our goal is therefore to verify that Assumption \ref{assum:assumption for killed PDMP} is satisfied.

It follows from Assumption \ref{assum:assumption for absorbed PCMPs} and Lemma \ref{lem:0 contained in open convex hull lemma} that every point in $U$ is both accessible and reverse-accessible from every other point, so that we have Part \ref{enum:PDMP accessibility condition} of Assumption \ref{assum:assumption for killed PDMP}. Part \ref{enum:assum for minorisation when start from compact set} of Assumption \ref{assum:assumption for killed PDMP} is immediate. 

We now seek to show that there exists a constant $C<\infty$ and time $T<\infty$ such that 
\begin{equation}\label{eq:transition density PCMP has bounded density}
P_T((x,i),\cdot)\leq C\Leb(\cdot)\quad \text{for all}\quad(x,i)\in U\times E.
\end{equation}

We write $\calA_t$ for the set of states visited by $I_s$ up to time $t$. For all $A\subseteq E$ and $t\geq 0$ we define the submarkovian kernel
\[
P^A_t((x,i),\cdot)=\Pm_{(x,i)}((X_t,I_t)\in \cdot,\tau_{\partial}>t,\calA_t=A).
\]
We have that
\[
X_t\in X_0+\conv(\{v_i:i\in \calA_t\}).
\]
We fix $A=\{i_0,\ldots,i_{\ell}\}\subseteq E$ and suppose that $\calA_t=A$. We write $(T^0,\ldots,T^{\ell})$ for the corresponding occupation times of $I_s$ prior to time $t$. There are therefore two possibilities:
\begin{enumerate}
\item
If $0\notin \overline{\conv}(\{v_i:i\in A\})$, then $d(X_0,X_t)\geq td(0,\overline{\conv}(\{v_i:i\in A\})\ra \infty$ as $t\ra\infty$. This implies that $P^A_t\equiv 0$ for all $t$ sufficiently large.
\item
If $0\in \overline{\conv}(\{v_i:i\in A\})$, then $\conv(\{v_i:i\in A\})$ is open by Lemma \ref{lem:0 contained in open convex hull lemma}. We have that $\text{span}(v_i-v_j:i,j\in A)\supseteq v_{i_0}-\conv(\{v_i:i\in A\})$, which is open. Therefore $\text{span}(v_i-v_j:i,j\in A)=\Rm^d$ so that
\[
\Rm^{\ell}\ni (t_1,\ldots,t_{\ell})\mapsto (t-\sum_{k=1}^{\ell}t_k)v_0+\sum_{k=1}^{\ell}t_kv_k\in \Rm^d
\]
is a surjective linear map. We may calculate from \cite[Theorem 4.3]{Sericola2000} that $\Pm_i((T^1_t,\ldots,T^{\ell}_t)\in \cdot,\calA_t=A)$ has a bounded density on $\Rm^{\ell}$ for all $i\in E$ (note that $T^0_t=t-\sum_{k=1}^{\ell}T^k_t$ automatically). It therefore follows that for all $t>0$ there exists $C^A_t<\infty$ such that 
\[
P^{\calA}_t((x,i),\cdot)\leq C^A_t\Leb(\cdot).
\]
\end{enumerate}
Since there are only finitely many subsets of $E$, we obtain \eqref{eq:transition density PCMP has bounded density}. It follows that $G_{T,T+1}((x,i),\cdot)=\delta_{(x,i)}G_{0,1}P_T\leq C\Leb(\cdot)$ for all $(x,i)\in U\times E$. We may repeat the above proof with $(\hat{X}_t)_{0\leq t<\tau_{\partial}}$, to obtain that Part \ref{enum:assum Gt minorised} of Assumption \ref{assum:assumption for killed PDMP} is satisfied, and that there exists a time $\hat{T}_2>0$ and constant $C<\infty$ such that $\hat{P}_{\hat{T}_2}((x,i),\cdot)\leq C\Leb(\cdot)$ for all $(x,i)\in U\times E$.
\qed

\section{Appendix}

We collect here the proofs of various technical propositions and lemmas, whose proof we have deferred to this appendix.

\subsection*{Proof of Proposition \ref{prop:density wrt pi-> density}}

We fix $0\leq t<\infty$ and $\mu\in\calM_{\infty}(\pi)$. We may take $C<\infty$ such that $-C\pi\leq \mu\leq C\pi$, so that $C\lambda^t\pi \leq \mu P_t\leq C\lambda^t\pi$. It follows that if $\mu\in \calM_{\infty}(\pi)$ then $\mu P_t\in \calM_{\infty}(\pi)$.

We now fix $\mu\in \calM(\pi)$, and define $\mu_n:=((-n)\vee\frac{d\mu}{d\pi}\wedge n)\pi\in\calM_{\infty}(\pi)$ for $n\in\Nm$. We therefore have that $\mu_nP_t\in \calM_{\infty}(\pi)$ for $n\in\Nm$. It follows from the dominated convergence theorem that $\mu_n\ra \mu$ in total variation, hence $\mu_nP_t\ra \mu P_t$ in total variation. Therefore $\mu P_t\in\calM(\pi)$.

\subsection*{Proof of Proposition \ref{prop:Pt well-defined on L1}}

We fix $0\leq t<\infty$ for the time being. Since $L^{1}(\pi)$ corresponds to equivalence classes of functions which agree $\pi$-almost everywhere, in order to establish that $P_tf$ is well-defined and contained in $L^1(\pi)$ for $f\in L^{\infty}(\pi)$, we must show that for $f\in\calB(\pi)$, $P_tf(x)$ is well-defined for $\pi$-almost every $x$, with $P_tf$ measurable and $\pi$-integrable (where it is defined), and that $P_tf=P_tg$ $\pi$-almost everywhere for any other $g=f$ $\pi$-almost everywhere.

For $f\in \calB(\pi)$, we define $N_f:=\{x\in\chi: P_t\lvert f\rvert(x)<\infty\}$. We firstly establish that $\pi(N_f)=0$ for all $f\in\calB(\pi)$. Given $f\in\calB_{\geq 0}(\pi)$, the monotone convergence theorem implies that
\begin{equation}\label{eq:Ptf non-neg f monotone}
\begin{split}
P_tf(x)=\lim_{n\ra\infty}P_t(f\wedge n)(x)\quad\text{for all}\quad x\\\text{and}\quad \pi(P_tf)=\lim_{N\ra\infty}\pi(P_t(f\wedge n))=\lambda^t\lim_{n\ra\infty}\pi(f\wedge n)=\lambda^t\pi(f).
\end{split}
\end{equation}
Thus $\pi(P_tf)<\infty$, so that $\pi(N_f)=0$ for $f\in\calB_{\geq 0}$. Since $N_f=N_{\lvert f\rvert}$, we have 
\begin{equation}
\pi(N_f)=0\quad\text{for all}\quad f\in\calB(\pi).
\end{equation}

We now establish that 
\begin{equation}\label{eq:indNfPtf in B(pi)}
\Ind_{N_f^c}P_tf\in \calB(\pi)\quad\text{with}\quad \pi(\lvert \Ind_{N_f^c}P_tf\rvert)\leq \lambda^t\pi(\lvert f\rvert)\quad\text{for all}\quad f\in \calB(\pi).
\end{equation}

For $f\in\calB_{\geq 0}(\pi)$, we have that $P_tf$ is a $[0,\infty]$-valued measurable function by \eqref{eq:Ptf non-neg f monotone}, so that we have \eqref{eq:indNfPtf in B(pi)} for $f\in \calB_{\geq 0}(\pi)$. We can write $\Ind_{N_f^c}P_tf=\Ind_{N_f^c}P_t(f\vee 0)-\Ind_{N_f^c}P_t((-f)\vee 0)$, hence we have \eqref{eq:indNfPtf in B(pi)} for all $f\in \calB(\pi)$.

We now observe for $f,g\in\calB(\chi)$ that 
\begin{equation}\label{eq:difference of Ptf and Ptg}
\lvert \Ind_{N_f^c\cap N_g^c}P_tf-\Ind_{N_f^c\cap N_g^c}P_tg\rvert\leq \Ind_{N_f^c\cap N_g^c}P_t\lvert f-g\rvert,
\end{equation}
so that $f=g$ $\pi$-almost everywhere implies that $P_tf=P_tg$ $\pi$-almost everywhere by \eqref{eq:indNfPtf in B(pi)}.

We have therefore established that $P_tf$ is well-defined and contained in $L^1(\pi)$ for $f\in L^1(\pi)$, with 
\begin{equation}\label{eq:cty of Pt fixed t}
\lvert\lvert P_tf-P_tg\rvert\rvert_{L^1(\pi)}\leq \lambda^t\lvert\lvert f-g\rvert\rvert_{L^1(\pi)}\quad\text{for}\quad f,g\in L^1(\pi)
\end{equation}
by \eqref{eq:indNfPtf in B(pi)} and \eqref{eq:difference of Ptf and Ptg}. 

We have that $\mu(P_tf)=(\mu P_t)(f)$ for $\mu\in\calP_{\infty}(\pi)$ and $f\in L_{\geq 0}^{\infty}(\pi)$. It follows from the monotone convergence theorem that this can be extended to all $\mu\in\calP_{\infty}(\pi)$ and $f\in L^1_{\geq 0}(\pi)$, hence by linearity it remains true for all $\mu\in\calP_{\infty}(\pi)$ and $f\in L^1(\pi)$.

We no longer consider fixed $t$. Since $(P_t)_{0\leq t}$ is a semigroup of linear operators on $\calB_b$, it defines a semigroup of linear operators on $(L^{\infty}(\pi),\lvert\lvert \cdot\rvert\rvert_{L^1(\pi)})$, hence by the density of $(L^{\infty}(\pi),\lvert\lvert \cdot\rvert\rvert_{L^1(\pi)})$ in $(L^{1}(\pi),\lvert\lvert \cdot\rvert\rvert_{L^1(\pi)})$ and \eqref{eq:cty of Pt fixed t} it defines a semigroup of linear operators on $L^1(\pi)$. Moreover these linear operators must be bounded (with $P_t$ having operator norm at most $\lambda^t$) by \eqref{eq:indNfPtf in B(pi)}

The fact that $P_t(L^1_{\geq 0}(\pi))\subseteq L^1_{\geq 0}(\pi)$ is an immediate consequence of \eqref{eq:Ptf non-neg f monotone}.
\qed

\subsection*{Proof of Proposition \ref{prop:right efn gives A2}}

We assume that $(X_t)_{0\leq t<\tau_{\partial}}$ satisfies \cite[Assumption (A1)]{Champagnat2014}, and that for some $t_1>0$ there exists a pointwise right eigenfunction $h$ for $P_{t_1}:\calB_b(\chi)\ra\calB_b(\chi)$ belonging to $\calB_b(\chi;\Rm_{>0})$. Then $h$ must have strictly positive eigenvalue, $\hat{\lambda}$ say. 

For $\mu\in\calP(\chi)$ and $n\in \Nm$ we therefore have
\[
\Pm_{\mu}(\tau_{\partial}>nt_1)=\frac{\mu(h)\hat{\lambda}^n}{\expE_{\mu}[h(X_{nt_1})\lvert \tau_{\partial}>nt_1]}.
\]
We now let $\nu\in\calP(\chi)$, $t_0>0$ and $c_0>0$ respectively be the probability measure, time and positive constant for which \cite[Assumption (A1)]{Champagnat2014} is satisfied. We fix $x\in \chi$, and apply the above with $\mu=\delta_x$ and $\mu=\nu$ to obtain
\[
\frac{\Pm_{x}(\tau_{\partial}>nt_1)}{\Pm_{\nu}(\tau_{\partial}>nt_1)}\leq \frac{h(x)\expE_{\nu}[h(X_{nt_1})\lvert \tau_{\partial}>nt_1]}{\nu(h)\expE_{x}[h(X_{nt_1})\lvert \tau_{\partial}>nt_1]}\leq \frac{\lvert\lvert h\rvert\rvert_{\infty}^2}{c_0(\nu(h))^2}\quad \text{for all}\quad nt_1\geq t_0,
\]
whence we obtain \cite[Assumption (A2)]{Champagnat2014}.
\qed

\subsection*{Proof of Lemma \ref{lem:eigenfunctions of eigenvalue lambda t difference of non-negative eigenfunctions}}

We fix $0\leq t<\infty$ and define $A:=\lambda^{-t}P_t$. We write $\phi=\phi_1-\phi_2$ whereby $\phi_1=\phi\vee 0$ and $\phi_2=(-\phi)\vee 0$. We firstly observe that
\[
(A\phi)\wedge 0=A\phi_1-(A\phi_1)\wedge (A\phi_2),\quad (-A\phi)\vee 0=A\phi_2-(A\phi_1)\wedge (A\phi_2).
\]

We note that $\lv\lv Af\rv\rv_{L^1(\pi)}=\lv\lv f\rv\rv_{L^1(\pi)}$ for all $f\geq 0$. Since $\phi$ is an $L^1(\pi)$-right eigengunction of $P_t$ of eigenvalue $1$, we have
\[
\begin{split}
\lv\lv \phi_1\rv\rv_{L^1(\pi)}+\lv\lv\phi_2\rv\rv_{L^1(\pi)}=\lv\lv \phi\rv\rv_{L^1(\pi)}=\lv\lv A\phi\rv\rv_{L^1(\pi)}\\=\lv\lv [A\phi_1-(A\phi_1)\wedge (A\phi_2)]-[A\phi_2-(A\phi_1)\wedge (A\phi_2)]\rv\rv_{L^1(\pi)}\\=\lv\lv A\phi_1-(A\phi_1)\wedge (A\phi_2)\rv\rv_{L^1(\pi)}+\lv\lv A\phi_2-(A\phi_1)\wedge (A\phi_2)\rv\rv_{L^1(\pi)}
=\lv\lv A\phi_1-A\phi_2\rv\rv_{L^1(\pi)}\\
\leq \lvert\lvert A\phi_1\rv\rv_{L^1(\pi)}+\lv\lv A\phi_2\rv\rv_{L^1(\pi)}=\lv\lv \phi_1\rv\rv_{L^1(\pi)}+\lv\lv\phi_2\rv\rv_{L^1(\pi)}.
\end{split}
\]
This implies that $(A\phi_1)\wedge (A\phi_2)=0$, so that
\[
\phi_1=\phi\vee 0=(A\phi)\vee 0=(A\phi_1-A\phi_2)\vee 0=A\phi_1\quad\text{and}\quad \phi_2=(-\phi)\vee 0=(A\phi_2-A\phi_1)\vee 0=A\phi_2,
\]
so that  $\phi_1$ and $\phi_2$ are non-negative $L^1(\pi)$-right eigenfunctions of $A$ of eigenvalue $1$, hence non-negative $L^1(\pi)$-right eigenfunctions of $P_t$ of eigenvalue $\lambda^t$, such that $\phi_1\wedge \phi_2=0$ and $\phi=\phi_1-\phi_2$.
\qed

\subsubsection*{Proof of Lemma \ref{lem:l infty limit of lower semi cts functions}}

We define the essential limit infimum of a non-negative Borel function, $g\in \calB(\chi;\Rm_{\geq 0})$, to be given by 
\begin{equation}
\essliminf_{x'\ra x} g(x'):=\lim_{r\ra 0}\essinf_{x'\in B(x,r)\setminus \{x\}}g(x'),\quad x\in \chi,
\end{equation}
where the essential infimum should be understood to mean essentially with respect to $\Lambda$.

We may take some Borel set $A$ such that $\Lambda(A^c)=0$, on which $u_n$ converges uniformly. We then define the following version of $f$,
\[
\hat{u}(x):=\begin{cases}
\lim_{n\ra \infty}u_n(x),\quad x\in A\\
0,\quad x\notin A
\end{cases}.
\]
We now define
\begin{equation}
u(x):=\essliminf_{x'\ra x}\hat{u}(x'),\quad x\in \chi.
\end{equation}
We claim that $u$ is our desired function.

We firstly observe that
\begin{equation}\label{eq:u>c is open for all c pf of lower semi cts lemma}
\{x\in \chi:u(x)>c\}\quad\text{is open for all $c\in\Rm$.}
\end{equation}
If follows, in particular, that $u$ is Borel-measurable and lower semi-continuous, the latter being equivalent to \eqref{eq:u>c is open for all c pf of lower semi cts lemma}. It is immediate by construction that $u$ is non-negative. Moreover, since $f_n\in L^{\infty}(\Lambda)$ for all $n$ and $f_n$ converges to $f$ in $L^{\infty}(\Lambda)$, $\hat{u}$ must be essentially bounded so that $u$ must be bounded.

We now check that $u=\hat{u}$ $\Lambda$-almost everywhere, so that $u$ must be a version of $f$. Since $\Lambda(A^c)=0$ and $\hat{u}_{\lvert_A}$ is lower semicontinuous, we have for all $x\in A$ that
\begin{equation}
\hat{u}(x)\leq \liminf_{\substack{x'\ra x\\ x'\in A}}\hat{u}(x')\leq \essliminf_{x'\ra x} \hat{u}(x)=u(x).
\end{equation}

Therefore $\hat{u}\leq u$ $\Lambda$-almost everywhere. We shall now establish that $\hat{u}\geq u$ $\Lambda$-almost everywhere. We fix arbitrary $0\leq a<b$, and seek to show that
\begin{equation}\label{eq: u greater than b, hat u less than a of Lebesgue measure 0}
\Lambda(S_{ab})=0\quad\text{whereby we define}\quad S_{ab}:=\{x:u(x)>b,\hat{u}(x)<a\}.
\end{equation}

We assume that $S_{ab}$ is non-empty, otherwise we are done. If $x\in S_{ab}$, then $u(x)>b$, so that there exists $r_x>0$ such that $\hat{u}(x')>a$ for $\Lambda$-almost every $x'\in B(x,r_x)$. Thus $\Lambda(S_{ab}\cap B(x,r_x))=0$ for all $x\in S_{ab}$. We see that $\{B(x,r_x)\cap S_{ab}:x\in S_{ab}\}$ forms a cover of $S_{ab}$. Since $\chi$ is a seperable metric space, so too must be $S_{ab}$, so that $S_{ab}$ must also be Lindel\"{o}f. Therefore we can take a countable subcover of  $\{B(x,r_x)\cap S_{ab}:x\in S_{ab}\}$, $\{B(x_n,r_{x_n}):n\in \Nm\}$, from which we conclude that $\Lambda(S_{ab})\leq \sum_n\Lambda(B(x_n,r_{x_n})\cap S_{ab})=0$. We have therefore established \eqref{eq: u greater than b, hat u less than a of Lebesgue measure 0}. 

Since $0\leq a<b$ is arbitrary, we have that $\hat{u}\geq u$ $\Lambda$-almost everywhere. Therefore $u=\hat{u}$ $\Lambda$-almost everywhere, so that $u$ is a version of $f$.

All that remains is to establish the maximality of $u$. We take some other bounded, non-negative, lower semicontinuous version of $f$, $\tilde{u}\in LC_b(\chi;\Rm_{\geq 0})$. Since both $\tilde{u}$ and $\hat{u}$ are versions of $f$, for all $x\in\chi$ we have that
\[
\tilde{u}(x)\leq \liminf_{x'\ra x}\tilde{u}(x')\leq \essliminf_{x'\ra x}\tilde{u}(x')=\essliminf_{x'\ra x}\hat{u}(x')=u(x).
\]
\qed

\subsection*{Proof of Lemma \ref{lem:0 contained in open convex hull lemma}}

We let $e_1,\ldots,e_d$ be unit basis vectors. We claim that  
\begin{equation}\label{eq:unit vector contained in closed convex hull lemma pf}
\text{there exists} \quad c_1>0\quad \text{such that}\quad c_1e_1\in \overline{\conv}(v_1,\ldots,v_n).
\end{equation}
By assumption, we may take $a_1,\ldots,a_n\geq 0$ such that
\[
\sum_{k=1}^na_kv_k=0,\quad \sum_{k=1}^na_k=1.
\]
We let $k_1,\ldots,k_m$ be the indices of those $a_k$ such that $a_k>0$. Since $v_{k_1},\ldots,v_{k_m}$ are linearly dependent, $m>d$, so that $v_{k_1},\ldots,v_{k_d}$ forms a basis for $\Rm^d$. Therefore there exists $b_1,\ldots,b_m\in \Rm$ such that
\[
\sum_{\ell=1}^mb_mv_{k_{\ell}}=e_1.
\]
We take $A>0$ such that $a_{k_{\ell}}+b_{\ell}>0$ for all $1\leq \ell\leq m$, and define $S:=\sum_{\ell=1}^m(a_{k_{\ell}}+b_{\ell})$. It follows that
\[
\frac{1}{S}e_1=\sum_{\ell=1}^m\frac{a_{k_{\ell}}+b_{\ell}}{S}v_{k_{\ell}}\in \overline{\conv}(v_1,\ldots,v_n).
\]
We have therefore established \eqref{eq:unit vector contained in closed convex hull lemma pf}. It follows that there exists $c^i_{\pm}$ for $1\leq i\leq d$ such that $c^i_{\pm}(\pm e_i)\in \overline{\conv}(v_1,\ldots,v_n)$. Therefore there exists $r>0$ such that
\[
B(0,r)\subseteq \overline{\conv}(\{e^i_{\pm}v_i\})\subseteq \overline{\conv}(v_1,\ldots,v_n),
\]
from which we conclude $0\in \conv(v_1,\ldots,v_n)$.

We now observe, using that $0\in \conv(v_1,\ldots,v_n)$, that for all $x\in \text{conv}(v_1,\ldots,v_n)$ there exists $\epsilon>0$ such that $(1+\epsilon)x=(1+\epsilon)x-\epsilon 0\in \conv(v_1,\ldots,v_n)$. It follows that $x\in (1-\frac{\epsilon}{1+\epsilon})[(1+\epsilon)x]+\frac{\epsilon}{1+\epsilon}B(0,r)\subseteq \conv(v_1,\ldots,v_n)$.
\qed

{\textbf{Acknowledgement:}}  This work was funded by grant 200020 196999 from the Swiss National Foundation. The author would like to thank Michel Bena{\"i}m for suggesting the author consider random diffeomorphisms, and for useful discussions with regard to these and PDMPs.

\bibliography{Library}
\bibliographystyle{plain}
\end{document}